\documentclass[10pt]{article}
\usepackage{pst-all}
\usepackage{tikz-cd}
\tikzcdset{scale cd/.style={every label/.append style={scale=#1},
    cells={nodes={scale=#1}}}}
\usetikzlibrary{calc,babel}
\tikzset{
  curve/.style={
    settings={#1},
    to path={
      (\tikztostart)
      .. controls ($(\tikztostart)!\pv{pos}!(\tikztotarget)!\pv{height}!270:(\tikztotarget)$)
      and ($(\tikztostart)!1-\pv{pos}!(\tikztotarget)!\pv{height}!270:(\tikztotarget)$)
      .. (\tikztotarget)\tikztonodes
    },
  },
  settings/.code={%
    \tikzset{quiver/.cd,#1}%
    \def\pv##1{\pgfkeysvalueof{/tikz/quiver/##1}}%
  },
  quiver/.cd,
  pos/.initial=0.35,
  height/.initial=0,
}

\usepackage{amsthm}
\usepackage{amsmath}
\usepackage{amssymb}
\usepackage{amsfonts}
\usepackage{amsxtra}
\usepackage{mathrsfs}
\usepackage{cite}
\usepackage{graphicx}
\usepackage{srcltx}
\usepackage{verbatim}
\usepackage{mathtools}
\usepackage{float}
\usepackage{yhmath}

\usepackage[all]{xy}
\usepackage{epsfig}
\usepackage{wrapfig}
\usepackage{subfig}
\usepackage{caption}
\usepackage[textwidth=6.5in,textheight=9.5in]{geometry}
\usepackage[pdfpagemode=FullScreen,bookmarks=true]{hyperref}

\theoremstyle{plain}
\newtheorem{thm}{Theorem}
\newtheorem{mainthm}{Theorem}

\numberwithin{thm}{subsection}

\newtheorem{prop}[thm]{Proposition}
\newtheorem*{prop*}{Proposition}
\newtheorem{cor}[thm]{Corollary}
\newtheorem*{cor*}{Corollary}
\newtheorem{lem}[thm]{Lemma}
\newtheorem*{lem*}{Lemma}

\theoremstyle{definition}
\newtheorem{defn}[thm]{Definition}
\newtheorem*{defn*}{Definition}
\newtheorem{exmp}[thm]{Example}
\newtheorem*{exmp*}{Example}
\newtheorem{conj}[thm]{Conjecture}
\newtheorem*{conj*}{Conjecture}

\newtheorem{rem}[thm]{Remark}
\newtheorem*{rem*}{Remark}
\newtheorem*{qst*}{Question}
\numberwithin{equation}{section}
\newcommand{\longeq}{\scalebox{3}[2]{=}}

\title{Birational Weyl Group Action on the Symplectic Groupoid and Cluster Algebras}
\author{Woojin Choi\footnote{Michigan State University, East Lansing, USA. Email: choiwoo7@msu.edu}}

\date{}
\begin{document}

\maketitle
\begin{abstract}
A. Bondal introduced the symplectic groupoid of triangular bilinear forms. This groupoid induces a Poisson structure on $\mathcal{A}_n$, the space of $n \times n$ unipotent upper-triangular matrices, governed by the classical $\mathfrak{so}(n)$ reflection equation. L. Chekhov and M. Shapiro described log-canonical coordinates on this symplectic groupoid via the $\mathcal{A}_n$-quiver.

In this paper, we introduce a birational Weyl group action on the symplectic groupoid, generated by cluster transformations associated with certain cycles of the quiver. We show that the matrix entries on $\mathcal{A}_n$ are invariant under the Weyl group action.

V. Fock and L. Chekhov defined a Poisson map $\phi_n$ from the Teichmüller space $\mathcal{T}_{g,s}$, with genus $g$ and $s \in \{1,2\}$ boundary components, to $\mathcal{A}_n$. To describe the cluster structure of $\operatorname{Im}(\phi_n)$, we aim to find a cluster Poisson reduction for $\mathcal{A}_n$. Since every element $A \in \operatorname{Im}(\phi_n)$ satisfies $\operatorname{rank}(A+A^T) \le 4$, this rank condition provides a natural criterion for our reduction. 

The solution set of this rank condition decomposes into distinct irreducible components. We show that the Weyl group acts transitively on these components, and hence the corresponding reductions are conjugate. Consequently, it suffices to determine the cluster Poisson reduction on a single component.

We also prove that, when \(n\) is even, the longest element of the Weyl group corresponds to a cluster Donaldson--Thomas transformation of the \(\mathcal A_n\)-quiver. This yields a canonical basis for \(\mathcal O(\mathcal X_{|\mathcal A_n|})\). Using this basis, we prove that the Poisson subalgebra of Weyl group invariants \(\mathcal O(\mathcal X_{|\mathcal A_n|})^{\mathbf W_n}\) is generated by the matrix entries on $\mathcal A_n$. However, the \(\mathcal A_n\)-quiver admits no reddening sequence when \(n\) is odd.

Finally, we consider the \(\Sigma_n\)-quiver, which is the frozen extension of the $\mathcal A_{n+1}$-quiver. It is used by J. Song in the cluster realization of the \(\imath\)quantum group of type \(\mathrm{AI}_n\). Using the Weyl group action, we determine the classical image of Song's embedding for odd \(n\) and identify it with a quotient of the Weyl-invariant subalgebra.
\end{abstract}

\tableofcontents

\section{Introduction}

\subsection{Background and Motivation}

Let $\mathcal A_n \subset GL_n$ be the space of unipotent upper-triangular matrices where unipotent means the diagonal entries are all units. The \textbf{symplectic groupoid of triangular bilinear forms} is the set of admissible pairs of matrices $(B, \mathbb A)$ such that
\begin{equation}
\mathcal M = \left\{(B,\mathbb A) : B \in GL_n, \text{ }\mathbb A \in \mathcal A_n,\text{ } B\mathbb A B^T \in \mathcal A_n \right\}.\end{equation}
The groupoid is equipped with the following standard maps:
\begin{equation}
\begin{aligned}
    \textbf{source map }&
    s: \mathcal{M} \to \mathcal{A}_n,\quad (B,\mathbb{A}) \longmapsto \mathbb{A},\\
    \textbf{target map }&
    t: \mathcal{M} \to \mathcal{A}_n, \quad (B,\mathbb{A}) \longmapsto B\mathbb{A}B^T,\\
    \textbf{injection }&
    e: \mathcal{A}_n \to \mathcal{M}, \quad \mathbb{A} \longmapsto (E,\mathbb{A}),\\
    \textbf{inversion }&
    i: \mathcal{M} \to \mathcal{M}, \quad (B,\mathbb{A}) \longmapsto (B^{-1},B\mathbb{A}B^T),\\
    \textbf{multiplication }&
    m: \mathcal{M}^{(2)} \to \mathcal{M}, \quad ((C,B\mathbb{A}B^T),(B,\mathbb{A})) \longmapsto (CB,\mathbb{A}).
\end{aligned}
\end{equation}

where $\mathcal M^{(2)}$ is the fibred square
\begin{equation}
\begin{tikzpicture}[x=0.75pt,y=0.75pt,yscale=-1,xscale=1]
%uncomment if require: \path (0,235); %set diagram left start at 0, and has height of 235

%Straight Lines [id:da9888091745010472] 
\draw    (271.82,93.58) -- (313.97,110.82) ;
\draw [shift={(315.82,111.58)}, rotate = 202.25] [color={rgb, 255:red, 0; green, 0; blue, 0 }  ][line width=0.75]    (10.93,-3.29) .. controls (6.95,-1.4) and (3.31,-0.3) .. (0,0) .. controls (3.31,0.3) and (6.95,1.4) .. (10.93,3.29)   ;

% Text Node
\draw (184,118) node    {$\mathcal M^{(2)}$};
% Text Node
\draw (257,90) node    {$\mathcal M$};
% Text Node
\draw (257,145) node    {$\mathcal M$};
% Text Node
\draw (328,119) node    {$\mathcal A_{n}$};
% Text Node
\draw (205,88) node [anchor=north west][inner sep=0.75pt]    {$p_{2}$};
% Text Node
\draw (211,137) node [anchor=north west][inner sep=0.75pt]    {$p_{1}$};
% Text Node
\draw (291,88) node [anchor=north west][inner sep=0.75pt]    {$s$};
% Text Node
\draw (290,137) node [anchor=north west][inner sep=0.75pt]    {$t$};
% Connection
\draw    (200,111.86) -- (241.13,96.09) ;
\draw [shift={(243,95.37)}, rotate = 159.02] [color={rgb, 255:red, 0; green, 0; blue, 0 }  ][line width=0.75]    (10.93,-3.29) .. controls (6.95,-1.4) and (3.31,-0.3) .. (0,0) .. controls (3.31,0.3) and (6.95,1.4) .. (10.93,3.29)   ;
% Connection
\draw    (200,123.92) -- (241.12,139.13) ;
\draw [shift={(243,139.82)}, rotate = 200.3] [color={rgb, 255:red, 0; green, 0; blue, 0 }  ][line width=0.75]    (10.93,-3.29) .. controls (6.95,-1.4) and (3.31,-0.3) .. (0,0) .. controls (3.31,0.3) and (6.95,1.4) .. (10.93,3.29)   ;
% Connection
\draw    (271,139.87) -- (312.12,124.81) ;
\draw [shift={(314,124.13)}, rotate = 159.89] [color={rgb, 255:red, 0; green, 0; blue, 0 }  ][line width=0.75]    (10.93,-3.29) .. controls (6.95,-1.4) and (3.31,-0.3) .. (0,0) .. controls (3.31,0.3) and (6.95,1.4) .. (10.93,3.29);
\end{tikzpicture}
\end{equation}

such that $p_1$ and $p_2$ are natural projections to the first and second components respectively. Then there is a symplectic form $\omega \in \Omega^2$ on $\mathcal M$ that satisfies the splitting condition: $m^*\omega = p_1^*\omega + p^*_2\omega$\cite{4}.

A symplectic form $\omega$ on $\mathcal{M}$ induces a Poisson bracket $\omega^{-1}$. Its pushforward $s_*(\omega^{-1})$ defines a Poisson structure on $\mathcal{A}_n$. This Poisson structure is called the \textbf{Bondal Poisson bracket} and was computed in \cite{4}. In terms of matrix entries, the Bondal Poisson bracket is given by
\begin{equation} \label{1.4}
\begin{aligned}
    \left\{a_{ik},a_{jl}\right\} &= 0\text{ for }i < k < j < l\text{ or }i < j < l < k,\\
    \left\{a_{ik},a_{jl}\right\} &= a_{ij}a_{kl} - a_{il}a_{kj}\text{ for }i < j <k <l,\\
    \left\{a_{ik},a_{kl}\right\} &= {1 \over 2}a_{ik}a_{kl} - a_{il}\text{ for }i < k < l,\\
    \left\{a_{ik},a_{jk}\right\} &= -{1 \over 2}a_{ik}a_{jk} + a_{ij}\text{ for }i < j < k,\\
    \left\{a_{ik},a_{il}\right\} &= -{1 \over 2}a_{ik}a_{il} + a_{kl}\text{ for }i < k < l.
\end{aligned}
\end{equation}

This coincides with the Nelson--Regge \cite{28}, Dubrovin \cite{39}, and Ugaglia brackets on Stokes matrices \cite{38,26}. The resulting Poisson algebra is also known as the classical $\mathfrak{so}(n)$ reflection equation algebra, which coincides with the one arising from the classical limit of the $\imath$quantum group of type $\mathrm{AI}_n$.

Moreover, a similar Poisson structure appears on the Teichmüller space $\mathcal{T}_{g,s}$. It is given by the Goldman bracket on a subset of geodesic functions associated with the Riemann surface $\Sigma_{g,s}$, where 
$g$ is the genus and $s \in \{1,2\}$ is the number of boundary components \cite{13}. Thus, we call the entries of the matrix $\mathbb{A} \in \mathcal{A}_n$ \textbf{formal geodesic functions} (Section~\ref{ch4.1}).

In \cite{1}, L. Chekhov and M. Shapiro constructed log-canonical coordinates on the $\mathcal{A}_n$-groupoid via a parametrization of the moduli space of flat $SL_n$-connections on a disk with three marked points on the boundary. Note that coordinates $\{x_i\}_{i=1}^n$ on an $n$-dimensional Poisson variety are called 
log-canonical if the Poisson bracket is given by $\{x_i,x_j\} = \lambda_{ij}x_ix_j$, where $\lambda_{ij}$ are constants.

To construct these log-canonical coordinates, they used transport matrices of a flat $SL_n$-connection on a disk with three marked points on the boundary. The moduli space of such connections is equipped with a Poisson structure derived from the Fock--Goncharov $SL_n$-quiver. The element $\mathbb{A} \in \mathcal{A}_n$ is constructed by a particular composition of transport matrices. Since this composition is generally not in $\mathcal{A}_n$, the parameters of the $SL_n$-quiver are modified to satisfy the compatibility condition. This process transforms the original $SL_n$-quiver into the \textbf{$\mathcal{A}_n$-quiver}. This quiver describes the log-canonical coordinates on the $\mathcal{A}_n$-groupoid (Section~\ref{Ch3.1}).
\begin{figure}[H]
    \centering
    \includegraphics[width=0.85\linewidth]{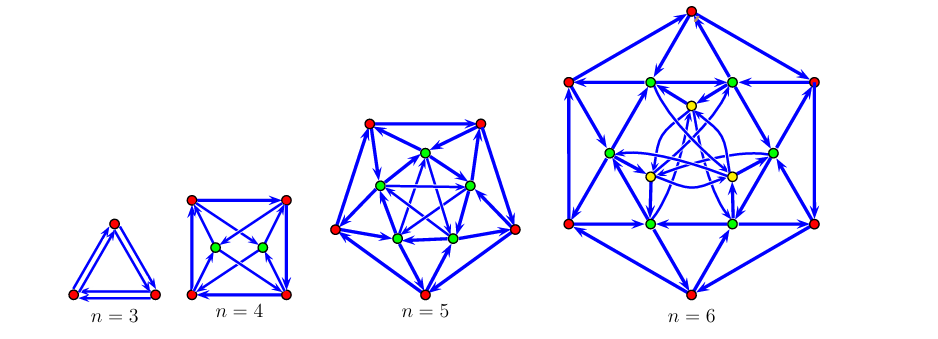}
    \caption{\textcolor{black}{\cite{1}} Examples of $\mathcal A_n$-quivers.}
    \label{Fig1}
\end{figure}

Cluster Poisson reductions have been widely studied in several contexts, including dimer models \cite{40} and the Painlevé equations \cite{41}. Such reductions enable us to derive refined, lower-dimensional cluster structures from their higher-dimensional counterparts.

As detailed in \cite{13}, there exists a Poisson map $\phi_n: \mathcal{T}_{\lfloor \frac{n-1}{2} \rfloor, \operatorname{par}(n)}
\to \mathcal{A}_n$, where $\operatorname{par}(n)=1$ for odd $n$ and $2$ otherwise. To recover the shear coordinates and the Goldman bracket of the underlying Teichmüller space, we seek to describe the cluster structure of $\operatorname{Im}(\phi_n) \subset \mathcal{A}_n$. To this end, we aim to find a cluster Poisson reduction, or equivalently a \textbf{coisotropic reduction}, for the symplectic groupoid. Since every element $\mathbb A \in \operatorname{Im}(\phi_n)$ satisfies the \textbf{rank condition} $\operatorname{rank}(\mathbb A+\mathbb A^T) \le 4$ \cite{28}, this constraint provides a natural criterion for our reduction.

In Fock--Goncharov parameters, the solution set $\mathcal{D}^n$ of the rank condition decomposes into distinct irreducible components \cite{17}. In particular, $\mathcal{D}^6$ splits into three distinct irreducible components \cite{16}. To describe the coisotropic reduction of $\mathcal{A}_6$, the authors of \cite{16} first constructed a reduction on one component. They then found sequences of cluster mutations inducing Poisson birational maps from the other components to this one. This allowed them to describe coisotropic reductions on all irreducible components.

Our goal is to generalize this approach to higher $n$. Namely, we aim to find cluster transformations that act transitively on the components of $\mathcal{D}^n$ as Poisson birational maps.

In this paper, we introduce a \textbf{birational Weyl group action} on the \(\mathcal A_n\)-quiver and denote the Weyl group by \(\mathbf W_n\). This action is analogous to the $\mathbb{Z}_2$-action on the moduli space of framed local systems studied in \cite{23}. From a combinatorial perspective, it is similar to Weyl group actions in which each reflection is realized as a specific sequence of cluster mutations along a chordless cycle of the quiver \cite{19,2}. We are indebted to M. Bershtein for this observation.

For any chordless cycle such that the number of arrows from any external vertex $v$ to the cycle equals the number of arrows from the cycle to $v$, we define the \textbf{reflection} (or \textbf{cycle mutation}) as a specific sequence of cluster mutations along the chordless cycle (Section~\ref{Ch3.3}).

The $\mathcal{A}_n$-quiver contains $\lfloor n/2 \rfloor$ main cycles with this property, as illustrated in Figure~\ref{Fig1} (where each cycle consists of vertices of the same color). When $n$ is even, every such cycle is chordless, allowing us to define the reflections naturally for all main cycles. However, when $n$ is odd, the innermost main cycle is not chordless. Consequently, we cannot directly define a reflection associated with this cycle.

To resolve this issue, we introduce the \textbf{doubled $\mathcal{A}_n$-quiver}. It is a $2:1$ quiver covering of the $\mathcal{A}_n$-quiver, with variables $\widetilde{U_{i,j}}$ and $U_{i,j}$ (Figure~\ref{Fig2}).

\begin{figure}[H]
    \centering
    \begin{tikzcd}[scale cd=0.7,column sep = tiny, row sep = small]
	&&&& \textcolor{rgb,255:red,255;green,0;blue,0}{{{U_{1,3}}}} \\
	&&& \textcolor{rgb,255:red,0;green,255;blue,0}{{{U_{2,4}}}} && \textcolor{rgb,255:red,255;green,0;blue,0}{{{U_{1,4}}}} &&&&&&&& \textcolor{rgb,255:red,255;green,0;blue,0}{{{U_{1,3}}}} \\
	&& \textcolor{rgb,255:red,0;green,0;blue,255}{{{\widetilde{U_{2,2}}}}} && \textcolor{rgb,255:red,0;green,255;blue,0}{{{U_{2,5}}}} && \textcolor{rgb,255:red,255;green,0;blue,0}{{{U_{1,5}}}} &&&&&& \textcolor{rgb,255:red,0;green,0;blue,255}{{{U_{2,2}}}} && \textcolor{rgb,255:red,255;green,0;blue,0}{{{U_{1,4}}}} \\
	& \textcolor{rgb,255:red,255;green,128;blue,0}{{{\widetilde{U_{1,2}}}}} && \textcolor{rgb,255:red,0;green,0;blue,255}{{{\widetilde{U_{2,3}}}}} && \textcolor{rgb,255:red,0;green,255;blue,0}{{{U_{2,1}}}} && \textcolor{rgb,255:red,255;green,0;blue,0}{{{U_{1,1}}}} &&&& \textcolor{rgb,255:red,255;green,128;blue,0}{{{\widetilde{U_{1,2}}}}} && \textcolor{rgb,255:red,0;green,0;blue,255}{{{\widetilde{U_{2,1}}}}} && \textcolor{rgb,255:red,255;green,0;blue,0}{{{U_{1,1}}}} \\
	{} && \textcolor{rgb,255:red,255;green,128;blue,0}{{{\widetilde{U_{1,3}}}}} && \textcolor{rgb,255:red,0;green,0;blue,255}{{{\widetilde{U_{2,4}}}}} && \textcolor{rgb,255:red,0;green,255;blue,0}{{{U_{2,2}}}} && \textcolor{rgb,255:red,255;green,0;blue,0}{{{U_{1,2}}}} &&&& \textcolor{rgb,255:red,255;green,128;blue,0}{{{\widetilde{U_{1,3}}}}} && \textcolor{rgb,255:red,0;green,0;blue,255}{{{\widetilde{U_{2,2}}}}} && \textcolor{rgb,255:red,255;green,0;blue,0}{{{U_{1,2}}}} && {} \\
	&&& \textcolor{rgb,255:red,255;green,128;blue,0}{{{\widetilde{U_{1,4}}}}} && \textcolor{rgb,255:red,0;green,0;blue,255}{{{\widetilde{U_{2,5}}}}} && \textcolor{rgb,255:red,0;green,255;blue,0}{{{U_{2,3}}}} && \textcolor{rgb,255:red,255;green,0;blue,0}{{{U_{1,3}}}} &&&& \textcolor{rgb,255:red,255;green,128;blue,0}{{{\widetilde{U_{1,4}}}}} && \textcolor{rgb,255:red,0;green,0;blue,255}{{{U_{2,1}}}} && \textcolor{rgb,255:red,255;green,0;blue,0}{{{U_{1,3}}}} \\
	&&&& \textcolor{rgb,255:red,255;green,128;blue,0}{{{\widetilde{U_{1,5}}}}} && \textcolor{rgb,255:red,0;green,0;blue,255}{{{\widetilde{U_{2,1}}}}} && \textcolor{rgb,255:red,0;green,255;blue,0}{{{U_{2,4}}}} &&&&&& \textcolor{rgb,255:red,255;green,128;blue,0}{{{\widetilde{U_{1,1}}}}} && \textcolor{rgb,255:red,0;green,0;blue,255}{{{U_{2,2}}}} \\
	&&&&& \textcolor{rgb,255:red,255;green,128;blue,0}{{{\widetilde{U_{1,1}}}}} && \textcolor{rgb,255:red,0;green,0;blue,255}{{{\widetilde{U_{2,2}}}}} &&&&&&&& \textcolor{rgb,255:red,255;green,128;blue,0}{{{\widetilde{U_{1,2}}}}} \\
	&&&&&& \textcolor{rgb,255:red,255;green,128;blue,0}{{{\widetilde{U_{1,2}}}}}
	\arrow[from=1-5, to=2-6]
	\arrow[dashed, from=2-4, to=1-5]
	\arrow[from=2-4, to=3-5]
	\arrow[from=2-6, to=2-4]
	\arrow[from=2-6, to=3-7]
	\arrow[from=2-14, to=3-15]
	\arrow[dashed, from=3-3, to=2-4]
	\arrow[from=3-3, to=4-4]
	\arrow[from=3-5, to=2-6]
	\arrow[from=3-5, to=3-3]
	\arrow[from=3-5, to=4-6]
	\arrow[from=3-7, to=3-5]
	\arrow[from=3-7, to=4-8]
	\arrow[dashed, from=3-13, to=2-14]
	\arrow[from=3-13, to=4-14]
	\arrow[from=3-15, to=3-13]
	\arrow[from=3-15, to=4-16]
	\arrow[dashed, from=4-2, to=3-3]
	\arrow[from=4-2, to=5-3]
	\arrow[from=4-4, to=3-5]
	\arrow[from=4-4, to=4-2]
	\arrow[from=4-4, to=5-5]
	\arrow[from=4-6, to=3-7]
	\arrow[from=4-6, to=4-4]
	\arrow[from=4-6, to=5-7]
	\arrow[from=4-8, to=4-6]
	\arrow[from=4-8, to=5-9]
	\arrow[dashed, from=4-12, to=3-13]
	\arrow[from=4-12, to=5-13]
	\arrow[from=4-14, to=3-15]
	\arrow[from=4-14, to=4-12]
	\arrow[from=4-14, to=5-15]
	\arrow[from=4-16, to=4-14]
	\arrow[from=4-16, to=5-17]
	\arrow[from=5-3, to=4-4]
	\arrow[dotted, no head, from=5-3, to=5-1]
	\arrow[from=5-3, to=6-4]
	\arrow[from=5-5, to=4-6]
	\arrow[from=5-5, to=5-3]
	\arrow[from=5-5, to=6-6]
	\arrow[from=5-7, to=4-8]
	\arrow[from=5-7, to=5-5]
	\arrow[from=5-7, to=6-8]
	\arrow[from=5-9, to=5-7]
	\arrow[dotted, no head, from=5-9, to=5-13]
	\arrow[from=5-9, to=6-10]
	\arrow[from=5-13, to=4-14]
	\arrow[from=5-13, to=6-14]
	\arrow[from=5-15, to=4-16]
	\arrow[from=5-15, to=5-13]
	\arrow[from=5-15, to=6-16]
	\arrow[from=5-17, to=5-15]
	\arrow[dotted, no head, from=5-17, to=5-19]
	\arrow[from=5-17, to=6-18]
	\arrow[from=6-4, to=5-5]
	\arrow[from=6-4, to=7-5]
	\arrow[from=6-6, to=5-7]
	\arrow[from=6-6, to=6-4]
	\arrow[from=6-6, to=7-7]
	\arrow[from=6-8, to=5-9]
	\arrow[from=6-8, to=6-6]
	\arrow[from=6-8, to=7-9]
	\arrow[from=6-10, to=6-8]
	\arrow[from=6-14, to=5-15]
	\arrow[from=6-14, to=7-15]
	\arrow[from=6-16, to=5-17]
	\arrow[from=6-16, to=6-14]
	\arrow[from=6-16, to=7-17]
	\arrow[from=6-18, to=6-16]
	\arrow[from=7-5, to=6-6]
	\arrow[from=7-5, to=8-6]
	\arrow[from=7-7, to=6-8]
	\arrow[from=7-7, to=7-5]
	\arrow[from=7-7, to=8-8]
	\arrow[dashed, from=7-9, to=6-10]
	\arrow[from=7-9, to=7-7]
	\arrow[from=7-15, to=6-16]
	\arrow[from=7-15, to=8-16]
	\arrow[dashed, from=7-17, to=6-18]
	\arrow[from=7-17, to=7-15]
	\arrow[from=8-6, to=7-7]
	\arrow[from=8-6, to=9-7]
	\arrow[dashed, from=8-8, to=7-9]
	\arrow[from=8-8, to=8-6]
	\arrow[dashed, from=8-16, to=7-17]
	\arrow[dashed, from=9-7, to=8-8]
\end{tikzcd}
    \caption{Examples of doubled $\mathcal{A}_n$-quivers. Vertices of the same color represent a main cycle.}
    \label{Fig2}
\end{figure}

Since every main cycle in the doubled quiver is chordless, we can naturally associate a reflection with each cycle. We then show that certain reflections descend to the standard quiver via the natural projection from the doubled quiver. After verifying that this descent is well-defined, we demonstrate that this construction induces a birational Weyl group action on the standard $\mathcal{A}_n$-quiver, including a reflection associated with the innermost cycle in the odd case (Section~\ref{Ch3.4}).

\subsection{Main Results} 
We first consider the relationship between the birational Weyl group action and the formal geodesic functions. In higher Teichm\"uller theory, the Fock--Goncharov--Shen Weyl group acts on the moduli space of framed local systems on an oriented surface by permuting the eigenvalues of the monodromy operators and hence preserving their traces. Motivated by this phenomenon, we prove that the birational Weyl group action constructed in this paper preserves the formal geodesic functions, which are the matrix entries of \(\mathbb A\in\mathcal A_n\) (Theorem~\ref{thm422}). 

We apply this theorem to the coisotropic reduction problem associated with the rank condition \( \operatorname{rank}(\mathbb A+\mathbb A^T)\leq 4. \) Recall that the solution locus \(\mathcal D^n\) of this rank condition decomposes into several irreducible components. We show that these components are conjugate under the birational Weyl group action. More precisely, Theorem~\ref{thm422} implies that \(\operatorname{rank}(\mathbb A+\mathbb A^T)\) is invariant under the Weyl group action (Proposition~\ref{prop516}). Since the irreducible components of \(\mathcal D^n\) are indexed by the Casimirs~\cite{17}, it suffices to prove that the Weyl group acts transitively on the corresponding Casimir data. This yields the following result. 

\begin{mainthm}[Corollary~\ref{cor523}] \label{thm:mainA} The Weyl group acts transitively on the irreducible components of the solution locus \(\mathcal D^n\) of the rank condition. Consequently, the coisotropic reductions associated with distinct irreducible components are conjugate under the Weyl group action. In particular, it suffices to determine the coisotropic reduction on a single irreducible component. \end{mainthm}

In~\cite{33}, M.~Kontsevich and Y.~Soibelman introduced Donaldson--Thomas (DT) invariants for three-dimensional Calabi--Yau categories equipped with stability conditions. A cluster variety has an associated three-dimensional Calabi--Yau category, and the corresponding DT transformation encodes these invariants as a formal automorphism of the cluster variety. In general, the DT transformation need not be a cluster transformation. 

A cluster transformation \(\omega\) is called a \textbf{cluster DT transformation} if it preserves the underlying quiver and its associated \(C\)-matrix is the negative identity matrix: \( C^\omega=-I \) (Section~\ref{ch6.1}). Whenever such a cluster transformation exists, it coincides with the Kontsevich--Soibelman DT transformation~\cite{34}. 

We prove that, when \(n\) is even, the longest element of the Weyl group corresponds to a cluster DT transformation on the cluster variety associated with the $\mathcal A_n$-quiver. The existence of this transformation yields a canonical basis for \(\mathcal O(\mathcal X_{|\mathcal A_n|})\), the ring of regular functions on \(\mathcal X_{|\mathcal A_n|}\)~\cite{21}. In contrast, we prove that the \(\mathcal A_n\)-quiver admits no reddening sequence when \(n\) is odd. Consequently, the cluster Poisson variety \(\mathcal X_{|\mathcal A_n|}\) admits no cluster DT transformation when \(n\) is odd. Note that a reddening sequence is a sequence of mutations whose associated \(C\)-matrix has only nonpositive entries.  

\begin{mainthm}[Theorems~\ref{thm621} and~\ref{thm626}] \label{thm:mainB} When \(n\) is even, the longest element of the Weyl group corresponds to a cluster DT transformation on the cluster variety associated with the $\mathcal A_n$-quiver. Consequently, \(\mathcal O(\mathcal X_{|\mathcal A_n|})\) admits a canonical basis. In contrast, when \(n\) is odd, the \(\mathcal A_n\)-quiver admits no reddening sequence and hence no cluster DT transformation. \end{mainthm}

A natural continuation of Theorem~\ref{thm422} is to determine whether the formal geodesic functions generate the entire algebra of Weyl group invariants. A useful precedent comes from the quantum group setting. Consider the quantum cluster algebra associated with framed \(PGL_{n+1}\)-local systems on a disk with two marked boundary points and one puncture. In~\cite{20}, an embedding of \(U_q(\mathfrak{sl}_{n+1})\) into this quantum cluster algebra was constructed. The cluster algebra admits the Fock--Goncharov--Shen Weyl group action, and the image of \(U_q(\mathfrak{sl}_{n+1})\) is invariant under this action. It was then proved in~\cite{3} that this embedding induces an isomorphism onto a quotient of the Weyl-invariant algebra. 

Motivated by this framework, we establish a classical analogue in our setting. When $n$ is even, the cluster DT transformation described above gives rise to a canonical theta basis of $\mathcal O(\mathcal X_{|\mathcal A_n|})$. Let $\mathcal O(\mathcal A_n)$ denote the Poisson algebra generated by the formal geodesic functions. Using the theta basis, we prove the inclusion \( \mathcal O(\mathcal X_{|\mathcal A_n|})^{\mathbf W_n} \subseteq \mathcal O(\mathcal A_n). \) Together with Theorem~\ref{thm422}, this yields \( \mathcal O(\mathcal X_{|\mathcal A_n|})^{\mathbf W_n} = \mathcal O(\mathcal A_n) \cap \mathcal O(\mathcal X_{|\mathcal A_n|}). \) In general, the individual formal geodesic functions need not belong to $\mathcal O(\mathcal X_{|\mathcal A_n|})$: by Theorem~\ref{thm416}, each of them has the form $M^{1/2}L$, where $M$ is a monomial and $L$ is a Laurent polynomial in every cluster.

\begin{mainthm}[Theorem~\ref{thm6412}]\label{thm:mainC} Assume that $n$ is even. Then \( \mathcal O(\mathcal X_{|\mathcal A_n|})^{\mathbf W_n} = \mathcal O(\mathcal A_n) \cap \mathcal O(\mathcal X_{|\mathcal A_n|}). \) In particular, every Weyl-invariant regular function is contained in the Poisson algebra generated by the formal geodesic functions, namely, the matrix entries of $\mathcal A_n$. \end{mainthm}

Originating in G.~Letzter's seminal work on quantum symmetric pairs~\cite{29}, \(\imath\)quantum groups provide a broad generalization of standard quantum groups. Their theory has led to several fundamental constructions, including universal \(K\)-matrices, \(\imath\)-Schur--Weyl duality, and \(\imath\)canonical bases~\cite{22,31}. 

Let \((\mathfrak g,\mathfrak g^\theta)\) be a symmetric pair, where \(\mathfrak g\) is a semisimple Lie algebra and \(\mathfrak g^\theta\) is the fixed-point subalgebra of an involutive automorphism \(\theta\). The associated quantum symmetric pair is denoted by \((U,U^\imath)\), where \(U\) is the Drinfeld--Jimbo quantum group associated with \(\mathfrak g\), and \(U^\imath\subset U\) is the corresponding coideal \(\imath\)quantum group. From this perspective, the Schrader--Shapiro construction for \(U_q(\mathfrak{sl}_{n+1})\) may be interpreted as a cluster realization of the \(\imath\)quantum group associated with the symmetric pair \( (\mathfrak{sl}_{n+1}\otimes\mathfrak{sl}_{n+1}, \mathfrak{sl}_{n+1}). \) This interpretation naturally motivates the study of cluster realizations of \(\imath\)quantum groups associated with other symmetric pairs. 

J.~Song constructed a quantum cluster realization of the \(\imath\)quantum group of type \(\mathrm{AI}_n\), corresponding to the symmetric pair \((\mathfrak{sl}_{n+1},\mathfrak{so}_{n+1})\)~\cite{37}. More precisely, the particular Casimir elements \(M_n-1\) and \(k_i^{-1}+1\) of the quantum cluster algebra define a central reduction ideal \(\mathcal I_q\), and Song's construction induces an injective algebra homomorphism \( \iota_q\colon U^\imath \hookrightarrow \mathcal O_q(\mathcal X_{|\Sigma_n|})/\mathcal I_q. \) Although this embedding is known to be injective, its precise image has remained undetermined. For odd \(n\), we resolve the classical analogue of this problem using the Weyl-invariant algebra techniques developed above. 

The \(\Sigma_n\)-quiver is obtained by adjoining frozen vertices along the first main cycle of the \(\mathcal A_{n+1}\)-quiver; see Section~\ref{ch7.2}. Consequently, the birational Weyl group action constructed in this paper extends naturally to the associated cluster Poisson variety. We denote the resulting Weyl group by \(\mathbf W_n^\imath\). Let \(\mathcal I\) be the classical central reduction ideal defined by the corresponding Casimir elements \(M_n-1\) and \(k_i^{-1}+1\). This is the classical counterpart of the central reduction used in~\cite{37}, and the classical limit of Song's construction induces an injective Poisson algebra homomorphism \( \iota\colon U_n^\imath \hookrightarrow \mathcal O(\mathcal X_{|\Sigma_n|})/\mathcal I. \) 

Using the theta basis and the minimal-multidegree elimination procedure developed in the proof of Theorem~\ref{thm6412}, we determine its image explicitly. Namely, \[ \iota(U_n^\imath) = \mathcal O(\mathcal X_{|\Sigma_n|})^{\mathbf W_n^\imath}/\mathcal I^{\mathbf W_n^\imath}. \] 

Thus, after the central reduction, the Weyl-invariant algebra gives the precise image of the classical limit of Song's embedding. 

\begin{mainthm}[Theorem~\ref{thm7311}]\label{thm:mainD} Assume that \(n\) is odd. Let \(U_n^\imath\) denote the classical limit of the \(\imath\)quantum group of type \(\mathrm{AI}_n\), and let \(\mathcal I\) be the central reduction ideal defined by the particular Casimir elements \(M_n-1\) and \(k_i^{-1}+1\). Then the classical limit of Song's cluster realization induces a Poisson algebra isomorphism \[ U_n^\imath \xrightarrow{\sim} \mathcal O(\mathcal X_{|\Sigma_n|})^{\mathbf W_n^\imath}/\mathcal I^{\mathbf W_n^\imath}. \]
\end{mainthm}

\subsection{Organization of the Paper} The paper is organized as follows. \begin{itemize} 

\item \textbf{Section 2.} We review the necessary preliminaries on cluster algebras, cluster Poisson varieties, and the Goldman bracket on Teichm\"uller spaces. 

\item \textbf{Section 3.} We review the log-canonical coordinates on the symplectic groupoid and the \(\mathcal A_n\)-quiver constructed in~\cite{1}. We then introduce the doubled \(\mathcal A_n\)-quiver and construct the birational Weyl group action on the standard \(\mathcal A_n\)-quiver. 

\item \textbf{Section 4.} We introduce the formal geodesic functions and prove that they are invariant under the birational Weyl group action (\textbf{Theorem~\ref{thm422}}). 

\item \textbf{Section 5.} We study the rank condition associated with the coisotropic reduction of the \(\mathcal A_n\)-groupoid. We prove that the Weyl group acts transitively on the irreducible components of the corresponding solution locus. Consequently, it suffices to construct the coisotropic reduction on a single irreducible component (\textbf{Corollary~\ref{cor523}}). 

\item \textbf{Section 6.} We introduce the tropicalization of the Weyl group action and study its relationship with the minimal multidegrees of formal geodesic functions. We prove that, when \(n\) is even, the longest element of the Weyl group is realized by a cluster Donaldson--Thomas transformation (\textbf{Theorem~\ref{thm621}}). Consequently, the corresponding cluster Poisson algebra admits a canonical theta basis. Using this basis, we prove that \[ \mathcal O(\mathcal X_{|\mathcal A_n|})^{\mathbf W_n} = \mathcal O(\mathcal A_n) \cap \mathcal O(\mathcal X_{|\mathcal A_n|}), \] so that the Weyl-invariant regular functions are precisely the square-root-free elements of the Poisson algebra generated by the formal geodesic functions (\textbf{Theorem~\ref{thm6412}}). In contrast, when \(n\) is odd, we prove that the \(\mathcal A_n\)-quiver admits no reddening sequence (\textbf{Theorem~\ref{thm626}}). 

\item \textbf{Section 7.} We introduce a frozen extension along the first main cycle of the \(\mathcal A_{n+1}\)-quiver and relate it to the cluster realization of the \(\imath\)quantum group of type \(\mathrm{AI}_n\). We then extend the birational Weyl group action to this frozen extension and denote the resulting Weyl group by \(\mathbf W_n^\imath\). 

Using this action, we determine the image of the cluster embedding in the classical limit for odd \(n\). More precisely, in \textbf{Theorem~\ref{thm7311}}, we prove that this image is Poisson isomorphic to \[ \mathcal O(\mathcal X_{|\Sigma_n|})^{\mathbf W_n^\imath}/\mathcal I^{\mathbf W_n^\imath} .\]

\item \textbf{Section 8.} We conclude the paper and discuss several directions for future research. \end{itemize}

\section{Preliminaries}
\subsection{Cluster Algebras}

A seed $\Sigma$ is a triple $(I, F, \epsilon)$, where $I$ is a finite set, $F$ is a subset of $I$, and $\epsilon = (\epsilon_{ij})_{i,j \in I}$ is a skew-symmetric $\frac{1}{2}\mathbb{Z}$-valued matrix such that $\epsilon_{ij} \in \mathbb{Z}$ unless $i, j \in F$. We refer to the matrix $\epsilon$ as the exchange matrix associated with the seed $\Sigma$. 

We associate two algebraic tori $\mathcal{X}_{\Sigma} = (\mathbb{C}^*)^{|I|}$ and $\mathcal{A}_{\Sigma} = (\mathbb{C}^*)^{|I|}$ equipped with coordinates $(X_i)_{i\in I}$ and $(A_i)_{i\in I}$ respectively to the seed $\Sigma$. 

The combinatorial data of the seed $\Sigma$ can be expressed via a quiver $Q_{\Sigma}$ with vertices labeled by elements of $I$ and exchange matrix $\epsilon$. Note that $Q_{\Sigma}$ does not allow oriented 1-cycles (loops)
or oriented 2-cycles. A vertex is said to be frozen if its index is in $F$. The mutation $\mu_k$ in the direction $k \in I-F$ over the quiver is defined in the following steps:
\begin{enumerate}
    \item Reverse all arrows incident to the vertex $k$;
    \item For each pair of arrows $i \to k$ and $k \to j$, draw an arrow $i \to j$;
    \item Delete all 2-cycles and loops.
\end{enumerate}

According to the mutation rules above, a mutation in the direction $k$ changes the matrix $\epsilon$ into the matrix

\begin{equation}
\epsilon_{ij}' =
\begin{cases}
-\epsilon_{ij} & \mbox{if }k \in \left\{i,j\right\}  \\
\epsilon_{ij}+ {|\epsilon_{ik}|\epsilon_{kj} + \epsilon_{ik}|\epsilon_{kj}| \over 2} & \mbox{otherwise.}
\end{cases}
\end{equation}

The mutation $\mu_k$ is intertwined with birational isomorphisms on two cluster tori defined by the formulas

\begin{equation}
(\mu_{k}^\mathcal X)^*(X_i') =
\begin{cases}
X_i^{-1} & \mbox{if }i = k  \\
X_i\left(1+X_k^{-\text{sgn}(\epsilon_{ki})}\right)^{-\epsilon_{ki}} & \mbox{if }i \ne k
\end{cases}
\end{equation}

and

\begin{equation} \label{2.3}
(\mu_{k}^\mathcal A)^*(A_{i}') =
\begin{cases}
{\prod_{j|\epsilon_{kj}>0} A_j^{\epsilon_{kj}} + \prod_{j|\epsilon_{kj}<0} A_j^{-\epsilon_{kj}} \over A_i} & \mbox{if }i = k  \\
A_i & \mbox{if }i \ne k
\end{cases}
\end{equation}

where $(X'_i)_{i \in I}$ and $(A'_i)_{i \in I}$ are coordinates of $\mathcal X_{\Sigma'}$ and $\mathcal A_{\Sigma'}$ respectively. Next, we define a cluster transformation and the notion of mutation equivalent.

\begin{defn}
    Two seeds $\Sigma$ and $\Sigma'$ are \textit{mutation equivalent} if there is a sequence of mutations that transforms one seed to another seed.
    A transformation obtained by composing mutations and permutations of the index set $I$, equipped with the corresponding birational isomorphisms, is called a \textit{cluster transformation}.
\end{defn}

\begin{exmp}\textit{(An example of a seed, an associated quiver, and mutations)}

Let the seed $\Sigma$ be a $(I,\emptyset,\epsilon)$ where
$$\epsilon = \begin{pmatrix}
0 & 1 & -1 \\
-1 & 0 & 1 \\
1 & -1 & 0
\end{pmatrix} \text{ and } I = \left\{1,2,3\right\}.$$

Then its associated quiver is as follows:
$$
\begin{tikzcd}
	& {\text{1}} \\
	{\text{3}} && {\text{2}}
	\arrow[from=1-2, to=2-3]
	\arrow[from=2-1, to=1-2]
	\arrow[from=2-3, to=2-1]
\end{tikzcd}
$$

The following quivers are obtained by applying $\mu_1$, $\mu_2$ and $\mu_3$ respectively:
$$
\begin{tikzcd}
	& {\text{1}} &&&& {\text{1}} &&&& {\text{1}} \\
	{\text{3}} && {\text{2}} && {\text{3}} && {\text{2}} && {\text{3}} && {\text{2}}
	\arrow[from=1-2, to=2-1]
	\arrow[from=1-10, to=2-9]
	\arrow[from=2-3, to=1-2]
	\arrow[from=2-5, to=2-7]
	\arrow[from=2-7, to=1-6]
	\arrow[from=2-9, to=2-11]
\end{tikzcd}
$$
\end{exmp}

Under the mutation $\mu_1$, initial tori $(X_1,X_2,X_3) \in (\mathbb C^*)^3$ and $(A_1,A_2,A_3) \in (\mathbb C^*)^3$ are transformed into $(X'_1,X'_2,X'_3) \in (\mathbb C^*)^3$ and $(A'_1,A'_2,A'_3) \in (\mathbb C^*)^3$ where $$(\mu_{1}^\mathcal X)^*(X'_1) = X_1^{-1},\text{ }(\mu_{1}^\mathcal X)^*(X'_2) = X_2(1+X_1^{-1})^{-1},\text{ } (\mu_{1}^\mathcal X)^*(X'_3) = X_3(1+X_1)$$
and
$$(\mu_{1}^\mathcal A)^*(A'_1) = {(A_2+A_3) \over A_1},\text{ }(\mu_{1}^\mathcal A)^*(A'_2) = A_2,\text{ } (\mu_{1}^\mathcal A)^*(A'_3) = A_3.$$

\begin{defn}\textit{(Cluster variety and its ring of regular functions)} 

The \textit{cluster $\mathcal X$-variety} $\mathcal X_{|\Sigma|}$ is a scheme obtained by gluing $\mathcal X$-tori for all seeds which are mutation equivalent to the seed $\Sigma$. $\mathcal{O}(\mathcal X_{|\Sigma|})$ is the ring of regular functions on the scheme $\mathcal X_{|\Sigma|}$. An element $f \in \mathcal{O}(\mathcal X_{|\Sigma|})$ is called a \textit{universal Laurent polynomial}. Note that such an element remains a Laurent polynomial under any finite sequence of cluster mutations.

The \textit{cluster $\mathcal A$-variety} $\mathcal A_{|\Sigma|}$ is a scheme obtained by gluing $\mathcal A$-tori for all seeds which are mutation equivalent to the seed $\Sigma$. We refer to the ring of regular functions on $\mathcal A_{|\Sigma|}$ as the upper cluster algebra and denote it by $\mathcal{O}(\mathcal A_{|\Sigma|})$. The element $g \in \mathcal{O}(\mathcal A_{|\Sigma|})$ is also called a universal Laurent polynomial and remains a Laurent polynomial under any finite sequence of cluster mutations.
\end{defn}

\begin{defn}\textit{(Classical ensemble map)} \label{defn214}
    
    Assume the matrix $\epsilon$ is an integer matrix. Then there is a regular map $p_{\Sigma}: \mathcal A_{\Sigma} \to \mathcal X_{\Sigma}$ which is called the \textit{classical ensemble map}, defined by 
    $$p_{\Sigma}^*(X_i) = \prod_{k \in I}A_k^{\epsilon_{ki}}.$$
\end{defn}

It is an isomorphism of algebraic tori if and only if $\det \epsilon = \pm 1$. Cluster mutations on $\mathcal A$-tori and $\mathcal X$-tori are compatible with each other via the ensemble map:
    \begin{equation} \label{2.4}
    \mu_k^\mathcal X \circ p_{\Sigma} = p_{\mu_k(\Sigma)}\circ \mu_{k}^\mathcal A.
    \end{equation}

Let $Q^{j}$ be the quiver obtained from $Q$ by freezing all vertices except a vertex $j \in J$. Let $\Sigma^j$ be its associated seed. The \textbf{upper bound} $\mathcal U(\mathcal A_\Sigma)$ is defined by
\begin{equation}
    \mathcal U(\mathcal A_\Sigma) = \bigcap_{j \in J}\mathcal{O}(\mathcal A_{|\Sigma^j|}),
\end{equation}
which includes $\mathcal{O}(\mathcal A_{|\Sigma|})$ because an element in the upper cluster algebra should remain a Laurent polynomial under 1-step mutations.

\begin{thm} \label{thm215}
\textcolor{black}{\cite[Corollary 1.9]{5}} If $\epsilon$ has full rank, the upper bound $\mathcal U(\mathcal A_\Sigma)$ is independent of the choice of a seed mutation equivalent to $\Sigma$, and hence is equal to $\mathcal{O}(\mathcal A_{|\Sigma|}).$
\end{thm}

From the theorem above, we derive the following corollary:

\begin{cor} \label{cor216} \cite[Lemma 3.8]{35}
    It is enough to check 1-step mutations to determine whether $f \in \mathcal{O}(\mathcal X_{|\Sigma|})$ or not. In other words,
    $$\mathcal{O}(\mathcal X_{|\Sigma|}) = \bigcap_{j \in J}\mathcal{O}(\mathcal X_{|\Sigma^j|}).$$
\end{cor}
\begin{proof}
    Consider the \textit{framing} $\hat{Q}$ of a quiver $Q$, obtained by adding a frozen vertex $\hat k$ for every vertex $k$ of $Q$ along with an arrow $\hat k \leftarrow k$. Then the determinant of the exchange matrix of $\hat Q$ is $1$. Hence, by Theorem~\ref{thm215}, we have the following chain of equations:
   $$
   \mathcal{O}(\mathcal X_{|\hat\Sigma|}) = \mathcal{O}(\mathcal A_{|\hat\Sigma|}) = \bigcap_{j \in J}\mathcal{O}(\mathcal A_{|(\hat\Sigma)^j|}) = \bigcap_{j \in J}\mathcal{O}(\mathcal X_{|(\hat\Sigma)^j|}).
   $$
   Note that $\mathcal{O}(\mathcal X_{|\Sigma|}) \simeq \mathcal{O}(\mathcal A_{|\Sigma|})$ when $\det \epsilon = \pm 1.$ As $\mathcal{O}(\mathcal X_{|\Sigma|})$ is a subring of $\mathcal{O}(\mathcal X_{|\hat \Sigma|})$ and $\mathcal{O}(X_{|\Sigma|}) \cap \mathcal{O}(X_{|(\hat\Sigma)^j|}) = \mathcal{O}(\mathcal X_{|\Sigma^j|})$, the proof is complete. \end{proof}

\begin{exmp}
    \textit{(Examples of universal Laurent polynomials)}

    Let seed $\Sigma$ be a $(I,\emptyset,\epsilon)$ where
$$\epsilon = \begin{pmatrix}
0 & 1 & 0 & -1 & 1 & -1 \\
-1 & 0 & 1 & 0 & -1 & 1\\
0 & -1 & 0 & 1 & 1 & -1 \\
1 & 0 & -1 & 0 & -1 & 1 \\
-1 & 1 & -1 & 1 & 0 & 0\\
1 & -1 & 1 & -1 & 0 & 0
\end{pmatrix} \text{ and } I = \left\{1,2,3,4,5,6\right\}.$$

Then the quiver is as follows:
$$
\begin{tikzcd}
	{\text{1}} &&&& {\text{2}} \\
	\\
	& {\text{6}} && {\text{5}} \\
	\\
	{\text{4}} &&&& {\text{3}}
	\arrow[from=1-1, to=1-5]
	\arrow[from=1-1, to=3-4]
	\arrow[from=1-5, to=3-2]
	\arrow[from=1-5, to=5-5]
	\arrow[from=3-2, to=1-1]
	\arrow[from=3-2, to=5-5]
	\arrow[from=3-4, to=1-5]
	\arrow[from=3-4, to=5-1]
	\arrow[from=5-1, to=1-1]
	\arrow[from=5-1, to=3-2]
	\arrow[from=5-5, to=3-4]
	\arrow[from=5-5, to=5-1]
\end{tikzcd}
$$
Consider elements $\mathcal C_1 = X_1X_2X_3X_4$, $\mathcal C_2 = X_5X_6$, and $\beta_1 = (X_1X_5X_2)\left(1+ {1 \over X_1}+{1 \over X_1X_5}+{1 \over X_1X_5X_2}\right)^2.$ Under a mutation $\mu_1:\Sigma' \to \Sigma$ on a vertex $1,$ we have 
$$(\mu_1^\mathcal X)^* (\beta_1) = {(X'_1X_2'X_5' + X_1'X_2' + X_1' + 1)^2 \over (X'_1X_2'X_5')},$$
$$(\mu_1^\mathcal X)^* (\mathcal C_1) = X'_2X'_3X_4',$$and
$$(\mu_1^\mathcal X)^* (\mathcal C_2) = X_1'X_5'X_6'$$
which are still Laurent polynomials in the cluster variables $X_i'$. By similar calculations, $\beta_1$, $\mathcal C_1$, and $\mathcal C_2$ remain Laurent polynomials after every one-step mutation. Thus, by Corollary~\ref{cor216}, they belong to $\mathcal{O}(\mathcal X_{|\Sigma|})$.
\end{exmp}

We now introduce the cluster modular group. Our birational Weyl group action corresponds to specific elements of this group.

\begin{defn}\label{defn218}
The \textit{cluster modular group} $\Gamma_{|\Sigma|}$ is the group generated by cluster transformations of the seed $\Sigma$ that preserve its quiver. An element of the cluster modular group is called a \textit{cluster modular transformation}.
\end{defn}

\subsection{Poisson Structure on Cluster Algebras}

\begin{defn}
    A \textit{Poisson algebra} $\mathcal F$ is a commutative algebra equipped with a Poisson bracket $\mathcal F \times \mathcal F \to \mathcal F$ that is a skew-symmetric bilinear map satisfying
    
\begin{enumerate}
    \item $\left\{fg,h\right\} = f\left\{g,h\right\} + \left\{f,h\right\}g$\textit{ (Leibniz identity),}
    \item $\left\{f,\left\{g,h\right\}\right\} + \left\{g,\left\{h,f\right\}\right\} + \left\{h,\left\{f,g\right\}\right\} =0$\textit{ (Jacobi identity).}
\end{enumerate}
\end{defn}

\begin{defn} \label{defn222}
A Casimir element in $\mathcal F$ is an element $c \in \mathcal F$ such that $\left\{c,p\right\} = 0$ for any $p \in \mathcal F.$ It is also called a \textit{Casimir}.   
\end{defn}

Assume that $\left\{\cdot,\cdot\right\}$ is a Poisson bracket on $\mathcal F$ which is a field of rational functions generated by $n$ functions $f_1, \cdots,f_n$. We call the coordinates $(f_1, \cdots ,f_n)$ \textit{log-canonical} if the Poisson bracket is given by $\left\{f_i,f_j\right\} = \lambda_{ij}f_if_j$ where $\lambda_{ij}$ are constants.

It is well-known from\textcolor{black}{\cite{6}} that for each torus $\mathcal X_{\Sigma}$, the algebra of rational functions $\mathcal K(\mathcal X_{\Sigma})$ is equipped with a Poisson bracket defined as \begin{equation} \label{2.6}
\left\{X_i,X_j\right\} = \epsilon_{ij}X_iX_j.
\end{equation}
Hence, the variables $(X_i)_{i \in I}$ form a set of log-canonical coordinates on an algebraic torus $\mathcal{X}_{\Sigma}$.

This Poisson bracket is compatible with cluster transformations. Specifically, for any cluster transformation $\tau: \Sigma' \to \Sigma$ and any functions $f,g \in \mathcal{K}(\mathcal{X}_{\Sigma})$, we have:$$
\{(\tau^\mathcal{X})^*f, (\tau^\mathcal{X})^*g\} = (\tau^\mathcal{X})^*(\{f,g\}).$$Here, the bracket on the left is computed in $\mathcal{K}(\mathcal{X}_{\Sigma'})$ via the exchange matrix of the seed $\Sigma'$ (i.e., $\{X'_i, X'_j\} = \epsilon'_{ij} X'_i X'_j$). Thus, a cluster $\mathcal{X}$-variety is also referred to as the \textit{cluster Poisson variety}. The following proposition shows that $\mathcal{O}(\mathcal{X}_{|\Sigma|})$ is closed under the Poisson bracket, and hence is a Poisson algebra:

\begin{prop} \label{prop223}
    $\left\{f,g\right\} \in \mathcal{O}(\mathcal X_{|\Sigma|})$ if $f,g \in \mathcal{O}(\mathcal X_{|\Sigma|})$. 
\end{prop}
\begin{proof}
    Consider an $\mathcal{X}$-variable cluster transformation $\mu: \mathcal{X}_{\Sigma'} \to \mathcal{X}_{\Sigma}$. We have $\{\mu^*f,\mu^*g\} = \mu^*(\{f,g\})$ because the Poisson bracket $\{\cdot,\cdot\}$ is compatible with any cluster transformation. 
    
    Since both $\mu^*f$ and $\mu^*g$ are Laurent polynomials in the cluster $\mathcal{X}$-variables $X_i'$, the bracket $\{\mu^*f,\mu^*g\}$ is also a Laurent polynomial in the variables $X_i'$ by the definition of the log-canonical Poisson bracket. Hence, $\mu^*(\{f,g\})$ is also a Laurent polynomial in the variables $X_i'$. This argument works for any cluster transformation $\mu$, so the claim follows.\end{proof}

\subsection{Goldman Bracket on Teichmüller Space $\mathcal T_{g,s}$ with $s \ge 1$} \label{Ch2.3}

Let $\Sigma_{g,s}$ be a two-dimensional oriented and compact surface where $g$ is the genus and $s$ is the number of boundary components. We call the surface \textit{hyperbolic} if it possesses a global Riemannian metric of constant curvature $-1$. In particular, this implies $2g - 2 + s > 0$.

By the \textit{Poincaré uniformization theorem}, the surface $\Sigma_{g,s}$ is isometric to a quotient of the upper half plane $\mathbb H = \left\{x + iy : x,y \in \mathbb R, y>0\right\}$ by a Fuchsian subgroup $\Delta_{g,s}$ of $PSL(2,\mathbb R)$. Note that the upper half plane is equipped with the global metric $ds^2 = {dx^2 + dy^2 \over y^2}$ whose curvature is $-1$. The Fuchsian group is generated by a finite number of \textit{hyperbolic} elements that have two distinct fixed points on the real line.

The \textit{ideal triangle} of the hyperbolic surface is a triangle with geodesic sides in $\mathbb H.$ Note that geodesics in $\mathbb H$ are either semicircles with center on the $x$-axis or an infinite vertical half-line.

\begin{figure}[H]
    \centering
\tikzset{every picture/.style={line width=0.7pt}} %set default line width to 0.75pt      

\begin{tikzpicture}[x=0.75pt,y=0.75pt,yscale=-1,xscale=1]

%Straight Lines [id:da3182048831354565] 
\draw    (161.82,219.82) -- (509.82,219.82) ;
\draw [shift={(511.82,219.82)}, rotate = 180] [color={rgb, 255:red, 0; green, 0; blue, 0 }  ][line width=0.75]    (10.93,-3.29) .. controls (6.95,-1.4) and (3.31,-0.3) .. (0,0) .. controls (3.31,0.3) and (6.95,1.4) .. (10.93,3.29)   ;
%Straight Lines [id:da4225042800404033] 
\draw    (225.82,76.82) -- (225.82,219.82) ;
%Straight Lines [id:da9584010603197256] 
\draw    (336.82,76.82) -- (336.82,219.82) ;
%Straight Lines [id:da667170500542643] 
\draw    (447.82,76.82) -- (447.82,140.82) -- (447.82,219.82) ;
%Shape: Arc [id:dp7292303014497719] 
\draw  [draw opacity=0] (197.02,218.83) .. controls (197.06,182.89) and (223.43,153.77) .. (255.94,153.79) .. controls (288.46,153.8) and (314.8,182.97) .. (314.78,218.93) .. controls (314.78,219.12) and (314.78,219.3) .. (314.78,219.49) -- (255.9,218.9) -- cycle ; \draw   (197.02,218.83) .. controls (197.06,182.89) and (223.43,153.77) .. (255.94,153.79) .. controls (288.46,153.8) and (314.8,182.97) .. (314.78,218.93) .. controls (314.78,219.12) and (314.78,219.3) .. (314.78,219.49) ;  
%Shape: Arc [id:dp8145877412437816] 
\draw  [draw opacity=0] (266.92,219.56) .. controls (266.98,160.8) and (313.3,113.21) .. (370.4,113.24) .. controls (427.51,113.27) and (473.79,160.95) .. (473.76,219.74) .. controls (473.76,220.06) and (473.76,220.39) .. (473.76,220.71) -- (370.34,219.68) -- cycle ; \draw   (266.92,219.56) .. controls (266.98,160.8) and (313.3,113.21) .. (370.4,113.24) .. controls (427.51,113.27) and (473.79,160.95) .. (473.76,219.74) .. controls (473.76,220.06) and (473.76,220.39) .. (473.76,220.71) ;  

\end{tikzpicture}
    \caption{Geodesics in the upper half plane}
    \label{Fig3}
\end{figure}
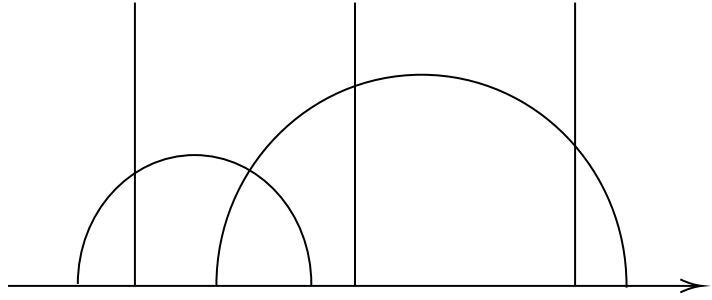

We cut infinite hyperboloids attached to the holes along a closed geodesic which is homeomorphic to a loop around each hole. Such a closed geodesic is called the \textit{bottleneck curve}. Then, we get a \textit{reduced surface}, which is an open Riemann surface with all infinite hyperboloids removed with bounding bottleneck curves. 

Let us define an \textit{ideal triangulation} of the reduced surface.

\begin{defn}
    An \textit{ideal triangulation} of a hyperbolic surface $\Sigma_{g,s}$ is a partition of its reduced surface into ideal triangles with vertices on the ideal line (real line $(x,0) \in \overline{\mathbb H}$). \end{defn}
\begin{figure}[H]
    \centering
\tikzset{every picture/.style={line width=0.75pt}} %set default line width to 0.75pt        

\tikzset{every picture/.style={line width=0.75pt}} %set default line width to 0.75pt        

\begin{tikzpicture}[x=0.75pt,y=0.75pt,yscale=-1,xscale=1]

%Shape: Arc [id:dp027843911806103683] 
\draw  [draw opacity=0][line width=1.5]  (184.76,218.28) .. controls (184.83,152.85) and (233.11,99.84) .. (292.62,99.87) .. controls (352.16,99.91) and (400.4,153) .. (400.37,218.47) .. controls (400.37,218.81) and (400.37,219.15) .. (400.36,219.48) -- (292.56,218.41) -- cycle ; \draw  [color={rgb, 255:red, 0; green, 255; blue, 0 }  ,draw opacity=1 ][line width=1.5]  (184.76,218.28) .. controls (184.83,152.85) and (233.11,99.84) .. (292.62,99.87) .. controls (352.16,99.91) and (400.4,153) .. (400.37,218.47) .. controls (400.37,218.81) and (400.37,219.15) .. (400.36,219.48) ;  
%Shape: Arc [id:dp18348644282702797] 
\draw  [draw opacity=0][line width=1.5]  (184.76,218.28) .. controls (184.8,177.87) and (217.12,145.14) .. (256.95,145.16) .. controls (296.8,145.18) and (329.09,177.97) .. (329.07,218.41) .. controls (329.07,218.63) and (329.07,218.86) .. (329.07,219.09) -- (256.91,218.37) -- cycle ; \draw  [color={rgb, 255:red, 255; green, 0; blue, 0 }  ,draw opacity=1 ][line width=1.5]  (184.76,218.28) .. controls (184.8,177.87) and (217.12,145.14) .. (256.95,145.16) .. controls (296.8,145.18) and (329.09,177.97) .. (329.07,218.41) .. controls (329.07,218.63) and (329.07,218.86) .. (329.07,219.09) ;  
%Shape: Arc [id:dp756348387911932] 
\draw  [draw opacity=0][line width=1.5]  (184.76,218.28) .. controls (184.85,122.59) and (249.92,45.07) .. (330.13,45.11) .. controls (410.37,45.16) and (475.38,122.8) .. (475.33,218.53) .. controls (475.33,218.99) and (475.33,219.45) .. (475.32,219.9) -- (330.04,218.46) -- cycle ; \draw  [line width=1.5]  (184.76,218.28) .. controls (184.85,122.59) and (249.92,45.07) .. (330.13,45.11) .. controls (410.37,45.16) and (475.38,122.8) .. (475.33,218.53) .. controls (475.33,218.99) and (475.33,219.45) .. (475.32,219.9) ;  
%Shape: Arc [id:dp7292303014497719] 
\draw  [draw opacity=0][fill={rgb, 255:red, 155; green, 155; blue, 155 }  ,fill opacity=1 ][line width=1.5]  (184.64,219.17) .. controls (184.66,194.27) and (200.74,174.1) .. (220.55,174.11) .. controls (240.38,174.12) and (256.44,194.32) .. (256.42,219.23) .. controls (256.42,219.35) and (256.42,219.46) .. (256.42,219.57) -- (220.53,219.21) -- cycle ; \draw  [line width=1.5]  (184.64,219.17) .. controls (184.66,194.27) and (200.74,174.1) .. (220.55,174.11) .. controls (240.38,174.12) and (256.44,194.32) .. (256.42,219.23) .. controls (256.42,219.35) and (256.42,219.46) .. (256.42,219.57) ;  
%Shape: Arc [id:dp6360087689692222] 
\draw  [draw opacity=0][fill={rgb, 255:red, 155; green, 155; blue, 155 }  ,fill opacity=1 ][line width=1.5]  (256.79,219.26) .. controls (256.82,194.36) and (272.89,174.19) .. (292.71,174.2) .. controls (312.53,174.21) and (328.59,194.41) .. (328.58,219.32) .. controls (328.58,219.43) and (328.58,219.55) .. (328.58,219.66) -- (292.69,219.3) -- cycle ; \draw  [line width=1.5]  (256.79,219.26) .. controls (256.82,194.36) and (272.89,174.19) .. (292.71,174.2) .. controls (312.53,174.21) and (328.59,194.41) .. (328.58,219.32) .. controls (328.58,219.43) and (328.58,219.55) .. (328.58,219.66) ;  
%Shape: Arc [id:dp10048465722635258] 
\draw  [draw opacity=0][fill={rgb, 255:red, 155; green, 155; blue, 155 }  ,fill opacity=1 ][line width=1.5]  (328.58,219.66) .. controls (328.6,194.76) and (344.68,174.59) .. (364.5,174.6) .. controls (384.32,174.61) and (400.38,194.81) .. (400.37,219.72) .. controls (400.37,219.83) and (400.37,219.95) .. (400.36,220.06) -- (364.47,219.7) -- cycle ; \draw  [line width=1.5]  (328.58,219.66) .. controls (328.6,194.76) and (344.68,174.59) .. (364.5,174.6) .. controls (384.32,174.61) and (400.38,194.81) .. (400.37,219.72) .. controls (400.37,219.83) and (400.37,219.95) .. (400.36,220.06) ;  
%Shape: Arc [id:dp2568462115956003] 
\draw  [draw opacity=0][fill={rgb, 255:red, 155; green, 155; blue, 155 }  ,fill opacity=1 ][line width=1.5]  (400.36,218.96) .. controls (400.39,193.45) and (417.06,172.79) .. (437.6,172.8) .. controls (458.16,172.81) and (474.81,193.5) .. (474.79,219.02) .. controls (474.79,219.14) and (474.79,219.26) .. (474.79,219.37) -- (437.58,219) -- cycle ; \draw  [line width=1.5]  (400.36,218.96) .. controls (400.39,193.45) and (417.06,172.79) .. (437.6,172.8) .. controls (458.16,172.81) and (474.81,193.5) .. (474.79,219.02) .. controls (474.79,219.14) and (474.79,219.26) .. (474.79,219.37) ;  
%Straight Lines [id:da3182048831354565] 
\draw [line width=1.5]    (160.82,219.82) -- (507.82,219.82) ;
\draw [shift={(510.82,219.82)}, rotate = 180] [color={rgb, 255:red, 0; green, 0; blue, 0 }  ][line width=1.5]    (14.21,-4.28) .. controls (9.04,-1.82) and (4.3,-0.39) .. (0,0) .. controls (4.3,0.39) and (9.04,1.82) .. (14.21,4.28)   ;
%Shape: Ellipse [id:dp7504182592912756] 
\draw  [fill={rgb, 255:red, 0; green, 0; blue, 255 }  ,fill opacity=1 ] (180.73,219.17) .. controls (180.73,217.01) and (182.48,215.26) .. (184.64,215.26) .. controls (186.79,215.26) and (188.55,217.01) .. (188.55,219.17) .. controls (188.55,221.33) and (186.79,223.08) .. (184.64,223.08) .. controls (182.48,223.08) and (180.73,221.33) .. (180.73,219.17) -- cycle ;
%Shape: Ellipse [id:dp10510556409208582] 
\draw  [fill={rgb, 255:red, 0; green, 0; blue, 255 }  ,fill opacity=1 ] (252.88,219.26) .. controls (252.88,217.1) and (254.63,215.35) .. (256.79,215.35) .. controls (258.95,215.35) and (260.7,217.1) .. (260.7,219.26) .. controls (260.7,221.42) and (258.95,223.17) .. (256.79,223.17) .. controls (254.63,223.17) and (252.88,221.42) .. (252.88,219.26) -- cycle ;
%Shape: Ellipse [id:dp3408649567458262] 
\draw  [fill={rgb, 255:red, 0; green, 0; blue, 255 }  ,fill opacity=1 ] (324.67,219.66) .. controls (324.67,217.5) and (326.42,215.75) .. (328.58,215.75) .. controls (330.74,215.75) and (332.49,217.5) .. (332.49,219.66) .. controls (332.49,221.82) and (330.74,223.57) .. (328.58,223.57) .. controls (326.42,223.57) and (324.67,221.82) .. (324.67,219.66) -- cycle ;
%Shape: Ellipse [id:dp17328926546160928] 
\draw  [fill={rgb, 255:red, 0; green, 0; blue, 255 }  ,fill opacity=1 ] (396.46,218.96) .. controls (396.46,216.8) and (398.21,215.05) .. (400.36,215.05) .. controls (402.52,215.05) and (404.27,216.8) .. (404.27,218.96) .. controls (404.27,221.12) and (402.52,222.87) .. (400.36,222.87) .. controls (398.21,222.87) and (396.46,221.12) .. (396.46,218.96) -- cycle ;
%Shape: Ellipse [id:dp5856054950715446] 
\draw  [fill={rgb, 255:red, 0; green, 0; blue, 255 }  ,fill opacity=1 ] (470.88,219.37) .. controls (470.88,217.21) and (472.63,215.46) .. (474.79,215.46) .. controls (476.95,215.46) and (478.7,217.21) .. (478.7,219.37) .. controls (478.7,221.53) and (476.95,223.28) .. (474.79,223.28) .. controls (472.63,223.28) and (470.88,221.53) .. (470.88,219.37) -- cycle ;

\end{tikzpicture}
    \caption{An example of an ideal triangulation on the reduced surface. The blue vertices are ideal points, green and red arcs are internal sides, and gray regions indicate the truncated infinite hyperboloids.}
    \label{Fig4}
\end{figure}
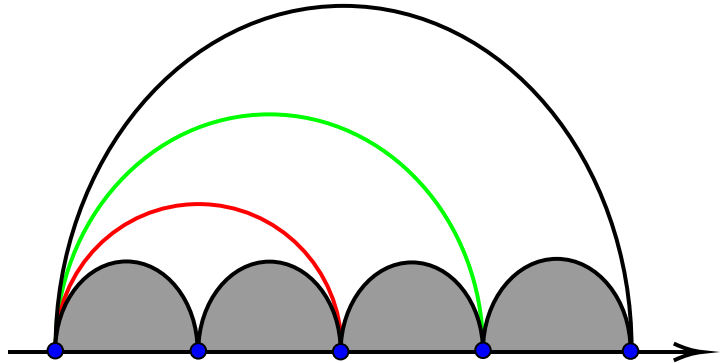
We introduce the \textit{Teichmüller space} and its fat graph description.

\begin{defn}
    Teichmüller space $\mathcal T_{g,s}$ is the space of complex structures on a genus $g$ surface with $s$ holes modulo diffeomorphisms homotopic to the identity.
\end{defn}
Consider a fat graph $\Gamma$ dual to an ideal triangulation of $\Sigma_{g,s}$, so that its vertices correspond to ideal triangles and its edges intersect the edges of the triangulation transversely at exactly one point. We associate a real variable $z_\alpha$ to each edge $\alpha$ of $\Gamma$ in order to parametrize the Teichmüller space. Then,

\begin{thm}
\textcolor{black}{\cite{14}} For a fat graph $\Gamma$ derived from $\Sigma_{g,s}$, there is a bijection between the set of points of $\mathcal T_{g,s}$ and the set $\mathbb R^{\text{the number of edges}}$ of assignments of real numbers to edges of the graph. 
\end{thm}

In this fat graph description, the canonical Poisson bracket $B_{WP}$ in the coordinates $\{z_\alpha\}$ is defined by

\begin{equation} \label{2.7}
    B_{WP} = \sum_v \sum_{i=1}^3 
    {\partial \over \partial z_{v_i}}\wedge{\partial \over \partial z_{v_{i+1}}}.
\end{equation}

Here, the sum is taken over all vertices $v$, and $v_1,v_2,v_3$ label the cyclically ordered edges adjacent to $v$, with indices taken modulo $3$. This Poisson bracket is degenerate; its Casimirs are the lengths of geodesics surrounding the holes.

We introduce the Poisson algebra generated by closed geodesics on the surface. To characterize these geodesics, we assign a matrix $X_{z_\alpha} \in \operatorname{PSL}(2,\mathbb R)$ to each edge $\alpha$, defined by
\begin{equation} \label{2.8}
    X_{z_\alpha}=\begin{pmatrix}
0 & -e^{z_\alpha/2} \\
e^{-z_\alpha/2} & 0
\end{pmatrix}.
\end{equation}
Furthermore, we define the following matrices corresponding to \textit{left} and \textit{right} turns in a closed path on the fat graph:
\begin{equation} \label{2.9}
    L =\begin{pmatrix}
0 & 1 \\
-1 & -1
\end{pmatrix},\text{ }
R=\begin{pmatrix}
1 & 1 \\
-1 & 0
\end{pmatrix}.\text{ }
\end{equation}

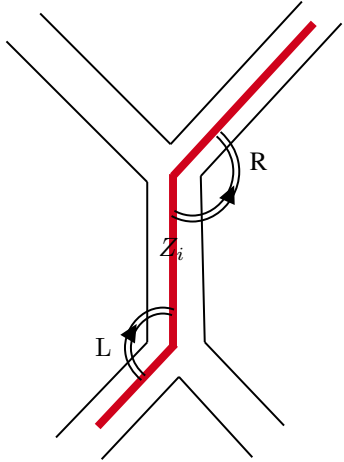
\begin{figure}[H]
    \centering
\tikzset{every picture/.style={line width=0.75pt}} %set default line width to 0.75pt        

\begin{tikzpicture}[x=0.75pt,y=0.75pt,yscale=-1,xscale=1]
%uncomment if require: \path (0,235); %set diagram left start at 0, and has height of 235

%Straight Lines [id:da7313205622171378] 
\draw    (247.82,18.67) -- (318.82,89.67) ;
%Straight Lines [id:da21272405370384717] 
\draw    (264.82,4.58) -- (330.58,70.58) ;
%Straight Lines [id:da532795477553455] 
\draw    (330.58,70.58) -- (396.82,-1.42) ;
%Straight Lines [id:da9053806753867502] 
\draw    (272.82,218.58) -- (318.82,170.58) ;
%Straight Lines [id:da677244691458253] 
\draw    (345.82,86.67) -- (416.82,9.58) ;
%Straight Lines [id:da9819879132407741] 
\draw    (347.82,170.49) -- (345.82,86.67) ;
%Straight Lines [id:da2008473729779886] 
\draw    (318.82,170.58) -- (318.82,89.67) ;
%Straight Lines [id:da616991751334333] 
\draw    (347.82,170.49) -- (387.82,212.58) ;
%Straight Lines [id:da011557859414001781] 
\draw    (334.82,188) -- (371.82,228.58) ;
%Straight Lines [id:da9938796833368464] 
\draw    (295.82,229.58) -- (334.82,188) ;
%Straight Lines [id:da6067701971645951] 
\draw [color={rgb, 255:red, 208; green, 2; blue, 27 }  ,draw opacity=1 ][line width=3]    (331.82,86.67) -- (402.82,9.58) ;
%Straight Lines [id:da8927086712592955] 
\draw [color={rgb, 255:red, 208; green, 2; blue, 27 }  ,draw opacity=1 ][line width=3]    (331.82,173.49) -- (331.82,85.67) ;
%Straight Lines [id:da928505724242965] 
\draw [color={rgb, 255:red, 208; green, 2; blue, 27 }  ,draw opacity=1 ][line width=3]    (293.82,213.07) -- (332.82,171.49) ;
%Curve Lines [id:da30623624313896536] 
\draw    (332.54,104.17) .. controls (332.98,104.42) and (333.43,104.64) .. (333.88,104.85) .. controls (336.51,106.06) and (339.14,106.61) .. (341.68,106.61) .. controls (348.94,106.61) and (355.42,102.13) .. (359.08,95.82) .. controls (361.07,92.4) and (362.23,88.43) .. (362.23,84.31) .. controls (362.23,82.8) and (362.07,81.28) .. (361.75,79.76) .. controls (360.77,75.2) and (358.26,70.64) .. (353.81,66.6)(331.1,106.81) .. controls (331.61,107.08) and (332.12,107.34) .. (332.63,107.57) .. controls (335.68,108.98) and (338.74,109.61) .. (341.68,109.61) .. controls (349.98,109.61) and (357.46,104.57) .. (361.67,97.33) .. controls (363.92,93.46) and (365.23,88.96) .. (365.23,84.31) .. controls (365.23,82.6) and (365.05,80.87) .. (364.68,79.13) .. controls (363.59,74.03) and (360.81,68.91) .. (355.83,64.38) ;
\draw [shift={(362.66,91.35)}, rotate = 114.12] [fill={rgb, 255:red, 0; green, 0; blue, 0 }  ][line width=0.08]  [draw opacity=0] (8.93,-4.29) -- (0,0) -- (8.93,4.29) -- cycle    ;
%Curve Lines [id:da7109753935870375] 
\draw    (315.86,189.64) .. controls (311.4,185.93) and (308.97,181.69) .. (308.04,177.47) .. controls (307.74,176.09) and (307.59,174.72) .. (307.59,173.36) .. controls (307.59,165.99) and (311.84,159.13) .. (317.77,155.59) .. controls (320.47,153.97) and (323.52,153.05) .. (326.67,153.05) .. controls (328.06,153.05) and (329.47,153.23) .. (330.88,153.61) .. controls (331.37,153.74) and (331.86,153.9) .. (332.34,154.08)(317.78,187.34) .. controls (313.93,184.13) and (311.78,180.48) .. (310.97,176.82) .. controls (310.72,175.66) and (310.59,174.5) .. (310.59,173.36) .. controls (310.59,167.06) and (314.24,161.19) .. (319.31,158.17) .. controls (321.54,156.83) and (324.06,156.05) .. (326.67,156.05) .. controls (327.8,156.05) and (328.94,156.19) .. (330.09,156.5) .. controls (330.49,156.61) and (330.89,156.74) .. (331.29,156.89) ;
\draw [shift={(313.49,161.29)}, rotate = 120.66] [fill={rgb, 255:red, 0; green, 0; blue, 0 }  ][line width=0.08]  [draw opacity=0] (8.93,-4.29) -- (0,0) -- (8.93,4.29) -- cycle    ;

% Text Node
\draw (369,73) node [anchor=north west][inner sep=0.75pt]   [align=left] {{\fontfamily{ptm}\selectfont R}};
% Text Node
\draw (291,167) node [anchor=north west][inner sep=0.75pt]   [align=left] {L};
% Text Node
\draw (323,116.4) node [anchor=north west][inner sep=0.75pt]    {$Z_{i}$};

\end{tikzpicture}
    \caption{The red line indicates a part of the path $P_{\gamma}$. The representation $[\gamma]$ is given by $[\cdots L X_{z_i} R \cdots]$ from the fat graph. As $L$, $R$, and $X_{z_i}$ lie in $PSL(2,\mathbb{R})$, the matrix product is in $PSL(2,\mathbb{R})$.}
\end{figure}

To define the Poisson algebra generated by closed geodesics, we state the following correspondence:
\begin{prop}
    \textcolor{black}{\cite[Theorem 1.6]{15}} \label{prop234}
    The following are in one-to-one correspondence:
    \begin{enumerate}
        \item Conjugacy classes $[\gamma]$ of hyperbolic elements of the Fuchsian group $\Delta_{g,s}$.
        \item Closed geodesics on the Riemann surface whose length $l_\gamma$ is given by an equation $\text{Tr}(\gamma) = 2\cosh({l_\gamma \over 2})$.
        \item Closed paths in a fat graph $\Gamma$.
        \item Conjugacy classes in $\pi_1(\Sigma_{g,s})$.
    \end{enumerate}
\end{prop}

Let $P_\gamma$ be a path in the fat graph $\Gamma$ corresponding to $[\gamma]$, and let $z_1,\ldots,z_n$ be the variables associated with the edges of $P_\gamma$. Recall that the conjugacy class of an element $\gamma \in \Delta_{g,s}$ is denoted by $[\gamma]$. As the path travels along the edges and turns left or right, it determines a product expression for $[\gamma]$ in terms of the matrices $X_{z_i}$, $L$, and $R$; see \cite{13}.

The geodesic functions $G_{\gamma}:=\operatorname{Tr}([\gamma])$ are in one-to-one correspondence with the closed geodesics on the surface, as established in Proposition~\ref{prop234}. These geodesic functions also generate a Poisson algebra induced by the Poisson bracket in \eqref{2.7}. This Poisson structure is called the \textit{Goldman bracket}.

To describe this Poisson structure, it suffices to consider a single intersection point of two geodesics by the Leibniz rule. Let $P,Q \in \Delta_{g,s}$, and suppose that the loops associated with $P$ and $Q$ intersect at exactly one point. The Poisson bracket at this intersection point is given by

\begin{equation} \label{2.10}
    \left\{G_P,G_Q\right\} = {1\over 2}G_{PQ} - {1 \over 2}G_{PQ^{-1}}
\end{equation}
\begin{figure}[H]
    \centering
    \includegraphics[width=0.5\linewidth]{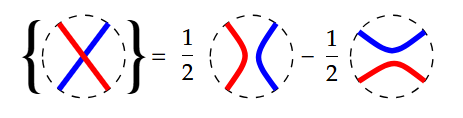}
    \caption{Graphical description of the Poisson bracket.}
\end{figure}

We also recall the classical skein relation:
\begin{equation} \label{2.11}
G_PG_Q = G_{PQ} + G_{PQ^{-1}}.
\end{equation}
It follows directly from the trace identity $\operatorname{Tr}(A)\operatorname{Tr}(B)
=
\operatorname{Tr}(AB)+\operatorname{Tr}(AB^{-1})$ for $A,B \in SL_2(\mathbb{C})$.
\begin{figure}[H]
    \centering
    \includegraphics[width=0.45\linewidth]{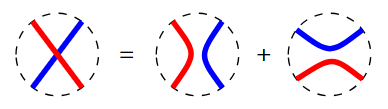}
    \caption{Graphical description of the skein relation.}
\end{figure}

\section{$\mathcal A_n$-quiver and its Birational Weyl Group Action}

\subsection{Log-canonical Coordinates for the Symplectic Groupoid and $\mathcal A_n$-quiver} \label{Ch3.1}

To describe a log-coordinate representation for the matrix entries of $\mathcal{A}_n$, we utilize Fock--Goncharov higher Teichmüller variables on an $SL_n$-quiver and a plabic graph $P_n$ dual to the $n$-triangular quiver. A detailed description is provided in \cite{1}.

\begin{figure}[H] 
    \centering
	\begin{tikzcd}[scale cd=0.7,column sep = tiny]
	&&&& \textcolor{rgb,255:red,0;green,0;blue,255}{{{Z_{0,0,4}}}} \\
	&&& \textcolor{rgb,255:red,0;green,0;blue,255}{{{Z_{1,0,3}}}} && \textcolor{rgb,255:red,0;green,0;blue,255}{{Z_{0,1,3}}} \\
	&& \textcolor{rgb,255:red,0;green,0;blue,255}{{{Z_{2,0,2}}}} && {Z_{1,1,2}} && \textcolor{rgb,255:red,0;green,0;blue,255}{{Z_{0,2,2}}} \\
	& \textcolor{rgb,255:red,0;green,0;blue,255}{{{Z_{3,0,1}}}} && {Z_{2,1,1}} && {Z_{1,2,1}} && \textcolor{rgb,255:red,0;green,0;blue,255}{{Z_{0,3,1}}} \\
	\textcolor{rgb,255:red,0;green,0;blue,255}{{{Z_{4,0,0}}}} && \textcolor{rgb,255:red,0;green,0;blue,255}{{{Z_{3,1,0}}}} && \textcolor{rgb,255:red,0;green,0;blue,255}{{{Z_{2,2,0}}}} && \textcolor{rgb,255:red,0;green,0;blue,255}{{{Z_{1,3,0}}}} && \textcolor{rgb,255:red,0;green,0;blue,255}{{{Z_{0,4,0}}}}
	\arrow[dashed, from=1-5, to=2-6]
	\arrow[dashed, from=2-4, to=1-5]
	\arrow[from=2-4, to=3-5]
	\arrow[from=2-6, to=2-4]
	\arrow[dashed, from=2-6, to=3-7]
	\arrow[dashed, from=3-3, to=2-4]
	\arrow[from=3-3, to=4-4]
	\arrow[from=3-5, to=2-6]
	\arrow[from=3-5, to=3-3]
	\arrow[from=3-5, to=4-6]
	\arrow[from=3-7, to=3-5]
	\arrow[dashed, from=3-7, to=4-8]
	\arrow[dashed, from=4-2, to=3-3]
	\arrow[from=4-2, to=5-3]
	\arrow[from=4-4, to=3-5]
	\arrow[from=4-4, to=4-2]
	\arrow[from=4-4, to=5-5]
	\arrow[from=4-6, to=3-7]
	\arrow[from=4-6, to=4-4]
	\arrow[from=4-6, to=5-7]
	\arrow[from=4-8, to=4-6]
	\arrow[dashed, from=4-8, to=5-9]
	\arrow[dashed, from=5-1, to=4-2]
	\arrow[from=5-3, to=4-4]
	\arrow[dashed, from=5-3, to=5-1]
	\arrow[from=5-5, to=4-6]
	\arrow[dashed, from=5-5, to=5-3]
	\arrow[from=5-7, to=4-8]
	\arrow[dashed, from=5-7, to=5-5]
	\arrow[dashed, from=5-9, to=5-7]
\end{tikzcd}
    \caption{4-triangular quiver with variables. Note that the dashed arrow has weight $1 \over 2$ and blue vertices are frozen.}
    \label{Fig8}
\end{figure}
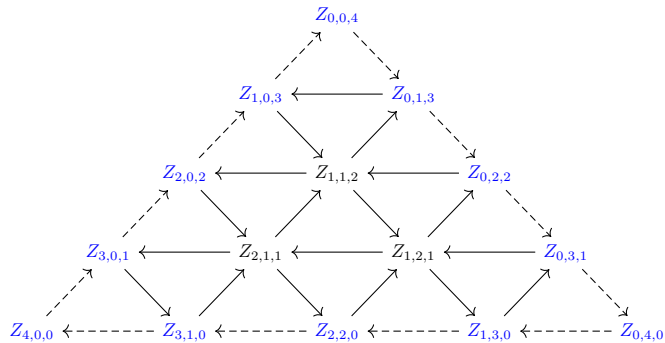
Consider an $n$-triangular quiver; see Figure~\ref{Fig8}. Vertices in the $k$th row from the top are labeled by $(i,k-i-1,n-k+1)$ for $i \in \{0,\ldots,k-1\}$, ordered from right to left. The vertex labeled by $(i,k-i-1,n-k+1)$ is assigned the variable $Z_{i,k-i-1,n-k+1}$. Vertices on the boundary of the triangle are frozen, and dashed arrows between frozen vertices have weight $1/2$.

By removing the three corner vertices corresponding to the variables $Z_{n,0,0}$, $Z_{0,n,0}$, and $Z_{0,0,n}$, we obtain a \textit{Fock--Goncharov $SL_n$-quiver}, or simply an $SL_n$-quiver. This quiver is naturally endowed with the Poisson structure defined in \eqref{2.6}.

\begin{figure}[H]
    \centering
    \begin{tikzcd}[scale cd=0.8,sep = tiny]
	&&&& \textcolor{rgb,255:red,255;green,51;blue,51}{{{1'}}} && \textcolor{rgb,255:red,255;green,0;blue,0}{\bullet} && \textcolor{rgb,255:red,255;green,51;blue,51}{1} \\
	&&&& \textcolor{rgb,255:red,51;green,173;blue,255}{{Z_{1,0,2}}} &&&& \textcolor{rgb,255:red,51;green,173;blue,255}{{Z_{0,1,2}}} \\
	&& \textcolor{rgb,255:red,255;green,51;blue,51}{{{2'}}} && \textcolor{rgb,255:red,255;green,0;blue,0}{\bullet} && \textcolor{rgb,255:red,255;green,0;blue,0}{\bullet} && \textcolor{rgb,255:red,255;green,0;blue,0}{\bullet} && \textcolor{rgb,255:red,255;green,51;blue,51}{2} \\
	&& \textcolor{rgb,255:red,51;green,173;blue,255}{{Z_{2,0,1}}} &&&& {Z_{1,1,1}} &&&& \textcolor{rgb,255:red,51;green,173;blue,255}{{Z_{0,2,1}}} \\
	\textcolor{rgb,255:red,255;green,51;blue,51}{{{3'}}} && \textcolor{rgb,255:red,255;green,0;blue,0}{\bullet} && \textcolor{rgb,255:red,255;green,0;blue,0}{\bullet} && \textcolor{rgb,255:red,255;green,0;blue,0}{\bullet} && \textcolor{rgb,255:red,255;green,0;blue,0}{\bullet} && \textcolor{rgb,255:red,255;green,0;blue,0}{\bullet} && \textcolor{rgb,255:red,255;green,51;blue,51}{3} \\
	&&&& \textcolor{rgb,255:red,51;green,173;blue,255}{{Z_{2,1,0}}} &&&& \textcolor{rgb,255:red,51;green,173;blue,255}{{Z_{1,2,0}}} \\
	&& \textcolor{rgb,255:red,255;green,51;blue,51}{{{1''}}} &&&& \textcolor{rgb,255:red,255;green,51;blue,51}{{{2''}}} &&&& \textcolor{rgb,255:red,255;green,51;blue,51}{{{3''}}}
	\arrow[color={rgb,255:red,255;green,0;blue,0}, from=1-7, to=1-5]
	\arrow[color={rgb,255:red,255;green,0;blue,0}, from=1-7, to=3-7]
	\arrow[color={rgb,255:red,255;green,0;blue,0}, from=1-9, to=1-7]
	\arrow[from=2-5, to=4-7]
	\arrow[from=2-9, to=2-5]
	\arrow[dashed, from=2-9, to=4-11]
	\arrow[color={rgb,255:red,255;green,0;blue,0}, from=3-5, to=3-3]
	\arrow[color={rgb,255:red,255;green,0;blue,0}, from=3-5, to=5-5]
	\arrow[color={rgb,255:red,255;green,0;blue,0}, from=3-7, to=3-5]
	\arrow[color={rgb,255:red,255;green,0;blue,0}, from=3-9, to=3-7]
	\arrow[color={rgb,255:red,255;green,0;blue,0}, from=3-9, to=5-9]
	\arrow[color={rgb,255:red,255;green,0;blue,0}, from=3-11, to=3-9]
	\arrow[dashed, from=4-3, to=2-5]
	\arrow[from=4-3, to=6-5]
	\arrow[from=4-7, to=2-9]
	\arrow[from=4-7, to=4-3]
	\arrow[from=4-7, to=6-9]
	\arrow[from=4-11, to=4-7]
	\arrow[color={rgb,255:red,255;green,0;blue,0}, from=5-3, to=5-1]
	\arrow[color={rgb,255:red,255;green,0;blue,0}, from=5-3, to=7-3]
	\arrow[color={rgb,255:red,255;green,0;blue,0}, from=5-5, to=5-3]
	\arrow[color={rgb,255:red,255;green,0;blue,0}, from=5-7, to=5-5]
	\arrow[color={rgb,255:red,255;green,0;blue,0}, from=5-7, to=7-7]
	\arrow[color={rgb,255:red,255;green,0;blue,0}, from=5-9, to=5-7]
	\arrow[color={rgb,255:red,255;green,0;blue,0}, from=5-11, to=5-9]
	\arrow[color={rgb,255:red,255;green,0;blue,0}, from=5-11, to=7-11]
	\arrow[color={rgb,255:red,255;green,0;blue,0}, from=5-13, to=5-11]
	\arrow[from=6-5, to=4-7]
	\arrow[from=6-9, to=4-11]
	\arrow[dashed, from=6-9, to=6-5]
\end{tikzcd}
    \caption{Fock--Goncharov $SL_3$-quiver and a graph $P_3$ which is dual to the $3$-triangular quiver.}
    \label{Fig9}
\end{figure}
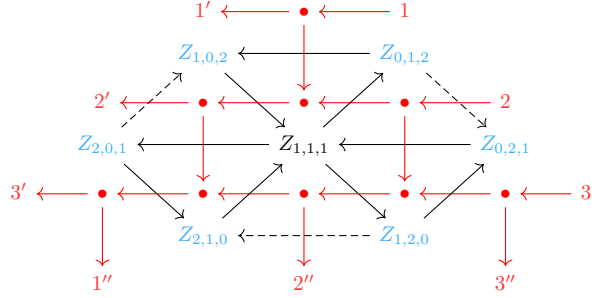

On the planar graph $P_n$ that is dual to the $n$-triangular quiver, we label the vertices on the right, left, and bottom sides from $1$ to $n$, $1'$ to $n'$, and $1''$ to $n''$ as shown in the orange graph in Figure~\ref{Fig9}.

Within these frameworks, transport matrices $T_{(1)}$, $T_{(2)}$ are defined as follows:

\begin{defn}
We define transport matrices $T_{(1)}$ and $T_{(2)}$ by

$$
(T_{(1)})_{ji} = \sum_{\text{oriented paths } p: i \to j'} w(p) \text{ and } (T_{(2)})_{ji} = \sum_{\text{oriented paths } p: i \to j''} w(p)
$$

where $w(p) = \prod_{v}Z_v$ such that the product is taken over all vertices $v$ in the $SL_n$-quiver that is to the right of the path $p$.
\end{defn}

\begin{exmp}
    Consider paths that start from $2$ and end at $3'$ of the orange quiver in Figure \ref{Fig9}. There are two such paths $p_1$ and $p_2$ and we compute $w(p_1) = Z_{1,0,2}Z_{0,1,2}Z_{2,0,1}Z_{1,1,1}$ and $w(p_2) = Z_{1,0,2}Z_{0,1,2}Z_{2,0,1}$. Thus, we get
    $$(T_{(1)})_{32} = Z_{1,0,2}Z_{0,1,2}Z_{2,0,1}Z_{1,1,1} + Z_{1,0,2}Z_{0,1,2}Z_{2,0,1}.$$

    Similarly, consider paths that start from $2$ and end at $1''$ of the orange quiver in Figure \ref{Fig9}. Again, there are two such paths $p'_1$ and $p'_2$, and we compute $w(p'_1) = Z_{1,0,2}Z_{0,1,2}Z_{2,0,1}Z_{1,1,1}$ and $w(p'_2) = Z_{1,0,2}Z_{0,1,2}Z_{2,0,1}$ Then, we have
    $$(T_{(2)})_{12} = Z_{1,0,2}Z_{0,1,2}Z_{2,0,1}Z_{1,1,1} + Z_{1,0,2}Z_{0,1,2}Z_{2,0,1}.$$
    
\end{exmp}

\begin{defn}
    The $n \times n$ matrix $S$ is defined by  $S_{ij} = (-1)^{n-i}\delta_{i,n+1-j}$, and the scalars $D_1$ and $D_2$ are defined as follows:

$$D_1 := \prod_{k=1}^n \prod_{i+j = n-k}Z_{i,j,k}^{k \over n} \text { and } D_2 := \prod_{k=1}^n \prod_{i+j = n-k}Z_{i,k,j}^{k \over n}.$$
\end{defn}

Then, we have the following theorem:

\begin{thm} \label{thm314} \textcolor{black}{\cite{1}}
Let $M_1 = ST_{(1)}D_1^{-1}$ and $M_2 = ST_{(2)}D_2^{-1}$. The matrix $\mathbb A := M_1^T M_2$ is an upper triangular matrix and its entries satisfy the Bondal Poisson bracket (\ref{1.4}) under the natural Poisson structure of the quiver defined in (\ref{2.6}).\end{thm}

In general, diagonal entries of $\mathbb A = M_1^TM_2$ are not necessarily units, so we impose compatibility conditions to make $\mathbb A$ unipotent. 

We have $\mathbb A_{i,j} = \sum_k (M_1^T)_{i,k}(M_2)_{k,j} = \sum_k (T_{(1)})_{k,i}(T_{(2)})_{k,j}$ up to the scalars $D_1$ and $D_2$. Moreover, all monomials of $(T_{(1)})_{k,i}$ contain $\prod_{i=1}^kZ_{i,0,n-i}$ and all monomials of $(T_{(2)})_{k,j}$ contain $\prod_{i=1}^kZ_{n-i,i,0}.$ These observations imply the matrix $\mathbb A$ depends on $Z_{i,0,n-i}Z_{n-i,i,0}$ but not each variable independently. 

Hence, we replace $Z_{i,0,n-i}Z_{n-i,i,0}$ by a new variable $\overline Z_i.$ In other words, we amalgamate two vertices $Z_{i,0,n-i}$ and $Z_{n-i,i,0}$. Let
\begin{equation}
    K_i := Z_{0,i,n-i}^2
    \prod_{j=1}^{i-1} Z_{j,i-j,n-i}\,\overline{Z}_i
    \prod_{j=1}^{n-i-1} Z_{j,i,n-i-j},
    \qquad i \in \{1,\ldots,n-1\}.
\end{equation}
Consider the locus defined by $K_i=1$ for all $i$ to make $\mathbb{A}$ unipotent. Each $K_i$ can be expressed as shown in Figure~\ref{Fig10}. Note that $\mathbb{A}_{i,i} = \prod_{j=1}^i K_j$ up to the scalars $D_1$ and $D_2$.

The remaining elements still form a log-canonical coordinate system on the locus; their Poisson bracket is determined by a quiver constructed by forgetting the variables $Z_{0,i,n-i}$ and amalgamating $Z_{i,0,n-i}$ with $Z_{n-i,i,0}$ for each $i$.

\begin{figure}[H]
    \centering
    \includegraphics[width=0.7\linewidth]{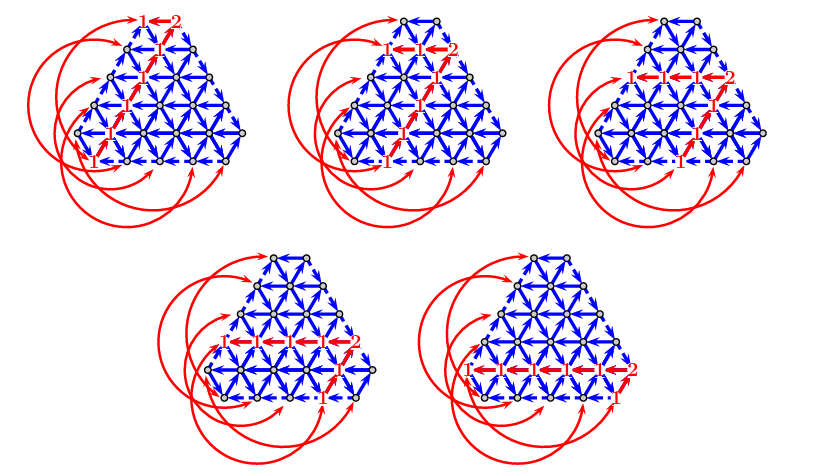}
    \caption{\textcolor{black}{\cite{1}} 
    Graphical description of $K_i$ for $SL_6$-quiver.}
    \label{Fig10}
\end{figure}

After unfreezing all vertices, we obtain a refined quiver called an \textbf{$\mathcal{A}_n$-quiver}. Examples of $\mathcal{A}_n$-quivers for $n=3,4,5,6$ are as follows:
\begin{figure}[H]
    \centering
    \includegraphics[width=0.8\linewidth]{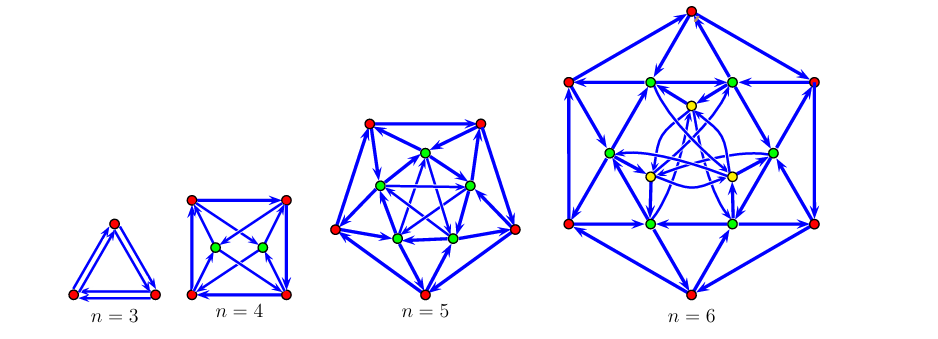}
    \caption{\textcolor{black}{\cite{1}} $\mathcal A_n$-quivers.}
    \label{Fig11}
\end{figure}

\begin{rem}(\textit{$\mathcal{A}_n$-quivers})
    \begin{enumerate}
        \item There are no dashed arrows due to the amalgamation. 
        \item The cluster Poisson coordinates associated with the $\mathcal{A}_n$-quiver are the log-canonical coordinates on the $\mathcal{A}_n$-groupoid.
        \item The $\mathcal{A}_3$-quiver coincides with the Markov quiver. Hence, it arises from a triangulation of a once-punctured torus. The $\mathcal{A}_4$-quiver arises from a triangulation of a twice-punctured torus.
    \end{enumerate}
\end{rem}

\subsection{Indices of $\mathcal A_n$-quiver} \label{Ch3.2}
In this section, we introduce a new indexing system for the cluster \(\mathcal X\)-variables, denoted by \(X_{l,j}\). We use \(A_{l,j}\) to denote the corresponding cluster \(\mathcal A\)-variable associated with the same vertex. Since our primary focus is the cluster \(\mathcal X\)-variety, we generally work with the variables \(X_{l,j}\), introducing their \(\mathcal A\)-counterparts only when needed. This reindexing is useful because the original Fock--Goncharov coordinates \(Z_{i,j,k}\) are not well suited for expressing the Weyl group action constructed below.

The $\mathcal{A}_n$-quiver contains $\lfloor n/2 \rfloor$ main cycles, which are distinguished by the red, green, and yellow vertices in Figure \ref{Fig11}. We relabel the cluster variables according to these cycles as follows: Using the Fock--Goncharov coordinates, we observe that the $l$th main cycle is (see Figure \ref{Fig12}),
$$Z_{(l,n-l-1,1)} \rightarrow Z_{(l,n-l,0)} = Z_{(n-l,0,l)} \rightarrow Z_{(n-l,1,l-1)} \rightarrow \cdots \rightarrow Z_{(n-l,l-1,1)} \rightarrow Z_{(n-l,l,0)} = Z_{(l,0,n-l)} \rightarrow $$$$Z_{(l,1,n-l-1)}\rightarrow \cdots \rightarrow Z_{(l,n-l-2,2)}\rightarrow Z_{(l,n-l-1,1)}$$
when $l \ne {n \over 2},$ or
$$Z_{(l,n-l-1,1)} \rightarrow Z_{(l,n-l,0)} = Z_{(n-l,0,l)} \rightarrow Z_{(n-l,1,l-1)} \rightarrow \cdots \rightarrow Z_{(n-l,l-1,1)} = Z_{(l,n-l-1,1)}$$
when $l = {n \over 2}.$ Note that the equalities above follow from the amalgamation discussed in the previous section.

\begin{figure}[H]
    \centering
    \begin{tikzcd}[scale cd=0.6,column sep = tiny]
	&&&& \textcolor{rgb,255:red,255;green,0;blue,0}{{Z_{1,0,5}}(X_{1,3})} \\
	&&& \textcolor{rgb,255:red,0;green,0;blue,255}{{Z_{2,0,4}}(X_{2,4})} && \textcolor{rgb,255:red,255;green,0;blue,0}{{Z_{1,1,4}}(X_{1,4})} \\
	&& \textcolor{rgb,255:red,0;green,255;blue,0}{{Z_{3,0,3}}(X_{3,2})} && \textcolor{rgb,255:red,0;green,0;blue,255}{{Z_{2,1,3}}(X_{2,5})} && \textcolor{rgb,255:red,255;green,0;blue,0}{{Z_{1,2,3}}(X_{1,5})} \\
	& \textcolor{rgb,255:red,0;green,0;blue,255}{{Z_{4,0,2}}(X_{2,2})} && \textcolor{rgb,255:red,0;green,255;blue,0}{{Z_{3,1,2}}(X_{3,3})} && \textcolor{rgb,255:red,0;green,0;blue,255}{{Z_{2,2,2}}(X_{2,6})} && \textcolor{rgb,255:red,255;green,0;blue,0}{{Z_{1,3,2}}(X_{1,6})} \\
	\textcolor{rgb,255:red,255;green,0;blue,0}{{Z_{5,0,1}}(X_{1,2})} && \textcolor{rgb,255:red,0;green,0;blue,255}{{Z_{4,1,1}}(X_{2,3})} && \textcolor{rgb,255:red,0;green,255;blue,0}{{Z_{3,2,1}}(X_{3,1})} && \textcolor{rgb,255:red,0;green,0;blue,255}{{Z_{2,3,1}}(X_{2,1})} && \textcolor{rgb,255:red,255;green,0;blue,0}{{Z_{1,4,1}}(X_{1,1})} \\
	& \textcolor{rgb,255:red,255;green,0;blue,0}{{Z_{5,1,0}}(X_{1,3})} && \textcolor{rgb,255:red,0;green,0;blue,255}{{Z_{4,2,0}}(X_{2,4})} && \textcolor{rgb,255:red,0;green,255;blue,0}{{Z_{3,3,0}}(X_{3,2})} && \textcolor{rgb,255:red,0;green,0;blue,255}{{Z_{2,4,0}}(X_{2,2})} && \textcolor{rgb,255:red,255;green,0;blue,0}{{Z_{1,5,0}}(X_{1,2})}
	\arrow[from=1-5, to=2-6]
	\arrow[dashed, from=2-4, to=1-5]
	\arrow[from=2-4, to=3-5]
	\arrow[from=2-6, to=2-4]
	\arrow[from=2-6, to=3-7]
	\arrow[dashed, from=3-3, to=2-4]
	\arrow[from=3-3, to=4-4]
	\arrow[from=3-5, to=2-6]
	\arrow[from=3-5, to=3-3]
	\arrow[from=3-5, to=4-6]
	\arrow[from=3-7, to=3-5]
	\arrow[from=3-7, to=4-8]
	\arrow[dashed, from=4-2, to=3-3]
	\arrow[from=4-2, to=5-3]
	\arrow[from=4-4, to=3-5]
	\arrow[from=4-4, to=4-2]
	\arrow[from=4-4, to=5-5]
	\arrow[from=4-6, to=3-7]
	\arrow[from=4-6, to=4-4]
	\arrow[from=4-6, to=5-7]
	\arrow[from=4-8, to=4-6]
	\arrow[from=4-8, to=5-9]
	\arrow[dashed, from=5-1, to=4-2]
	\arrow[from=5-1, to=6-2]
	\arrow[from=5-3, to=4-4]
	\arrow[from=5-3, to=5-1]
	\arrow[from=5-3, to=6-4]
	\arrow[from=5-5, to=4-6]
	\arrow[from=5-5, to=5-3]
	\arrow[from=5-5, to=6-6]
	\arrow[from=5-7, to=4-8]
	\arrow[from=5-7, to=5-5]
	\arrow[from=5-7, to=6-8]
	\arrow[from=5-9, to=5-7]
	\arrow[from=5-9, to=6-10]
	\arrow[from=6-2, to=5-3]
	\arrow[from=6-4, to=5-5]
	\arrow[dashed, from=6-4, to=6-2]
	\arrow[from=6-6, to=5-7]
	\arrow[dashed, from=6-6, to=6-4]
	\arrow[from=6-8, to=5-9]
	\arrow[dashed, from=6-8, to=6-6]
	\arrow[dashed, from=6-10, to=6-8]
\end{tikzcd}
    \caption{The $2$nd cycle (blue vertices) as $Z_{2,3,1} \to Z_{2,4,0} = Z_{4,0,2} \to Z_{4,1,1} \to Z_{4,2,0} = Z_{2,0,4} \to Z_{2,1,3} \to Z_{2,2,2} \to Z_{2,3,1}$ and $3$rd cycle (green vertices) as $Z_{3,2,1} \to Z_{3,3,0} = Z_{3,0,3} \to Z_{3,1,2} \to Z_{3,2,1}$ in the $\mathcal A_6$-quiver.}
    \label{Fig12}
\end{figure}
We label the cluster variables of the $l$th cycle as $X_{l,j}$ according to their cyclic order, with $l \in \{1,\dots,\lfloor n/2 \rfloor\}$ and $j \in \{1,\dots,N_l\}$. Here, $N_l$ denotes the length of the $l$th main cycle; $N_l=n/2$ if $l=n/2$, and $N_l=n$ otherwise. Specifically, for $l \ne n/2$, we 
label the variables as follows:
\[
\begin{gathered}
X_{l,1}:=Z_{(l,n-l-1,1)},\quad
X_{l,2}:=Z_{(l,n-l,0)}=Z_{(n-l,0,l)},\quad
X_{l,3}:=Z_{(n-l,1,l-1)},\quad \ldots,\quad
X_{l,l}:=Z_{(n-l,l-1,1)},\\
X_{l,l+1}:=Z_{(n-l,l,0)}=Z_{(l,0,n-l)},\quad
X_{l,l+2}:=Z_{(l,1,n-l-1)},\quad \ldots,\quad
X_{l,n}:=Z_{(l,n-l-2,2)}.
\end{gathered}
\]
The collection $I_l$ of cluster variable indices for the $l$th cycle is given by
\begin{equation} \label{3.2}
I_l = \{n(l-1)+1, n(l-1)+2,\dots,n(l-1)+N_l\}.
\end{equation}
We also extend the definition of the elements $X_{l,j}$ to all 
$j \in \mathbb{Z}$ by periodicity:
\[
X_{l,j+N_l}=X_{l,j}.
\]

\begin{prop}\label{prop321}
\textit{(Cycle symmetry of the $\mathcal{A}_n$-quiver)}

For any integer $k$, let $T_k$ be the shift operator defined by shifting the 
indices of each variable by $k$:
\[
T_k(X_{l,j}) = X_{l,j+k}
\quad \text{for all } 
l \in \left\{1,\ldots,\left\lfloor \frac{n}{2} \right\rfloor\right\}
\text{ and } j \in \{1,\ldots,N_l\}.
\]
The resulting quiver is isomorphic to the original quiver via this relabeling.
\end{prop}

\begin{proof}
Take a second copy of the $\mathcal{A}_n$-quiver and flip it vertically. Then glue this copy to the bottom side of the original quiver so that the amalgamated variables on the corresponding sides of the two quivers match. This \textit{glued $\mathcal{A}_n$-quiver} is isomorphic to the standard $\mathcal{A}_n$-quiver up to arrow multiplicity \cite{1}.
    $$\begin{tikzcd}[scale cd=0.65,column sep = tiny, row sep = small]
	&&&& \textcolor{rgb,255:red,255;green,0;blue,0}{{{X_{1,3}}}} \\
	&&& \textcolor{rgb,255:red,0;green,255;blue,0}{{{X_{2,4}}}} && \textcolor{rgb,255:red,255;green,0;blue,0}{{{X_{1,4}}}} &&&&&&&& \textcolor{rgb,255:red,255;green,0;blue,0}{{{X_{1,3}}}} \\
	&& \textcolor{rgb,255:red,0;green,255;blue,0}{{{{X_{2,2}}}}} && \textcolor{rgb,255:red,0;green,255;blue,0}{{{X_{2,5}}}} && \textcolor{rgb,255:red,255;green,0;blue,0}{{{X_{1,5}}}} &&&&&& \textcolor{rgb,255:red,0;green,255;blue,0}{{{X_{2,2}}}} && \textcolor{rgb,255:red,255;green,0;blue,0}{{{X_{1,4}}}} \\
	& \textcolor{rgb,255:red,255;green,0;blue,0}{{{{X_{1,2}}}}} && \textcolor{rgb,255:red,0;green,255;blue,0}{{{{X_{2,3}}}}} && \textcolor{rgb,255:red,0;green,255;blue,0}{{{X_{2,1}}}} && \textcolor{rgb,255:red,255;green,0;blue,0}{{{X_{1,1}}}} &&&& \textcolor{rgb,255:red,255;green,0;blue,0}{{{{X_{1,2}}}}} && \textcolor{rgb,255:red,0;green,255;blue,0}{{{{X_{2,1}}}}} && \textcolor{rgb,255:red,255;green,0;blue,0}{{{X_{1,1}}}} \\
	{} && \textcolor{rgb,255:red,255;green,0;blue,0}{{{{X_{1,3}}}}} && \textcolor{rgb,255:red,0;green,255;blue,0}{{{{X_{2,4}}}}} && \textcolor{rgb,255:red,0;green,255;blue,0}{{{X_{2,2}}}} && \textcolor{rgb,255:red,255;green,0;blue,0}{{{X_{1,2}}}} &&&& \textcolor{rgb,255:red,255;green,0;blue,0}{{{{X_{1,3}}}}} && \textcolor{rgb,255:red,0;green,255;blue,0}{{{{X_{2,2}}}}} && \textcolor{rgb,255:red,255;green,0;blue,0}{{{X_{1,2}}}} && {} \\
	&&& \textcolor{rgb,255:red,255;green,0;blue,0}{{{{X_{1,4}}}}} && \textcolor{rgb,255:red,0;green,255;blue,0}{{{{X_{2,5}}}}} && \textcolor{rgb,255:red,0;green,255;blue,0}{{{X_{2,3}}}} && \textcolor{rgb,255:red,255;green,0;blue,0}{{{X_{1,3}}}} &&&& \textcolor{rgb,255:red,255;green,0;blue,0}{{{{X_{1,4}}}}} && \textcolor{rgb,255:red,0;green,255;blue,0}{{{X_{2,1}}}} && \textcolor{rgb,255:red,255;green,0;blue,0}{{{X_{1,3}}}} \\
	&&&& \textcolor{rgb,255:red,255;green,0;blue,0}{{{{X_{1,5}}}}} && \textcolor{rgb,255:red,0;green,255;blue,0}{{{{X_{2,1}}}}} && \textcolor{rgb,255:red,0;green,255;blue,0}{{{X_{2,4}}}} &&&&&& \textcolor{rgb,255:red,255;green,0;blue,0}{{{{X_{1,1}}}}} && \textcolor{rgb,255:red,0;green,255;blue,0}{{{X_{2,2}}}} \\
	&&&&& \textcolor{rgb,255:red,255;green,0;blue,0}{{{{X_{1,1}}}}} && \textcolor{rgb,255:red,0;green,255;blue,0}{{{{X_{2,2}}}}} &&&&&&&& \textcolor{rgb,255:red,255;green,0;blue,0}{{{{X_{1,2}}}}} \\
	&&&&&& \textcolor{rgb,255:red,255;green,0;blue,0}{{{{X_{1,2}}}}}
	\arrow[from=1-5, to=2-6]
	\arrow[dashed, from=2-4, to=1-5]
	\arrow[from=2-4, to=3-5]
	\arrow[from=2-6, to=2-4]
	\arrow[from=2-6, to=3-7]
	\arrow[from=2-14, to=3-15]
	\arrow[dashed, from=3-3, to=2-4]
	\arrow[from=3-3, to=4-4]
	\arrow[from=3-5, to=2-6]
	\arrow[from=3-5, to=3-3]
	\arrow[from=3-5, to=4-6]
	\arrow[from=3-7, to=3-5]
	\arrow[from=3-7, to=4-8]
	\arrow[dashed, from=3-13, to=2-14]
	\arrow[from=3-13, to=4-14]
	\arrow[from=3-15, to=3-13]
	\arrow[from=3-15, to=4-16]
	\arrow[dashed, from=4-2, to=3-3]
	\arrow[from=4-2, to=5-3]
	\arrow[from=4-4, to=3-5]
	\arrow[from=4-4, to=4-2]
	\arrow[from=4-4, to=5-5]
	\arrow[from=4-6, to=3-7]
	\arrow[from=4-6, to=4-4]
	\arrow[from=4-6, to=5-7]
	\arrow[from=4-8, to=4-6]
	\arrow[from=4-8, to=5-9]
	\arrow[dashed, from=4-12, to=3-13]
	\arrow[from=4-12, to=5-13]
	\arrow[from=4-14, to=3-15]
	\arrow[from=4-14, to=4-12]
	\arrow[from=4-14, to=5-15]
	\arrow[from=4-16, to=4-14]
	\arrow[from=4-16, to=5-17]
	\arrow[from=5-3, to=4-4]
	\arrow[dotted, no head, from=5-3, to=5-1]
	\arrow[from=5-3, to=6-4]
	\arrow[from=5-5, to=4-6]
	\arrow[from=5-5, to=5-3]
	\arrow[from=5-5, to=6-6]
	\arrow[from=5-7, to=4-8]
	\arrow[from=5-7, to=5-5]
	\arrow[from=5-7, to=6-8]
	\arrow[from=5-9, to=5-7]
	\arrow[dotted, no head, from=5-9, to=5-13]
	\arrow[from=5-9, to=6-10]
	\arrow[from=5-13, to=4-14]
	\arrow[from=5-13, to=6-14]
	\arrow[from=5-15, to=4-16]
	\arrow[from=5-15, to=5-13]
	\arrow[from=5-15, to=6-16]
	\arrow[from=5-17, to=5-15]
	\arrow[dotted, no head, from=5-17, to=5-19]
	\arrow[from=5-17, to=6-18]
	\arrow[from=6-4, to=5-5]
	\arrow[from=6-4, to=7-5]
	\arrow[from=6-6, to=5-7]
	\arrow[from=6-6, to=6-4]
	\arrow[from=6-6, to=7-7]
	\arrow[from=6-8, to=5-9]
	\arrow[from=6-8, to=6-6]
	\arrow[from=6-8, to=7-9]
	\arrow[from=6-10, to=6-8]
	\arrow[from=6-14, to=5-15]
	\arrow[from=6-14, to=7-15]
	\arrow[from=6-16, to=5-17]
	\arrow[from=6-16, to=6-14]
	\arrow[from=6-16, to=7-17]
	\arrow[from=6-18, to=6-16]
	\arrow[from=7-5, to=6-6]
	\arrow[from=7-5, to=8-6]
	\arrow[from=7-7, to=6-8]
	\arrow[from=7-7, to=7-5]
	\arrow[from=7-7, to=8-8]
	\arrow[dashed, from=7-9, to=6-10]
	\arrow[from=7-9, to=7-7]
	\arrow[from=7-15, to=6-16]
	\arrow[from=7-15, to=8-16]
	\arrow[dashed, from=7-17, to=6-18]
	\arrow[from=7-17, to=7-15]
	\arrow[from=8-6, to=7-7]
	\arrow[from=8-6, to=9-7]
	\arrow[dashed, from=8-8, to=7-9]
	\arrow[from=8-8, to=8-6]
	\arrow[dashed, from=8-16, to=7-17]
	\arrow[dashed, from=9-7, to=8-8]
\end{tikzcd}$$
    In this setting, the cyclic symmetry is immediate.\end{proof}

\subsection{Cycle mutations on $\mathcal A_n$-quiver} \label{Ch3.3}

Let $\mathbf{q}$ be a quiver labeled by $I$ and containing a subquiver $\mathbf{q}_N$ with $N$ vertices, which is a chordless cycle labeled clockwise by the set $J$.

Assume that, for each vertex $v \notin \mathbf{q}_N$, the number of arrows from $v$ to $\mathbf{q}_N$ equals the number of arrows from $\mathbf{q}_N$ to $v$. In other words, we assume
\begin{equation} \label{3.3}
    \sum_{j \in J}\epsilon_{vj}=0
    \quad \text{for all } v \in I- J.
\end{equation}

\begin{figure}[H]
    \centering
    \begin{tikzcd}[scale cd=0.8, column sep = small]
	&&&&&&& \bullet && \bullet \\
	\bullet &&&& \bullet \\
	\bullet &&& \bullet && \bullet && \bullet && \bullet \\
	{\mathbf q_2} &&&& {\mathbf q_3} &&&& {\mathbf q_4}
	\arrow[from=1-8, to=1-10]
	\arrow[from=1-10, to=3-10]
	\arrow[from=2-5, to=3-6]
	\arrow[from=3-4, to=2-5]
	\arrow[from=3-6, to=3-4]
	\arrow[from=3-8, to=1-8]
	\arrow[from=3-10, to=3-8]
\end{tikzcd}
    \caption{Examples of $\mathbf q_2,\mathbf q_3$, and $\mathbf q_4.$}
\end{figure}
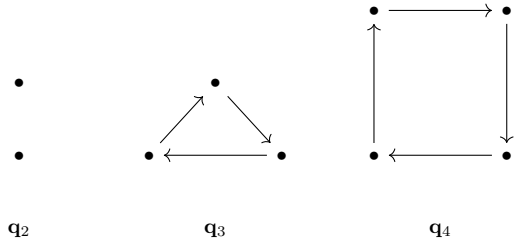

Consider a set $\left\{j_1, j_2, \cdots, j_N\right\}$ which is a permutation of the set $J$. We define the \textbf{cycle mutation} over the $\mathbf q_N$ by
\begin{equation} \label{3.4}
    \tau_{J}^{} := \mu_{j_1}\circ \cdots \circ\mu_{j_{N-1}}\circ \pi_{j_{N-1},j_N}\circ \mu_{j_{N-1}} \circ \cdots \circ \mu_{j_1}
\end{equation}

where $\mu_k$ is the cluster mutation in direction $k$ and $\pi_{j_{N-1},j_N}$ is the transformation switching the labels $j_{N-1}$ and $j_{N}$. Then, we have the following theorem:

\begin{thm} \label{thm331}
    \textcolor{black}{\cite[Theorem 7.1 and Theorem 7.7]{2}} The cluster transformation $\tau_{J}^{}$ preserves $\mathbf q$ and does not depend on the order of $\left\{j_1, j_2, \cdots, j_N\right\}.$
\end{thm}

To compute the image under $\tau_{J}^*$, we define a vector $c_v$ for each $v \in I-J$.

\begin{defn}\label{defn332}
Let $J=\{j_1,j_2,\ldots,j_N\}$ with $j_1<j_2<\cdots<j_N$. For each $v\in I - J$, there is a unique vector $c_v=(c_{v,1},\ldots,c_{v,N})$ such that
\begin{enumerate}
    \item $c_{v,i}-c_{v,i-1}=\epsilon_{v,j_i}$ for all $i=1,\ldots,N$, with 
    the convention $c_{v,0}=c_{v,N}$.
    \item $\min(c_{v,1},\ldots,c_{v,N})=0$.
\end{enumerate}
\end{defn}

From now on, we consider the case where $\mathbf{q}$ is an $\mathcal{A}_n$-quiver with $I=\{1,\ldots,n(n-1)/2\}$ and $J=I_l$, as defined in \eqref{3.2}. We exclude the case where $n$ is odd and $l=\lfloor n/2\rfloor$, since the corresponding main cycle is not chordless. Then we have the following result (note that the main cycles satisfy \eqref{3.3}):

\begin{thm} \cite[Theorem 7.7]{2} \label{thm333} 
The action on the cluster \(\mathcal A\)-variables is \[ \tau_{I_l}^*(A_{i,j}) = \begin{cases} A_{i,j}W_l, & i=l,\\[4pt] A_{i,j}, & i\neq l. \end{cases} \] The action on the cluster \(\mathcal X\)-variables is \[ \tau_{I_l}^*(X_{i,j}) = \begin{cases} \dfrac{X_{i,j}}{Y_{i,j}Y_{i,j-1}}, & i=l,\\[10pt] X_{i,j}\displaystyle\prod_{k=1}^{N_l} (Y_{l,k})^{c_{q,k}}, & i\neq l, \end{cases} \] where \(q=n(i-1)+j\). 
Here, \[ Y_{i,j} := X_{i,j}\frac{F_{i,j-1}}{F_{i,j}} \] with \[ F_{i,j} := 1+X_{i,j} +X_{i,j}X_{i,j-1} +\cdots +X_{i,j}X_{i,j-1}\cdots X_{i,j-N_i+2} \]
and
\[ W_l := \sum_{r=1}^{N_l} \frac{\prod_{q\notin I_l} A_q^{c_{q,r}}}{A_{l,r}A_{l,r+1}}. \]
\end{thm}

Using Theorem~\ref{thm333}, we describe the image of the birational action more explicitly. To simplify notation, we write $\tau_{l}^{}$ for $\tau_{I_l}$.

\begin{prop} \label{prop334} Let \(n\geq 5\), and exclude the innermost main cycle \(l=\lfloor n/2\rfloor\) when \(n\) is odd. The action on the cluster \(\mathcal X\)-variables is as follows: \[ \tau_l^*(X_{i,j})= \begin{cases} X_{l-1,j}Y_{l,j}, & i=l-1,\\[0.6em] \dfrac{X_{l,j}} {Y_{l,j}Y_{l,j-1}}, & i=l,\\[0.8em] X_{l+1,j}Y_{l,j-1}, & i=l+1,\ \text{unless } n=2m \text{ and } l=m-1,\\[0.6em] X_{m,j}\,Y_{m-1,j-1}Y_{m-1,m+j-1}, & i=m,\ n=2m,\ l=m-1,\\[0.6em] X_{i,j}, & i\notin\{l-1,l,l+1\}. \end{cases} \]

The action on the cluster \(\mathcal A\)-variables is \[\tau_l^*(A_{i,j})= \begin{cases} A_{i,j}W_l, & i=l,\\[0.3em] A_{i,j}, & i\ne l. \end{cases} \]

where \(W_l\) is given explicitly as follows: \[ W_l= \begin{cases} \displaystyle \sum_{r=1}^{n} \frac{A_{2,r+1}} {A_{1,r}A_{1,r+1}}, & l=1,\\[1.2em] \displaystyle \sum_{r=1}^{n} \frac{A_{l-1,r}A_{l+1,r+1}} {A_{l,r}A_{l,r+1}}, & 2\le l<\left\lfloor \frac{n}{2}\right\rfloor,\\[1.2em] \displaystyle \sum_{r=1}^{m} \frac{A_{m-1,r}A_{m-1,r+m}} {A_{m,r}A_{m,r+1}}, & l=m,\quad n=2m. \end{cases} \]\end{prop}

\begin{proof} 
Assume that $n=2m$. We first work on $\mathcal X$-variables. We divide the proof into two cases depending on $l$:
\begin{enumerate}
    \item $l \in \{1, \dots, m-2\} \cup \{m\}$: Let the indices for $X_{l-1,j}$ and $X_{l+1,j}$ be $q_1$ and $q_2$ respectively. Consider the following subquiver:
    $$\begin{tikzcd}[scale cd=0.8,column sep = tiny, row sep = small]
	& {X_{l,j-1}} \\
	{X_{l+1,j}} && {X_{l,j}} && {X_{l-1,j}} \\
	&&& {X_{l,j+1}}
	\arrow[from=1-2, to=2-3]
	\arrow[from=2-1, to=1-2]
	\arrow[from=2-3, to=2-1]
	\arrow[from=2-3, to=3-4]
	\arrow[from=2-5, to=2-3]
	\arrow[from=3-4, to=2-5]
\end{tikzcd}$$
    
    This describes the adjacency relations between $X_{l-1,j}$, $X_{l+1,j}$, and the $l$th main cycle. The other variables in the $l$th main cycle are not adjacent to $X_{l-1,j}$ or $X_{l+1,j}$ (see Figure \ref{Fig14}). Then, we have
\[
c_{q_1} = (0, \dots, 0, \underset{j\text{th}}{1}, 0, \dots, 0)
\]
\[
c_{q_2} = (0, \dots, 0, \underset{(j-1)\text{th}}{1}, 0, \dots, 0)
\]
since $(\epsilon_{q_1,k})_{k \in I_l}= (0, \dots, 0, \underset{j\text{th}}{1}, -1,0, \dots, 0)$ and $(\epsilon_{q_2,k})_{k \in I_l} = (0, \dots, 0, \underset{(j-1)\text{th}}{1}, -1, 0,\dots, 0)$. This implies, by Definition~\ref{defn332} and Theorem~\ref{thm333}, 
    \[
\tau_l^*(X_{l-1,j}) = X_{l-1,j}Y_{l,j} \text{ and } \tau_l^*(X_{l+1,j}) = X_{l+1,j}Y_{l,j-1}.
\]
It is clear that $\tau_l^*(X_{l,j}) = \frac{X_{l,j}}{Y_{l,j}Y_{l,j-1}}$ and $\tau_l^*(X_{i,j}) = X_{i,j}$ for $i \notin \{l-1, l, l+1\}$.
    \item $l = m-1$: Let the indices for $X_{m-2,j}$ and $X_{m,j}$ be $q_1$ and $q_2$ respectively. Consider the following subquiver:
    
    $$\begin{tikzcd}[scale cd=0.8,column sep = tiny, row sep = small]
	&&& {X_{m-1,j-1}} \\
	{X_{m-1,m+j-1}} && {X_{m,j}} && {X_{m-1,j}} && {X_{m-2,j}} \\
	& {X_{m-1,m+j}} &&&& {X_{m-1,j+1}}
	\arrow[from=1-4, to=2-5]
	\arrow[from=2-1, to=3-2]
	\arrow[from=2-3, to=1-4]
	\arrow[from=2-3, to=2-1]
	\arrow[from=2-5, to=2-3]
	\arrow[from=2-5, to=3-6]
	\arrow[from=2-7, to=2-5]
	\arrow[from=3-2, to=2-3]
	\arrow[from=3-6, to=2-7]
\end{tikzcd}$$
    This describes the adjacency relations between $X_{m-2,j}$, $X_{m,j}$, and the $m-1$th main cycle. The other variables in the $m-1$th main cycle are not adjacent to $X_{m-2,j}$ or $X_{m,j}$ (see Figure \ref{Fig12} and \ref{Fig14}). Then, we have
\[
c_{q_1} = (0, \dots, 0, \underset{j\text{th}}{1}, 0, \dots, 0)
\]
\[
c_{q_2} = (0, \dots, 0, \underset{(j-1)\text{th}}{1}, 0\dots, 0,\underset{(m+j-1)\text{th}}{1},0,\dots,0)
\]
as $(\epsilon_{q_1,k})_{k \in I_{m-1}}= (0, \dots, 0, \underset{j\text{th}}{1}, -1,0, \dots, 0)$ and 

$(\epsilon_{q_2,k})_{k \in I_{m-1}} = (0, \dots, 0, \underset{(j-1)\text{th}}{1}, -1, 0,\dots,0,\underset{(m+j-1)\text{th}}{1}, -1,0,\dots ,0)$. This implies, by Definition~\ref{defn332} and Theorem~\ref{thm333}, 
    \[
\begin{aligned}
\tau_{m-1}^*(X_{m-2,j}) &= X_{m-2,j}Y_{m-1,j}\text{ and }\tau_{m-1}^*(X_{m,j}) = X_{m,j}Y_{m-1,j-1}Y_{m-1,m+j-1},
\end{aligned}
\]
It is clear that $\tau_{m-1}^*(X_{m-1,j}) = \frac{X_{m-1,j}}{Y_{m-1,j}Y_{m-1,j-1}}$ and $\tau_{m-1}^*(X_{i,j}) = X_{i,j}$ for $i \notin \{m-2, m-1, m\}$.
\end{enumerate}

It remains to describe the action on the cluster \(\mathcal A\)-variables. By Theorem~\ref{thm333}, every \(\mathcal A\)-variable belonging to the \(l\)th main cycle is multiplied by 
\[ W_l = \sum_{r=1}^{N_l} \frac{Q_{l,r}}{A_{l,r}A_{l,r+1}}, \qquad Q_{l,r} = \prod_{q\notin I_l}A_q^{c_{q,r}}, \]
whereas all \(\mathcal A\)-variables outside the cycle remain invariant. 

For \(l=1\), the only external vertex whose vector \(c_q\) has a nonzero \(r\)th component is \(A_{2,r+1}\). More precisely, this component is equal to \(1\), while the \(r\)th components of the vectors associated with all other vertices \(A_{2,i}\) are zero. It is because $c_{A_{2,j}} = (0, \dots, 0, \underset{(j-1)\text{th}}{1}, 0, \dots, 0)$ when $l$ = 1 by calculations above. Hence \[ Q_{1,r}=A_{2,r+1}, \] and therefore \[ W_1 = \sum_{r=1}^{n} \frac{A_{2,r+1}} {A_{1,r}A_{1,r+1}}. \] 

For $2\leq l<\left\lfloor\frac n2\right\rfloor,$ the calculations above show that the external vertices whose vectors \(c_q\) have a nonzero \(r\)th component are \(A_{l-1,r}\) and \(A_{l+1,r+1}\). In both cases, the corresponding component is equal to \(1\). Thus \[ Q_{l,r} = A_{l-1,r}A_{l+1,r+1}, \] and \[ W_l = \sum_{r=1}^{n} \frac{A_{l-1,r}A_{l+1,r+1}} {A_{l,r}A_{l,r+1}}. \] 

Finally, suppose that \(n=2m\) and \(l=m\). By a calculation similar to the one above, the external vertices whose vectors \(c_q\) have a nonzero \(r\)th component are \[ A_{m-1,r} \qquad\text{and}\qquad A_{m-1,r+m}. \] Again, both components are equal to \(1\). Consequently, \[ Q_{m,r} = A_{m-1,r}A_{m-1,r+m}, \] and hence \[ W_m = \sum_{r=1}^{m} \frac{A_{m-1,r}A_{m-1,r+m}} {A_{m,r}A_{m,r+1}}. \] This proves the formulas for the cluster \(\mathcal A\)-variables. For odd \(n\), the proof is identical for every chordless main cycle. This completes the proof. \end{proof}

\subsection{Birational Weyl Group Action on $\mathcal A_n$-quiver} \label{Ch3.4}
\begin{defn} 
    Take another copy of an $\mathcal A_n$-quiver and flip it down. Then, glue this copy to the bottom side of the original quiver so that the corresponding amalgamated variables on the two sides are identified. This glued quiver is isomorphic to the $\mathcal{A}_n$-quiver up to arrow multiplicity. It is illustrated as follows:
    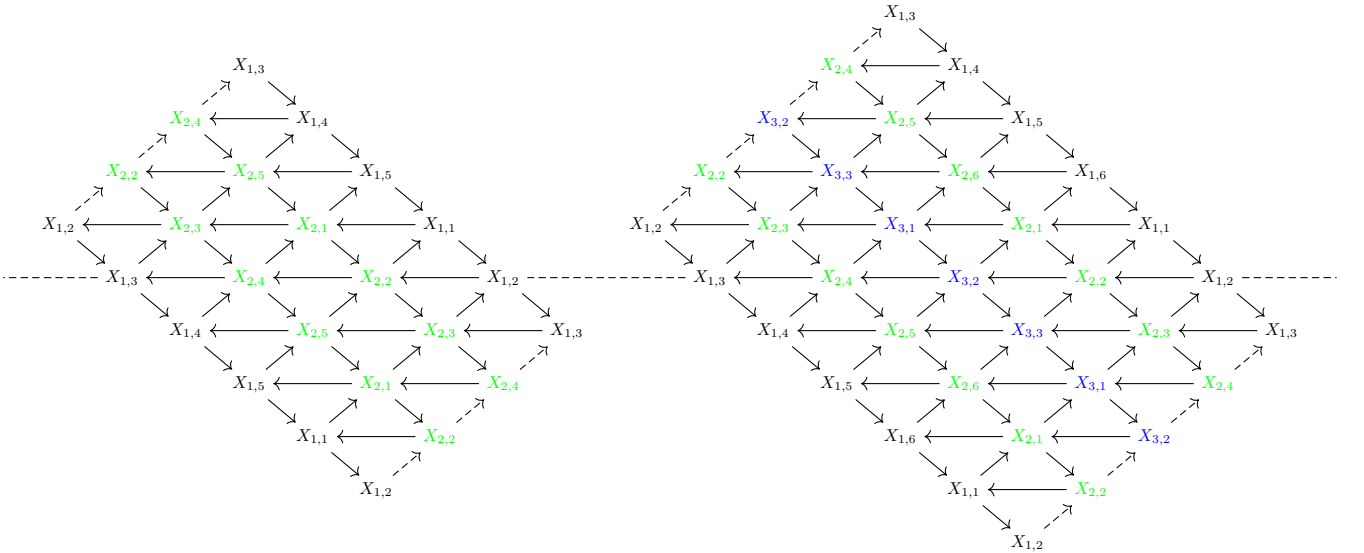
\begin{figure}[H]
    \centering
    \begin{tikzcd}[scale cd=0.6,column sep = tiny, row sep = small]
	&&&&&&&&&&&&&&&& {X_{1,3}} \\
	&&&&& {X_{1,3}} &&&&&&&&&& \textcolor{rgb,255:red,0;green,255;blue,0}{{X_{2,4}}} && {X_{1,4}} \\
	&&&& \textcolor{rgb,255:red,0;green,255;blue,0}{{X_{2,4}}} && {X_{1,4}} &&&&&&&& \textcolor{rgb,255:red,0;green,0;blue,255}{{{{X_{3,2}}}}} && \textcolor{rgb,255:red,0;green,255;blue,0}{{X_{2,5}}} && {X_{1,5}} \\
	&&& \textcolor{rgb,255:red,0;green,255;blue,0}{{X_{2,2}}} && \textcolor{rgb,255:red,0;green,255;blue,0}{{X_{2,5}}} && {X_{1,5}} &&&&&& \textcolor{rgb,255:red,0;green,255;blue,0}{{{{X_{2,2}}}}} && \textcolor{rgb,255:red,0;green,0;blue,255}{{X_{3,3}}} && \textcolor{rgb,255:red,0;green,255;blue,0}{{X_{2,6}}} && {X_{1,6}} \\
	&& {X_{1,2}} && \textcolor{rgb,255:red,0;green,255;blue,0}{{X_{2,3}}} && \textcolor{rgb,255:red,0;green,255;blue,0}{{X_{2,1}}} && {X_{1,1}} &&&& {X_{1,2}} && \textcolor{rgb,255:red,0;green,255;blue,0}{{X_{2,3}}} && \textcolor{rgb,255:red,0;green,0;blue,255}{{X_{3,1}}} && \textcolor{rgb,255:red,0;green,255;blue,0}{{{{X_{2,1}}}}} && {{{X_{1,1}}}} \\
	{} &&& {X_{1,3}} && \textcolor{rgb,255:red,0;green,255;blue,0}{{X_{2,4}}} && \textcolor{rgb,255:red,0;green,255;blue,0}{{X_{2,2}}} && {X_{1,2}} &&&& {X_{1,3}} && \textcolor{rgb,255:red,0;green,255;blue,0}{{X_{2,4}}} && \textcolor{rgb,255:red,0;green,0;blue,255}{{{{X_{3,2}}}}} && \textcolor{rgb,255:red,0;green,255;blue,0}{{{{X_{2,2}}}}} && {{{X_{1,2}}}} &&& {} \\
	&&&& {X_{1,4}} && \textcolor{rgb,255:red,0;green,255;blue,0}{{X_{2,5}}} && \textcolor{rgb,255:red,0;green,255;blue,0}{{X_{2,3}}} && {X_{1,3}} &&&& {X_{1,4}} && \textcolor{rgb,255:red,0;green,255;blue,0}{{X_{2,5}}} && \textcolor{rgb,255:red,0;green,0;blue,255}{{{{X_{3,3}}}}} && \textcolor{rgb,255:red,0;green,255;blue,0}{{{{X_{2,3}}}}} && {X_{1,3}} \\
	&&&&& {X_{1,5}} && \textcolor{rgb,255:red,0;green,255;blue,0}{{X_{2,1}}} && \textcolor{rgb,255:red,0;green,255;blue,0}{{X_{2,4}}} &&&&&& {X_{1,5}} && \textcolor{rgb,255:red,0;green,255;blue,0}{{{{X_{2,6}}}}} && \textcolor{rgb,255:red,0;green,0;blue,255}{{{{X_{3,1}}}}} && \textcolor{rgb,255:red,0;green,255;blue,0}{{X_{2,4}}} \\
	&&&&&& {X_{1,1}} && \textcolor{rgb,255:red,0;green,255;blue,0}{{X_{2,2}}} &&&&&&&& {{{X_{1,6}}}} && \textcolor{rgb,255:red,0;green,255;blue,0}{{{{X_{2,1}}}}} && \textcolor{rgb,255:red,0;green,0;blue,255}{{{{X_{3,2}}}}} \\
	&&&&&&& {X_{1,2}} &&&&&&&&&& {{{X_{1,1}}}} && \textcolor{rgb,255:red,0;green,255;blue,0}{{{{X_{2,2}}}}} \\
	&&&&&&&&&&&&&&&&&& {X_{1,2}}
	\arrow[from=1-17, to=2-18]
	\arrow[from=2-6, to=3-7]
	\arrow[dashed, from=2-16, to=1-17]
	\arrow[from=2-16, to=3-17]
	\arrow[from=2-18, to=2-16]
	\arrow[from=2-18, to=3-19]
	\arrow[dashed, from=3-5, to=2-6]
	\arrow[from=3-5, to=4-6]
	\arrow[from=3-7, to=3-5]
	\arrow[from=3-7, to=4-8]
	\arrow[dashed, from=3-15, to=2-16]
	\arrow[from=3-15, to=4-16]
	\arrow[from=3-17, to=2-18]
	\arrow[from=3-17, to=3-15]
	\arrow[from=3-17, to=4-18]
	\arrow[from=3-19, to=3-17]
	\arrow[from=3-19, to=4-20]
	\arrow[dashed, from=4-4, to=3-5]
	\arrow[from=4-4, to=5-5]
	\arrow[from=4-6, to=3-7]
	\arrow[from=4-6, to=4-4]
	\arrow[from=4-6, to=5-7]
	\arrow[from=4-8, to=4-6]
	\arrow[from=4-8, to=5-9]
	\arrow[dashed, from=4-14, to=3-15]
	\arrow[from=4-14, to=5-15]
	\arrow[from=4-16, to=3-17]
	\arrow[from=4-16, to=4-14]
	\arrow[from=4-16, to=5-17]
	\arrow[from=4-18, to=3-19]
	\arrow[from=4-18, to=4-16]
	\arrow[from=4-18, to=5-19]
	\arrow[from=4-20, to=4-18]
	\arrow[from=4-20, to=5-21]
	\arrow[dashed, from=5-3, to=4-4]
	\arrow[from=5-3, to=6-4]
	\arrow[from=5-5, to=4-6]
	\arrow[from=5-5, to=5-3]
	\arrow[from=5-5, to=6-6]
	\arrow[from=5-7, to=4-8]
	\arrow[from=5-7, to=5-5]
	\arrow[from=5-7, to=6-8]
	\arrow[from=5-9, to=5-7]
	\arrow[from=5-9, to=6-10]
	\arrow[dashed, from=5-13, to=4-14]
	\arrow[from=5-13, to=6-14]
	\arrow[from=5-15, to=4-16]
	\arrow[from=5-15, to=5-13]
	\arrow[from=5-15, to=6-16]
	\arrow[from=5-17, to=4-18]
	\arrow[from=5-17, to=5-15]
	\arrow[from=5-17, to=6-18]
	\arrow[from=5-19, to=4-20]
	\arrow[from=5-19, to=5-17]
	\arrow[from=5-19, to=6-20]
	\arrow[from=5-21, to=5-19]
	\arrow[from=5-21, to=6-22]
	\arrow[from=6-4, to=5-5]
	\arrow[dashed, no head, from=6-4, to=6-1]
	\arrow[from=6-4, to=7-5]
	\arrow[from=6-6, to=5-7]
	\arrow[from=6-6, to=6-4]
	\arrow[from=6-6, to=7-7]
	\arrow[from=6-8, to=5-9]
	\arrow[from=6-8, to=6-6]
	\arrow[from=6-8, to=7-9]
	\arrow[from=6-10, to=6-8]
	\arrow[from=6-10, to=7-11]
	\arrow[from=6-14, to=5-15]
	\arrow[dashed, no head, from=6-14, to=6-10]
	\arrow[from=6-14, to=7-15]
	\arrow[from=6-16, to=5-17]
	\arrow[from=6-16, to=6-14]
	\arrow[from=6-16, to=7-17]
	\arrow[from=6-18, to=5-19]
	\arrow[from=6-18, to=6-16]
	\arrow[from=6-18, to=7-19]
	\arrow[from=6-20, to=5-21]
	\arrow[from=6-20, to=6-18]
	\arrow[from=6-20, to=7-21]
	\arrow[from=6-22, to=6-20]
	\arrow[dashed, no head, from=6-22, to=6-25]
	\arrow[from=6-22, to=7-23]
	\arrow[from=7-5, to=6-6]
	\arrow[from=7-5, to=8-6]
	\arrow[from=7-7, to=6-8]
	\arrow[from=7-7, to=7-5]
	\arrow[from=7-7, to=8-8]
	\arrow[from=7-9, to=6-10]
	\arrow[from=7-9, to=7-7]
	\arrow[from=7-9, to=8-10]
	\arrow[from=7-11, to=7-9]
	\arrow[from=7-15, to=6-16]
	\arrow[from=7-15, to=8-16]
	\arrow[from=7-17, to=6-18]
	\arrow[from=7-17, to=7-15]
	\arrow[from=7-17, to=8-18]
	\arrow[from=7-19, to=6-20]
	\arrow[from=7-19, to=7-17]
	\arrow[from=7-19, to=8-20]
	\arrow[from=7-21, to=6-22]
	\arrow[from=7-21, to=7-19]
	\arrow[from=7-21, to=8-22]
	\arrow[from=7-23, to=7-21]
	\arrow[from=8-6, to=7-7]
	\arrow[from=8-6, to=9-7]
	\arrow[from=8-8, to=7-9]
	\arrow[from=8-8, to=8-6]
	\arrow[from=8-8, to=9-9]
	\arrow[dashed, from=8-10, to=7-11]
	\arrow[from=8-10, to=8-8]
	\arrow[from=8-16, to=7-17]
	\arrow[from=8-16, to=9-17]
	\arrow[from=8-18, to=7-19]
	\arrow[from=8-18, to=8-16]
	\arrow[from=8-18, to=9-19]
	\arrow[from=8-20, to=7-21]
	\arrow[from=8-20, to=8-18]
	\arrow[from=8-20, to=9-21]
	\arrow[dashed, from=8-22, to=7-23]
	\arrow[from=8-22, to=8-20]
	\arrow[from=9-7, to=8-8]
	\arrow[from=9-7, to=10-8]
	\arrow[dashed, from=9-9, to=8-10]
	\arrow[from=9-9, to=9-7]
	\arrow[from=9-17, to=8-18]
	\arrow[from=9-17, to=10-18]
	\arrow[from=9-19, to=8-20]
	\arrow[from=9-19, to=9-17]
	\arrow[from=9-19, to=10-20]
	\arrow[dashed, from=9-21, to=8-22]
	\arrow[from=9-21, to=9-19]
	\arrow[dashed, from=10-8, to=9-9]
	\arrow[from=10-18, to=9-19]
	\arrow[from=10-18, to=11-19]
	\arrow[dashed, from=10-20, to=9-21]
	\arrow[from=10-20, to=10-18]
	\arrow[dashed, from=11-19, to=10-20]
\end{tikzcd}

\caption{Glued $\mathcal{A}_n$-quivers for $n=5,6$.}
\label{Fig14}
\end{figure}

    Note that the dotted line in the center indicates the line along which the triangles are glued. On the quiver above, there are multiple vertices corresponding to each vertex $X_{i,j}$. We unfold each vertex $X_{i,j}$ to $\widetilde{U_{i,j}}$ and $U_{i,j}$ to obtain the following quiver (See also Figure \ref{Fig2}):

    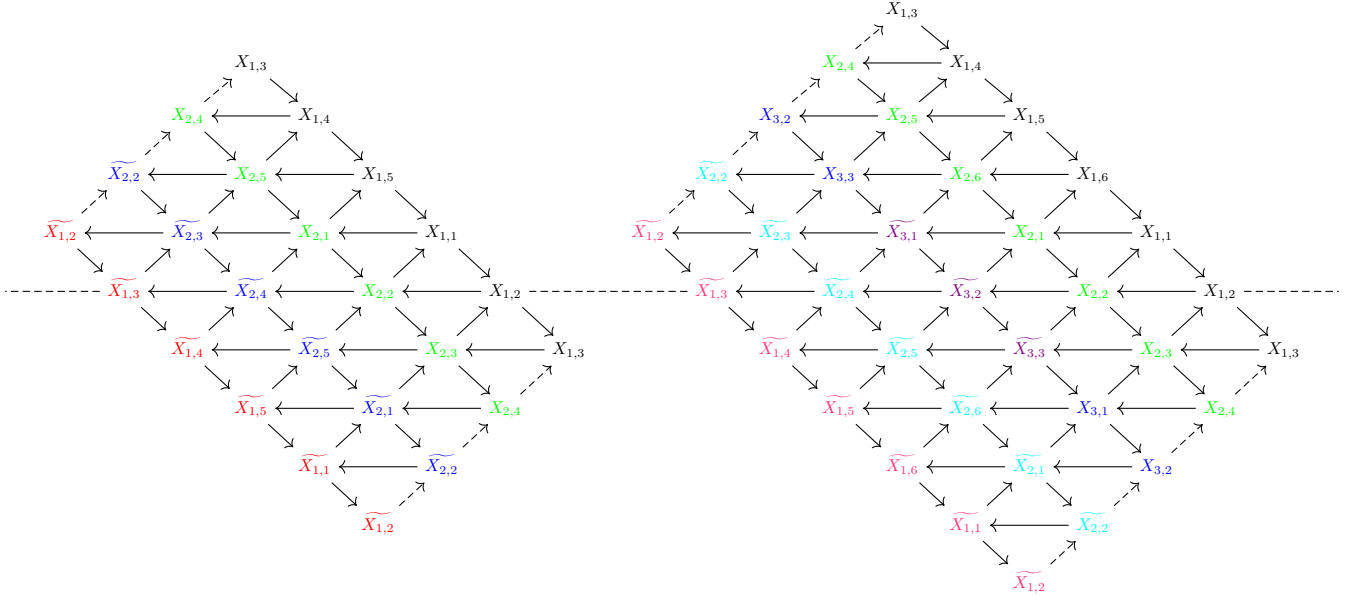
\begin{figure}[H]
    \centering
	\begin{tikzcd}[scale cd=0.6,column sep = tiny, row sep = small]
	&&&&&&&&&&&&&&&& {U_{1,3}} \\
	&&&&& {U_{1,3}} &&&&&&&&&& \textcolor{rgb,255:red,0;green,255;blue,0}{{U_{2,4}}} && {U_{1,4}} \\
	&&&& \textcolor{rgb,255:red,0;green,255;blue,0}{{U_{2,4}}} && {U_{1,4}} &&&&&&&& \textcolor{rgb,255:red,0;green,0;blue,255}{{{{U_{3,2}}}}} && \textcolor{rgb,255:red,0;green,255;blue,0}{{U_{2,5}}} && {U_{1,5}} \\
	&&& \textcolor{rgb,255:red,0;green,0;blue,255}{{\widetilde{U_{2,2}}}} && \textcolor{rgb,255:red,0;green,255;blue,0}{{U_{2,5}}} && {U_{1,5}} &&&&&& \textcolor{rgb,255:red,0;green,255;blue,255}{{\widetilde{{U_{2,2}}}}} && \textcolor{rgb,255:red,0;green,0;blue,255}{{U_{3,3}}} && \textcolor{rgb,255:red,0;green,255;blue,0}{{U_{2,6}}} && {U_{1,6}} \\
	&& \textcolor{rgb,255:red,255;green,0;blue,0}{{\widetilde{U_{1,2}}}} && \textcolor{rgb,255:red,0;green,0;blue,255}{{\widetilde{U_{2,3}}}} && \textcolor{rgb,255:red,0;green,255;blue,0}{{U_{2,1}}} && {U_{1,1}} &&&& \textcolor{rgb,255:red,255;green,51;blue,116}{{\widetilde{U_{1,2}}}} && \textcolor{rgb,255:red,0;green,255;blue,255}{{\widetilde{U_{2,3}}}} && \textcolor{rgb,255:red,128;green,0;blue,128}{{\widetilde{U_{3,1}}}} && \textcolor{rgb,255:red,0;green,255;blue,0}{{{{U_{2,1}}}}} && {{{U_{1,1}}}} \\
	{} &&& \textcolor{rgb,255:red,255;green,0;blue,0}{{\widetilde{U_{1,3}}}} && \textcolor{rgb,255:red,0;green,0;blue,255}{{\widetilde{U_{2,4}}}} && \textcolor{rgb,255:red,0;green,255;blue,0}{{U_{2,2}}} && {U_{1,2}} &&&& \textcolor{rgb,255:red,255;green,51;blue,116}{{\widetilde{U_{1,3}}}} && \textcolor{rgb,255:red,0;green,255;blue,255}{{\widetilde{U_{2,4}}}} && \textcolor{rgb,255:red,128;green,0;blue,128}{{\widetilde{{U_{3,2}}}}} && \textcolor{rgb,255:red,0;green,255;blue,0}{{{{U_{2,2}}}}} && {{{U_{1,2}}}} &&& {} \\
	&&&& \textcolor{rgb,255:red,255;green,0;blue,0}{{\widetilde{U_{1,4}}}} && \textcolor{rgb,255:red,0;green,0;blue,255}{{\widetilde{U_{2,5}}}} && \textcolor{rgb,255:red,0;green,255;blue,0}{{U_{2,3}}} && {U_{1,3}} &&&& \textcolor{rgb,255:red,255;green,51;blue,116}{{\widetilde{U_{1,4}}}} && \textcolor{rgb,255:red,0;green,255;blue,255}{{\widetilde{U_{2,5}}}} && \textcolor{rgb,255:red,128;green,0;blue,128}{{\widetilde{{U_{3,3}}}}} && \textcolor{rgb,255:red,0;green,255;blue,0}{{{{U_{2,3}}}}} && {U_{1,3}} \\
	&&&&& \textcolor{rgb,255:red,255;green,0;blue,0}{{\widetilde{U_{1,5}}}} && \textcolor{rgb,255:red,0;green,0;blue,255}{{\widetilde{U_{2,1}}}} && \textcolor{rgb,255:red,0;green,255;blue,0}{{U_{2,4}}} &&&&&& \textcolor{rgb,255:red,255;green,51;blue,116}{{\widetilde{U_{1,5}}}} && \textcolor{rgb,255:red,0;green,255;blue,255}{{\widetilde{{U_{2,6}}}}} && \textcolor{rgb,255:red,0;green,0;blue,255}{{{{U_{3,1}}}}} && \textcolor{rgb,255:red,0;green,255;blue,0}{{U_{2,4}}} \\
	&&&&&& \textcolor{rgb,255:red,255;green,0;blue,0}{{\widetilde{U_{1,1}}}} && \textcolor{rgb,255:red,0;green,0;blue,255}{{\widetilde{U_{2,2}}}} &&&&&&&& \textcolor{rgb,255:red,255;green,51;blue,116}{{\widetilde{{U_{1,6}}}}} && \textcolor{rgb,255:red,0;green,255;blue,255}{{\widetilde{{U_{2,1}}}}} && \textcolor{rgb,255:red,0;green,0;blue,255}{{{{U_{3,2}}}}} \\
	&&&&&&& \textcolor{rgb,255:red,255;green,0;blue,0}{{\widetilde{U_{1,2}}}} &&&&&&&&&& \textcolor{rgb,255:red,255;green,51;blue,116}{{\widetilde{{U_{1,1}}}}} && \textcolor{rgb,255:red,0;green,255;blue,255}{{\widetilde{{U_{2,2}}}}} \\
	&&&&&&&&&&&&&&&&&& \textcolor{rgb,255:red,255;green,51;blue,116}{{\widetilde{U_{1,2}}}}
	\arrow[from=1-17, to=2-18]
	\arrow[from=2-6, to=3-7]
	\arrow[dashed, from=2-16, to=1-17]
	\arrow[from=2-16, to=3-17]
	\arrow[from=2-18, to=2-16]
	\arrow[from=2-18, to=3-19]
	\arrow[dashed, from=3-5, to=2-6]
	\arrow[from=3-5, to=4-6]
	\arrow[from=3-7, to=3-5]
	\arrow[from=3-7, to=4-8]
	\arrow[dashed, from=3-15, to=2-16]
	\arrow[from=3-15, to=4-16]
	\arrow[from=3-17, to=2-18]
	\arrow[from=3-17, to=3-15]
	\arrow[from=3-17, to=4-18]
	\arrow[from=3-19, to=3-17]
	\arrow[from=3-19, to=4-20]
	\arrow[dashed, from=4-4, to=3-5]
	\arrow[from=4-4, to=5-5]
	\arrow[from=4-6, to=3-7]
	\arrow[from=4-6, to=4-4]
	\arrow[from=4-6, to=5-7]
	\arrow[from=4-8, to=4-6]
	\arrow[from=4-8, to=5-9]
	\arrow[dashed, from=4-14, to=3-15]
	\arrow[from=4-14, to=5-15]
	\arrow[from=4-16, to=3-17]
	\arrow[from=4-16, to=4-14]
	\arrow[from=4-16, to=5-17]
	\arrow[from=4-18, to=3-19]
	\arrow[from=4-18, to=4-16]
	\arrow[from=4-18, to=5-19]
	\arrow[from=4-20, to=4-18]
	\arrow[from=4-20, to=5-21]
	\arrow[dashed, from=5-3, to=4-4]
	\arrow[from=5-3, to=6-4]
	\arrow[from=5-5, to=4-6]
	\arrow[from=5-5, to=5-3]
	\arrow[from=5-5, to=6-6]
	\arrow[from=5-7, to=4-8]
	\arrow[from=5-7, to=5-5]
	\arrow[from=5-7, to=6-8]
	\arrow[from=5-9, to=5-7]
	\arrow[from=5-9, to=6-10]
	\arrow[dashed, from=5-13, to=4-14]
	\arrow[from=5-13, to=6-14]
	\arrow[from=5-15, to=4-16]
	\arrow[from=5-15, to=5-13]
	\arrow[from=5-15, to=6-16]
	\arrow[from=5-17, to=4-18]
	\arrow[from=5-17, to=5-15]
	\arrow[from=5-17, to=6-18]
	\arrow[from=5-19, to=4-20]
	\arrow[from=5-19, to=5-17]
	\arrow[from=5-19, to=6-20]
	\arrow[from=5-21, to=5-19]
	\arrow[from=5-21, to=6-22]
	\arrow[from=6-4, to=5-5]
	\arrow[dashed, no head, from=6-4, to=6-1]
	\arrow[from=6-4, to=7-5]
	\arrow[from=6-6, to=5-7]
	\arrow[from=6-6, to=6-4]
	\arrow[from=6-6, to=7-7]
	\arrow[from=6-8, to=5-9]
	\arrow[from=6-8, to=6-6]
	\arrow[from=6-8, to=7-9]
	\arrow[from=6-10, to=6-8]
	\arrow[from=6-10, to=7-11]
	\arrow[from=6-14, to=5-15]
	\arrow[dashed, no head, from=6-14, to=6-10]
	\arrow[from=6-14, to=7-15]
	\arrow[from=6-16, to=5-17]
	\arrow[from=6-16, to=6-14]
	\arrow[from=6-16, to=7-17]
	\arrow[from=6-18, to=5-19]
	\arrow[from=6-18, to=6-16]
	\arrow[from=6-18, to=7-19]
	\arrow[from=6-20, to=5-21]
	\arrow[from=6-20, to=6-18]
	\arrow[from=6-20, to=7-21]
	\arrow[from=6-22, to=6-20]
	\arrow[dashed, no head, from=6-22, to=6-25]
	\arrow[from=6-22, to=7-23]
	\arrow[from=7-5, to=6-6]
	\arrow[from=7-5, to=8-6]
	\arrow[from=7-7, to=6-8]
	\arrow[from=7-7, to=7-5]
	\arrow[from=7-7, to=8-8]
	\arrow[from=7-9, to=6-10]
	\arrow[from=7-9, to=7-7]
	\arrow[from=7-9, to=8-10]
	\arrow[from=7-11, to=7-9]
	\arrow[from=7-15, to=6-16]
	\arrow[from=7-15, to=8-16]
	\arrow[from=7-17, to=6-18]
	\arrow[from=7-17, to=7-15]
	\arrow[from=7-17, to=8-18]
	\arrow[from=7-19, to=6-20]
	\arrow[from=7-19, to=7-17]
	\arrow[from=7-19, to=8-20]
	\arrow[from=7-21, to=6-22]
	\arrow[from=7-21, to=7-19]
	\arrow[from=7-21, to=8-22]
	\arrow[from=7-23, to=7-21]
	\arrow[from=8-6, to=7-7]
	\arrow[from=8-6, to=9-7]
	\arrow[from=8-8, to=7-9]
	\arrow[from=8-8, to=8-6]
	\arrow[from=8-8, to=9-9]
	\arrow[dashed, from=8-10, to=7-11]
	\arrow[from=8-10, to=8-8]
	\arrow[from=8-16, to=7-17]
	\arrow[from=8-16, to=9-17]
	\arrow[from=8-18, to=7-19]
	\arrow[from=8-18, to=8-16]
	\arrow[from=8-18, to=9-19]
	\arrow[from=8-20, to=7-21]
	\arrow[from=8-20, to=8-18]
	\arrow[from=8-20, to=9-21]
	\arrow[dashed, from=8-22, to=7-23]
	\arrow[from=8-22, to=8-20]
	\arrow[from=9-7, to=8-8]
	\arrow[from=9-7, to=10-8]
	\arrow[dashed, from=9-9, to=8-10]
	\arrow[from=9-9, to=9-7]
	\arrow[from=9-17, to=8-18]
	\arrow[from=9-17, to=10-18]
	\arrow[from=9-19, to=8-20]
	\arrow[from=9-19, to=9-17]
	\arrow[from=9-19, to=10-20]
	\arrow[dashed, from=9-21, to=8-22]
	\arrow[from=9-21, to=9-19]
	\arrow[dashed, from=10-8, to=9-9]
	\arrow[from=10-18, to=9-19]
	\arrow[from=10-18, to=11-19]
	\arrow[dashed, from=10-20, to=9-21]
	\arrow[from=10-20, to=10-18]
	\arrow[dashed, from=11-19, to=10-20]
\end{tikzcd}
    \caption{Doubled $\mathcal A_5$-quiver and $\mathcal A_6$-quiver.}
    \label{Fig15}
\end{figure}
    We refer to this quiver as the \textit{doubled $\mathcal{A}_n$-quiver}, or simply the $\mathcal{A}_n^{\mathrm{dbl}}$-quiver. Note that the doubled $\mathcal A_n$-quiver is a $2:1$ quiver covering of the standard $\mathcal{A}_n$-quiver. 
\end{defn}

\begin{exmp}
    An example of the doubled $\mathcal A_3$-quiver is as follows:

    $$
    \begin{tikzcd}[scale cd=0.7,column sep = tiny, row sep = small]
	{\widetilde{U_{1,3}}} &&&& {\widetilde{U_{1,1}}} \\
	&& {U_{1,2}} \\
	& {U_{1,1}} && {U_{1,3}} \\
	\\
	&& {\widetilde{U_{1,2}}}
	\arrow[curve={height=-12pt}, from=1-1, to=1-5]
	\arrow[from=1-1, to=3-2]
	\arrow[from=1-5, to=2-3]
	\arrow[curve={height=-18pt}, from=1-5, to=5-3]
	\arrow[from=2-3, to=1-1]
	\arrow[from=2-3, to=3-4]
	\arrow[from=3-2, to=2-3]
	\arrow[from=3-2, to=5-3]
	\arrow[from=3-4, to=1-5]
	\arrow[from=3-4, to=3-2]
	\arrow[curve={height=-18pt}, from=5-3, to=1-1]
	\arrow[from=5-3, to=3-4]
\end{tikzcd}
$$
    The doubled $\mathcal A_3$-quiver is identical to the quiver from a triangulation of the 4-punctured sphere; see\textcolor{black}{\cite{25}}.
\end{exmp}

Next, we consider natural cycle mutations on each main cycle of the doubled quiver.

\begin{prop} \label{prop341}
Consider the doubled $\mathcal A_n$-quiver. On each cycle $U_{i,1} \to U_{i,2} \to \cdots \to U_{i,n} \to U_{i,1}$ and $\widetilde{U_{i,1}} \to \widetilde{U_{i,2}} \to \cdots \to \widetilde{U_{i,n}} \to \widetilde{U_{i,1}}$ for $i \ne n/2$, there are natural cycle mutations $f_i$ and $\widetilde{f_i}$ associated to the cycles (\ref{3.4}). Specifically,
$$f_i :=  \mu_{j^i_1}\circ \cdots \circ\mu_{j^i_{n-1}}\circ \pi_{j^i_{n-1},j^i_{n}}\circ \mu_{j^i_{n-1}} \circ \cdots \circ \mu_{j^i_1}$$
and
$$\widetilde{f_i} := \mu_{\widetilde{j^i_1}}\circ \cdots \circ\mu_{\widetilde{j^i_{n-1}}}\circ \pi_{\widetilde{j^i_{n-1}},\widetilde{j^i_{n}}}\circ \mu_{\widetilde{j^i_{n-1}}} \circ \cdots \circ \mu_{\widetilde{j^i_1}}$$
where $\left\{j^i_1, j^i_2, \cdots, j^i_n\right\}$ is a permutation of the set of indices of the variables $\left\{U_{i,j}\right\}_{j=1}^n$ and $\left\{\widetilde{j^i_1}, \widetilde{j^i_2}, \cdots, \widetilde{j^i_n}\right\}$ is a permutation of the set of indices of the variables $\left\{\widetilde{U_{i,j}}\right\}_{j=1}^n$.

When $n$ is even, the cycle $\widetilde{U_{{n \over 2},1}} \to \widetilde{U_{{n \over 2},2}} \to \cdots \to \widetilde{U_{{n \over 2},{n \over 2}}} \to U_{{n \over 2},1} \to  U_{{n \over 2},2} \to \cdots \to U_{{n \over 2},{n \over 2}} \to \widetilde{U_{{n \over 2},1}}$ induces a cycle mutation $f_{{n \over 2}} = \widetilde{f_{{n \over 2}}}$. Specifically,
$$f_{{n \over 2}} = \widetilde{f_{{n \over 2}}} :=  \mu_{j^{n/2}_1}\circ \cdots \circ\mu_{j^{n/2}_{n-1}}\circ \pi_{j^{n/2}_{n-1},j^{n/2}_{n}}\circ \mu_{j^{n/2}_{n-1}} \circ \cdots \circ \mu_{j^{n/2}_1}$$
where $\left\{j^{n/2}_1, j^{n/2}_2, \cdots, j^{n/2}_n\right\}$ is a permutation of the set of indices of the variables $\left\{U_{{n \over 2},j}\right\}_{j=1}^{n/2} \cup \left\{\widetilde{U_{{n \over 2},j}}\right\}_{j=1}^{n/2}$.
\end{prop}
\begin{proof} A direct calculation shows that the cycles satisfy condition \eqref{3.3} and are chordless. \end{proof}

\begin{defn}
    Let $m = \lfloor n/2 \rfloor$. We define $s_i := f_i \widetilde{f_i}$ for all $1 \le i < m$. For the final index $m$, we set:$$s_m := \begin{cases}
f_m & \text{if } n \text{ is even}, \\
f_m \widetilde{f_m} f_m & \text{if } n \text{ is odd}.
\end{cases}
$$
\end{defn}

\begin{prop} \label{prop345}
    Let $\mathbf W_n$ be the group generated by the birational transformations $s_i^*$ on the rational function field $\mathcal K(\mathcal X_{|\mathcal A_n^{dbl}|})$. Then $\mathbf W_n$ is a quotient of the Weyl group of type $B_{\lfloor n/2 \rfloor}$, where the generators $s_i^*$ correspond to the simple reflections.
\end{prop}
\begin{proof}
    Let $m = \lfloor n/2 \rfloor$. By\textcolor{black}{\cite[Theorem 4.4]{19}}, we have the following relations:
    $$(f_if_{i+1})^3 = (\widetilde{f_i}\widetilde{f_{i+1}})^3 = \mathrm{id} \quad \text{for } 1 \le i < m,$$
    and $$(f_m\widetilde{f_m})^3 = \mathrm{id}.$$ 
    
    These relations are the same as those for two adjacent cycles in the ladder configuration considered in\cite{19}.
    
    All $f_i$ and $\widetilde{f_i}$ are clearly involutions ($f_i^2 = (\widetilde{f_i})^2 = \mathrm{id}$). Furthermore, we have $(f_if_j)^2 = (\widetilde{f_i}\widetilde{f_j})^2 = \mathrm{id}$ for $|i-j| \ge 2$ and $(f_i\widetilde{f_j})^2 = \mathrm{id}$ for all pairs $(i,j) \neq (m,m)$.

    Hence, we can directly verify the following relations: $(s_m s_{m-1})^4 = \mathrm{id}$; $(s_i s_{i-1})^3 = \mathrm{id}$ for $2 \le i \le m-1$; $s_i^2 = \mathrm{id}$; and $(s_i s_j)^2 = \mathrm{id}$ for $|i-j| \ge 2$ (since these are direct calculations, we omit the details). These relations establish that $\mathbf W_n$ is a quotient of the Weyl group of type ${B}_{\lfloor n/2 \rfloor}$. \end{proof}

We are now ready to define the \textbf{birational Weyl group action} on the \(\mathcal A_n\)-quiver.

\begin{defn} Let \(g\) be the involution of \(\mathcal K(\mathcal X_{|\mathcal A_n^{\mathrm{dbl}}|})\) that exchanges \(U_{i,j}\) and \(\widetilde {U_{i,j}}\). Let \( \mathbb Q_+ \bigl(\mathcal X_{|\mathcal A_n^{\mathrm{dbl}}|}\bigr) \) denote the semifield of positive rational functions in the initial cluster variables. We define the diagonal projection \( \operatorname{pr}: \mathbb Q_+ \bigl(\mathcal X_{|\mathcal A_n^{\mathrm{dbl}}|}\bigr) \longrightarrow \mathbb Q_+ \bigl(\mathcal X_{|\mathcal A_n|}\bigr) \) by
\[ \operatorname{pr}(U_{i,j})=X_{i,j}, \qquad \operatorname{pr}(\widetilde {U_{i,j}})=X_{i,j}. \]

This projection is well defined because positive rational functions admit subtraction-free expressions, whose denominators remain nonzero after the projection. \end{defn}

\begin{lem} For each \(i\), the action \(s_i^*\) commutes with \(g\) on \(\mathcal K(\mathcal X_{|\mathcal A_n^{\mathrm{dbl}}|})\); that is,
\( g\circ s_i^*=s_i^*\circ g. \) \end{lem}

\begin{proof}
For \(i\leq n/2\) when \(n\) is even, and for \(i\leq\lfloor n/2\rfloor-1\) when \(n\) is odd, the assertion follows immediately from the symmetry of the doubled \(\mathcal A_n\)-quiver interchanging \(U_{k,l}\) and \(\widetilde {U_{k,l}}\). For \(i=\lfloor n/2\rfloor\) when \(n\) is odd, it follows from the relation \( f_i\widetilde f_i f_i = \widetilde f_i f_i\widetilde f_i \) together with the same symmetry.
\end{proof}

\begin{defn} \label{defn348}
Since \(s_i^*\) is a cluster transformation, it restricts to an automorphism of the positive semifield. Moreover, \[ \operatorname{pr}\circ g=\operatorname{pr} \qquad\text{and}\qquad g\circ s_i^*=s_i^*\circ g. \] 

It follows that \[ \operatorname{pr}\bigl(s_i^*(U_{k,l})\bigr) = \operatorname{pr}\bigl(s_i^*(\widetilde {U_{k,l}})\bigr). \] 

Therefore, \(s_i^*\) induces a positive birational transformation of the standard \(\mathcal A_n\)-quiver, defined by \[ \overline{s}_i^*(X_{k,l}) := \operatorname{pr}\bigl(s_i^*(U_{k,l})\bigr). \] 

The transformation \(\overline{s}_i^*\) extends uniquely to an automorphism of \(\mathcal K(\mathcal X_{|\mathcal A_n|})\). We refer to the induced action as the \textit{birational Weyl group action} on the \(\mathcal A_n\)-quiver. By abuse of notation, we write \(s_i^*\) instead of \(\overline{s}_i^*\). \end{defn}

Let \(L_n\) be the group generated by the birational actions \(\tau_i^*\) defined in~\eqref{3.4} on \(\mathcal K(\mathcal X_{|\mathcal A_n|})\). Then \(L_n\subseteq\mathbf W_n\). When \(n\) is even, the symmetry of the doubled quiver gives \(s_i^*=\tau_i^*\) for \(1\leq i\leq\lfloor n/2\rfloor\), and hence \(L_n=\mathbf W_n\). When \(n\) is odd, the same equality holds for \(1\leq i<\lfloor n/2\rfloor\), whereas \(s_{\lfloor n/2\rfloor}^*\notin L_n\). Therefore, \(L_n\subsetneq\mathbf W_n\).

\begin{prop}\label{prop349}
The birational Weyl group action is compatible with the Poisson structure on the standard $\mathcal{A}_n$-quiver.
\end{prop}

\begin{proof}
Recall that $\operatorname{pr}$ is a $2:1$ covering map satisfying $\operatorname{pr}(\widetilde{U_{i,j}})=X_{i,j}$ and $\operatorname{pr}(U_{i,j})=X_{i,j}$, and that $g$ is the involution exchanging $\widetilde{U_{i,j}}$ and $U_{i,j}$ in $\mathcal{K}(\mathcal{X}_{|\mathcal{A}_n^{dbl}|})$. For the initial variables, the Poisson bracket on the standard quiver satisfies
\[
\{X_{i,j},X_{k,l}\}
=
\{\operatorname{pr}(\widetilde{U_{i,j}}),
  \operatorname{pr}(\widetilde{U_{k,l}})\}
=
\operatorname{pr}\left(
\{\widetilde{U_{i,j}},\widetilde{U_{k,l}}\}
+
\{\widetilde{U_{i,j}},g(\widetilde{U_{k,l}})\}
\right).
\]
The second equality, which we call the descent formula, extends by the Leibniz and quotient rules to every rational expression whose denominator remains nonzero under the diagonal projection. In particular, it applies to the positive rational expressions arising from the cluster transformations \(s_i^*\). Moreover, their Poisson brackets remain well defined under the diagonal projection, because taking a Poisson bracket introduces only powers of the original denominators, which remain nonzero after projection. 

We need to show that $s_r^*\{X_{i,j},X_{k,l}\}=\{s_r^*X_{i,j},s_r^*X_{k,l}\}.$ Applying $s_r^*$ to the bracket, we obtain
\begin{align*}
s_r^*\{X_{i,j},X_{k,l}\}
&=s_r^*\left(
\operatorname{pr}\left(
\{\widetilde{U_{i,j}},\widetilde{U_{k,l}}\}
+
\{\widetilde{U_{i,j}},g(\widetilde{U_{k,l}})\}
\right)\right)\\
&=\operatorname{pr}\left(
s_r^*\left(
\{\widetilde{U_{i,j}},\widetilde{U_{k,l}}\}
+
\{\widetilde{U_{i,j}},g(\widetilde{U_{k,l}})\}
\right)\right)
&&(\text{since } s_r^*\circ\operatorname{pr}=\operatorname{pr}\circ s_r^*)\\
&=\operatorname{pr}\left(
\{s_r^*\widetilde{U_{i,j}},s_r^*\widetilde{U_{k,l}}\}
+
\{s_r^*\widetilde{U_{i,j}},s_r^*g(\widetilde{U_{k,l}})\}
\right)
&&(\text{since } s_r^* \text{ is Poisson})\\
&=\operatorname{pr}\left(
\{s_r^*\widetilde{U_{i,j}},s_r^*\widetilde{U_{k,l}}\}
+
\{s_r^*\widetilde{U_{i,j}},g(s_r^*\widetilde{U_{k,l}})\}
\right)
&&(\text{since } s_r^*\circ g=g\circ s_r^*)\\
&=
\{\operatorname{pr}(s_r^*\widetilde{U_{i,j}}),
  \operatorname{pr}(s_r^*\widetilde{U_{k,l}})\}
&&(\text{by the descent formula})\\
&=
\{s_r^*(\operatorname{pr}\widetilde{U_{i,j}}),
  s_r^*(\operatorname{pr}\widetilde{U_{k,l}})\}
&&(\text{since } \operatorname{pr}\circ s_r^*=s_r^*\circ\operatorname{pr})\\
&=
\{s_r^*X_{i,j},s_r^*X_{k,l}\}.
\end{align*}
Thus the relation holds for all initial cluster variables. Since the Poisson bracket satisfies the Leibniz rule and $s_r^*$ is an algebra homomorphism, $s_r^*$ preserves the bracket for all rational functions in the field.
\end{proof}

\subsection{Casimirs of the $\mathcal A_n$-quiver} \label{Ch3.5}
In this subsection, we introduce the Casimirs of the $\mathcal{A}_n$-quiver and characterize the action of the Weyl group on these elements.

\begin{defn} \label{defn351}
    A Casimir of the quiver is defined as a Casimir element (Definition~\ref{defn222}) with respect to the Poisson bracket given in \eqref{2.6}.
\end{defn}

\begin{prop} \label{prop352}
    For each $i$, define $\mathcal{C}_i := \prod_{j=1}^{N_i} X_{i,j}$ and $\mathcal{K}_i := \prod_{j=i}^{\lfloor n/2 \rfloor} \mathcal{C}_{j}$, where $N_i$ is the length of the $i$th cycle. The elements $\mathcal{K}_i$ are Casimirs of the $\mathcal{A}_n$-quiver and are globally defined monomials (i.e., they remain monomials in any cluster).
\end{prop}
\begin{proof}
    Let $I$ denote the set of all variable indices, and let $I_i$ denote the subset of indices belonging to the $i$th main cycle. We then have
    \begin{equation} \label{3.5}
        \sum_{j \in I_i}\epsilon_{kj} = 0 \quad \text{for all } k \in I.
    \end{equation}
    This condition implies that the elements $\mathcal{C}_i$ are Casimirs and globally defined monomials. Since each $\mathcal{K}_i$ is a product of the $\mathcal{C}_j$, we conclude that the $\mathcal{K}_i$ are also Casimirs and globally defined monomials.
\end{proof}

Our primary interest lies in the action of the Weyl group on the set $\{\mathcal{K}_i^{\pm 1}\}_{i=1}^{\lfloor n/2 \rfloor}$. The following proposition describes this action.

\begin{prop}\label{prop353} Let \(m=\lfloor n/2\rfloor\). For \(i<m\), the transformation \(s_i^*\) interchanges \(\mathcal K_i\) and \(\mathcal K_{i+1}\), while \(s_m^*\) inverts the last element: \( s_m^*(\mathcal K_m)=\mathcal K_m^{-1}. \) Thus, the Weyl group action on the Casimirs is faithful. Together with Proposition~\ref{prop345}, this implies that $\mathbf W_n$ is isomorphic to the Weyl group of type $B_m$. \end{prop}

\begin{proof}
We have the identity
\[
\prod_{j=1}^{N_i} Y_{i,j} = \prod_{j=1}^{N_i} X_{i,j},
\]
where $N_i$ is the length of the $i$th cycle. This identity follows immediately from the relation $Y_{i,j}=X_{i,j}\frac{F_{i,j-1}}{F_{i,j}}$. From this, the images under the Weyl group action can be verified by direct calculation.
\end{proof}

\begin{cor}\label{cor354}
The Weyl group acts by permutations and inversions on the set $\{\mathcal{K}^{\pm1}_i\}_{i=1}^{m}$. Therefore, any Weyl group invariant element generated by these Casimirs can be expressed as a polynomial in the elementary symmetric functions of $\{\mathcal{K}_i+\mathcal{K}_i^{-1}\}_{i=1}^m$.
\end{cor}

\section{Formal Geodesic Functions and Their Invariance under the Weyl Group Action}
\subsection{Formal Geodesic Functions} \label{ch4.1}
The main goal of this section is to introduce the formal geodesic functions and prove their invariance under the birational Weyl-group action.
\begin{defn} \label{defn411}
    Each entry $\mathbb A_{ij}$ with $i < j$ is called a \textit{formal geodesic function}. In particular, each element $\mathbb A_{i,i+1}$ for $i \in \left\{1,\cdots,n-1\right\}$ is called an \textit{elementary formal geodesic function}.
\end{defn}
\begin{figure}[H]
\centering
\includegraphics[width=0.65\linewidth]{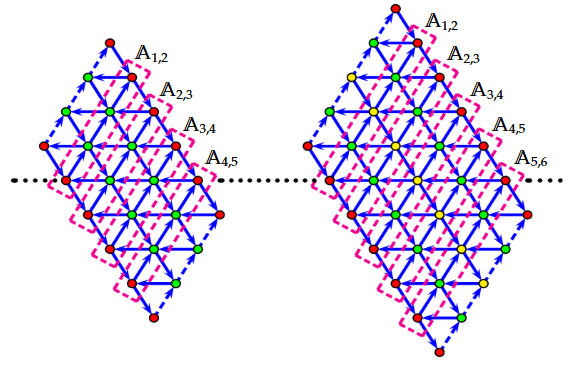}
\caption{Main cycles and elementary geodesic functions are represented on the union of two triangles, each of which is a copy of the $\mathcal A_n$-quiver for $n=5$ and $6$. The dotted line in the center indicates the line of the triangle gluing. The cluster variables in the red dashed block correspond to the formal geodesic functions and each collection of colored vertices represents the main cycle. The indices of cluster variables are denoted in Figure~\ref{Fig14}.} 
\label{Fig16}
\end{figure}

A convenient way to express elementary geodesic functions is the process depicted in \textbf{Figure~\ref{Fig16}}. 

We take another copy of an $\mathcal A_n$-quiver, flip the quiver down, and glue the quiver to the original one along the bottom side of the latter in a way that amalgamated variables on the sides of the two quivers match. The resulting quiver is equal to the glued $\mathcal{A}_n$-quiver discussed in Sections \ref{Ch3.2} and \ref{Ch3.4}.

In the quiver, the red dashed block corresponds to $\mathbb A_{i,i+1}$ and each collection of colored vertices represents the main cycle; see\textcolor{black}{\cite{1}} for more details.
\begin{rem}\textit{(Formal geodesic functions and classical geodesic functions)}
    \begin{enumerate}
        \item As stated in the introduction, there exists a Poisson map $\phi_n: \mathcal T_{\lfloor {n-1\over 2} \rfloor, \text{par}(n)} \to \mathcal A_n$. When $n=3$ and $n=4$, these two spaces are locally isomorphic because their dimensions coincide.
        
        Recall that the dimension of the Teichmüller space corresponds to the number of edges in its fat graph, given by $6g + 3s - 6$. Using the relation $n = 2g + s$, this simplifies to $3n - 6$. On the other hand, the dimension of $\mathcal A_n$ is $n(n-1)/2$. Since these dimensions are equal for $n=3$ and $n=4$, the formal geodesic functions can be identified with the geodesic functions in these cases.

        \item Formal geodesic functions satisfy the \textbf{Bondal Poisson bracket (\ref{1.4})} by \textit{Theorem~\ref{thm314}}, which is a generalization of the Goldman bracket. However, formal geodesic functions are not equal to geodesic functions in Teichmüller space since the matrix $\mathbb A$ does not enjoy the rank condition in general.
    \end{enumerate}
\end{rem}

We will sometimes omit the word \textit{formal} when the meaning is clear from the context.

Let us define the bracket notation $\langle\,\cdots\,\rangle$ as follows:
$$ \langle a_1, a_2, \dots, a_n \rangle := (a_1 a_2 \cdots a_n)^{\frac{1}{2}} \left( 1 + \frac{1}{a_1} + \frac{1}{a_1 a_2} + \cdots + \frac{1}{a_1 a_2 \cdots a_n} \right). $$

Let $m = \lfloor \frac{n}{2} \rfloor$. The elementary geodesic function $\mathbb{A}_{n-1,n}$ is defined along a certain path of variables (enclosed by the red dashed block in Figure \ref{Fig16}).

When $n$ is even, the path is given by:
$$ X_{1,1} \longrightarrow X_{2,1} \longrightarrow \cdots \longrightarrow X_{m,1} \longrightarrow X_{m-1,m} \longrightarrow \cdots \longrightarrow X_{2,3} \longrightarrow X_{1,2}. $$
The corresponding elementary geodesic function is:
$$ \mathbb{A}_{n-1,n} = \langle X_{1,1}, X_{2,1}, \dots, X_{m,1}, X_{m-1,m}, \dots, X_{2,3}, X_{1,2} \rangle, $$
which we denote by $\langle X_{j,1}|_{j=1}^{m}, X_{m-j,m-j+1}|_{j=1}^{m-1} \rangle$.

When $n$ is odd, the path is given by:
$$ X_{1,1} \longrightarrow X_{2,1} \longrightarrow \cdots \longrightarrow X_{m,1} \longrightarrow X_{m,m+1} \longrightarrow X_{m-1,m} \longrightarrow \cdots \longrightarrow X_{2,3} \longrightarrow X_{1,2}. $$
In this case, the corresponding elementary geodesic function is:
$$ \mathbb{A}_{n-1,n} = \langle X_{1,1}, X_{2,1}, \dots, X_{m,1}, X_{m,m+1}, X_{m-1,m}, \dots, X_{2,3}, X_{1,2} \rangle, $$
which we denote by $\langle X_{j,1}|_{j=1}^{m}, X_{m-j,m-j+1}|_{j=0}^{m-1} \rangle$.
\begin{figure}[H]
\centering
\includegraphics[width=0.9\linewidth]{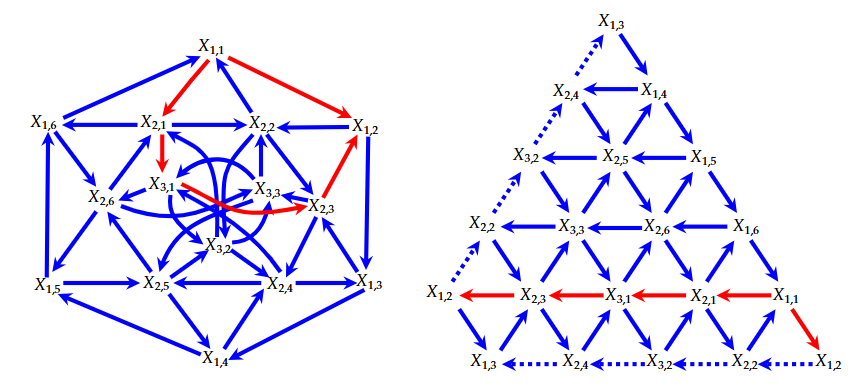}
\caption{Two expressions of the $\mathcal A_6$-quiver. The red arrows represent the element $\mathbb A_{5,6}.$} 
\label{Fig17}
\end{figure}

\begin{exmp}\textit{(Examples of elementary geodesic functions)}

    From Figure~\ref{Fig16}, we have 
    $$\mathbb A_{1,2} = \langle X_{1,3},X_{2,3},X_{3,3},X_{2,5},X_{1,4} \rangle,\text{ }\mathbb A_{2,3} = \langle X_{1,4},X_{2,4},X_{3,1},X_{2,6},X_{1,5} \rangle,$$
    $$\mathbb A_{3,4} = \langle X_{1,5},X_{2,5},X_{3,2},X_{2,1},X_{1,6} \rangle ,\text{ }\mathbb A_{4,5} = \langle X_{1,6},X_{2,6},X_{3,3},X_{2,2},X_{1,1} \rangle,\text{ and }\mathbb A_{5,6} = \langle X_{1,1},X_{2,1},X_{3,1},X_{2,3},X_{1,2} \rangle.$$
\end{exmp}

\begin{defn}\label{defn414}
Let $\mathcal{O}(\mathcal{A}_n)$ denote the Poisson algebra generated by the 
elementary geodesic functions. It contains all formal geodesic functions by 
the Bondal Poisson bracket relations in \eqref{1.4}.
\end{defn}

We prove that each matrix entry of $\mathbb A$ has the form $M^{1/2}L$ where $M$ is a monomial and $L$ is a Laurent polynomial in any cluster. 

\begin{lem} \label{lem416}
   Each $\mathbb A_{i,i+1}$ has the form $M^{1/2}L$ where $M$ is a monomial and $L$ is a Laurent polynomial in any cluster.
\end{lem}
\begin{proof}
    We first verify the lemma for $\mathbb{A}_{n-1,n} = \langle X_{j,1}\big|_{j=1}^{m}, X_{m-j,m-j+1}\big|_{j=1}^{m-1} \rangle$ when $n = 2m$. 

Consider the \textit{framing} of the $\mathcal{A}_n$-quiver, obtained by adding a frozen vertex $\hat{v}$ and an arrow $\hat{v} \leftarrow v$ for each mutable vertex $v$. We denote this framed quiver by $\hat{\mathcal{A}}_n$. Let $(X_{i,j})$ and $(A_{i,j})$ be the mutable $\mathcal{X}$- and $\mathcal{A}$-variables, respectively, and let $(\hat{X}_{i,j})$ and $(\hat{A}_{i,j})$ be the corresponding frozen variables on the $\hat{\mathcal{A}}_n$-quiver. The pullback under the map $p_{\hat{\mathcal{A}}_n}^*$ satisfies the following relation:
\begin{equation}
 p_{\hat{\mathcal{A}}_n}^* \left( \prod_{j=1}^{m} X_{j,1} \prod_{j=1}^{m-1} X_{m-j,m-j+1} \right) = \left( \prod_{j=1}^{m} \hat{A}_{j,1} \prod_{j=1}^{m-1} \hat{A}_{m-j,m-j+1} \right) \left( \frac{A_{1,2}}{A_{1,1}} \right)^2.
\end{equation}

To establish this, let $I$ be the index set of all variables in the $\mathcal{A}_n$-quiver, and let $I_G \subset I$ be the subset of indices corresponding to the variables that define $\mathbb{A}_{n-1,n}$ (see Figure \ref{Fig17}). This relation follows from the property of the exchange matrix elements $\epsilon_{ij}$:
\begin{equation}
\sum_{i \in I_G} \epsilon_{ij} = 0 \quad \text{for all } j \in I \setminus \{1,2\},
\end{equation}
while $\sum_{i \in I_G} \epsilon_{i1} = -2$ and $\sum_{i \in I_G} \epsilon_{i2} = 2$. Note that the indices $1$ and $2$ correspond to the variables $X_{1,1}$ and $X_{1,2}$, respectively \eqref{3.2}. 

Consequently, by the definition of the symbol $\langle \cdots \rangle$, we obtain:
\begin{equation}
p_{\hat{\mathcal{A}}_n}^*(\mathbb{A}_{n-1,n}) = \left(\prod_{j=1}^{m}\hat{A}_{j,1}\prod_{j=1}^{m-1}\hat{A}_{m-j,m-j+1}\right)^{1/2} \cdot f,
\end{equation}
where $f$ is a Laurent polynomial in the $\mathcal{A}$-variables of the $\hat{\mathcal{A}}_n$-quiver.

For the framed quiver, the determinant of the exchange matrix is $1$. Therefore, $f \in \mathcal{O}(\mathcal{A}_{|\hat{\mathcal{A}}_n|})$ if and only if $f$ remains a Laurent polynomial in every cluster adjacent to the initial one. Since $\mathcal{A}$-variable mutations applied to $p_{\hat{\mathcal{A}}_n}^*(\mathbb{A}_{n-1,n})$ only affect the factor $f$, checking the adjacent clusters is sufficient to conclude that $p_{\hat{\mathcal{A}}_n}^*(\mathbb{A}_{n-1,n})$ takes the form $M^{1/2}L$ in any cluster, where $M$ is a monomial and $L$ is a Laurent polynomial.

Furthermore, as the determinant is $1$, the map $p_{\hat{\mathcal{A}}_n}^*$ is an isomorphism of tori. This allows us to pull this result back to $\mathbb{A}_{n-1,n}$ using the compatibility relation \eqref{2.4}:
\begin{equation}
p_{\hat{\mathcal{A}}_n}^* \circ(\tau^\mathcal{X})^* = (\tau^\mathcal{A})^*\circ p_{\hat{\mathcal{A}}_n}^*,
\end{equation}
where $\tau$ denotes a cluster transformation. Thus, it suffices to verify the adjacent clusters to ensure that $\mathbb{A}_{n-1,n}$ maintains the form $M^{1/2}L$ in every cluster.

To check adjacent clusters, we consider the following four cases (see Figure \ref{Fig17}):

\begin{enumerate}
    \item \textbf{A mutation of a vertex adjacent to the path defining the geodesic function}:

    Without loss of generality, we can assume the mutation is of $X_{1,n}$. Then, $(\mu^\mathcal X_{X_{1,n}})^*\left(X_{1,1}\right) = X'_{1,1}\left({X'_{1,n} \over 1+X'_{1,n}}\right)$, and $(\mu^\mathcal X_{X_{1,n}})^*\left(X_{2,1}\right) = X'_{2,1}\left(1 + X'_{1,n}\right)$. This induces
    $$(\mu^\mathcal X_{X_{1,n}})^*\left(1+{1 \over X_{1,1}} + {1 \over X_{1,1}X_{2,1}}\right) = 1 + {1+X'_{1,n} \over X'_{1,1}X'_{1,n}} + {1 \over X'_{1,1}X'_{1,n}X'_{2,1}}$$
    and 
    $$(\mu^\mathcal X_{X_{1,n}})^*\left(X_{1,1}X_{2,1}\right)= X'_{1,1}X'_{1,n}X'_{2,1}.$$
    Hence, the geodesic function still has the form $M^{1/2}L$ under the mutation.

    \item \textbf{A mutation of a vertex in the path for the geodesic function except $X_{1,1}$ and $X_{1,2}$}:

   Without loss of generality, we can assume the mutation is of $X_{2,1}$. Then, $(\mu^\mathcal X_{X_{2,1}})^*\left(X_{1,1}\right) = X'_{1,1}\left(1+X'_{2,1}\right)$, $(\mu^\mathcal X_{X_{2,1}})^*\left(X_{2,1}\right) = (X'_{2,1})^{-1}$, and $(\mu^\mathcal X_{X_{2,1}})^*\left(X_{3,1}\right) = X'_{3,1}\left({X'_{2,1} \over 1+X'_{2,1}}\right)$. Note that other variables are invariant. Thus, we have
   $$(\mu^\mathcal X_{X_{2,1}})^*\left(1 + {1\over X_{1,1}} + {1\over X_{1,1}X_{2,1}} + {1\over X_{1,1}X_{2,1}X_{3,1}}\right) = 1 + {1 \over (X'_{2,1}+1)X'_{1,1}} + {X'_{2,1} \over (X'_{2,1}+1)X'_{1,1}} +{1 \over X'_{1,1}X'_{3,1}}$$
   $$= 1+ {1 \over X'_{1,1}} + {1 \over X'_{1,1}X'_{3,1}}$$
   and 
   $$(\mu^\mathcal X_{X_{2,1}})^*\left(X_{1,1}X_{2,1}X_{3,1}\right) = X'_{1,1}X'_{3,1}.$$
    Hence, the geodesic function still has the form $M^{1/2}L$ under the mutation.
    \item \textbf{A mutation of $X_{1,1}$}:
    
    $(\mu^\mathcal X_{X_{1,1}})^*\left(X_{1,1}\right)=(X'_{1,1})^{-1},$ $(\mu^\mathcal X_{X_{1,1}})^*\left(X_{1,2}\right)=(X'_{1,2})\left({X'_{1,1} \over 1+X'_{1,1}}\right),$ and $(\mu^\mathcal X_{X_{1,1}})^*\left(X_{2,1}\right) = (X'_{2,1})\left({X'_{1,1} \over 1+X'_{1,1}}\right)$. 
    
    Then, we get
   $$(\mu^\mathcal X_{X_{1,1}})^*\left(1 + {1\over X_{1,1}} + {1\over X_{1,1}X_{2,1}} + \cdots + {1 \over \prod_{j=1}^{m}X_{j,1}\prod_{j=1}^{m-1}X_{m-j,m-j+1}}\right)$$
   $$= 1 + X'_{1,1} + {1 + X'_{1,1} \over X'_{2,1}} + \cdots + {(1+X'_{1,1})^2 \over \prod_{j=1}^{m}X'_{j,1}\prod_{j=1}^{m-1}X'_{m-j,m-j+1}}$$
   and 
   $$(\mu^\mathcal X_{X_{1,1}})^*\left(\prod_{j=1}^{m}X_{j,1}\prod_{j=1}^{m-1}X_{m-j,m-j+1}\right) = {\prod_{j=1}^{m}X'_{j,1}\prod_{j=1}^{m-1}X'_{m-j,m-j+1} \over (1 + X'_{1,1})^2}.$$
    Hence, $(\mu^\mathcal X_{X_{1,1}})^*\left(\mathbb A_{n-1,n}\right)$ is equal to the
    $$ \sqrt{\prod_{j=1}^{m}X'_{j,1}\prod_{j=1}^{m-1}X'_{m-j,m-j+1} \over (1 + X'_{1,1})^2}\cdot{\left(1 + X'_{1,1} + {1 + X'_{1,1} \over X'_{2,1}} + \cdots + {(1+X'_{1,1})^2 \over \prod_{j=1}^{m}X'_{j,1}\prod_{j=1}^{m-1}X'_{m-j,m-j+1}}\right)}.$$
    Thus,  the geodesic function still has the form $M^{1/2}L$ under the mutation. Note that the factor $1 + X'_{1,1}$ cancels out.
    
    \item \textbf{A mutation of $X_{1,2}$}: This is similar to the previous case.
    
\end{enumerate}
Thus, $\mathbb A_{n-1,n}$ has the form $M^{1/2}L$ in any cluster of $\mathcal X_{|\hat{\mathcal A_n}|}$. This implies $\mathbb A_{n-1,n}$ has the form $M^{1/2}L$ in any cluster of $\mathcal X_{|{\mathcal A_n}|}$ because $\mathcal X$-variable mutations do not affect frozen variables $\hat{ X}_{i,j}$. For the other elementary geodesic functions, the proofs are the same, which completes the proof.\end{proof}

\begin{thm} \label{thm416}Any formal geodesic function $\mathbb A_{i,j}$ has the form $M^{1/2}L$, where $M$ is a monomial and $L$ is a Laurent polynomial in any cluster.
\end{thm}
\begin{proof} The statement holds for elementary geodesic functions by the previous lemma. We proceed by induction on \(j-i\). Using the Bondal Poisson relation~\eqref{1.4}, we have \[ \mathbb A_{i,j} = \frac12\mathbb A_{i,j-1}\mathbb A_{j-1,j} - \{\mathbb A_{i,j-1},\mathbb A_{j-1,j}\}. \] 

By induction, write \[ \mathbb A_{i,j-1}=M_1^{1/2}L_1, \qquad \mathbb A_{j-1,j}=M_2^{1/2}L_2, \] where \(M_1,M_2\) are monomials and \(L_1,L_2\) are Laurent polynomials in the chosen cluster. Since the cluster variables are log-canonical, the Leibniz rule implies \[ \{M_1^{1/2}L_1,M_2^{1/2}L_2\} = (M_1M_2)^{1/2}L_3 \] for some Laurent polynomial \(L_3\). Therefore, \[ \mathbb A_{i,j} = (M_1M_2)^{1/2} \left(\frac12L_1L_2-L_3\right), \] which has the required form. Since the same argument applies in every cluster, the result follows. \end{proof}

\subsection{Invariance of Formal Geodesic Functions under the Weyl Group Action}

In Fock--Goncharov higher Teichmüller theory, the Fock--Goncharov--Shen Weyl group acts on the moduli space of framed local systems on an oriented surface by permuting the eigenvalues of the monodromy operators around the punctures. As a result, the traces of these monodromy operators are invariant under the 
Weyl group action.

Analogously, the birational Weyl group action considered in this paper preserves the formal geodesic functions. We first prove the following lemma.

\begin{lem} \label{lem421}
Let $Y_{i,j} = X_{i,j}{F_{i,j-1} \over F_{i,j}}$ and $F_{i,j} := 1 + X_{i,j} + X_{i,j}X_{i,j-1} + \cdots + X_{i,j}X_{i,j-1} \cdots X_{i,j-N_i+2}$. Then, 
    \[ \frac1{Y_{i,j}}+\frac{Y_{i,j-1}}{X_{i,j}} = 1+\frac1{X_{i,j}}, \qquad i\in\{1,\ldots,\lfloor n/2\rfloor\}. \]
\end{lem}
\begin{proof}
Substituting the definition of $Y_{i,j}$ into the claim, the equation is equivalent to verifying the following identity:
\[
F_{i,j} + X_{i,j-1}F_{i,j-2} = F_{i,j-1}(1+X_{i,j}).
\]
We begin by expanding the left-hand side using the definition of $F_{i,j}$:
\begin{align*}
F_{i,j} + X_{i,j-1}F_{i,j-2} &= \left( 1 + X_{i,j} + X_{i,j}X_{i,j-1} + \cdots + X_{i,j}X_{i,j-1} \cdots X_{i,j-N_i+2} \right) \\
&\quad + X_{i,j-1}\left( 1 + X_{i,j-2} + X_{i,j-2}X_{i,j-3} + \cdots + X_{i,j-2}X_{i,j-3} \cdots X_{i,j-2-N_i+2} \right).
\end{align*}
We now regroup the terms to factor $(1+X_{i,j})$. By utilizing the cyclic property $X_{i,j-N_i} = X_{i,j}$, the expression can be rewritten as:
\begin{align*}
&= (1+X_{i,j}) + X_{i,j-1} \bigg[ X_{i,j}(1+X_{i,j-2} + \cdots + X_{i,j-2}X_{i,j-3}\cdots X_{i,j-N_i+2}) \\
&\quad \quad + (1 + X_{i,j-2} + \cdots + X_{i,j-2}X_{i,j-3}\cdots X_{i,j-N_i+2}) \\
&\quad \quad + X_{i,j-2}X_{i,j-3} \cdots X_{i,j-N_i+1}(1+X_{i,j-N_i}) \bigg] \\[1em]
&= (1+X_{i,j}) + X_{i,j-1} \bigg[ (1+X_{i,j})(1+X_{i,j-2} + \cdots + X_{i,j-2}X_{i,j-3}\cdots X_{i,j-N_i+2}) \\
&\quad \quad + X_{i,j-2}X_{i,j-3} \cdots X_{i,j-N_i+1}(1+X_{i,j}) \bigg].
\end{align*}
Finally, factoring out $(1+X_{i,j})$ from the entire expression induces:
\begin{align*}
F_{i,j} + X_{i,j-1}F_{i,j-2} &= (1+X_{i,j}) \left( 1 + X_{i,j-1} + X_{i,j-1}X_{i,j-2} + \cdots + X_{i,j-1}\cdots X_{i,j-1-N_i+2} \right) \\[0.5em]
&= (1+X_{i,j})F_{i,j-1}.
\end{align*}
This proves the identity.
\end{proof}

\begin{thm} \label{thm422}
    Consider the Weyl group $\mathbf W_n$ generated by $\left\{s_1^*, \dots, s^*_{\lfloor n/2 \rfloor}\right\}$. The Weyl group action preserves formal geodesic functions. Note that this preservation depends on a choice of square root; otherwise, the formal geodesic functions are preserved only up to sign.
\end{thm}
\begin{proof}

    For the cases $n=3$ and $n=4$, we have verified the result via computer software. Hence, we may assume $n \ge 5$.
    
    It suffices to consider only one elementary geodesic function to prove the theorem for all elementary geodesic functions. This follows from the cyclic symmetry of the $\mathcal A_n$-quiver (\textbf{Proposition~\ref{prop321}}) and from the fact that the birational actions do not depend on the choice of the order of cluster mutations (\textbf{Theorem~\ref{thm331}}). 
    
    Consider an elementary geodesic function $\left\langle X_{1,1}, X_{2,1}, \cdots, X_{m,1}, X_{m-1,m}, \cdots, X_{2,3}, X_{1,2} \right\rangle$ when $n = 2m$. The proof is divided into the following four cases according to the choice of the main cycle:
\begin{enumerate}
    \item $1$st cycle: In this case, we consider the action $s_1^* = \tau_{1}^*$.
    
    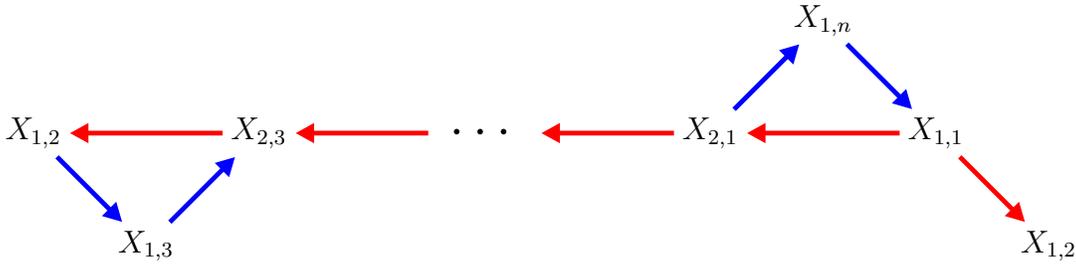
\begin{figure}[H]
    \centering
\begin{tikzpicture}[
    % 1. Global Scale
    scale=1.5,
    % 2. Styles
    vertex/.style={anchor=center, inner sep=3pt, font=\large},
    redarrow/.style={->, color=red, line width=1.8pt, -{Triangle[length=2.5mm, width=2.5mm]}},
    bluearrow/.style={->, color=blue, line width=1.8pt, -{Triangle[length=2.5mm, width=2.5mm]}}
]

    % --- 3. Place Vertices on a Grid ---
    
    % === LEFT BLOCK ===
    \node[vertex] (X_1_2L) at (0, 0)  {$X_{1,2}$};
    \node[vertex] (X_2_3)  at (2, 0)  {$X_{2,3}$};
    \node[vertex] (X_1_3)  at (1, -1) {$X_{1,3}$};

    % === CENTER GAP ===
    % Instead of a line, we place a Node containing dots at x=4
    \node[scale=2] (Dots) at (4, 0) {$\cdots$};

    % We define coordinates for where the arrows stop/start near the dots
    % Arrow entering from Right stops at x=4.5
    \coordinate (Arrow_Out_R) at (4.5, 0);
    
    % Arrow emerging to Left starts at x=3.5
    \coordinate (Arrow_In_L) at (3.5, 0);

    % === RIGHT BLOCK ===
    \node[vertex] (X_2_1)  at (6, 0) {$X_{2,1}$};
    \node[vertex] (X_1_1)  at (8, 0) {$X_{1,1}$};
    \node[vertex] (X_1_n)  at (7, 1) {$X_{1,n}$};
    \node[vertex] (X_1_2R) at (9, -1) {$X_{1,2}$};

    % --- 4. Draw Edges ---

    % === RED SPINE (Right-to-Left Flow) ===
    
    % 1. Far Right
    \draw[redarrow] (X_1_1)   -- (X_1_2R);
    \draw[redarrow] (X_1_1)   -- (X_2_1);
    
    % 2. Arrow pointing to the Dots
    \draw[redarrow] (X_2_1)   -- (Arrow_Out_R);
    
    % 3. Arrow emerging from the Dots
    \draw[redarrow] (Arrow_In_L) -- (X_2_3);

    % 4. Far Left
    \draw[redarrow] (X_2_3)   -- (X_1_2L);

    % === BLUE CYCLES ===
    
    % Left Triangle
    \draw[bluearrow] (X_1_2L) -- (X_1_3);
    \draw[bluearrow] (X_1_3)  -- (X_2_3);

    % Right Triangle
    \draw[bluearrow] (X_2_1)  -- (X_1_n);
    \draw[bluearrow] (X_1_n)  -- (X_1_1);

\end{tikzpicture}
        \caption{This expresses adjacency relations between the elementary geodesic function and the 1st cycle. Here, the red arrows indicate the path that defines an elementary geodesic function. Note that the other variables of the first cycle which are not shown in this figure, are not connected to the red path. }
    \label{Fig18}
    \end{figure}

    See \textbf{Figure~\ref{Fig18}}: among variables $X_{j,1}$ and $X_{m-j,m-j+1}$ that define the elementary geodesic function, only $X_{1,1},X_{2,1},X_{2,3}$ and $X_{1,2}$ are adjacent to the $1$st cycle, so other variables are invariant under $\tau_{1}^*.$

    For $X_{1,1}$ and $X_{1,2}$, direct calculation gives
    $$\tau_{1}^*(X_{1,1}) = {X_{1,1}\over Y_{1,1}Y_{1,n}} \text{ and } \tau_{1}^*(X_{1,2}) = {X_{1,2}\over Y_{1,2}Y_{1,1}}$$
    since the variables lie in the first cycle (\textbf{Theorem~\ref{thm333}}).

    For $X_{2,1},$ see vertices on the right-hand side of the \textbf{Figure \ref{Fig18}}. $X_{2,1}$ is only adjacent to $X_{1,1}$ and $X_{1,N}$ among the variables of the cycle. Therefore, we have $c_{n+1} = (0,0,\cdots,0,1)$ because $(\epsilon_{n+1j})_{j \in I_1} = (-1,0,\cdots,0,1).$ Recall that the index of $X_{2,1}$ is $n+1$. Thus, $\tau_{1}^*(X_{2,1}) = X_{2,1}Y_{1,N}$ by \textbf{Theorem~\ref{thm333}}.

    For $X_{2,3},$ see the vertices on the left-hand side of \textbf{Figure \ref{Fig18}}. $X_{2,3}$ is only adjacent to $X_{1,2}$ and $X_{1,3}$ among the variables of the cycle. Therefore, $c_{n+3} = (0,1,0,\cdots,0,0)$ because $(\epsilon_{n+3j})_{j \in I_1} = (0,1,-1,0\cdots,0,0).$ Note that the index of $X_{2,3}$ is $n+3$. Thus, $\tau_{1}^*(X_{2,3}) = X_{2,3}Y_{1,2}$ by \textbf{Theorem~\ref{thm333}}.
    
    As a result, $\tau_{1}^*(\left\langle X_{1,1}, X_{2,1}, \cdots, X_{m,1}, X_{m-1,m}, \cdots, X_{2,3}, X_{1,2} \right\rangle)$ is, 
    $$\sqrt{{X_{1,1}\over Y_{1,1}Y_{1,n}}X_{2,1}Y_{1,n} X_{3,1} \cdots X_{m,1}X_{m-1,m} \cdots X_{3,4}X_{2,3}Y_{1,2}{X_{1,2}\over Y_{1,2}Y_{1,1}}}\left(1 + {Y_{1,1}Y_{1,n} \over X_{1,1}} + {Y_{1,1} \over X_{1,1}X_{2,1}} + {Y_{1,1} \over X_{1,1}X_{2,1}X_{3,1}} +\right.$$
    $$\left.\cdots + {Y_{1,1} \over \prod_{j=1}^m X_{j,1} \prod_{j=1}^{m-3}X_{m-j,m-j+1}}\left(1+{1 \over X_{2,3}Y_{1,2}} + {Y_{1,1} \over X_{2,3}X_{1,2}}\right)\right)$$
    $$= \sqrt{\prod_{j=1}^m X_{j,1} \prod_{j=1}^{m-1}X_{m-j,m-j+1}} {1 \over \sqrt{(Y_{1,1})^2}}Y_{1,1}\left({1 \over Y_{1,1}} + {Y_{1,n} \over X_{1,1}} + {1 \over X_{1,1}X_{2,1}} + \cdots + {1 \over \prod_{j=1}^m X_{j,1} \prod_{j=1}^{m-2}X_{m-j,m-j+1}}\right.$$
    $$\left.\left({1 \over Y_{1,2}} + {Y_{1,1} \over X_{1,2}}\right)\right) = \sqrt{\prod_{j=1}^m X_{j,1} \prod_{j=1}^{m-1}X_{m-j,m-j+1}} \left(1+ {1 \over X_{1,1}} + \cdots +  {1 \over \prod_{j=1}^m X_{j,1} \prod_{j=1}^{m-2}X_{m-j,m-j+1}}\left(1 + {1 \over X_{1,2}}\right)\right)$$
    $$= \left\langle X_{1,1}, X_{2,1}, \cdots, X_{m,1}, X_{m-1,m}, \cdots, X_{2,3}, X_{1,2} \right\rangle,$$
    where the second equality is by \textbf{Lemma~\ref{lem421}.}
    \item $i$th cycle where $2 \le i \le m-2$: In this case, we consider the action $s_i^*=\tau_{i}^*.$
    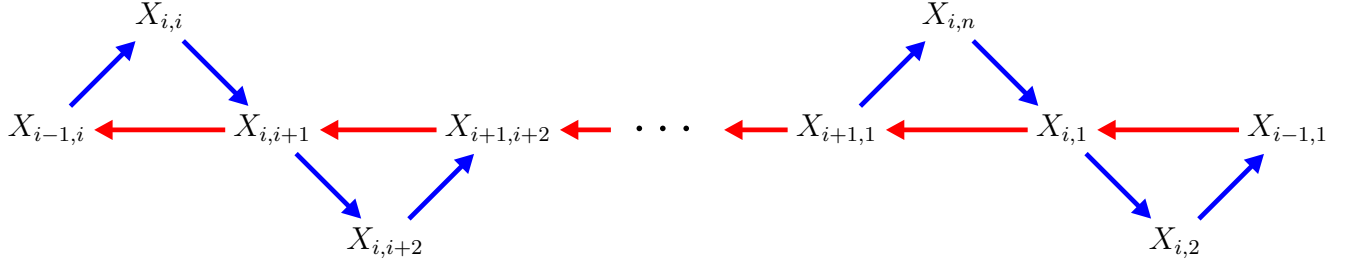
\begin{figure}[H]
        \begin{tikzpicture}[
    % 1. Global Scale
    scale=1.5,
    % 2. Styles
    vertex/.style={anchor=center, inner sep=3pt, font=\large},
    redarrow/.style={->, color=red, line width=1.8pt, -{Triangle[length=2.5mm, width=2.5mm]}},
    bluearrow/.style={->, color=blue, line width=1.8pt, -{Triangle[length=2.5mm, width=2.5mm]}}
]

    % --- 3. Place Vertices on a Grid ---
    
    % === LEFT CLUSTER ===
    \node[vertex] (X_i-1_i)   at (0, 0) {$X_{i-1,i}$};
    \node[vertex] (X_i_i+1)   at (2, 0) {$X_{i,i+1}$};
    \node[vertex] (X_i+1_i+2) at (4, 0) {$X_{i+1,i+2}$};
    
    \node[vertex] (X_i_i)     at (1, 1) {$X_{i,i}$};      % Peak
    \node[vertex] (X_i_i+2)   at (3, -1) {$X_{i,i+2}$};   % Valley

    % === CENTER GAP (The "Dots" area) ===
    % The space is between x=4 and x=7.
    % We place the ellipsis in the exact center: x=5.5
    \node[scale=2] (Dots) at (5.5, 0) {$\cdots$};

    % Define stops/starts to create blank space
    % Arrow from Right stops at x=6.0
    \coordinate (Arrow_In_Stop) at (6.0, 0);

    % Arrow to Left starts at x=5.0
    \coordinate (Arrow_Out_Start) at (5.0, 0);

    % === RIGHT CLUSTER ===
    \node[vertex] (X_i+1_1)   at (7, 0) {$X_{i+1,1}$};
    \node[vertex] (X_i_1)     at (9, 0) {$X_{i,1}$};
    \node[vertex] (X_i-1_1)   at (11, 0) {$X_{i-1,1}$};
    
    \node[vertex] (X_i_n)     at (8, 1) {$X_{i,n}$};      % Peak
    \node[vertex] (X_i_2)     at (10, -1) {$X_{i,2}$};    % Valley

    % --- 4. Draw Edges ---
    
    % -- RED SPINE (Right to Left Flow) --
    
    % 1. Inside Right Cluster
    \draw[redarrow] (X_i-1_1) -- (X_i_1);
    \draw[redarrow] (X_i_1)   -- (X_i+1_1);
    
    % 2. THE CENTER CONNECTION
    % Pointing TOWARDS the dots (stops at 6.0)
    \draw[redarrow] (X_i+1_1) -- (Arrow_In_Stop);
    
    % Emerging FROM the dots (starts at 5.0)
    \draw[redarrow] (Arrow_Out_Start) -- (X_i+1_i+2);
    
    % 3. Inside Left Cluster
    \draw[redarrow] (X_i+1_i+2) -- (X_i_i+1);
    \draw[redarrow] (X_i_i+1)   -- (X_i-1_i);

    % -- BLUE TRIANGLE EDGES --

    % Left Cluster Triangles
    \draw[bluearrow] (X_i-1_i) -- (X_i_i);
    \draw[bluearrow] (X_i_i)   -- (X_i_i+1);
    \draw[bluearrow] (X_i_i+1) -- (X_i_i+2);
    \draw[bluearrow] (X_i_i+2) -- (X_i+1_i+2);

    % Right Cluster Triangles
    \draw[bluearrow] (X_i+1_1) -- (X_i_n);
    \draw[bluearrow] (X_i_n)   -- (X_i_1);
    \draw[bluearrow] (X_i_1)   -- (X_i_2);
    \draw[bluearrow] (X_i_2)   -- (X_i-1_1);

\end{tikzpicture}

        \caption{The quiver expresses adjacency relations between the elementary geodesic function and $i$th cycle.}
    \label{Fig19}
    \end{figure}

    By \textbf{Theorem~\ref{thm333}}, we have
    $$\tau_{i}^*(X_{i,1}) = {X_{i,1} \over Y_{i,1}Y_{i,n}} \text{ and } \tau_{i}^*(X_{i,i+1}) = {X_{i,i+1} \over Y_{i,i+1}Y_{i,i}},$$
    since these variables are in the $i$th cycle. Similar to \textit{Case 1}, we also get
    $$\tau_{i}^*(X_{i-1,1}) = X_{i-1,1}Y_{i,1}, \text{ }\tau_{i}^*(X_{i+1,1}) = X_{i+1,1}Y_{i,n}, \text{ }\tau_{i}^*(X_{i-1,i}) = X_{i-1,i}Y_{i,i},$$
    $$\text{ and } \tau_{i}^*(X_{i+1,i+2}) = X_{i+1,i+2}Y_{i,i+1}.$$

    The other variables are invariant under the $\tau_{i}^*$ because they are not adjacent to the cycle of $I_i$; See \textbf{Figure \ref{Fig19}}. Hence, $\tau_{i}^*(\left\langle X_{1,1}, X_{2,1}, \cdots, X_{m,1}, X_{m-1,m}, \cdots, X_{2,3}, X_{1,2} \right\rangle)$ is,
    $$\sqrt{X_{1,1}X_{2,1}\cdots X_{i-1,1}Y_{i,1}{X_{i,1}\over Y_{i,1}Y_{i,n}}X_{i+1,1}Y_{i,n} \cdots X_{i+1,i+2}Y_{i,i+1}{X_{i,i+1}\over Y_{i,i+1}Y_{i,i}}X_{i-1,i}Y_{i,i}\cdots X_{2,3}X_{1,2}}\left(1 + {1 \over X_{1,1}} \right.$$
    $$+ \cdots + {1 \over \prod_{j=1}^{i-2}X_{j,1}}\left(1 + {1 \over X_{i-1,1}Y_{i,1}} + {Y_{i,n} \over X_{i-1,1}X_{i,1}} + {1 \over X_{i-1,1}X_{i,1}X_{i+1,1}}\right) + \cdots +$$
    $${1 \over \prod_{j=1}^m X_{j,1} \prod_{j=1}^{m-i-2}X_{m-j,m-j+1}}\left(1 + {1 \over X_{i+1,i+2}Y_{i,i+1}} + {Y_{i,i} \over X_{i+1,i+2}X_{i,i+1}} + {1 \over X_{i+1,i+2}X_{i,i+1}X_{i-1,i}}\right) + \cdots + $$
    $$\left.{1 \over \prod_{j=1}^m X_{j,1} \prod_{j=1}^{m-1}X_{m-j,m-j+1}}\right) = \left\langle X_{1,1}, X_{2,1}, \cdots, X_{m,1}, X_{m-1,m}, \cdots, X_{2,3}, X_{1,2} \right\rangle,$$
    since an identity
    $${1 \over X_{i-1,1}Y_{i,1}} + {Y_{i,n} \over X_{i-1,1}X_{i,1}} = {1 \over X_{i-1,1}}\left(1+ {1 \over X_{i,1}}\right), $$
    $${1 \over X_{i+1,i+2}Y_{i,i+1}} + {Y_{i,i} \over X_{i+1,i+2}X_{i,i+1}} = {1 \over X_{i+1,i+2}}\left(1+ {1 \over X_{i,i+1}}\right)$$
    from \textbf{Lemma~\ref{lem421}}, and the other identity
    $$\sqrt{X_{1,1}X_{2,1}\cdots X_{i-1,1}Y_{i,1}{X_{i,1}\over Y_{i,1}Y_{i,n}}X_{i+1,1}Y_{i,n} \cdots X_{i+1,i+2}Y_{i,i+1}{X_{i,i+1}\over Y_{i,i+1}Y_{i,i}}X_{i-1,i}Y_{i,i}\cdots X_{2,3}X_{1,2}}$$
    $$= \sqrt{\prod_{j=1}^{m}X_{j,1}\prod_{j=1}^{m-1}X_{m-j,m-j+1}}.$$
    
    \item $m-1$th cycle: In this case, we consider the action $s_{m-1}^* = \tau_{m-1}^*.$
    
    \begin{figure}[H]

\begin{tikzpicture}[
    % 1. Global Scale
    scale=1.6,
    % 2. Styles
    vertex/.style={anchor=center, inner sep=3pt, font=\large},
    redarrow/.style={->, color=red, line width=1.8pt, -{Triangle[length=2.5mm, width=2.5mm]}},
    bluearrow/.style={->, color=blue, line width=1.8pt, -{Triangle[length=2.5mm, width=2.5mm]}}
]

    % --- 3. Place Vertices on a Grid ---
    
    % The Horizontal Spine (Left to Right at y=0)
    \node[vertex] (X_m-2_m-1) at (0, 0) {$X_{m-2,m-1}$};
    \node[vertex] (X_m-1_m)   at (2, 0) {$X_{m-1,m}$};
    \node[vertex] (X_m_1)     at (4, 0) {$X_{m,1}$};
    \node[vertex] (X_m-1_1)   at (6, 0) {$X_{m-1,1}$};
    \node[vertex] (X_m-2_1)   at (8, 0) {$X_{m-2,1}$};

    % The "Peaks" (Top Row, y=1)
    % Placed between the first two and the third/fourth spine nodes
    \node[vertex] (X_m-1_m-1) at (1, 1) {$X_{m-1,m-1}$};
    \node[vertex] (X_m-1_n)   at (5, 1) {$X_{m-1,n}$};

    % The "Valleys" (Bottom Row, y=-1)
    % Placed between the second/third and fourth/fifth spine nodes
    \node[vertex] (X_m-1_m+1) at (3, -1) {$X_{m-1,m+1}$};
    \node[vertex] (X_m-1_2)   at (7, -1) {$X_{m-1,2}$};

    % --- 4. Draw Edges ---
    
    % -- Red Spine Arrows (Flow: Right to Left) --
    \draw[redarrow] (X_m-2_1)   -- (X_m-1_1);
    \draw[redarrow] (X_m-1_1)   -- (X_m_1);
    \draw[redarrow] (X_m_1)     -- (X_m-1_m);
    \draw[redarrow] (X_m-1_m)   -- (X_m-2_m-1);

    % -- Blue Triangle Edges --
    
    % 1. First Peak (Left)
    \draw[bluearrow] (X_m-2_m-1) -- (X_m-1_m-1); % Up-Right
    \draw[bluearrow] (X_m-1_m-1) -- (X_m-1_m);   % Down-Right

    % 2. First Valley (Middle-Left)
    \draw[bluearrow] (X_m-1_m)   -- (X_m-1_m+1); % Down-Right
    \draw[bluearrow] (X_m-1_m+1) -- (X_m_1);     % Up-Right

    % 3. Second Peak (Middle-Right)
    \draw[bluearrow] (X_m_1)     -- (X_m-1_n);   % Up-Right
    \draw[bluearrow] (X_m-1_n)   -- (X_m-1_1);   % Down-Right

    % 4. Second Valley (Right)
    \draw[bluearrow] (X_m-1_1)   -- (X_m-1_2);   % Down-Right
    \draw[bluearrow] (X_m-1_2)   -- (X_m-2_1);   % Up-Right

\end{tikzpicture}
        \centering
        \caption{The quiver expresses adjacency relations between the elementary geodesic function and $m-1$th cycle.}
        \label{Fig20}
    \end{figure}
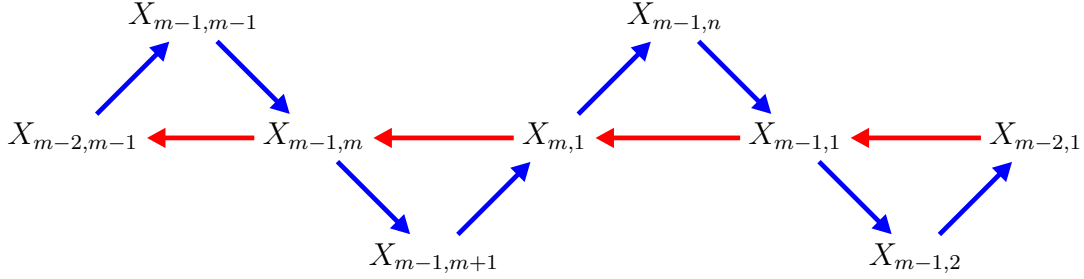

    By \textbf{Theorem~\ref{thm333}}, we have
    $$\tau_{m-1}^*(X_{m-1,1}) = {X_{m-1,1} \over Y_{m-1,1}Y_{m-1,n}}\text{ and }\tau_{m-1}^*(X_{m-1,m}) = {X_{m-1,m} \over Y_{m-1,m}Y_{m-1,m-1}}.$$
    Similar to \textit{Case 1}, we get
    $$\tau_{m-1}^*(X_{m-2,1}) = X_{m-2,1}Y_{m-1,1}, \text{ and }\tau_{m-1}^*(X_{m-2,m-1}) = X_{m-2,m-1}Y_{m-1,m-1}.$$

    $X_{m,1}$ is slightly different. From Figure~\ref{Fig20}, $(\epsilon_{n(m-1)+1,j})_{j \in I_{m-1}} = (-1,0,\cdots,0,1,-1,0,\cdots,0,1),$ so $c_{n(m-1)+1} = (0,\cdots,0,1,0,\cdots,0,1)$. This implies $\tau_{m-1}^*(X_{m,1}) = X_{m,1}Y_{m-1,m}Y_{m-1,n}$ by  \textbf{Theorem~\ref{thm333}.} Note that the index of $X_{m,1}$ is $n(m-1)+1$.

    From the equation above and \textbf{Lemma~\ref{lem421}}, we have
    $$\tau_{m-1}^*\left(\left\langle X_{1,1}, X_{2,1}, \cdots, X_{m,1}, X_{m-1,m}, \cdots, X_{2,3}, X_{1,2} \right\rangle\right) =\left\langle X_{1,1}, X_{2,1}, \cdots, X_{m,1}, X_{m-1,m}, \cdots, X_{2,3}, X_{1,2} \right\rangle.$$
    We omit the calculation since it is similar to the previous cases.

    \item $m$th cycle: In this case, we consider the action $s_m^* = \tau_{m}^*$.
    
    \begin{figure}[H]
\begin{tikzpicture}[
    % 1. Global Scale
    scale=1.5,
    % 2. Styles
    vertex/.style={anchor=center, inner sep=3pt, font=\large},
    redarrow/.style={->, color=red, line width=2pt, -{Triangle[length=2.5mm, width=2.5mm]}},
    bluearrow/.style={->, color=blue, line width=2pt, -{Triangle[length=2.5mm, width=2.5mm]}}
]

    % --- 3. Place Vertices on a Grid ---
    
    % The Spine (Left to Right at y=0)
    \node[vertex] (X_m-1_m) at (0, 0) {$X_{m-1,m}$};
    \node[vertex] (X_m_1)   at (2, 0) {$X_{m,1}$};
    \node[vertex] (X_m-1_1) at (4, 0) {$X_{m-1,1}$};

    % The "Peak" (Top Left)
    % Placed horizontally between the first two nodes
    \node[vertex] (X_mm)    at (1, 1) {$X_{m,m}$};

    % The "Valley" (Bottom Right)
    % Placed horizontally between the last two nodes
    \node[vertex] (X_m_2)   at (3, -1) {$X_{m,2}$};

    % --- 4. Draw Edges ---
    
    % Red Horizontal Arrows (Right to Left flow)
    \draw[redarrow] (X_m-1_1) -- (X_m_1);
    \draw[redarrow] (X_m_1)   -- (X_m-1_m);

    % Blue Top Triangle
    \draw[bluearrow] (X_m-1_m) -- (X_mm);  % Up-Right
    \draw[bluearrow] (X_mm)    -- (X_m_1); % Down-Right

    % Blue Bottom Triangle
    \draw[bluearrow] (X_m_1)   -- (X_m_2);   % Down-Right
    \draw[bluearrow] (X_m_2)   -- (X_m-1_1); % Up-Right

\end{tikzpicture}
        \centering
        \caption{The quiver expresses adjacency relations between the elementary geodesic function and $m$th cycle.}
    \label{Fig21}
    \end{figure}
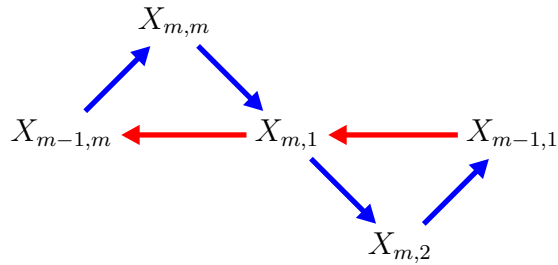

    Similar to \textit{Case 1},
    $$\tau_{m}^*(X_{m,1}) = {X_{m,1} \over Y_{m,1}Y_{m,m}}, \text{ }\tau_{m}^*(X_{m-1,1}) = X_{m-1,1}Y_{m,1}, \text{ and }\tau_{m}^*(X_{m-1,m}) = X_{m-1,m}Y_{m,m}.$$

    From the equation above and \textbf{Lemma~\ref{lem421}}, we have 
    $$\tau_{m}^*\left(\left\langle X_{1,1}, X_{2,1}, \cdots, X_{m,1}, X_{m-1,m}, \cdots, X_{2,3}, X_{1,2} \right\rangle\right) =\left\langle X_{1,1}, X_{2,1}, \cdots, X_{m,1}, X_{m-1,m}, \cdots, X_{2,3}, X_{1,2} \right\rangle.$$
    We omit the calculation since it is very similar to the previous cases.
\end{enumerate}
    Let us now assume $n = 2m+1$. For $i \le m-1,$ $s_i^* = \tau_i^*$ on $\mathcal K(\mathcal X_{|\mathcal A_n|})$, so the proof is identical to the \textit{Case 1} and \textit{Case 2} when $n$ is even. To deal with the case of $s_m^*$, we consider a doubled $\mathcal A_n$-quiver and the elementary geodesic function $\left\langle \widetilde{U_{1,1}}, \cdots, \widetilde{U_{m,1}}, U_{m,m+1},U_{m-1,m}, \cdots, U_{1,2} \right\rangle$.
    \begin{figure}[H]
\centering
\tikzset{every picture/.style={line width=0.75pt}} %set default line width to 0.75pt        
\begin{tikzpicture}[
    % 1. Global Scale
    scale=1.5, 
    % 2. Define Styles for consistency
    vertex/.style={anchor=center, inner sep=3pt, font=\normalsize},
    redarrow/.style={->, color=red, line width=2pt, -{Triangle[length=3mm, width=3mm]}},
    bluearrow/.style={->, color=blue, line width=2pt, -{Triangle[length=3mm, width=3mm]}}
]
    % --- 3. Place Vertices on a Grid ---
    % I used a coordinate system where:
    % y=2 is top row, y=1 is middle, y=0 is bottom
    % x spacing is 2 units horizontally
    
    % Middle Row (y=1)
    \node[vertex] (X_m-1_1) at (0, 1)   {$\widetilde{U_{m-1,1}}$};
    \node[vertex] (X_m_1)   at (2, 1)   {$\widetilde{U_{m,1}}$};
    \node[vertex] (X_mm1)   at (4, 1)   {$U_{m,m+1}$};
    \node[vertex] (X_m-1_m) at (6, 1)   {$U_{m-1,m}$};

    % Top Row (y=2) - Staggered (x values are odd numbers)
    \node[vertex] (X_m-1_n) at (3, 2.2) {$\widetilde{U_{m,n}}$};
    \node[vertex] (X_mm)    at (5, 2.2) {$U_{m,m}$};

    % Bottom Row (y=0) - Staggered
    \node[vertex] (X_m-1_2) at (1, -0.2) {$\widetilde{U_{m,2}}$};
    \node[vertex] (X_mm2)   at (3, -0.2) {$U_{m,m+2}$};

    % --- 4. Draw Edges ---
    
    % Red Horizontal Arrows (Middle Row)
    \draw[redarrow] (X_m-1_1) -- (X_m_1);
    \draw[redarrow] (X_m_1)   -- (X_mm1);
    \draw[redarrow] (X_mm1)   -- (X_m-1_m);

    % Blue Diagonal Arrows (Top Half)
    \draw[bluearrow] (X_m-1_m) -- (X_mm);    % Up-Left
    \draw[bluearrow] (X_mm)    -- (X_mm1);   % Down-Left
    \draw[bluearrow] (X_mm1)   -- (X_m-1_n); % Up-Left
    \draw[bluearrow] (X_m-1_n) -- (X_m_1);   % Down-Left

    % Blue Diagonal Arrows (Bottom Half)
    \draw[bluearrow] (X_mm1)   -- (X_mm2);   % Down-Left
    \draw[bluearrow] (X_mm2)   -- (X_m_1);   % Up-Left
    \draw[bluearrow] (X_m_1)   -- (X_m-1_2); % Down-Left
    \draw[bluearrow] (X_m-1_2) -- (X_m-1_1); % Up-Left
\end{tikzpicture}
        \caption{The quiver expresses adjacency relations between the elementary geodesic function and cycles associated with $s_m = f_m\widetilde{f_m}f_m$ on the doubled $\mathcal A_n$-quiver.}
        \label{Fig22}
    \end{figure}
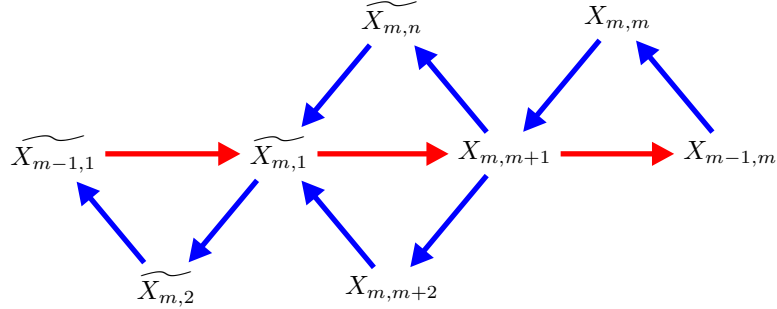

    We are going to prove that $f_m^*$ and $(\widetilde{f_m})^*$ preserve $\left\langle \widetilde{U_{1,1}}, \cdots, \widetilde{U_{m,1}}, U_{m,m+1},U_{m-1,m}, \cdots, U_{1,2} \right\rangle$. In this case, we temporarily introduce the variables $Y_{i,j}$ and $\widetilde{Y_{i,j}}$ in the doubled quiver by$$Y_{i,j} = U_{i,j} \frac{F_{i,j-1}}{F_{i,j}} \quad \text{and} \quad \widetilde{Y_{i,j}} = \widetilde{U_{i,j}} \frac{\widetilde{F_{i,j-1}}}{\widetilde{F_{i,j}}},$$where the functions $F_{i,j}$ and $\widetilde{F_{i,j}}$ are defined as the following partial sums:\begin{align*}F_{i,j} &:= 1 + U_{i,j} + U_{i,j}U_{i,j-1} + \dots + U_{i,j}U_{i,j-1} \cdots U_{i,j-N_i+2}, \\
    \widetilde{F_{i,j}} &:= 1 + \widetilde{U_{i,j}} + \widetilde{U_{i,j}}\widetilde{U_{i,j-1}} + \dots + \widetilde{U_{i,j}}\widetilde{U_{i,j-1}} \cdots \widetilde{U_{i,j-N_i+2}}.\end{align*}
    For the case of the $(\widetilde{f_m})^*$, we have (see Figure \ref{Fig22})
    $$(\widetilde{f_m})^*(\widetilde{U_{m-1,1}}) = \widetilde{U_{m-1,1}}\widetilde{Y_{m,1}}, \text{ }(\widetilde{f_m})^*(\widetilde{U_{m,1}}) = {\widetilde{U_{m,1}} \over \widetilde{Y_{m,1}}\widetilde{Y_{m,n}}}, \text{ and }(\widetilde{f_m})^*(U_{m,m+1}) = U_{m,m+1}\widetilde{Y_{m,n}}.$$
    This can be obtained by the same argument as in \textbf{Proposition~\ref{prop334}}. Hence, similar to the previous cases, we have $(\widetilde{f_m})^*\left(\left\langle \widetilde{U_{1,1}}, \cdots, \widetilde{U_{m,1}}, U_{m,m+1},U_{m-1,m}, \cdots, U_{1,2} \right\rangle\right) = \left\langle \widetilde{U_{1,1}}, \cdots, \widetilde{U_{m,1}}, U_{m,m+1},U_{m-1,m}, \cdots, U_{1,2} \right\rangle.$ Note that we also have
    $${1 \over \widetilde{Y_{i,j}}} + {\widetilde{Y_{i,j-1}} \over \widetilde{U_{i,j}}} = 1 + {1 \over \widetilde{U_{i,j}}}.$$

    For the case of the $f_m^*$, we have (see Figure \ref{Fig22})
    $$f_m^*(U_{m-1,m}) = U_{m-1,m}Y_{m,m}, \text{ }f_m^*(U_{m,m+1}) = {U_{m,m+1} \over Y_{m,m+1}Y_{m,m}}, \text{ and }f_m^*(\widetilde{U_{m,1}}) = \widetilde{U_{m,1}}Y_{m,m+1}.$$
    This can be obtained by the same technique as in \textbf{Proposition~\ref{prop334}.} Hence, similar to the previous cases, we obtain $f_m^*\left(\left\langle \widetilde{U_{1,1}}, \cdots, \widetilde{U_{m,1}}, U_{m,m+1},U_{m-1,m}, \cdots, U_{1,2} \right\rangle\right) = \left\langle \widetilde{U_{1,1}}, \cdots, \widetilde{U_{m,1}}, U_{m,m+1},U_{m-1,m}, \cdots, U_{1,2} \right\rangle.$
    
    Consequently, we have 
    $$s_m^*\left(\left\langle \widetilde{U_{1,1}}, \cdots, \widetilde{U_{m,1}}, U_{m,m+1},U_{m-1,m}, \cdots, U_{1,2} \right\rangle\right) = \left\langle \widetilde{U_{1,1}}, \cdots, \widetilde{U_{m,1}}, U_{m,m+1},U_{m-1,m}, \cdots, U_{1,2} \right\rangle$$ 
    on $\mathcal K(\mathcal X_{|\mathcal A_n^{dbl}|})$ since $(s_m)^*=(f_m\widetilde{f_m}f_m)^* =(f_m)^*(\widetilde{f_m})^*(f_m)^*$. Since $\operatorname{{pr}} \circ s_m^* = s_m^*\circ \operatorname{{pr}}$, we conclude 
    $$s_m^*\left(\left\langle X_{1,1}, \cdots, X_{m,1}, X_{m,m+1},X_{m-1,m}, \cdots, X_{1,2} \right\rangle\right) = \left\langle X_{1,1}, \cdots, X_{m,1}, X_{m,m+1},X_{m-1,m}, \cdots, X_{1,2} \right\rangle.$$
    
    Thus,
    $$s_{i}^*(\mathbb A_{q,q+1}) = \mathbb A_{q,q+1}$$ 
     for any $i \in \left\{1, 2, \cdots,\left\lfloor {n \over 2} \right\rfloor\right\}$ and $q \in \left\{1,2,\cdots,n-1\right\}.$ 
     
     Consider the third equation in the Bondal Poisson bracket (\ref{1.4}):
$$\begin{aligned}
    \left\{\mathbb A_{i,k},\mathbb A_{k,l}\right\} &= {1 \over 2}\mathbb A_{i,k}\mathbb A_{k,l} - \mathbb A_{i,l}\text{ for }i < k < l.\\
\end{aligned}$$

    Let $k = q+1$ and $l = q+2$. We get

    $$\mathbb A_{q,q+2} = {1 \over 2}\mathbb A_{q,q+1}\mathbb A_{q+1,q+2}-\left\{\mathbb A_{q,q+1},\mathbb A_{q+1,q+2}\right\}.$$

    Since $s^*_i$ preserves the quiver and is compatible with the Poisson bracket (\textbf{Proposition~\ref{prop349}}), we get
    $$s^*_i(\mathbb A_{q,q+2}) = {1 \over 2}(s^*_i\mathbb A_{q,q+1})(s^*_i\mathbb A_{q+1,q+2})-\left\{s^*_i\mathbb A_{q,q+1},s^*_i\mathbb A_{q+1,q+2}\right\} = {1 \over 2}(\mathbb A_{q,q+1})(\mathbb A_{q+1,q+2})-\left\{\mathbb A_{q,q+1},\mathbb A_{q+1,q+2}\right\}$$
    $$= \mathbb A_{q,q+2}.$$
    
    By performing this procedure repeatedly, for any $r$, we have 
    $$s^*_i(\mathbb A_{q,q+r}) = \mathbb A_{q,q+r}.$$
    
    This implies that all $\mathbb A_{i,j}$ for $i< j$ are preserved under the Weyl group action. \end{proof}
\begin{rem} \label{rem423}
We want to find new coordinates of the once-punctured torus where geodesic length functions of the surface are preserved under the new coordinates. 
    
    Let $x, y, z$ be fixed. Consider the equation system
    $$
    \begin{aligned}
        \langle a,b \rangle &= \langle x,y \rangle, \\
        \langle b,c \rangle &= \langle y,z \rangle, \\
        \langle c,a \rangle &= \langle z,x \rangle.
    \end{aligned}
    $$
   With the assistance of computer algebra software (Maple), we verified that this system has exactly two solutions. The trivial solution is clearly $(a,b,c) = (x,y,z)$. The second, nontrivial solution can be recovered via the Weyl group action. Specifically, we consider the action of $s_1^*$ on $\mathcal K(\mathcal X_{|\mathcal A_3|})$ where the initial cluster variables are $x, y, z$.
    
    By the previous theorem, this action preserves the bracket $\langle \cdots \rangle$. Consequently, the triple $(s_1^*x, s_1^*y, s_1^*z)$ is the other non-trivial solution. Explicitly, the action is given by:
    $$
    \begin{aligned}
        s_1^*(x) &= {(x^2y^2z^2 + x^2yz^2 + 2x^2yz + x^2y + 2xy + y + 1)^2 \over x(x^2y^2z^2 + x^2y^2z + 2xy^2z + y^2z + 2yz + z + 1)^2},\\
        s_1^*(y) &= {(x^2y^2z^2 + x^2y^2z + 2xy^2z + y^2z + 2yz + z + 1)^2 \over y(x^2y^2z^2 + xy^2z^2 + 2xyz^2 + xz^2 + 2xz + x + 1)^2}, \\
        s_1^*(z) &= {(x^2y^2z^2 + xy^2z^2 + 2xyz^2 + xz^2 + 2xz + x + 1)^2 \over z(x^2y^2z^2 + x^2yz^2 + 2x^2yz + x^2y + 2xy + y + 1)^2}.
    \end{aligned}
    $$
    Since the cluster Poisson variety $\mathcal X_{|\mathcal{A}_3|}$ corresponds to a once-punctured torus, and the bracket $\langle \cdots \rangle$ represents the geodesic length functions, we conclude that $(s_1^*x, s_1^*y, s_1^*z)$ provides the desired new coordinate system. Note that the action of $s_1^*$ corresponds to the $\mathbb{Z}_2$-action on the once-punctured torus that reverses the orientation of the surface.
    
    Generalizing to higher $n$, we consider the system of equations
    $$\mathbb{A}_{p,q}(X_{i,j}) = b_{p,q}$$
    where $b_{p,q} \in \mathbb C$. The Weyl group $\mathbf W_n$ generates non-trivial solutions to this system. Specifically, for any $w \in W_n$, the tuple $(w^*X_{i,j})$ serves as a solution if $(X_{i,j})$ is a solution of the system.

    We can show that, for generic values of \(b_{p,q}\), the system has exactly \(\lvert \mathbf W_n\rvert\) distinct solutions. In particular, we obtain \[ \mathbb{C}(\mathcal X) = \operatorname{Frac} \bigl(\mathcal O(\mathcal X_{|\mathcal A_n|})\bigr). \]

Here, \(\mathcal O(\mathcal X_{|\mathcal A_n|})\) denotes the ring of regular functions on \(\mathcal X_{|\mathcal A_n|}\), \(\operatorname{Frac}(\mathcal O(\mathcal X_{|\mathcal A_n|}))\) denotes its field of fractions, and \(\mathbb C(\mathcal X)\) denotes the field of rational functions generated by an initial cluster \((X_{i,j})\). This implies that \[ \mathbb{C}(\mathcal X)^{\mathbf W_n} = \operatorname{Frac} \bigl(\mathcal O(\mathcal A_n)\bigr) \cap \mathbb C(\mathcal X) \] because the Weyl group acts faithfully on the set of Casimirs $\mathcal K_i$. The proof of this result will be presented in a forthcoming paper. \end{rem}

\begin{rem}\label{rem424}
If we allow an arbitrary choice of square roots, the action of an element $s \in \mathbf W_n$ preserves the generators $\mathbb{A}_{i,i+1}$ up to a sign. Specifically, we have
$$s(\mathbb{A}_{i,i+1}) = \sigma_i \mathbb{A}_{i,i+1} \quad \text{for } i = 1, \dots, n-1,$$
where $\sigma_i \in \{\pm 1\}$. Recall that the generators satisfy the Bondal Poisson relation \eqref{1.4}:
$$\mathbb{A}_{ij} = \frac{1}{2}\mathbb{A}_{ik}\mathbb{A}_{kj} - \{\mathbb{A}_{ik}, \mathbb{A}_{kj}\}.$$
Since $s$ is a Poisson map, applying this relation inductively yields
$$s(\mathbb{A}_{ij}) = \left( \prod_{k=i}^{j-1} \sigma_k \right) \mathbb{A}_{ij}.$$
Now, define a diagonal matrix $E = \operatorname{diag}(\epsilon_1, \dots, \epsilon_n)$ by
$$\epsilon_1 = 1 \quad \text{and} \quad \epsilon_m = \prod_{k=1}^{m-1} \sigma_k \quad \text{for } m = 2, \dots, n.$$
The action on the matrix $\mathbb{A} = (\mathbb{A}_{ij})$ can be expressed in matrix form as $s(\mathbb{A}) = E \mathbb{A} E^T$. Indeed, using $\sigma_k^2 = 1$, a direct computation shows that
$$(E \mathbb{A} E^T)_{ij} = \epsilon_i \epsilon_j \mathbb{A}_{ij} = \left( \prod_{k=1}^{i-1} \sigma_k \right) \left( \prod_{k=1}^{j-1} \sigma_k \right) \mathbb{A}_{ij} = \left( \prod_{k=i}^{j-1} \sigma_k \right) \mathbb{A}_{ij}.$$
\end{rem}

\section{Transitivity of Coisotropic Reductions under the Weyl Group Action}
In this section, we investigate the conjugacy of coisotropic reductions on the symplectic groupoid. This reduction procedure aims to recover the cluster structure of certain Teichmüller spaces.
\subsection{Geometric Leaf of the Symplectic Groupoid and Rank Condition}

In\textcolor{black}{\cite{12}} and\textcolor{black}{\cite{13}}, the map $\phi_n: \mathcal T_{\left\lfloor {n-1 \over 2} \right\rfloor,par(n)} \to \mathcal A_n$ is constructed as follows: For an element $\left\{z_\alpha\right\} \in \mathcal T_{\left\lfloor {n-1 \over 2} \right\rfloor,par(n)}$,

\begin{equation}
    \phi_n(\left\{z_\alpha\right\}) = \begin{pmatrix}
1 &  & G_{ij}(z_\alpha) \\
 & \cdots &  \\
0 &  & 1
\end{pmatrix}
\end{equation}
where $G_{ij}(z_\alpha)$ is the value of the geodesic function $G_{ij}$ at $\left\{z_\alpha\right\} \in \mathcal T_{\left\lfloor {n-1 \over 2} \right\rfloor,par(n)}$. 

\begin{defn}
    $\text{Im}(\phi_n)$ is called a \textit{geometric leaf} of the symplectic groupoid.
\end{defn}

\begin{thm}\textcolor{black}{\cite{28}} Every $n \times n$ matrix $A \in \text{Im}(\phi_n)$ satisfies the rank condition $\text{rank}(A+A^T) \le 4$.
\end{thm}

This theorem implies the natural criterion for our coisotropic reduction is the rank condition $\text{rank}(\mathbb A + \mathbb A^T) \le 4$. The reduction based on this rank condition is studied in\textcolor{black}{\cite{16}} and\textcolor{black}{\cite{18}}.

To describe the solution set of the rank condition, we consider an $SL_n$-quiver and the graph $P_n$ (Section \ref{Ch3.1}). We then reverse the direction of horizontal arrows of $P_n$ to get another graph $\hat P_n$ (Figure \ref{Fig23}).

\begin{defn}
We define the $n \times n$ transport matrix $T_{(3)}$ by

$$
(T_{(3)})_{ji} = \sum_{\text{oriented paths } p: i' \to j''} w(p)
$$
where $w(p) = \prod_{v}Z_v$ such that the product is taken over all vertices $v$ in the $SL_n$-quiver that is to the right of the path $p$.

\end{defn}

\begin{figure}[H]
    \centering
	\begin{tikzcd}[scale cd=0.8,sep = tiny]
	&&&& \textcolor{rgb,255:red,255;green,51;blue,51}{{{1'}}} && \textcolor{rgb,255:red,255;green,0;blue,0}{\bullet} && \textcolor{rgb,255:red,255;green,51;blue,51}{1} \\
	&&&& \textcolor{rgb,255:red,51;green,136;blue,255}{{Z_{1,0,2}}} &&&& \textcolor{rgb,255:red,51;green,136;blue,255}{{Z_{0,1,2}}} \\
	&& \textcolor{rgb,255:red,255;green,51;blue,51}{{{2'}}} && \textcolor{rgb,255:red,255;green,0;blue,0}{\bullet} && \textcolor{rgb,255:red,255;green,0;blue,0}{\bullet} && \textcolor{rgb,255:red,255;green,0;blue,0}{\bullet} && \textcolor{rgb,255:red,255;green,51;blue,51}{2} \\
	&& \textcolor{rgb,255:red,51;green,136;blue,255}{{Z_{2,0,1}}} &&&& {Z_{1,1,1}} &&&& \textcolor{rgb,255:red,51;green,136;blue,255}{{Z_{0,2,1}}} \\
	\textcolor{rgb,255:red,255;green,51;blue,51}{{{3'}}} && \textcolor{rgb,255:red,255;green,0;blue,0}{\bullet} && \textcolor{rgb,255:red,255;green,0;blue,0}{\bullet} && \textcolor{rgb,255:red,255;green,0;blue,0}{\bullet} && \textcolor{rgb,255:red,255;green,0;blue,0}{\bullet} && \textcolor{rgb,255:red,255;green,0;blue,0}{\bullet} && \textcolor{rgb,255:red,255;green,51;blue,51}{3} \\
	&&&& \textcolor{rgb,255:red,51;green,136;blue,255}{{Z_{2,1,0}}} &&&& \textcolor{rgb,255:red,51;green,136;blue,255}{{Z_{1,2,0}}} \\
	&& \textcolor{rgb,255:red,255;green,51;blue,51}{{{1''}}} &&&& \textcolor{rgb,255:red,255;green,51;blue,51}{{{2''}}} &&&& \textcolor{rgb,255:red,255;green,51;blue,51}{{{3''}}}
	\arrow[color={rgb,255:red,255;green,0;blue,0}, from=1-5, to=1-7]
	\arrow[color={rgb,255:red,255;green,0;blue,0}, from=1-7, to=1-9]
	\arrow[color={rgb,255:red,255;green,0;blue,0}, from=1-7, to=3-7]
	\arrow[from=2-5, to=4-7]
	\arrow[from=2-9, to=2-5]
	\arrow[dashed, from=2-9, to=4-11]
	\arrow[color={rgb,255:red,255;green,0;blue,0}, from=3-3, to=3-5]
	\arrow[color={rgb,255:red,255;green,0;blue,0}, from=3-5, to=3-7]
	\arrow[color={rgb,255:red,255;green,0;blue,0}, from=3-5, to=5-5]
	\arrow[color={rgb,255:red,255;green,0;blue,0}, from=3-7, to=3-9]
	\arrow[color={rgb,255:red,255;green,0;blue,0}, from=3-9, to=3-11]
	\arrow[color={rgb,255:red,255;green,0;blue,0}, from=3-9, to=5-9]
	\arrow[dashed, from=4-3, to=2-5]
	\arrow[from=4-3, to=6-5]
	\arrow[from=4-7, to=2-9]
	\arrow[from=4-7, to=4-3]
	\arrow[from=4-7, to=6-9]
	\arrow[from=4-11, to=4-7]
	\arrow[color={rgb,255:red,255;green,0;blue,0}, from=5-1, to=5-3]
	\arrow[color={rgb,255:red,255;green,0;blue,0}, from=5-3, to=5-5]
	\arrow[color={rgb,255:red,255;green,0;blue,0}, from=5-3, to=7-3]
	\arrow[color={rgb,255:red,255;green,0;blue,0}, from=5-5, to=5-7]
	\arrow[color={rgb,255:red,255;green,0;blue,0}, from=5-7, to=5-9]
	\arrow[color={rgb,255:red,255;green,0;blue,0}, from=5-7, to=7-7]
	\arrow[color={rgb,255:red,255;green,0;blue,0}, from=5-9, to=5-11]
	\arrow[color={rgb,255:red,255;green,0;blue,0}, from=5-11, to=5-13]
	\arrow[color={rgb,255:red,255;green,0;blue,0}, from=5-11, to=7-11]
	\arrow[from=6-5, to=4-7]
	\arrow[from=6-9, to=4-11]
	\arrow[dashed, from=6-9, to=6-5]
\end{tikzcd}
    \caption{$SL_3$-quiver and $\hat P_3$.}
    \label{Fig23}
\end{figure}

\begin{thm}\textcolor{black}{\cite{17}}
Consider an $n \times n$ matrix with entries $S_{ij} = (-1)^{n-i}\delta_{i,n+1-j}$ and let $M_1 = ST_{(1)}D_1^{-1}$, $M_2 = ST_{(2)}D_2^{-1}$, and $M_3 = ST_{(3)}D_3^{-1}$, where
$$
D_1 := \prod_{k=1}^n \prod_{i+j = n-k}Z_{i,j,k}^{\frac{k}{n}}, \quad
D_2 := \prod_{k=1}^n \prod_{i+j = n-k}Z_{i,k,j}^{\frac{k}{n}}, \quad \text{and} \quad
D_3 := \prod_{k=1}^n \prod_{i+j = n-k}Z_{k,i,j}^{\frac{k}{n}}.
$$
Then the following \textit{groupoid condition} holds:
$$
M_2 = M_3 M_1.
$$
\end{thm}

From the groupoid condition, we have
\begin{equation}
\begin{split}
    \mathbb{A} &= M_1^T M_2 = T_{(1)}^T S^T D_1^{-1} M_3 M_1 = T_{(1)}^T S^T S T_{(3)} S T_{(1)} (D_1^2 D_3)^{-1} \\
    &= (-1)^{n-1} T_{(1)}^T T_{(3)} S T_{(1)} (D_1^2 D_3)^{-1}.
\end{split}
\end{equation}
This implies,
\begin{equation}\label{6.3}
    \text{rank}(\mathbb A + \mathbb A^T) = \text{rank}((T_{(1)})^T(T_{(3)}S + (T_{(3)}S)^T)T_{(1)}) = \text{rank}(T_{(3)}S + (T_{(3)}S)^T)
\end{equation}
where the last equality follows because $T_{(1)}$ is nondegenerate and nonzero scalars can be disregarded when considering the rank of the matrix.

\begin{prop} \label{prop515}
\textcolor{black}{\cite{17}} The matrix $\mathcal M := T_{(3)}S + (T_{(3)}S)^T$ is an $n \times n$ lower anti-diagonal matrix. Furthermore, it is a symmetric matrix and $\mathcal M_{i,n-i+1} = \mathcal M_{n-i+1,i} = 0$ if and only if $\mathcal K_i = 1$ where $\mathcal C_i = \prod_{j=1}^{N_i}X_{i,j}$ and $\mathcal K_i = \prod_{j=i}^{\left\lfloor {n \over 2} \right\rfloor} \mathcal C_{j}$ such that $N_i$ is the length of the $i$th cycle.
\end{prop}
    
\begin{prop}\label{prop516}
    The rank of $\mathcal{M}$ is invariant under the Weyl group action.
\end{prop}
\begin{proof}
By Remark~\ref{rem424}, $s(\mathbb{A}) = E \mathbb{A} E^T$ for any $s \in W_n$. Hence, the rank of $\mathbb A+\mathbb A^T$ is invariant under the action $s$. This implies the rank of $\mathcal{M}$ is invariant under the Weyl group action by \eqref{6.3}. \end{proof}

\begin{defn}
Let $\mathcal{D}^n \subset \mathcal{X}_{\mathcal{A}_n}$ be the solution set of the rank condition. For an index set $I = \{i_1, \dots, i_{e(n)}\} \subset \{1, \dots, \lfloor n/2 \rfloor\}$ of size $e(n) = \lfloor n/2 \rfloor - \operatorname{par}(n)$, the subvariety $\mathcal{D}^n_I$ is the irreducible component of $\mathcal{D}^n$ satisfying the Casimir conditions $\mathcal{K}_i = 1$ for all $i \in I$. Here, the parity function is defined by
\begin{equation*}
\operatorname{par}(n) = 
\begin{cases} 
1 & \text{if } n \text{ is odd,} \\ 
2 & \text{if } n \text{ is even.} 
\end{cases}
\end{equation*}
\end{defn}

Proposition~\ref{prop515} implies that $\mathcal{D}^n$ decomposes into the components $\mathcal{D}^n_I$ where $I$ ranges over all subsets of $\{1, \dots, \lfloor n/2 \rfloor\}$ of size $e(n)$.

\begin{exmp}\label{exmp518}
(\textit{Rank condition for $n = 6$})

$$
\mathcal{M} = \begin{pmatrix}
0 & 0 & 0 & 0 & 0 & m_{16} \\
0 & 0 & 0 & 0 & m_{25} & m_{26} \\
0 & 0 & 0 & m_{34} & m_{35} & m_{36} \\
0 & 0 & m_{34} & m_{44} & m_{45} & m_{46} \\
0 & m_{25} & m_{35} & m_{45} & m_{55} & m_{56} \\
m_{16} & m_{26} & m_{36} & m_{46} & m_{56} & m_{66}
\end{pmatrix}.
$$
The solution set $\mathcal D^6$ of the rank condition splits into the following three irreducible components $\mathcal{D}^6_{\{1\}}$, $\mathcal{D}^6_{\{2\}}$, and $\mathcal{D}^6_{\{3\}}$:

\begin{enumerate}
    \item $\mathcal{D}^6_{\{3\}}$ satisfies the Casimir condition $\mathcal{K}_3 = 1 \longleftrightarrow m_{34} = 0$. Thus, the component is defined by 
    $$m_{34} = m_{44} = 0.$$
    
    \item $\mathcal{D}^6_{\{2\}}$ satisfies the Casimir condition $\mathcal{K}_2 = 1 \longleftrightarrow m_{25} = 0$. Thus, the component is defined by 
    $$
    m_{25} = \det \begin{pmatrix}
    0 & m_{34} & m_{35}\\
    m_{34} & m_{44} & m_{45}\\
    m_{35} & m_{45} & m_{55}\\
    \end{pmatrix} = 0.
    $$

    \item $\mathcal{D}^6_{\{1\}}$ satisfies the Casimir condition $\mathcal{K}_1 = 1 \longleftrightarrow m_{16} = 0$. Thus, the component is defined by 
    $$
    m_{16} = \det \begin{pmatrix}
    0 & 0 & 0 & m_{25} & m_{26}\\
    0 & 0 & m_{34} & m_{35} & m_{36}\\
    0 & m_{34} & m_{44} & m_{45} & m_{46}\\
    m_{25} & m_{35} & m_{45} & m_{55} & m_{56}\\
    m_{26} & m_{36} & m_{46} & m_{56} & m_{66}\\
    \end{pmatrix} = 0.
    $$
\end{enumerate}
\end{exmp}

\begin{exmp}(\textit{Rank condition for $n = 7$})

$$\mathcal M = \begin{pmatrix}
0 & 0 & 0 & 0 & 0 & 0 & m_{17}\\
0 & 0 & 0 & 0 & 0 &m_{26} & m_{27}\\
0 & 0 & 0 & 0& m_{35} & m_{36} & m_{37}\\
0 & 0 & 0& m_{44} & m_{45} & m_{46} & m_{47}\\
0 & 0 & m_{35} & m_{45} & m_{55} & m_{56} & m_{57}\\
0 & m_{26} & m_{36} & m_{46} & m_{56} & m_{66} & m_{67}\\
m_{17} & m_{27} & m_{37} & m_{47} & m_{57} & m_{67} & m_{77} \\
\end{pmatrix}$$

The solution set $\mathcal D^7$ of the rank condition splits into the following three irreducible components $\mathcal{D}^7_{\{1,2\}}$, $\mathcal{D}^7_{\{1,3\}}$, and $\mathcal{D}^7_{\{2,3\}}$:
\begin{enumerate}
    \item $\mathcal{D}^7_{\{2,3\}}$ satisfies the Casimir condition $\mathcal K_2 = \mathcal K_3 = 1 \longleftrightarrow m_{26} = m_{35} = 0$. Thus, the component is defined by
    $$m_{26} = m_{35} = m_{36} = \det \begin{pmatrix}
m_{44} & m_{45} & m_{46} \\
m_{45} & m_{55} & m_{56} \\
m_{46} & m_{56} & m_{66} \\
\end{pmatrix} = 0$$.

    \item $\mathcal{D}^7_{\{1,3\}}$ satisfies the Casimir condition $\mathcal K_1 = \mathcal K_3 = 1 \longleftrightarrow m_{17} = m_{35} = 0$. Thus, the component is defined by
    
    $$m_{17} = m_{35} = \det \begin{pmatrix}
m_{26} & m_{36} \\
m_{27} & m_{37}
\end{pmatrix} = \det \begin{pmatrix}
0 & 0 & 0 & m_{36} &m_{37}\\
0 & m_{44} & m_{45} & m_{46} & m_{47}\\
0 & m_{45} & m_{55} & m_{56} & m_{57}\\
m_{36} & m_{46} & m_{56} & m_{66} & m_{67}\\
m_{37} & m_{47} & m_{57} & m_{67} & m_{77}\\
\end{pmatrix} = 0$$.

\item $\mathcal{D}^7_{\{1,2\}}$ satisfies the Casimir condition $\mathcal K_1 = \mathcal K_2 = 1 \longleftrightarrow m_{17} = m_{26} = 0$. Thus, the component is defined by
$$m_{17} = m_{26} = m_{27} = \det \begin{pmatrix}
0 & 0 & m_{35} & m_{36} &m_{37}\\
0 & m_{44} & m_{45} & m_{46} & m_{47}\\
m_{35} & m_{45} & m_{55} & m_{56} & m_{57}\\
m_{36} & m_{46} & m_{56} & m_{66} & m_{67}\\
m_{37} & m_{47} & m_{57} & m_{67} & m_{77}\\
\end{pmatrix} = 0$$.
\end{enumerate}\end{exmp}

\subsection{The Birational Weyl Group Action Conjugates the Coisotropic Reductions}

We show that the coisotropic reductions are conjugate under the birational Weyl group action. We first recall the following definition (Definition~\ref{defn218}): The cluster modular group $\Gamma_{|\Sigma|}$ is the group generated by cluster transformations on a seed $\Sigma$ that preserve the initial quiver. Elements of this group are referred to as cluster modular transformations.

Let $n \ge 5$ and $m = \lfloor n/2 \rfloor$. The group generated by the actions $\{s_i^*\}_{i=1}^{m-1}$ acts on the set of Casimirs $\{\mathcal{K}_j\}_{j=1}^m$ as the symmetric group $S_m$. Specifically, each generator $s_i^*$ acts as the simple transposition $\mathcal{K}_i \longleftrightarrow \mathcal{K}_{i+1}$ on the set of Casimirs (Proposition~\ref{prop353}). This implies that the Weyl group action on $\{\mathcal{K}_j\}_{j=1}^m$ is transitive, leading to the following result:

\begin{thm}
For $n \ge 5$, the irreducible components $\mathcal{D}^n_I$ of $\mathcal{D}^n$ are birationally isomorphic under the Weyl group action. Specifically, for any index set $I = \{i_1, \dots, i_{e(n)}\}$ satisfying $1 \le i_1 < \dots < i_{e(n)} \le \lfloor n/2 \rfloor$ with $e(n) = \lfloor n/2 \rfloor - \operatorname{par}(n)$, the components $\mathcal{D}^n_I$ are permuted transitively by the Weyl group action.
\end{thm}

\begin{proof} First, suppose that \(n=2m\). Since the symmetric group \(S_m\) acts transitively on \(\{1,\dots,m\}\), there exists a permutation \(\sigma\in S_m\) such that \[ \sigma\bigl(\{i_1,\dots,i_{m-2}\}\bigr)=\{3,\dots,m\}. \] 

By Proposition~\ref{prop353}, there is an element \(\eta^*\in W_n\) that maps the Casimir data \[ \{\mathcal K_{i_1},\dots,\mathcal K_{i_{m-2}}\} \quad\text{to}\quad \{\mathcal K_3,\dots,\mathcal K_m\}. \] 

Moreover, the Weyl group action preserves the rank of \(\mathcal M\) by Proposition~\ref{prop516}. Hence, \(\eta^*\) preserves the rank condition \( \operatorname{rank}(\mathcal M)\leq 4 \) while mapping the corresponding Casimir conditions to one another. Therefore, the cluster transformation \(\eta\) restricts to a birational map \[ \mathcal D^n_{\{3,\dots,m\}} \dashrightarrow \mathcal D^n_{\{i_1,\dots,i_{m-2}\}}. \] 

Since the image of a dense open subset is dense and irreducible, this establishes a birational isomorphism between the two components. For odd \(n=2m+1\), the same argument shows that \( \mathcal D^n_{\{i_1,\dots,i_{m-1}\}} \) is birationally isomorphic to \( \mathcal D^n_{\{2,\dots,m\}} \) under the Weyl group action. 

Consequently, every irreducible component of the form \( \mathcal D^n_{\{i_1,\dots,i_{e(n)}\}} \) is birationally isomorphic to the component \( \mathcal D^n_{\{ 1+\operatorname{par}(n),\dots,\lfloor n/2\rfloor \}}. \) Thus, the Weyl group acts transitively on the irreducible components, and all such components are birationally isomorphic to one another. \end{proof}

\begin{exmp}
(\textit{Transitivity of irreducible components under Weyl group actions for $n=6,7$})

Applying the Weyl group action to the defining equations of $\mathcal D^n$, we can directly verify the following results using computer-assisted calculations:
\begin{enumerate}
    \item $s_2$ maps $\mathcal D^6_{\left\{2\right\}}$ to $\mathcal D^6_{\left\{3\right\}}$ and $s_2\circ s_1$ maps $\mathcal D^6_{\left\{1\right\}}$ to $\mathcal D^6_{\left\{3\right\}}$.
    \item $s_1$ maps $\mathcal D^7_{\left\{1,3\right\}}$ to $\mathcal D^7_{\left\{2,3\right\}}$ and $s_2 \circ s_1$ maps $\mathcal D^7_{\left\{1,2\right\}}$ to $\mathcal D^7_{\left\{2,3\right\}}$.
\end{enumerate}
\end{exmp}

\begin{cor} \label{cor523}
    The coisotropic reductions are conjugate to each other under the birational Weyl group action. Consequently, it suffices to consider a single irreducible component of $\mathcal D^n$ for the coisotropic reduction.
\end{cor}

\begin{proof}
    For $i < \lfloor \frac{n}{2} \rfloor$, the elements $s_i = \tau_i$ are cluster modular transformations; hence, they preserve the $\mathcal A_n$-quiver. Furthermore, the irreducible components are Poisson birationally isomorphic to each other under these transformations. Consequently, the coisotropic reductions are conjugate under the Weyl group action.\end{proof}
    
For $n \le 6$, the existence of the coisotropic reduction on a single component was shown in\textcolor{black}{\cite{16}}. This induces the existence of a coisotropic reduction for all components when $n \le 6$ by Corollary~\ref{cor523}. Consequently, the cluster structure of $\mathcal T_{2,2}$ is recovered through the coisotropic reduction.

For $n = 8$ and $10$, we also found the coisotropic reduction on a single component. Consequently, cluster structures of $\mathcal T_{3,2}$ and $\mathcal T_{4,2}$ are recovered through the coisotropic reductions; these results will be presented in a future paper. 

Hence, for higher $n$, we expect a similar result:

\begin{conj} \label{conj524}
There exists a coisotropic reduction on a single irreducible component of the rank-condition locus. This would imply the existence of a coisotropic reduction on the entire solution set of the rank condition by \textit{Corollary~\ref{cor523}}. Furthermore, this coisotropic reduction recovers the cluster structure of $\mathcal T_{\left\lfloor {n-1 \over 2} \right\rfloor,par(n)}$. \end{conj}

\section{Cluster DT Transformation and Weyl-invariants on $\mathcal A_n$-quiver} 
In this section, we show that, for even \(n\), the longest element of the Weyl group corresponds to a cluster DT transformation. This provides a canonical basis of \(\mathcal O(\mathcal X_{|\mathcal A_n|})\), which we use to characterize its Weyl-invariant subalgebra in terms of the formal geodesic functions (Theorem~\ref{thm6412}). This characterization can also be viewed as a cluster-theoretic analogue of the Chevalley restriction theorem (Remark~\ref{rem6414}). In contrast, when \(n\) is odd, no reddening sequence exists.

\subsection{Tropicalization of Cluster $\mathcal X$-variables} \label{ch6.1}

\begin{defn} \label{defn611}
    A set $\mathbb P$ is a semifield if it is a multiplicative abelian group with the addition $\oplus$, which is commutative, associative, and distributive. In particular, the following two semifields will be important:
    \begin{enumerate}
        \item On the algebraic torus $\mathcal X_{\Sigma}$, the universal semifield $\mathbb P_{\text{univ}}(\mathcal X_{\Sigma})$ is the semifield consisting of all the rational functions in the variables $X_i$ over $\mathbb Q$ with subtraction-free expressions.
        \item The tropical semifield $\mathbb P_{\text{trop}}(\mathcal X_{\Sigma})$ is a multiplicative free abelian group with the following special addition $\oplus:$
        \begin{equation}\prod_{i\in I}X_i^{a_i} \oplus \prod_{i\in I}X_i^{b_i} := \prod_{i\in I}X_i^{\min(a_i,b_i)}.\end{equation}
    \end{enumerate}

    Define a natural homomorphism of semifields
    \begin{equation}\pi_{\text{trop}}: \mathbb P_{\text{univ}}(\mathcal X_{\Sigma}) \to \mathbb P_{\text{trop}}(\mathcal X_{\Sigma})\end{equation}
    where $\pi_{\text{trop}}(X_i) = X_i$ and $\pi_{\text{trop}}(c) = 1$ for $c \in \mathbb Q_{>0}.$ We can also denote a homomorphism $\pi_{\text{trop}}(\cdot)$ by $[\cdot]$. Specifically, $[x] = \pi_{\text{trop}}(x)$ where $x \in \mathbb P_{\text{univ}}(\mathcal X_{\Sigma})$.
\end{defn}

\begin{thm}[{\textcolor{black}{\cite[Proposition 3.13, Corollary 6.3]{7}}}]
Let $\Sigma = (I, F, \epsilon)$ be the initial seed. For every cluster transformation $\tau : \Sigma \to \Sigma'$, there exists an $|I| \times |I|$ integer matrix $C^{\tau} = (c^{\tau}_{ij})$ such that
\[
    [\tau^*(X_i')] = \prod_{j \in I} X_j^{c^{\tau}_{ji}}.
\]
The matrix $C^{\tau}$ is called the \textit{$C$-matrix} of a cluster transformation $\tau$, and its columns are referred to as the \textit{$c$-vectors} of $\tau$.
\end{thm}

For a given $C$-matrix $C^\tau$ with a cluster transformation $\tau : \Sigma \to \Sigma'$, the mutation in the direction $k$ changes the matrix $C'$ to the matrix
\begin{equation}c^{\mu_k\circ\tau}_{ij} =
\begin{cases}
-c^\tau_{ij}, & \mbox{if }j = k  \\
c^\tau_{ij} + c^\tau_{ik}\max(0,\epsilon'_{kj}) + \max(0,-c^\tau_{ik})\epsilon'_{kj} & \mbox{otherwise.}
\end{cases}
\end{equation} 

We now introduce a reddening sequence.
\begin{defn}\label{defn613}
\cite{36} A cluster transformation is called a \textit{reddening sequence} if all entries of its $C$-matrix are nonpositive. Furthermore, if the $C$-matrix is exactly $-I$ and the underlying quiver is preserved, the transformation is called a \textit{cluster Donaldson--Thomas (DT) transformation}.
\end{defn}

\begin{prop} \label{prop614}
    If Q admits a reddening sequence, any quiver $Q'$ mutation equivalent to $Q$ also admits a reddening sequence, following from the definition of the $C$-matrix. 
\end{prop}

\begin{thm} \label{thm615}
    \textcolor{black}{\cite{9}} Suppose that a quiver $Q$ admits a reddening sequence. Then, every full subquiver $Q'$ of $Q$ also admits a reddening sequence. Hence, $Q$ does not admit a reddening sequence if it contains a full subquiver that does not admit a reddening sequence.
\end{thm}

The sign-coherence of the c-vectors, which plays a crucial role in proving the existence of a reddening sequence for the $\mathcal A_n$-quiver, is as follows:

\begin{thm} \label{thm616}
    \textcolor{black}{\cite[Theorem 1.7]{8}} Every column of a $C$-matrix is a nonzero vector and its nonzero entries are either all positive or all negative.
\end{thm}

\begin{exmp} \label{exmp617}
    \textit{(Reddening sequences on $\mathcal A_3$-quiver and $\mathcal A_4$-quiver)}

Consider the following $\mathcal A_4$-quiver:
$$
\begin{tikzcd}
	{\text{1}} &&&& {\text{2}} \\
	\\
	& {\text{6}} && {\text{5}} \\
	\\
	{\text{4}} &&&& {\text{3}}
	\arrow[from=1-1, to=1-5]
	\arrow[from=1-1, to=3-4]
	\arrow[from=1-5, to=3-2]
	\arrow[from=1-5, to=5-5]
	\arrow[from=3-2, to=1-1]
	\arrow[from=3-2, to=5-5]
	\arrow[from=3-4, to=1-5]
	\arrow[from=3-4, to=5-1]
	\arrow[from=5-1, to=1-1]
	\arrow[from=5-1, to=3-2]
	\arrow[from=5-5, to=3-4]
	\arrow[from=5-5, to=5-1]
\end{tikzcd}
$$

It admits a reddening sequence: Let $\tau_1 : = \mu_1\mu_2\mu_4\mu_3\mu_2\mu_1$ and $\tau_2 : = \mu_6\mu_5.$ Then $\tau_1\tau_2\tau_1\tau_2$ is a reddening sequence; see Theorem~\ref{thm621}.

On the other hand, the following $\mathcal A_3$-quiver, which is also called the Markov quiver, does not admit a reddening sequence: see\textcolor{black}{\cite{10}}.

$$
\begin{tikzcd}
	& {\text{1}} \\
	{\text{2}} && {\text{3}}
	\arrow[Rightarrow, from=1-2, to=2-3]
	\arrow[Rightarrow, from=2-1, to=1-2]
	\arrow[Rightarrow, from=2-3, to=2-1]
\end{tikzcd}
$$
\end{exmp}

\subsection{Cluster DT transformation on $\mathcal A_n$-quiver}

We show that there exists a cluster DT transformation on the $\mathcal A_n$-quiver for even $n$.

\begin{thm} \label{thm621}
    Assume $n = 2m$. Then the cluster transformation $(\tau_1\tau_2\cdots\tau_{m-1}\tau_m)^m$ is a cluster DT transformation on the $\mathcal A_n$-quiver. Here, the longest element of $\mathbf W_n$ corresponds to the cluster DT transformation since $\mathbf W_n$ is a $B_m$-type Weyl group.
\end{thm}

\begin{lem} \label{lem622}
Denote $w := \tau_1\tau_2\cdots\tau_{m-1}\tau_m$ and $w_0 := w^m.$ Then, 
$$w_0^*(\mathcal C_i) = \mathcal C_i^{-1}$$
with $i \in \left\{1,\cdots,m\right\}.$
\end{lem}

\begin{proof}

The identity $\prod_{j=1}^{N_i} Y_{i,j} = \prod_{j=1}^{N_i} X_{i,j}$ where $N_i$ is the length of the $i$th cycle follows immediately from the definition $$Y_{i,j} = X_{i,j}{F_{i,j-1} \over F_{i,j}}$$
where $F_{i,j} := 1 + X_{i,j} + X_{i,j}X_{i,j-1} + \dots + X_{i,j}X_{i,j-1}\dots X_{i,j-N_i+2}$. Consider $\mathcal C_i = \prod_{j=1}^{2m}X_{i,j}$ where $i \in \left\{2,\cdots,m-1\right\}.$ We have 
$$\tau_{i-1}^*\cdots\tau_2^*\tau_1^*\mathcal C_i = \prod_{j=1}^{2m}X_{i,j}Y_{i-1,j-1} =  \prod_{j=1}^{2m}X_{i,j}\prod_{j=1}^{2m}Y_{i-1,j} = \prod_{j=1}^{2m}X_{i,j}\prod_{j=1}^{2m}X_{i-1,j}.$$

This implies

$$\tau_{i}^*\tau_{i-1}^*\cdots\tau_2^*\tau_1^*\mathcal C_i = \prod_{j=1}^{2m}{X_{i,j} \over Y_{i,j}Y_{i,j-1}}\prod_{j=1}^{2m}X_{i-1,j}Y_{i,j} = \prod_{j=1}^{2m}X_{i-1,j} \to w^*\mathcal C_i = \prod_{j=1}^{2m}X_{i-1,j} = \mathcal C_{i-1}.$$

On the other hand, consider $\mathcal C_1 = \prod_{j=1}^{2m}X_{1,j}.$ We get, $$\tau_{1}^*\mathcal C_1 = \prod_{j=1}^{2m}{X_{1,j}\over Y_{1,j}Y_{1,j-1}} = {1 \over \prod_{j=1}^{2m}{X_{1,j}}} \to \tau_2^*\tau_1^*\mathcal C_1 = {1 \over \prod_{j=1}^{2m}{X_{1,j}Y_{2,j}}} = {1 \over \prod_{j=1}^{2m}{X_{1,j}X_{2,j}}} \to \cdots $$
$$\to w^*\mathcal C_1 = {1 \over \prod_{j=1}^{2m}{X_{1,j}X_{2,j}\cdots X_{m,j}}} = {1 \over \mathcal C_m\prod_{j=1}^m\mathcal C_j}.$$

Next, consider $\mathcal C_m = \prod_{j=1}^mX_{m,j}.$
We have $$\tau_{m-1}^*\cdots\tau_2^*\tau_1^*\mathcal C_m = \prod_{j=1}^mX_{m,j}Y_{m-1,j-1}Y_{m-1,j-1+m} = \prod_{j=1}^mX_{m,j}\prod_{j=1}^{2m}Y_{m-1,j} = \prod_{j=1}^mX_{m,j}\prod_{j=1}^{2m}X_{m-1,j} $$

$$\to w^*\mathcal C_m = \prod_{j=1}^m{X_{m,j} \over Y_{m,j}Y_{m,j-1}}\prod_{j=1}^{2m}X_{m-1,j}Y_{m,j} = \prod_{j=1}^mX_{m,j}\prod_{j=1}^{2m}X_{m-1,j} = \mathcal C_m\mathcal C_{m-1}.$$

In summary, $w^*\mathcal C_1 = {1 /( \mathcal C_m\prod_{j=1}^m\mathcal C_j)}, w^*\mathcal C_i =\mathcal C_{i-1},$ and $w^*\mathcal C_m = \mathcal C_m\mathcal C_{m-1}.$

By repeatedly applying the equations above, we obtain the following expression for $(w_0)^*(\mathcal C_m) =(w^*)^m(\mathcal C_m)$:

$$\mathcal C_m \xrightarrow{w^*}\mathcal C_{m}\mathcal C_{m-1}\xrightarrow{w^*}\mathcal C_{m}\mathcal C_{m-1}\mathcal C_{m-2} \xrightarrow{w^*} \cdots \xrightarrow{w^*}\mathcal C_{m}\mathcal C_{m-1}\cdots \mathcal C_2\mathcal C_1 \xrightarrow{w^*} {\prod_{j=1}^m\mathcal C_j \over \mathcal C_m\prod_{j=1}^mC_j} = {1 \over \mathcal C_m}.$$

Similarly, we have
$$\mathcal C_i \xrightarrow{w^*} \mathcal C_{i-1} \xrightarrow{w^*} \cdots \xrightarrow{w^*}\mathcal C_1 \xrightarrow{w^*} {1 \over \mathcal C_m\prod_{j=1}^m\mathcal C_j} \xrightarrow{w^*}{\mathcal C_m\prod_{j=1}^m\mathcal C_j \over \mathcal C_m\mathcal C_{m-1}\prod_{j=1}^m\mathcal C_j} = {1 \over \mathcal C_{m-1}} \xrightarrow{w^*} {1 \over \mathcal C_{m-2}} \xrightarrow{w^*} \cdots \xrightarrow{w^*} {1 \over \mathcal C_i},$$
and 
$$\mathcal C_1 \xrightarrow{w^*} {1 \over \mathcal C_m\prod_{j=1}^m\mathcal C_j} \xrightarrow{w^*}{\mathcal C_m\prod_{j=1}^mC_j \over \mathcal C_m\mathcal C_{m-1}\prod_{j=1}^m\mathcal C_j} = {1 \over \mathcal C_{m-1}} \xrightarrow{w^*} {1 \over \mathcal C_{m-2}} \xrightarrow{w^*} \cdots \xrightarrow{w^*} {1 \over \mathcal C_1}.$$

Thus, we finally get 
$$(w_0)^*\mathcal C_i = {1 \over \mathcal C_i}$$ 
for any $i \in \left\{1,\cdots,m\right\}.$\end{proof}

\begin{proof}
    \textit{(Proof of Theorem~\ref{thm621})}

Let $(X_{i,j})$ be an initial $\mathcal X$-torus and $(X_{i,j}')$ be an $\mathcal X$-torus after applying a cluster transformation $w_0.$ Let $\mathcal C_i' = \prod_{j=1}^{N_i} X_{i,j}'.$ Then, $w_0^*(\mathcal C_i') = {1 \over \mathcal C_i}$ because the pullback of \(C_i'\) under \(w_0\) is \(C_i^{-1}\) by Lemma~\ref{lem622}.

Consider the $C$-matrix of $w_0$ and assume $j \in \left\{1,\cdots,n\right\}.$ Each $j$th column corresponds to the tropicalization of $w_0^*(X_{1,j}')$ and has a sign-coherence property \textit{(Theorem~\ref{thm616})}. Moreover, these columns are identical up to permutations: It is a consequence of both the cyclic symmetry of the $\mathcal A_n$-quiver (Proposition~\ref{prop321}) and the fact that each $\tau_k$ is independent of the order of mutations \textit{(Theorem~\ref{thm331})}.

Since $w_0^*(\mathcal C_1') = {1 \over \mathcal C_1}$, 
the sum of all entries in the first $n$ columns of the $C$-matrix of $w_0$ is $-n$. This implies entries of a column should be $0$ except a single entry of $-1$ due to the sign-coherence of the $C$-matrix. Moreover, this non-zero entry must be located within the first $n$ rows due to $w_0^*(\mathcal C_1') = {1 \over \mathcal C_1}$. By applying this argument on every cycle, we conclude that the $C$-matrix of $w_0$ has the following block-diagonal form:

$$\begin{pmatrix}
    -P_{\sigma_1} & 0 & \dots & 0 \\
    0 & -P_{\sigma_2} & \dots & 0 \\
    \vdots & \vdots & \ddots & \vdots \\
    0 & 0 & \dots & -P_{\sigma_m}
\end{pmatrix}$$

where $P_{\sigma_i}$ is a $N_i \times N_i$ permutation matrix with respect to a permutation \(\sigma_i\in S_{N_i}\). Recall that $N_i$ is the length of the $i$th cycle. This implies $[w_0^*(X'_{i,j})] = (X_{i,\sigma_i(j)})^{-1}$.

Let us denote the following squares of elementary geodesic functions as:
$$\beta_{n} :=\left(\left\langle X_{j,2m}\Big|_{j=1}^{m}, X_{m-j,m-j+2m}\Big|_{j=1}^{m-1} \right\rangle\right)^2,\text{  }\text{  }\beta_{1} :=\left(\left\langle X_{j,1}\Big|_{j=1}^{m}, X_{m-j,m-j+1}\Big|_{j=1}^{m-1} \right\rangle\right)^2$$
and
$$\beta'_{n} :=\left(\left\langle X'_{j,2m}\Big|_{j=1}^{m}, X'_{m-j,m-j+2m}\Big|_{j=1}^{m-1} \right\rangle\right)^2,\text{  }\text{  }\beta'_{1} :=\left(\left\langle X'_{j,1}\Big|_{j=1}^{m}, X'_{m-j,m-j+1}\Big|_{j=1}^{m-1} \right\rangle\right)^2$$

From Theorem~\ref{thm422}, we know $w_0^*(\beta'_1) = \beta_1$ and $w_0^*(\beta'_{n}) = \beta_{n}$ since every elementary geodesic function is invariant under any $\tau_i^*$. By applying the tropicalization on the equations $\beta_1 = w_0^*(\beta'_1)$ and $\beta_{n} = w_0^*(\beta'_{n})$, we have
$$\prod_{j=1}^{m}(X_{j,1})^{-1}\prod_{j=1}^{m-1}(X_{m-j,m-j+1})^{-1} = \prod_{j=1}^{m}[w_0^*(X'_{j,1})]\prod_{j=1}^{m-1}[w_0^*(X'_{m-j,m-j+1})]$$
and
$$\prod_{j=1}^{m}(X_{j,2m})^{-1}\prod_{j=1}^{m-1}(X_{m-j,m-j+2m})^{-1} = \prod_{j=1}^{m}[w_0^*(X'_{j,2m})]\prod_{j=1}^{m-1}[w_0^*(X'_{m-j,m-j+2m})]$$

since $w_0^*$ is a homomorphism and $[w_0^*(X'_{i,r})] = (X_{i,\sigma_i(r)})^{-1}$ from the block-diagonal form above.

$(X_{1,1})^{-1}$ is the only variable that appears in both $[\beta_1]$ and $[\beta_{n}]$ among variables $\left\{X_{1,j}\right\}_{j=1}^n$ (see the left sides of equations above). Similarly, $[w_0^*(X_{1,1}')]$ is the only variable that appears in both $[w_0^*(\beta_1')]$ and $[w_0^*(\beta_{n}')]$ among variables $\left\{[w_0^*(X_{1,j}')]\right\}_{j=1}^n$ (see the right sides of equations above).

Hence, we conclude that 
$$[w_0^*(X_{1,1}')] = (X_{1,1})^{-1}.$$ 
Applying the same argument to all variables in first $m-1$ cycles, we find that $P_{\sigma_i} = I$ for all $i \in \left\{1,\cdots,m-1\right\}$: It is possible because of the cyclic symmetry of the $\mathcal A_n$-quiver.

For the $m$th cycle, see the following equations derived from the equation $\beta_1 = w_0^*(\beta'_1)$:
$$\prod_{j=1}^{m}(X_{j,1})^{-1}\prod_{j=1}^{m-1}(X_{m-j,m-j+1})^{-1} = \prod_{j=1}^{m}[w_0^*(X'_{j,1})]\prod_{j=1}^{m-1}[w_0^*(X'_{m-j,m-j+1})].$$

We have $[w_0^*(X'_{i,j})] = (X_{i,j})^{-1}$ for $i \le m-1$, so we have 
$$(X_{m,1})^{-1} = [w_0^*(X'_{m,1})].$$
Applying the same argument to the other variables in the $m$th cycle, we conclude that $P_{\sigma_m} = I.$ This shows that the $C$-matrix of $w_0$ is $-I$, so it is the cluster DT transformation. \end{proof}

\begin{exmp}\textit{(Example: $n=4$)}

To provide a concrete illustration of Theorem~\ref{thm621}, we present the proof for the case $n=4$.

Consider $w_0 = \tau_1 \tau_2 \tau_1 \tau_2$ and the Casimir elements $\mathcal C_1 = X_{1,1}X_{1,2}X_{1,3}X_{1,4}$ and $\mathcal C_2 = X_{2,1}X_{2,2}$. Let $(X_{i,j})$ denote the initial $\mathcal X$-torus and $(X'_{i,j})$ denote the $\mathcal X$-torus after applying the cluster transformation $w_0$.

Consider the $C$-matrix of $w_0$ and assume $j \in \left\{1,2,3,4\right\}.$ Each $j$th column corresponds to the tropicalization of $w_0^*(X_{1,j}')$ and has a sign-coherence property \textit{(Theorem~\ref{thm616})}. Moreover, these columns are identical up to permutations: It is a consequence of both the cyclic symmetry of the $\mathcal A_4$-quiver (Proposition~\ref{prop321}) and the fact that $\tau_1$ and $\tau_2$ are independent of the order of mutations\textit{(Theorem~\ref{thm331})}. 

Similarly, for $j \in \{5,6\}$, the same properties hold for $w_0^*(X'_{2,1})$ and $w_0^*(X'_{2,2})$. 

Consequently, we can write $[w_0^*(X'_{i,j})] = (X_{i,\sigma_i(j)})^{-1}$, where $\sigma_i \in S_{N_i}$ are permutations. This relationship arises from the fact that the action of $w_0$ satisfies the following identity for each $i \in \{1, 2\}$:$$\prod_{j = 1}^{N_i} w_0^*(X'_{i,j}) = w_0^* \left( \prod_{j = 1}^{N_i} X'_{i,j} \right) = \prod_{j=1}^{N_i} (X_{i,j})^{-1}.$$
Next, analogous to the general proof, we consider the square of the elementary formal geodesic functions:
\begin{align*}
    \beta_1 &= \langle X_{1,1},X_{2,1},X_{1,2}\rangle^2 \\
    &= (X_{1,1}X_{2,1}X_{1,2})\left(1+\frac{1}{X_{1,1}} + \frac{1}{X_{1,1}X_{2,1}} + \frac{1}{X_{1,1}X_{2,1}X_{1,2}}\right)^2, \\
    \beta_4 &= \langle X_{1,4},X_{2,2},X_{1,1}\rangle^2 \\
    &= (X_{1,4}X_{2,2}X_{1,1})\left(1+\frac{1}{X_{1,4}} + \frac{1}{X_{1,4}X_{2,2}} + \frac{1}{X_{1,4}X_{2,2}X_{1,1}}\right)^2.
\end{align*}
Note that $X_{2,4} = X_{2,2}$. Similarly, we define $\beta'_1$ and $\beta'_4$ using the coordinates $(X'_{i,j})$. We have that $w_0^*(\beta_1') = \beta_1$ and $w_0^*(\beta_4') = \beta_4$ by Theorem~\ref{thm422}.

Applying the tropicalization $[\cdot]$, the condition $[w_0^*(\beta_1')] = [\beta_1]$ yields the relation between the monomials:
\[
    [w_0^*(X'_{1,1})][w_0^*(X'_{2,1})][w_0^*(X'_{1,2})] = (X_{1,1})^{-1}(X_{2,1})^{-1}(X_{1,2})^{-1}.
\]
Substituting $[w_0^*(X'_{1,j})] = (X_{1,\sigma_1(j)})^{-1}$ and $[w_0^*(X'_{2,j})] = (X_{2,\sigma_2(j)})^{-1}$, this becomes:
\[
    (X_{1,\sigma_1(1)})^{-1}(X_{2,\sigma_2(1)})^{-1}(X_{1,\sigma_1(2)})^{-1} = (X_{1,1})^{-1}(X_{2,1})^{-1}(X_{1,2})^{-1}.
\]
Similarly, the condition $[w_0^*(\beta_4')] = [\beta_4]$ implies:
\[
    (X_{1,\sigma_1(4)})^{-1}(X_{2,\sigma_2(2)})^{-1}(X_{1,\sigma_1(1)})^{-1} = (X_{1,4})^{-1}(X_{2,2})^{-1}(X_{1,1})^{-1}.
\]
Comparing the variables in these two equations necessitates $\sigma_1(1) = 1$. Consequently, we obtain $[w_0^*(X'_{1,1})] = X_{1,1}^{-1}$. Generalizing this comparison to other elementary geodesic functions, it follows that $[w_0^*(X'_{1,j})] = X_{1,j}^{-1}$ for all $j$. Finally, given that $w_0^*(X'_{1,j}) = X_{1,j}^{-1}$ for all $j$, the equations above immediately imply $[w_0^*(X'_{2,j})] = X_{2,j}^{-1}$ for all $j$. This logic demonstrates that the $C$-matrix of $w_0$ is $-I$.
\end{exmp}

\begin{rem}\textit{(Cluster DT transformation on the doubled $\mathcal A_n$-quiver)}

For arbitrary $n$, let $m = \lfloor n/2 \rfloor$. The cluster transformation $s_0 := (s_1\cdots s_{m})^{m}$ acts as the cluster DT transformation on the doubled $\mathcal{A}_n$-quiver.

This result follows from reasoning analogous to Theorem~\ref{thm621}. Specifically, we observe:
\[
    s_0^*(\mathcal{C}_i) = \mathcal{C}_i^{-1} \quad \text{and} \quad s_0^*(\widetilde{\mathcal{C}_i}) = \left(\widetilde{\mathcal{C}_i}\right)^{-1}.
\]
Here, $\mathcal{C}_i = \prod_j U_{i,j}$ and $\widetilde{\mathcal{C}}_i = \prod_j \widetilde{U_{i,j}}$. Then, we use Theorem~\ref{thm616} and the cyclic symmetry of the doubled quiver. Furthermore, we consider the elementary geodesic functions depending on the parity of $n$.

If $n$ is even, we consider:
\[
    \left\langle \widetilde{U_{1,i}}, \dots, \widetilde{U_{m,i}}, U_{m-1,m-1+i}, \dots, U_{1,1+i} \right\rangle.
\]

If $n$ is odd, we consider:
\[
    \left\langle \widetilde{U_{1,i}}, \dots, \widetilde{U_{m,i}}, U_{m,m+i}, \dots,U_{1,1+i} \right\rangle.
\]
Applying tropicalization to these functions implies that $s_0$ is indeed the DT transformation. Consequently, the doubled $\mathcal{A}_n$-quiver admits the theta basis. We omit the detailed verification as the proof is very analogous to Theorem~\ref{thm621}.
\end{rem}

However, the proof of Theorem~\ref{thm621} fails for standard $\mathcal A_n$-quiver with odd $n$ because the innermost cycle of the quiver is not chordless: In fact, there is no reddening sequence when $n$ is odd.

\begin{exmp}
    (\textit{Nonexistence of the reddening sequence for $n = 5$})
    
    In the $\mathcal A_5$-quiver, the innermost cycle is the left-hand side quiver in the following figure:
    \begin{figure}[H]
        \centering
        \begin{tikzcd}[scale cd=0.8]
	& 1 && 2 &&&&& 1 && 2 \\
	\\
	5 &&&& 3 &&& 5 &&&& 3 \\
	\\
	&& 4 &&&&&&& 4
	\arrow[from=1-2, to=1-4]
	\arrow[from=1-2, to=3-5]
	\arrow[from=1-4, to=3-5]
	\arrow[from=1-4, to=5-3]
	\arrow[from=1-9, to=3-8]
	\arrow[from=1-11, to=5-10]
	\arrow[from=3-1, to=1-2]
	\arrow[from=3-1, to=1-4]
	\arrow[from=3-5, to=3-1]
	\arrow[from=3-5, to=5-3]
	\arrow[Rightarrow, from=3-8, to=5-10]
	\arrow[from=3-12, to=1-11]
	\arrow[Rightarrow, from=3-12, to=3-8]
	\arrow[from=5-3, to=1-2]
	\arrow[from=5-3, to=3-1]
	\arrow[from=5-10, to=1-9]
	\arrow[Rightarrow, from=5-10, to=3-12]
\end{tikzcd}
\caption{An innermost cycle of the $\mathcal A_5$-quiver.}
    \label{Fig24}
    \end{figure}
    After applying mutations of $4,1, 2$, the quiver is transformed into the right-hand side quiver. This new quiver includes the Markov quiver as its full subquiver.

    Since the Markov quiver does not admit a reddening sequence (\textit{Example~\ref{exmp617}}), the right-hand side quiver does not admit a reddening sequence by \textit{Theorem~\ref{thm615}}. Hence, the left-hand side quiver also does not admit a reddening sequence by \textit{Proposition~\ref{prop614}}. As a result, $\mathcal A_5$-quiver does not admit a reddening sequence by \textit{Theorem~\ref{thm615}}, since the left-hand side quiver in \textbf{Figure \ref{Fig24}} is its full subquiver.
\end{exmp}

\begin{thm} \label{thm626}
    There is no reddening sequence for the $\mathcal A_n$-quiver for any odd $n = 2m + 1$.
\end{thm}
\begin{proof}
    We can express the innermost cycle as follows:

    \begin{figure}[H]
        \centering
	\begin{tikzcd}[scale cd=1,sep = tiny]
	& 1 &&&&&& 1 \\
	2 && 3 &&&& 2 && 3 \\
	& 4 && 5 &&&& 4 && 5 \\
	&& 6 && 7 &&&& 6 && 7 \\
	&&& \cdots && \cdots && {\longeq} && \cdots && \cdots \\
	&&&& {\text{\small{2m-2}}} && {\text{\small{2m-1}}} &&&& {\text{\small{2m-2}}} && {\text{\small{2m-1}}} \\
	&&&&& {\text{\small{2m}}} && {\text{\small{2m+1}}} &&&& {\text{\small{2m}}} && {\text{\small{2m+1}}} \\
	&&&&&& 1 && 2
	\arrow[from=1-2, to=2-3]
	\arrow[from=1-8, to=2-9]
	\arrow[curve={height=-24pt}, from=1-8, to=7-14]
	\arrow[dashed, from=2-1, to=1-2]
	\arrow[from=2-1, to=3-2]
	\arrow[from=2-3, to=2-1]
	\arrow[from=2-3, to=3-4]
	\arrow[from=2-7, to=1-8]
	\arrow[from=2-7, to=3-8]
	\arrow[from=2-9, to=2-7]
	\arrow[from=2-9, to=3-10]
	\arrow[from=3-2, to=2-3]
	\arrow[from=3-2, to=4-3]
	\arrow[from=3-4, to=3-2]
	\arrow[from=3-4, to=4-5]
	\arrow[from=3-8, to=2-9]
	\arrow[from=3-8, to=4-9]
	\arrow[from=3-10, to=3-8]
	\arrow[from=3-10, to=4-11]
	\arrow[from=4-3, to=3-4]
	\arrow[from=4-3, to=5-4]
	\arrow[from=4-5, to=4-3]
	\arrow[from=4-5, to=5-6]
	\arrow[from=4-9, to=3-10]
	\arrow[from=4-9, to=5-10]
	\arrow[from=4-11, to=4-9]
	\arrow[from=4-11, to=5-12]
	\arrow[from=5-4, to=4-5]
	\arrow[from=5-4, to=6-5]
	\arrow[from=5-6, to=6-7]
	\arrow[from=5-10, to=4-11]
	\arrow[from=5-10, to=6-11]
	\arrow[from=5-12, to=6-13]
	\arrow[from=6-5, to=5-6]
	\arrow[from=6-5, to=7-6]
	\arrow[from=6-7, to=6-5]
	\arrow[from=6-7, to=7-8]
	\arrow[from=6-11, to=5-12]
	\arrow[from=6-11, to=7-12]
	\arrow[from=6-13, to=6-11]
	\arrow[from=6-13, to=7-14]
	\arrow[from=7-6, to=6-7]
	\arrow[from=7-6, to=8-7]
	\arrow[from=7-8, to=7-6]
	\arrow[from=7-8, to=8-9]
	\arrow[curve={height=-70pt}, from=7-12, to=1-8]
	\arrow[from=7-12, to=6-13]
	\arrow[curve={height=-50pt}, from=7-14, to=2-7]
	\arrow[from=7-14, to=7-12]
	\arrow[from=8-7, to=7-8]
	\arrow[dashed, from=8-9, to=8-7]
\end{tikzcd}
\caption{An innermost cycle of the $\mathcal A_n$-quiver.}
    \label{Fig25}
    \end{figure}

    Consider a sequence of mutations $\mu_{2m}\mu_{2m-1}\mu_{2m-2}\cdots\mu_5\mu_4\mu_3$ on the right-hand side quiver above. This transforms the quiver to the following quiver:

    \begin{figure}[H]
        \centering
	\begin{tikzcd}[scale cd=0.95,sep = tiny]
	& 1 \\
	2 && 3 \\
	& 4 && 5 \\
	&& 6 && 7 \\
	&&& \cdots && \cdots \\
	&&&& {\text{\small{2m-2}}} && {\text{\small{2m-1}}} \\
	&&&&& {\text{\small{2m}}} && {\text{\small{2m+1}}}
	\arrow[Rightarrow, from=1-2, to=3-2]
	\arrow[Rightarrow, from=2-1, to=1-2]
	\arrow[from=2-3, to=1-2]
	\arrow[from=2-3, to=3-4]
	\arrow[from=3-2, to=2-1]
	\arrow[from=3-2, to=2-3]
	\arrow[from=3-2, to=4-3]
	\arrow[from=3-4, to=3-2]
	\arrow[from=3-4, to=4-5]
	\arrow[from=4-3, to=3-4]
	\arrow[from=4-3, to=5-4]
	\arrow[from=4-5, to=4-3]
	\arrow[from=4-5, to=5-6]
	\arrow[from=5-4, to=4-5]
	\arrow[from=5-4, to=6-5]
	\arrow[from=5-6, to=6-7]
	\arrow[from=6-5, to=5-6]
	\arrow[from=6-5, to=7-6]
	\arrow[from=6-7, to=6-5]
	\arrow[from=7-6, to=6-7]
	\arrow[from=7-6, to=7-8]
	\arrow[curve={height=-60pt}, from=7-8, to=2-1]
\end{tikzcd}
\caption{A quiver after applying mutations $\mu_{2m}\mu_{2m-1}\mu_{2m-2}\cdots\mu_5\mu_4\mu_3$ on the quiver of Figure~\ref{Fig25}.}
    \label{Fig26}
    \end{figure}
    Next, a sequence of mutations $\mu_{2m+1}\mu_{2m}\mu_{2m-2}\cdots\mu_{10}\mu_8\mu_6$ transforms a quiver in Figure~\ref{Fig26} into a quiver in Figure~\ref{Fig27}. 
    
\begin{figure}[H]
        \centering
	\begin{tikzcd}[scale cd=0.95,sep = tiny]
	& 1 \\
	2 && 3 \\
	& 4 && 5 \\
	&& 6 && 7 \\
	&&& \cdots && \cdots \\
	&&&& {\text{\small{2m-4}}} && {\text{\small{2m-3}}} \\
	&&&&& {\text{\small{2m-2}}} && {\text{\small{2m-1}}} \\
	&&&&&& {\text{\small{2m}}} && {\text{\small{2m+1}}}
	\arrow[Rightarrow, from=1-2, to=3-2]
	\arrow[Rightarrow, from=2-1, to=1-2]
	\arrow[curve={height=30pt}, from=2-1, to=4-3]
	\arrow[from=2-3, to=1-2]
	\arrow[from=2-3, to=3-4]
	\arrow[Rightarrow, from=3-2, to=2-1]
	\arrow[from=3-2, to=2-3]
	\arrow[from=3-4, to=4-3]
	\arrow[from=3-4, to=4-5]
	\arrow[from=4-3, to=3-2]
	\arrow[from=4-3, to=5-4]
	\arrow[from=4-5, to=5-4]
	\arrow[from=4-5, to=5-6]
	\arrow[from=5-4, to=3-4]
	\arrow[from=5-4, to=6-5]
	\arrow[from=5-6, to=6-5]
	\arrow[from=5-6, to=6-7]
	\arrow[from=6-5, to=7-6]
	\arrow[from=6-7, to=7-6]
	\arrow[from=6-7, to=7-8]
	\arrow[from=7-6, to=5-6]
	\arrow[from=7-6, to=8-7]
	\arrow[from=7-8, to=8-7]
	\arrow[from=8-7, to=6-7]
	\arrow[from=8-7, to=8-9]
	\arrow[from=8-9, to=7-8]
\end{tikzcd}
\caption{A quiver after applying mutations $\mu_{2m+1}\mu_{2m}\mu_{2m-2}\cdots\mu_{10}\mu_8\mu_6$ on the quiver of Figure~\ref{Fig26}}
    \label{Fig27}
    \end{figure}

    The Markov quiver does not admit a reddening sequence; therefore, by \textit{Theorem~\ref{thm615}}, the quiver in \textbf{Figure \ref{Fig27}} also fails to admit a reddening sequence.
    
    Since the quivers in \textbf{Figure \ref{Fig25}} and \textbf{Figure \ref{Fig27}} are mutation equivalent, \textit{Proposition~\ref{prop614}} implies that the quiver in \textbf{Figure \ref{Fig25}} also does not admit a reddening sequence. Hence, we conclude that the $\mathcal{A}_n$-quiver does not possess a reddening sequence, as it contains the quiver in \textbf{Figure \ref{Fig25}} as a full subquiver.\end{proof}

\subsection{Tropicalization of the Weyl Group Action and Multidegrees of Formal Geodesic Functions} \label{ch6.3}

In this section, we study the tropical Weyl group action on the cluster \(\mathcal A\)-variety. We first recall the tropicalization of cluster \(\mathcal A\)-transformations, following~\cite{3}, and apply it to the Weyl group action. We then establish a decomposition of tropical coordinate vectors that will be used in the next section, which concerns Weyl-invariant functions. More precisely, we show that every \( \mathbf a(l)\in\mathbb Z^{n(n-1)/2} \) can be written as an integral linear combination of the multidegrees of the Casimirs and a nonnegative integral linear combination of the minimal multidegrees of the formal geodesic functions.

Let \( \mathbb Q^+(\mathcal A_{|\Sigma|}) \) denote the semifield of positive rational functions on the cluster \(\mathcal A\)-variety \(\mathcal A_{|\Sigma|}\). A nonzero rational function \(F\) on \(\mathcal A_{|\Sigma|}\) is called positive if it can be expressed in one, and hence every, cluster torus as a ratio of polynomials with nonnegative integral coefficients. The set \(\mathbb Q^+(\mathcal A_{|\Sigma|})\) is closed under addition, multiplication, and division. Let \(\mathbb Z^t\) denote the tropical semifield. Its tropical multiplication and tropical addition are defined by \[ a \cdot b:=a+b, \qquad a\oplus b:=\min\{a,b\}. \] 

\begin{defn} \label{defn631}
The set of tropical points of \(\mathcal A_{|\Sigma|}\) is 
\[ \mathcal A_{|\Sigma|}(\mathbb Z^t) := \operatorname{Hom}_{\mathrm{semifield}} \left( \mathbb Q^+(\mathcal A_{|\Sigma|}), \mathbb Z^t \right). \]

For \( F\in\mathbb Q^+(\mathcal A_{|\Sigma|}), \) its tropicalization is the function \( F^t: \mathcal A_{|\Sigma|}(\mathbb Z^t) \longrightarrow\mathbb Z \) defined by \( F^t(l):=l(F). \) \end{defn} 

Fix a seed \( \Sigma \) with cluster \(\mathcal A\)-coordinates \( \{A_i\}_{i\in I}. \) Every tropical point \( l\in\mathcal A_{|\Sigma|}(\mathbb Z^t) \) is represented in this seed by the tropical coordinate vector \[ \mathbf a_\Sigma(l) = \bigl[A_i^t(l)\bigr]_{i\in I} = \bigl[l(A_i)\bigr]_{i\in I}. \]

For every positive rational function \(F\) on \(\mathcal A_{\Sigma'}\), a cluster $\mathcal A$-transformation induces a piecewise-linear transformation \( \tau^t: \mathcal A_{\Sigma}(\mathbb Z^t) \longrightarrow \mathcal A_{\Sigma'}(\mathbb Z^t) \) defined by \[ F^t(\tau^t(l)) = (\tau^*F)^t(l). \] 

We now apply this construction to the Weyl group action introduced in Section~\ref{Ch3.3}. From now on, we abbreviate $\mathcal A_{|\mathcal A_n|}(\mathbb Z^t)$ and $\mathbf a_{\mathcal A_n}(l)$ as $\mathcal A(\mathbb Z^t)$ and $\mathbf a(l)$, respectively. By Theorem~\ref{thm333}, the action of $\tau_k$ on the cluster $\mathcal A$-variables is given by \[ \tau_k^*(A_{i,j}) = \begin{cases} A_{i,j}W_k, & i=k,\\ A_{i,j}, & i\neq k. \end{cases} \]

\begin{prop}\label{prop632} 
Let \(l\in\mathcal A(\mathbb Z^t)\), and set \( a_{i,j}:=A_{i,j}^t(l). \) By the explicit formulas in Proposition~\ref{prop334}, the tropicalization of \(W_k\) is given by, 
\[ W_k^t(l)= \begin{cases} \displaystyle \min_{1\le r\le n} \left\{ a_{2,r+1}-a_{1,r}-a_{1,r+1} \right\}, & k=1,\\[4mm] \displaystyle \min_{1\le r\le n} \left\{ a_{k-1,r}+a_{k+1,r+1} -a_{k,r}-a_{k,r+1} \right\}, & 2\le k<\left\lfloor\frac n2\right\rfloor. \end{cases} \] 

If \(n=2m\), then for the innermost cycle we additionally have 
\[ W_m^t(l) = \min_{1\le r\le m} \left\{ a_{m-1,r}+a_{m-1,r+m} -a_{m,r}-a_{m,r+1} \right\}. \] 

Consequently, the tropical action of \(\tau_k\) is given by 

\[ A_{i,j}^t\bigl(\tau_k^t(l)\bigr) = \begin{cases} a_{i,j}+W_k^t(l), & i=k,\\ a_{i,j}, & i\neq k. \end{cases} \] 

In particular, \[ \tau_k^t(l)=l \quad\Longleftrightarrow\quad W_k^t(l)=0. \] 

For odd \(n\), the proposition applies only to \(1\le k<\lfloor n/2\rfloor\). \end{prop}

\begin{proof} By Proposition~\ref{prop334}, \(W_k\) is a sum of Laurent monomials. Tropicalization replaces addition by the minimum and multiplication by ordinary addition. Tropicalizing the explicit formulas for \(W_k\) in Proposition~\ref{prop334} therefore gives the stated expressions for \(W_k^t(l)\). 

Moreover, Proposition~\ref{prop334} gives \[ \tau_k^*(A_{i,j}) = \begin{cases} A_{i,j}W_k, & i=k,\\ A_{i,j}, & i\neq k. \end{cases} \] 

Applying tropicalization yields \[ A_{i,j}^t\bigl(\tau_k^t(l)\bigr) = \begin{cases} a_{i,j}+W_k^t(l), & i=k,\\ a_{i,j}, & i\neq k. \end{cases} \] 

Hence, \(\tau_k^t(l)=l\) if and only if \(W_k^t(l)=0\). For odd \(n\), Proposition~\ref{prop334} applies only to the chordless main cycles, namely, \(1\leq k<\lfloor n/2\rfloor\). \end{proof}

We next show that every \( \mathbf a(l)\in\mathbb Z^{n(n-1)/2} \) can be written as an integral linear combination of the multidegrees of Casimir monomials and a nonnegative integral linear combination of the unique minimal multidegrees of formal geodesic functions. To formulate and prove this decomposition, we first recall the description of formal geodesic functions established in~\cite{1}.

We glue two copies of the $\mathcal A_n$-quiver and consider its dual quiver. To compute $\mathbb{A}_{p,q}$, we sum over all paths starting at the northeastern vertex $q$ and terminating at the southwestern vertex $p'$.

For each path from $q$ to $p'$, consider the parallelogram $\blacklozenge_{p,q}$ defined by the vertices $q$ and $p'$. The cluster variables within this parallelogram contribute as follows: Those lying above the path have exponent $1/2$, while those lying below the path have exponent $-1/2$. Variables outside the parallelogram do not contribute to $\mathbb A_{p,q}$; see Figure \ref{Fig28} below.

\begin{figure}[H]
    \centering 
	\begin{tikzcd}[scale cd=0.6,column sep = tiny, row sep = tiny]
	&&&&& {X_{1,3}} \\
	&&&&& \bullet && \textcolor{rgb,255:red,255;green,0;blue,0}{2} \\
	&&&& {X_{2,4}} && {X_{1,4}} &&&&&&&& {X_{1,3}} \\
	&&&& \bullet & \bullet & \bullet && \textcolor{rgb,255:red,255;green,0;blue,0}{3} &&&&&& \bullet && \textcolor{rgb,255:red,255;green,0;blue,0}{2} \\
	&&& {X_{2,2}} && {X_{2,5}} && {X_{1,5}} &&&&&& {X_{2,2}} && {X_{1,4}} \\
	&&& \bullet & \bullet & \bullet & \bullet & \bullet && \textcolor{rgb,255:red,255;green,0;blue,0}{4} &&&& \bullet & \bullet & \bullet && \textcolor{rgb,255:red,255;green,0;blue,0}{3} \\
	&& {X_{1,2}} && {X_{2,3}} && {X_{2,1}} && {X_{1,1}} &&&& {X_{1,2}} && {X_{2,1}} && {X_{1,1}} \\
	&&& \bullet & \bullet & \bullet & \bullet & \bullet & \bullet && \textcolor{rgb,255:red,255;green,0;blue,0}{5} &&& \bullet & \bullet & \bullet & \bullet && \textcolor{rgb,255:red,255;green,0;blue,0}{4} \\
	{} & {} && {X_{1,3}} && {X_{2,4}} && {X_{2,2}} && {X_{1,2}} &&&& {X_{1,3}} && {X_{2,2}} && {X_{1,2}} && {} \\
	&& \textcolor{rgb,255:red,255;green,0;blue,0}{{{1'}}} && \bullet & \bullet & \bullet & \bullet & \bullet & \bullet &&& \textcolor{rgb,255:red,255;green,0;blue,0}{{{1'}}} && \bullet & \bullet & \bullet & \bullet \\
	&&&& {X_{1,4}} && {X_{2,5}} && {X_{2,3}} && {X_{1,3}} &&&& {X_{1,4}} && {X_{2,1}} && {X_{1,3}} \\
	&&& \textcolor{rgb,255:red,255;green,0;blue,0}{{{2'}}} && \bullet & \bullet & \bullet & \bullet & \bullet &&&& \textcolor{rgb,255:red,255;green,0;blue,0}{{{2'}}} && \bullet & \bullet & \bullet \\
	&&&&& {X_{1,5}} && {X_{2,1}} && {X_{2,4}} &&&&&& {X_{1,1}} && {X_{2,2}} \\
	&&&& \textcolor{rgb,255:red,255;green,0;blue,0}{{{3'}}} && \bullet & \bullet & \bullet &&&&&& \textcolor{rgb,255:red,255;green,0;blue,0}{{{3'}}} && \bullet \\
	&&&&&& {X_{1,1}} && {X_{2,2}} &&&&&&&& {X_{1,2}} \\
	&&&&& \textcolor{rgb,255:red,255;green,0;blue,0}{{{4'}}} && \bullet \\
	&&&&&&& {X_{1,2}}
	\arrow[from=1-6, to=3-7]
	\arrow[color={rgb,255:red,255;green,0;blue,0}, from=2-6, to=4-6]
	\arrow[color={rgb,255:red,255;green,0;blue,0}, from=2-8, to=2-6]
	\arrow[color={rgb,255:red,255;green,0;blue,0}, from=2-8, to=4-7]
	\arrow[dashed, from=3-5, to=1-6]
	\arrow[from=3-5, to=5-6]
	\arrow[from=3-7, to=3-5]
	\arrow[from=3-7, to=5-8]
	\arrow[from=3-15, to=5-16]
	\arrow[color={rgb,255:red,255;green,0;blue,0}, from=4-5, to=6-5]
	\arrow[color={rgb,255:red,255;green,0;blue,0}, from=4-6, to=4-5]
	\arrow[color={rgb,255:red,255;green,0;blue,0}, from=4-7, to=4-6]
	\arrow[color={rgb,255:red,255;green,0;blue,0}, from=4-7, to=6-7]
	\arrow[color={rgb,255:red,255;green,0;blue,0}, from=4-9, to=4-7]
	\arrow[color={rgb,255:red,255;green,0;blue,0}, from=4-9, to=6-8]
	\arrow[color={rgb,255:red,255;green,0;blue,0}, from=4-15, to=6-15]
	\arrow[color={rgb,255:red,255;green,0;blue,0}, from=4-17, to=4-15]
	\arrow[color={rgb,255:red,255;green,0;blue,0}, from=4-17, to=6-16]
	\arrow[dashed, from=5-4, to=3-5]
	\arrow[from=5-4, to=7-5]
	\arrow[from=5-6, to=3-7]
	\arrow[from=5-6, to=5-4]
	\arrow[from=5-6, to=7-7]
	\arrow[from=5-8, to=5-6]
	\arrow[from=5-8, to=7-9]
	\arrow[dashed, from=5-14, to=3-15]
	\arrow[from=5-14, to=7-15]
	\arrow[from=5-16, to=5-14]
	\arrow[from=5-16, to=7-17]
	\arrow[color={rgb,255:red,255;green,0;blue,0}, from=6-4, to=8-4]
	\arrow[color={rgb,255:red,255;green,0;blue,0}, from=6-5, to=6-4]
	\arrow[color={rgb,255:red,255;green,0;blue,0}, from=6-6, to=6-5]
	\arrow[color={rgb,255:red,255;green,0;blue,0}, from=6-6, to=8-6]
	\arrow[color={rgb,255:red,255;green,0;blue,0}, from=6-7, to=6-6]
	\arrow[color={rgb,255:red,255;green,0;blue,0}, from=6-8, to=6-7]
	\arrow[color={rgb,255:red,255;green,0;blue,0}, from=6-8, to=8-8]
	\arrow[color={rgb,255:red,255;green,0;blue,0}, from=6-10, to=6-8]
	\arrow[color={rgb,255:red,255;green,0;blue,0}, from=6-10, to=8-9]
	\arrow[color={rgb,255:red,255;green,0;blue,0}, from=6-14, to=8-14]
	\arrow[color={rgb,255:red,255;green,0;blue,0}, from=6-15, to=6-14]
	\arrow[color={rgb,255:red,255;green,0;blue,0}, from=6-16, to=6-15]
	\arrow[color={rgb,255:red,255;green,0;blue,0}, from=6-16, to=8-16]
	\arrow[color={rgb,255:red,255;green,0;blue,0}, from=6-18, to=6-16]
	\arrow[color={rgb,255:red,255;green,0;blue,0}, from=6-18, to=8-17]
	\arrow[dashed, from=7-3, to=5-4]
	\arrow[from=7-3, to=9-4]
	\arrow[from=7-5, to=5-6]
	\arrow[from=7-5, to=7-3]
	\arrow[from=7-5, to=9-6]
	\arrow[from=7-7, to=5-8]
	\arrow[from=7-7, to=7-5]
	\arrow[from=7-7, to=9-8]
	\arrow[from=7-9, to=7-7]
	\arrow[from=7-9, to=9-10]
	\arrow[dashed, from=7-13, to=5-14]
	\arrow[from=7-13, to=9-14]
	\arrow[from=7-15, to=5-16]
	\arrow[from=7-15, to=7-13]
	\arrow[from=7-15, to=9-16]
	\arrow[from=7-17, to=7-15]
	\arrow[from=7-17, to=9-18]
	\arrow[color={rgb,255:red,255;green,0;blue,0}, from=8-4, to=10-3]
	\arrow[color={rgb,255:red,255;green,0;blue,0}, from=8-5, to=8-4]
	\arrow[color={rgb,255:red,255;green,0;blue,0}, from=8-5, to=10-5]
	\arrow[color={rgb,255:red,255;green,0;blue,0}, from=8-6, to=8-5]
	\arrow[color={rgb,255:red,255;green,0;blue,0}, from=8-7, to=8-6]
	\arrow[color={rgb,255:red,255;green,0;blue,0}, from=8-7, to=10-7]
	\arrow[color={rgb,255:red,255;green,0;blue,0}, from=8-8, to=8-7]
	\arrow[color={rgb,255:red,255;green,0;blue,0}, from=8-9, to=8-8]
	\arrow[color={rgb,255:red,255;green,0;blue,0}, from=8-9, to=10-9]
	\arrow[color={rgb,255:red,255;green,0;blue,0}, from=8-11, to=8-9]
	\arrow[color={rgb,255:red,255;green,0;blue,0}, from=8-11, to=10-10]
	\arrow[color={rgb,255:red,255;green,0;blue,0}, from=8-14, to=10-13]
	\arrow[color={rgb,255:red,255;green,0;blue,0}, from=8-15, to=8-14]
	\arrow[color={rgb,255:red,255;green,0;blue,0}, from=8-15, to=10-15]
	\arrow[color={rgb,255:red,255;green,0;blue,0}, from=8-16, to=8-15]
	\arrow[color={rgb,255:red,255;green,0;blue,0}, from=8-17, to=8-16]
	\arrow[color={rgb,255:red,255;green,0;blue,0}, from=8-17, to=10-17]
	\arrow[color={rgb,255:red,255;green,0;blue,0}, from=8-19, to=8-17]
	\arrow[color={rgb,255:red,255;green,0;blue,0}, from=8-19, to=10-18]
	\arrow[from=9-4, to=7-5]
	\arrow[dashed, no head, from=9-4, to=9-1]
	\arrow[from=9-4, to=11-5]
	\arrow[from=9-6, to=7-7]
	\arrow[from=9-6, to=9-4]
	\arrow[from=9-6, to=11-7]
	\arrow[from=9-8, to=7-9]
	\arrow[from=9-8, to=9-6]
	\arrow[from=9-8, to=11-9]
	\arrow[from=9-10, to=9-8]
	\arrow[from=9-10, to=11-11]
	\arrow[from=9-14, to=7-15]
	\arrow[dashed, from=9-14, to=9-10]
	\arrow[from=9-14, to=11-15]
	\arrow[from=9-16, to=7-17]
	\arrow[from=9-16, to=9-14]
	\arrow[from=9-16, to=11-17]
	\arrow[from=9-18, to=9-16]
	\arrow[from=9-18, to=11-19]
	\arrow[dashed, no head, from=9-20, to=9-18]
	\arrow[color={rgb,255:red,255;green,0;blue,0}, from=10-5, to=10-3]
	\arrow[color={rgb,255:red,255;green,0;blue,0}, from=10-5, to=12-4]
	\arrow[color={rgb,255:red,255;green,0;blue,0}, from=10-6, to=10-5]
	\arrow[color={rgb,255:red,255;green,0;blue,0}, from=10-6, to=12-6]
	\arrow[color={rgb,255:red,255;green,0;blue,0}, from=10-7, to=10-6]
	\arrow[color={rgb,255:red,255;green,0;blue,0}, from=10-8, to=10-7]
	\arrow[color={rgb,255:red,255;green,0;blue,0}, from=10-8, to=12-8]
	\arrow[color={rgb,255:red,255;green,0;blue,0}, from=10-9, to=10-8]
	\arrow[color={rgb,255:red,255;green,0;blue,0}, from=10-10, to=10-9]
	\arrow[color={rgb,255:red,255;green,0;blue,0}, from=10-10, to=12-10]
	\arrow[color={rgb,255:red,255;green,0;blue,0}, from=10-15, to=10-13]
	\arrow[color={rgb,255:red,255;green,0;blue,0}, from=10-15, to=12-14]
	\arrow[color={rgb,255:red,255;green,0;blue,0}, from=10-16, to=10-15]
	\arrow[color={rgb,255:red,255;green,0;blue,0}, from=10-16, to=12-16]
	\arrow[color={rgb,255:red,255;green,0;blue,0}, from=10-17, to=10-16]
	\arrow[color={rgb,255:red,255;green,0;blue,0}, from=10-18, to=10-17]
	\arrow[color={rgb,255:red,255;green,0;blue,0}, from=10-18, to=12-18]
	\arrow[from=11-5, to=9-6]
	\arrow[from=11-5, to=13-6]
	\arrow[from=11-7, to=9-8]
	\arrow[from=11-7, to=11-5]
	\arrow[from=11-7, to=13-8]
	\arrow[from=11-9, to=9-10]
	\arrow[from=11-9, to=11-7]
	\arrow[from=11-9, to=13-10]
	\arrow[from=11-11, to=11-9]
	\arrow[from=11-15, to=9-16]
	\arrow[from=11-15, to=13-16]
	\arrow[from=11-17, to=9-18]
	\arrow[from=11-17, to=11-15]
	\arrow[from=11-17, to=13-18]
	\arrow[from=11-19, to=11-17]
	\arrow[color={rgb,255:red,255;green,0;blue,0}, from=12-6, to=12-4]
	\arrow[color={rgb,255:red,255;green,0;blue,0}, from=12-6, to=14-5]
	\arrow[color={rgb,255:red,255;green,0;blue,0}, from=12-7, to=12-6]
	\arrow[color={rgb,255:red,255;green,0;blue,0}, from=12-7, to=14-7]
	\arrow[color={rgb,255:red,255;green,0;blue,0}, from=12-8, to=12-7]
	\arrow[color={rgb,255:red,255;green,0;blue,0}, from=12-9, to=12-8]
	\arrow[color={rgb,255:red,255;green,0;blue,0}, from=12-9, to=14-9]
	\arrow[color={rgb,255:red,255;green,0;blue,0}, from=12-10, to=12-9]
	\arrow[color={rgb,255:red,255;green,0;blue,0}, from=12-16, to=12-14]
	\arrow[color={rgb,255:red,255;green,0;blue,0}, from=12-16, to=14-15]
	\arrow[color={rgb,255:red,255;green,0;blue,0}, from=12-17, to=12-16]
	\arrow[color={rgb,255:red,255;green,0;blue,0}, from=12-17, to=14-17]
	\arrow[color={rgb,255:red,255;green,0;blue,0}, from=12-18, to=12-17]
	\arrow[from=13-6, to=11-7]
	\arrow[from=13-6, to=15-7]
	\arrow[from=13-8, to=11-9]
	\arrow[from=13-8, to=13-6]
	\arrow[from=13-8, to=15-9]
	\arrow[dashed, from=13-10, to=11-11]
	\arrow[from=13-10, to=13-8]
	\arrow[from=13-16, to=11-17]
	\arrow[from=13-16, to=15-17]
	\arrow[dashed, from=13-18, to=11-19]
	\arrow[from=13-18, to=13-16]
	\arrow[color={rgb,255:red,255;green,0;blue,0}, from=14-7, to=14-5]
	\arrow[color={rgb,255:red,255;green,0;blue,0}, from=14-7, to=16-6]
	\arrow[color={rgb,255:red,255;green,0;blue,0}, from=14-8, to=14-7]
	\arrow[color={rgb,255:red,255;green,0;blue,0}, from=14-8, to=16-8]
	\arrow[color={rgb,255:red,255;green,0;blue,0}, from=14-9, to=14-8]
	\arrow[color={rgb,255:red,255;green,0;blue,0}, from=14-17, to=14-15]
	\arrow[from=15-7, to=13-8]
	\arrow[from=15-7, to=17-8]
	\arrow[dashed, from=15-9, to=13-10]
	\arrow[from=15-9, to=15-7]
	\arrow[dashed, from=15-17, to=13-18]
	\arrow[color={rgb,255:red,255;green,0;blue,0}, from=16-8, to=16-6]
	\arrow[dashed, from=17-8, to=15-9]
\end{tikzcd}
    \caption{$\mathcal A_n$-quivers and their dual quiver for $n=4,5$. The dual quiver is expressed via red arrows.}
    \label{Fig28}
\end{figure}

\begin{exmp} \label{exmp431}
    \textit{(Formal geodesic function examples for n = 4)}
    
    Let us calculate $\mathbb{A}_{1,3}$ and $\mathbb{A}_{2,4}$ (see Figure \ref{Fig28}). The parallelogram $\blacklozenge_{1,3}$ defined by vertices $3$ and $1'$ contains the variables $X_{1,1}, X_{2,1}, X_{2,2},$ and $X_{1,3}$. Summing the contributions from all paths from $3$ to $1'$, we obtain:
$$
\begin{aligned}
\mathbb{A}_{1,3} &= (X_{1,1})^{-1/2}(X_{2,1})^{-1/2}(X_{2,2})^{-1/2}(X_{1,3})^{-1/2} 
    + (X_{1,1})^{-1/2}(X_{2,1})^{1/2}(X_{2,2})^{-1/2}(X_{1,3})^{-1/2} \\
    &\quad + (X_{1,1})^{-1/2}(X_{2,1})^{1/2}(X_{2,2})^{-1/2}(X_{1,3})^{1/2} 
    + (X_{1,1})^{1/2}(X_{2,1})^{1/2}(X_{2,2})^{-1/2}(X_{1,3})^{-1/2} \\
    &\quad + (X_{1,1})^{1/2}(X_{2,1})^{1/2}(X_{2,2})^{-1/2}(X_{1,3})^{1/2} 
    + (X_{1,1})^{1/2}(X_{2,1})^{1/2}(X_{2,2})^{1/2}(X_{1,3})^{1/2}.
\end{aligned}
$$

    Similarly, the parallelogram $\blacklozenge_{2,4}$ defined by $4$ and $2'$ contains variables $X_{1,2},X_{2,2},X_{2,1},$ and $X_{1,4}$. Summing the contributions from all paths from $4$ to $2'$, we obtain:
    
$$
\begin{aligned}
\mathbb{A}_{2,4} &= (X_{1,2})^{-1/2}(X_{2,2})^{-1/2}(X_{2,1})^{-1/2}(X_{1,4})^{-1/2} 
    + (X_{1,2})^{-1/2}(X_{2,2})^{1/2}(X_{2,1})^{-1/2}(X_{1,4})^{-1/2} \\
    &\quad + (X_{1,2})^{-1/2}(X_{2,2})^{1/2}(X_{2,1})^{-1/2}(X_{1,4})^{1/2} 
    + (X_{1,2})^{1/2}(X_{2,2})^{1/2}(X_{2,1})^{-1/2}(X_{1,4})^{-1/2} \\
    &\quad + (X_{1,2})^{1/2}(X_{2,2})^{1/2}(X_{2,1})^{-1/2}(X_{1,4})^{1/2} 
    + (X_{1,2})^{1/2}(X_{2,2})^{1/2}(X_{2,1})^{1/2}(X_{1,4})^{1/2}.
\end{aligned}
$$
\end{exmp}

\begin{defn} \label{defn634}
    Let $f$ be a Laurent polynomial in the cluster variables $X_{i,j}^{1/2}$ with a unique minimal multidegree. We denote this multidegree by $\operatorname{val}(f)$. Note that each formal geodesic function $\mathbb A_{p,q}$ (Definition~\ref{defn411}) and each Casimir element $\mathcal K_i$ (Proposition~\ref{prop352}) has a unique minimal multidegree.
\end{defn}

\begin{exmp} \label{exmp635}
Recall that $\blacklozenge_{p,q}$ denotes the parallelogram defined by vertices $q$ and $p'$. The unique minimal multidegree of $\mathbb{A}_{p,q}$ is given by
\begin{equation} \label{6.4}
    \operatorname{val}(\mathbb{A}_{p,q}) = \operatorname{val}\left( \prod_{X_{i,j} \in \blacklozenge_{p,q}} X_{i,j}^{-1/2} \right).
\end{equation}
When $n=4$, $\operatorname{val}(\mathbb{A}_{1,3}) = [-1/2, 0, -1/2, 0, -1/2, -1/2]^T$ and $\operatorname{val}(\mathbb{A}_{2,4}) = [0, -1/2, 0, -1/2, -1/2, -1/2]^T$ (see Example \ref{exmp431}). Furthermore, for the Casimirs $\mathcal{K}_1 = X_{1,1}X_{1,2}X_{1,3}X_{1,4}X_{2,1}X_{2,2}$ and $\mathcal{K}_2 = X_{2,1}X_{2,2}$, we have $\operatorname{val}(\mathcal{K}_1) = [1, 1, 1, 1, 1, 1]^T$ and $\operatorname{val}(\mathcal{K}_2) = [0, 0, 0, 0, 1, 1]^T$.
\end{exmp}

\begin{defn} The \textit{(formal) geodesic degree matrix}, denoted by \(\mathbb A_{\deg}\), is the \(\frac{n(n-1)}{2}\times\frac{n(n-1)}{2}\) matrix whose columns are the unique minimal multidegrees \(\operatorname{val}(\mathbb A_{p,q})\) of the formal geodesic functions. It has the block decomposition 
\[ \mathbb A_{\deg} = \begin{bmatrix} \mathbb A_1 & \mathbb A_2 & \cdots & \mathbb A_{\lfloor n/2\rfloor} \end{bmatrix}. \] 

For \(d\neq n/2\), the block \(\mathbb A_d\) is defined by \[ \mathbb A_d = \begin{bmatrix} \operatorname{val}(\mathbb A_{n-d,n}) & \operatorname{val}(\mathbb A_{1,n-d+1}) & \cdots & \operatorname{val}(\mathbb A_{d,n}) & \operatorname{val}(\mathbb A_{1,d+1}) & \cdots & \operatorname{val}(\mathbb A_{n-d-1,n-1}) \end{bmatrix}. \] 

When \(n=2m\), the innermost block is \[ \mathbb A_m = \begin{bmatrix} \operatorname{val}(\mathbb A_{m,2m}) & \operatorname{val}(\mathbb A_{1,m+1}) & \cdots & \operatorname{val}(\mathbb A_{m-1,2m-1}) \end{bmatrix}. \] 

We also write \( \mathbf u=[u_{p,q}]^T\in \mathbb Z_{\geq0}^{\,n(n-1)/2}, \) where the entries \(u_{p,q}\) are ordered according to the corresponding columns \(\operatorname{val}(\mathbb A_{p,q})\) of \(\mathbb A_{\deg}\). Similar to $\mathbb A_{\text{deg}}$, it has a block decomposition \( \mathbf u = \begin{bmatrix} \mathbf u^T_1 & \mathbf u^T_2 & \cdots & \mathbf u^T_{\lfloor n/2\rfloor} \end{bmatrix}^T \) where each $\mathbf u_i$ is an $N_i$-dimensional column vector. \end{defn} 

\begin{exmp} \textit{(Geodesic degree matrices for \(n=4,5\)).} For \(n=4\), we have \[ \mathbb A_{\deg} = \begin{bmatrix} \operatorname{val}(\mathbb A_{3,4}) & \operatorname{val}(\mathbb A_{1,4}) & \operatorname{val}(\mathbb A_{1,2}) & \operatorname{val}(\mathbb A_{2,3}) & \operatorname{val}(\mathbb A_{2,4}) & \operatorname{val}(\mathbb A_{1,3}) \end{bmatrix}, \] and hence \[ \mathbf u = [u_{3,4},u_{1,4},u_{1,2},u_{2,3},u_{2,4},u_{1,3}]^T = \begin{bmatrix} \mathbf u^T_1 & \mathbf u^T_{2} \end{bmatrix}^T. \] 

Explicitly, \[ \mathbb A_{\deg} = \begin{bmatrix} -\frac12 & 0 & 0 & -\frac12 & 0 & -\frac12 \\ -\frac12 & -\frac12 & 0 & 0 & -\frac12 & 0 \\ 0 & -\frac12 & -\frac12 & 0 & 0 & -\frac12 \\ 0 & 0 & -\frac12 & -\frac12 & -\frac12 & 0 \\ -\frac12 & 0 & -\frac12 & 0 & -\frac12 & -\frac12 \\ 0 & -\frac12 & 0 & -\frac12 & -\frac12 & -\frac12 \end{bmatrix}. \] 

For \(n=5\), we have \[ \mathbb A_{\deg} = \begin{bmatrix} \operatorname{val}(\mathbb A_{4,5}) & \operatorname{val}(\mathbb A_{1,5}) & \operatorname{val}(\mathbb A_{1,2}) & \operatorname{val}(\mathbb A_{2,3}) & \operatorname{val}(\mathbb A_{3,4}) & \operatorname{val}(\mathbb A_{3,5}) & \operatorname{val}(\mathbb A_{1,4}) & \operatorname{val}(\mathbb A_{2,5}) & \operatorname{val}(\mathbb A_{1,3}) & \operatorname{val}(\mathbb A_{2,4}) \end{bmatrix}, \] with the corresponding vector \[ \mathbf u = [u_{4,5},u_{1,5},u_{1,2},u_{2,3},u_{3,4}, u_{3,5},u_{1,4},u_{2,5},u_{1,3},u_{2,4}]^T = \begin{bmatrix} \mathbf u^T_1 & \mathbf u^T_{2} \end{bmatrix}^T. \] \end{exmp} 

For the desired decomposition, we seek integers \(v_i\in\mathbb Z\) and nonnegative integers \(u_{p,q}\in\mathbb Z_{\geq0}\) such that 
\begin{equation}\label{6.5} \sum_{i=1}^{\lfloor n/2\rfloor} v_i\,\operatorname{val}(\mathcal K_i) + \mathbb A_{\deg}\mathbf u = \mathbf a(l). \end{equation} 

To solve~\eqref{6.5}, we need the inverse matrix of \(\mathbb A_{\deg}\). We therefore introduce a matrix \(\mathbb B_{\deg}\) and subsequently prove that \( \mathbb B_{\deg}^T=\mathbb A_{\deg}^{-1}. \)
\begin{defn} \label{defn638}
    The matrix $\mathbb{B}_{\mathrm{deg}}$ is the block matrix
    \begin{equation*}
        \mathbb{B}_{\mathrm{deg}} = \begin{bmatrix}
        \mathbb{B}_1 & \mathbb{B}_2 & \cdots & \mathbb{B}_{\lfloor n/2 \rfloor}
        \end{bmatrix}.
    \end{equation*}
    
    Here, each submatrix $\mathbb{B}_k$ is of size $\frac{n(n-1)}{2} \times n$, except for the last block $\mathbb{B}_{\lfloor n/2 \rfloor}$ which is $\frac{n(n-1)}{2} \times \frac{n}{2}$ if $n$ is even. The $r$th column of each block, denoted $\mathbb{B}_{k,r}$, is defined as follows:
    \begin{itemize}
        \item For $k=1$:
        \[
            \mathbb{B}_{1,r} = \operatorname{val}\left(\frac{X_{2,r+1}}{X_{1,r}X_{1,r+1}}\right).
        \]
        \item For $2 \le k < \lfloor n/2 \rfloor$:
        \[
            \mathbb{B}_{k,r} = \operatorname{val}\left(\frac{X_{k+1,r+1}X_{k-1,r}}{X_{k,r}X_{k,r+1}}\right).
        \]
        \item For the last block $k = \lfloor n/2 \rfloor$:
        \[
            \mathbb{B}_{\lfloor n/2 \rfloor,r} = 
            \begin{cases} 
                 \operatorname{val}\left(\displaystyle \frac{X_{k,r+k+1}X_{k-1,r}}{X_{k, r}X_{k,r+1}}\right) & \text{if } n \text{ is odd}, \\[15pt]
                 \operatorname{val}\left(\displaystyle \frac{X_{k-1,r+k}X_{k-1,r}}{X_{k,r}X_{k,r+1}}\right) & \text{if } n \text{ is even}.
            \end{cases}
        \]
    \end{itemize}
\end{defn}

\begin{cor}\label{cor639}
Let $l \in \mathcal{A}(\mathbb{Z}^t)$ and let $\mathbf{a}(l)$ be its tropical coordinate vector. For $1 \le k < \lfloor n/2 \rfloor$, we have
$$W_k^t(l) = \min_{1 \le r \le N_k} \left\{ \mathbb{B}_{k,r}^T \mathbf{a}(l) \right\},$$
where $N_k$ is the length of the $k$th main cycle. If $n$ is even, this also holds for $k = \lfloor n/2 \rfloor$ (see Proposition~\ref{prop632}).\end{cor}

\begin{defn}
    \textit{(Geometric description of $\mathbb{B}_{\mathrm{deg}}$ via tiny polygons)}
    
    Each column $\mathbb{B}_{k,r}$ corresponds to a \textit{tiny polygon} $\Box_{k,r}$ on the $\mathcal{A}_n$-quiver. This polygon is defined by the vertex set $V_{k,r}$, which is partitioned into two disjoint subsets $V_{k,r} = V^+_{k,r} \sqcup V^-_{k,r}$. The subset $V^-_{k,r}$ is defined as
\[
    V^-_{k,r} = \{X_{k,r}, X_{k,r+1}\}.
\]
The complementary set $V^+_{k,r}$ is determined by the block index $k$ as follows:
    \begin{itemize}
        \item \textbf{Case $k=1$:} 
        \[ V^+_{1,r} = \{X_{2,r+1}\} \]
        
        \item \textbf{Case $1 < k < \lfloor n/2 \rfloor$:} 
        \[ V^+_{k,r} = \{X_{k-1,r}, X_{k+1,r+1}\} \]
    
        \item \textbf{Case $k = \lfloor n/2 \rfloor$:} 
        \[
            V^+_{k,r} = 
            \begin{cases} 
                \{X_{k-1,r}, X_{k,r+k+1}\} & \text{if } n \text{ is odd}, \\ 
                \{X_{k-1,r+k}, X_{k-1,r}\} & \text{if } n \text{ is even}.
            \end{cases}
        \]
    \end{itemize}
    
    Note that the variables in $V^+_{k,r}$ correspond to the entries with value $1$ in the column $\mathbb{B}_{k,r}$, whereas the variables in $V^-_{k,r}$ correspond to the entries with value $-1$.
\end{defn}

\begin{rem}\textit{(Cyclic symmetry of $\mathbb A_{\text{deg}}$ and $\mathbb B_{\text{deg}}$)} \label{rem6311}

Recall the shift operator $T_1$ from Proposition~\ref{prop321}. We observe that $T_1$ cyclically permutes the columns of $\mathbb{A}_k$. Specifically, for the case $k \neq n/2$, the action of $T_1$ yields the following cycle:$$\mathbb{A}_{n-k,n} \xrightarrow{T_1} \mathbb{A}_{1, n-k+1} \xrightarrow{T_1} \cdots \xrightarrow{T_1} \mathbb{A}_{k,n} \xrightarrow{T_1} \mathbb{A}_{1, 1+k} \xrightarrow{T_1} \cdots \xrightarrow{T_1} \mathbb{A}_{n-k-1, n-1} \xrightarrow{T_1} \mathbb{A}_{n-k,n}.$$

Analogously, the columns of $\mathbb{B}_k$ are cyclically permuted by $T_1$, satisfying the relation $T_1(V_{k,r}) = V_{k,r+1}$.\end{rem}
    
\begin{exmp}
\textit{(Example of tiny polygons for $n=6$)}

    Consider the following $\mathcal A_6$-quiver:
    $$\begin{tikzcd}[scale cd=0.6,column sep = tiny, row sep = tiny]
	&&&& {X_{1,3}} \\
	&&& {X_{2,4}} && \textcolor{rgb,255:red,0;green,255;blue,0}{{X_{1,4}}} \\
	&& \textcolor{rgb,255:red,0;green,255;blue,255}{{{{X_{3,2}}}}} && \textcolor{rgb,255:red,0;green,255;blue,0}{{X_{2,5}}} && \textcolor{rgb,255:red,0;green,255;blue,0}{{X_{1,5}}} \\
	& \textcolor{rgb,255:red,0;green,255;blue,255}{{{{X_{2,2}}}}} && \textcolor{rgb,255:red,255;green,0;blue,0}{{X_{3,3}}} && \textcolor{rgb,255:red,255;green,0;blue,0}{{X_{2,6}}} && {X_{1,6}} \\
	{X_{1,2}} && \textcolor{rgb,255:red,255;green,0;blue,0}{{X_{2,3}}} && \textcolor{rgb,255:red,255;green,0;blue,0}{{X_{3,1}}} && \textcolor{rgb,255:red,0;green,255;blue,255}{{{{X_{2,1}}}}} && \textcolor{rgb,255:red,0;green,255;blue,255}{{{{X_{1,1}}}}} \\
	& {X_{1,3}} && {X_{2,4}} && \textcolor{rgb,255:red,0;green,255;blue,255}{{{{X_{3,2}}}}} && \textcolor{rgb,255:red,0;green,255;blue,255}{{{{X_{2,2}}}}} && {{{X_{1,2}}}} \\
	&& \textcolor{rgb,255:red,0;green,255;blue,0}{{X_{1,4}}} && \textcolor{rgb,255:red,0;green,255;blue,0}{{X_{2,5}}} && \textcolor{rgb,255:red,255;green,0;blue,0}{{{{X_{3,3}}}}} && \textcolor{rgb,255:red,255;green,0;blue,0}{{{{X_{2,3}}}}} && {X_{1,3}} \\
	&&& \textcolor{rgb,255:red,0;green,255;blue,0}{{X_{1,5}}} && \textcolor{rgb,255:red,255;green,0;blue,0}{{{{X_{2,6}}}}} && \textcolor{rgb,255:red,255;green,0;blue,0}{{{{X_{3,1}}}}} && {X_{2,4}} \\
	&&&& {{{X_{1,6}}}} && \textcolor{rgb,255:red,0;green,255;blue,255}{{{X_{2,1}}}} && \textcolor{rgb,255:red,0;green,255;blue,255}{{{{X_{3,2}}}}} \\
	&&&&& \textcolor{rgb,255:red,0;green,255;blue,255}{{{{X_{1,1}}}}} && \textcolor{rgb,255:red,0;green,255;blue,255}{{{{X_{2,2}}}}} \\
	&&&&&& {X_{1,2}}
	\arrow[from=1-5, to=2-6]
	\arrow[dashed, from=2-4, to=1-5]
	\arrow[from=2-4, to=3-5]
	\arrow[from=2-6, to=2-4]
	\arrow[color={rgb,255:red,0;green,255;blue,0}, from=2-6, to=3-7]
	\arrow[dashed, from=3-3, to=2-4]
	\arrow[from=3-3, to=4-4]
	\arrow[color={rgb,255:red,0;green,255;blue,0}, from=3-5, to=2-6]
	\arrow[from=3-5, to=3-3]
	\arrow[from=3-5, to=4-6]
	\arrow[color={rgb,255:red,0;green,255;blue,0}, from=3-7, to=3-5]
	\arrow[from=3-7, to=4-8]
	\arrow[color={rgb,255:red,0;green,255;blue,255}, dashed, from=4-2, to=3-3]
	\arrow[from=4-2, to=5-3]
	\arrow[from=4-4, to=3-5]
	\arrow[from=4-4, to=4-2]
	\arrow[color={rgb,255:red,255;green,0;blue,0}, from=4-4, to=5-5]
	\arrow[from=4-6, to=3-7]
	\arrow[color={rgb,255:red,255;green,0;blue,0}, from=4-6, to=4-4]
	\arrow[from=4-6, to=5-7]
	\arrow[from=4-8, to=4-6]
	\arrow[from=4-8, to=5-9]
	\arrow[dashed, from=5-1, to=4-2]
	\arrow[from=5-1, to=6-2]
	\arrow[color={rgb,255:red,255;green,0;blue,0}, from=5-3, to=4-4]
	\arrow[from=5-3, to=5-1]
	\arrow[from=5-3, to=6-4]
	\arrow[color={rgb,255:red,255;green,0;blue,0}, from=5-5, to=4-6]
	\arrow[color={rgb,255:red,255;green,0;blue,0}, from=5-5, to=5-3]
	\arrow[from=5-5, to=6-6]
	\arrow[from=5-7, to=4-8]
	\arrow[from=5-7, to=5-5]
	\arrow[color={rgb,255:red,0;green,255;blue,255}, from=5-7, to=6-8]
	\arrow[color={rgb,255:red,0;green,255;blue,255}, from=5-9, to=5-7]
	\arrow[from=5-9, to=6-10]
	\arrow[from=6-2, to=5-3]
	\arrow[from=6-2, to=7-3]
	\arrow[from=6-4, to=5-5]
	\arrow[from=6-4, to=6-2]
	\arrow[from=6-4, to=7-5]
	\arrow[color={rgb,255:red,0;green,255;blue,255}, from=6-6, to=5-7]
	\arrow[from=6-6, to=6-4]
	\arrow[from=6-6, to=7-7]
	\arrow[color={rgb,255:red,0;green,255;blue,255}, from=6-8, to=5-9]
	\arrow[color={rgb,255:red,0;green,255;blue,255}, from=6-8, to=6-6]
	\arrow[from=6-8, to=7-9]
	\arrow[from=6-10, to=6-8]
	\arrow[from=6-10, to=7-11]
	\arrow[from=7-3, to=6-4]
	\arrow[color={rgb,255:red,0;green,255;blue,0}, from=7-3, to=8-4]
	\arrow[from=7-5, to=6-6]
	\arrow[color={rgb,255:red,0;green,255;blue,0}, from=7-5, to=7-3]
	\arrow[from=7-5, to=8-6]
	\arrow[from=7-7, to=6-8]
	\arrow[from=7-7, to=7-5]
	\arrow[color={rgb,255:red,255;green,0;blue,0}, from=7-7, to=8-8]
	\arrow[from=7-9, to=6-10]
	\arrow[color={rgb,255:red,255;green,0;blue,0}, from=7-9, to=7-7]
	\arrow[from=7-9, to=8-10]
	\arrow[from=7-11, to=7-9]
	\arrow[color={rgb,255:red,0;green,255;blue,0}, from=8-4, to=7-5]
	\arrow[from=8-4, to=9-5]
	\arrow[color={rgb,255:red,255;green,0;blue,0}, from=8-6, to=7-7]
	\arrow[from=8-6, to=8-4]
	\arrow[from=8-6, to=9-7]
	\arrow[color={rgb,255:red,255;green,0;blue,0}, from=8-8, to=7-9]
	\arrow[color={rgb,255:red,255;green,0;blue,0}, from=8-8, to=8-6]
	\arrow[from=8-8, to=9-9]
	\arrow[dashed, from=8-10, to=7-11]
	\arrow[from=8-10, to=8-8]
	\arrow[from=9-5, to=8-6]
	\arrow[from=9-5, to=10-6]
	\arrow[from=9-7, to=8-8]
	\arrow[from=9-7, to=9-5]
	\arrow[color={rgb,255:red,0;green,255;blue,255}, from=9-7, to=10-8]
	\arrow[dashed, from=9-9, to=8-10]
	\arrow[color={rgb,255:red,0;green,255;blue,255}, from=9-9, to=9-7]
	\arrow[color={rgb,255:red,0;green,255;blue,255}, from=10-6, to=9-7]
	\arrow[from=10-6, to=11-7]
	\arrow[color={rgb,255:red,0;green,255;blue,255}, dashed, from=10-8, to=9-9]
	\arrow[color={rgb,255:red,0;green,255;blue,255}, from=10-8, to=10-6]
	\arrow[dashed, from=11-7, to=10-8]
\end{tikzcd}$$
The green ($\Box_{1,4}$), mint ($\Box_{2,1}$), and red ($\Box_{3,3}$) tiny polygons correspond to the columns $\mathbb{B}_{1,4}$, $\mathbb{B}_{2,1}$, and $\mathbb{B}_{3,3}$, respectively. The figure above depicts the \textit{glued} $\mathcal A_n$-quiver, which is isomorphic to the standard $\mathcal A_n$-quiver up to arrow multiplicity. Since this quiver is obtained by gluing two $\mathcal A_n$-quivers, tiny polygons are represented as two separate pieces in the quiver.
\end{exmp}
\begin{prop} \label{prop6313}
   Let $\Diamond_{p,q}$ denote the plane figure generated by the cluster variables contained in the parallelogram $\blacklozenge_{p,q}$. Recall that $\blacklozenge_{p,q}$ is defined by the vertices $q$ and $p'$ in the dual quiver (see Figure \ref{Fig29}). Moreover, let $N^+$ (resp. $N^-$) be the number of variables shared by the plane figure $\Diamond_{p,q}$ and the set $V^+_{k,r}$ (resp. $V^-_{k,r}$). Then, the inner product $\mathbb{B}_{k,r}^T \cdot \operatorname{val}(\mathbb{A}_{p,q})$ is given by$$-\frac{1}{2} \left( N^+ - N^- \right).$$
\end{prop}

\begin{figure}[H]
    \centering
	\begin{tikzcd}[scale cd=0.6,column sep = tiny, row sep = tiny]
	&&&& {X_{1,3}} \\
	&&& {X_{2,4}} && \textcolor{rgb,255:red,0;green,255;blue,0}{{X_{1,4}}} \\
	&& \textcolor{rgb,255:red,0;green,255;blue,255}{{{{X_{3,2}}}}} && \textcolor{rgb,255:red,0;green,255;blue,0}{{X_{2,5}}} && {X_{1,5}} \\
	& \textcolor{rgb,255:red,0;green,255;blue,255}{{{{X_{2,2}}}}} && \textcolor{rgb,255:red,0;green,255;blue,0}{{X_{3,3}}} && {X_{2,6}} && {X_{1,6}} \\
	{X_{1,2}} && \textcolor{rgb,255:red,0;green,255;blue,0}{{X_{2,3}}} && {X_{3,1}} && \textcolor{rgb,255:red,0;green,255;blue,255}{{{{X_{2,1}}}}} && \textcolor{rgb,255:red,0;green,255;blue,255}{{{{X_{1,1}}}}} \\
	& \textcolor{rgb,255:red,0;green,255;blue,0}{{X_{1,3}}} && {X_{2,4}} && \textcolor{rgb,255:red,0;green,255;blue,255}{{{{X_{3,2}}}}} && \textcolor{rgb,255:red,128;green,0;blue,255}{{{{X_{2,2}}}}} && \textcolor{red}{{{{X_{1,2}}}}} \\
	&& {X_{1,4}} && {X_{2,5}} && \textcolor{red}{{{{X_{3,3}}}}} && \textcolor{red}{{{{X_{2,3}}}}} && {X_{1,3}} \\
	&&& {X_{1,5}} && \textcolor{red}{{{{X_{2,6}}}}} && \textcolor{red}{{{{X_{3,1}}}}} && {X_{2,4}} \\
	&&&& \textcolor{red}{{{{X_{1,6}}}}} && \textcolor{rgb,255:red,128;green,0;blue,255}{{{{X_{2,1}}}}} && \textcolor{rgb,255:red,0;green,255;blue,255}{{{{X_{3,2}}}}} \\
	&&&&& \textcolor{rgb,255:red,0;green,255;blue,255}{{{{X_{1,1}}}}} && \textcolor{rgb,255:red,0;green,255;blue,255}{{{{X_{2,2}}}}} \\
	&&&&&& {X_{1,2}}
	\arrow[from=1-5, to=2-6]
	\arrow[dashed, from=2-4, to=1-5]
	\arrow[from=2-4, to=3-5]
	\arrow[from=2-6, to=2-4]
	\arrow[from=2-6, to=3-7]
	\arrow[dashed, from=3-3, to=2-4]
	\arrow[from=3-3, to=4-4]
	\arrow[color={rgb,255:red,0;green,255;blue,0}, from=3-5, to=2-6]
	\arrow[from=3-5, to=3-3]
	\arrow[from=3-5, to=4-6]
	\arrow[from=3-7, to=3-5]
	\arrow[from=3-7, to=4-8]
	\arrow[color={rgb,255:red,0;green,255;blue,255}, dashed, from=4-2, to=3-3]
	\arrow[from=4-2, to=5-3]
	\arrow[color={rgb,255:red,0;green,255;blue,0}, from=4-4, to=3-5]
	\arrow[from=4-4, to=4-2]
	\arrow[from=4-4, to=5-5]
	\arrow[from=4-6, to=3-7]
	\arrow[from=4-6, to=4-4]
	\arrow[from=4-6, to=5-7]
	\arrow[from=4-8, to=4-6]
	\arrow[from=4-8, to=5-9]
	\arrow[dashed, from=5-1, to=4-2]
	\arrow[from=5-1, to=6-2]
	\arrow[color={rgb,255:red,0;green,255;blue,0}, from=5-3, to=4-4]
	\arrow[from=5-3, to=5-1]
	\arrow[from=5-3, to=6-4]
	\arrow[from=5-5, to=4-6]
	\arrow[from=5-5, to=5-3]
	\arrow[from=5-5, to=6-6]
	\arrow[from=5-7, to=4-8]
	\arrow[from=5-7, to=5-5]
	\arrow[color={rgb,255:red,0;green,255;blue,255}, from=5-7, to=6-8]
	\arrow[color={rgb,255:red,0;green,255;blue,255}, from=5-9, to=5-7]
	\arrow[from=5-9, to=6-10]
	\arrow[color={rgb,255:red,0;green,255;blue,0}, from=6-2, to=5-3]
	\arrow[from=6-2, to=7-3]
	\arrow[from=6-4, to=5-5]
	\arrow[from=6-4, to=6-2]
	\arrow[from=6-4, to=7-5]
	\arrow[color={rgb,255:red,0;green,255;blue,255}, from=6-6, to=5-7]
	\arrow[from=6-6, to=6-4]
	\arrow[from=6-6, to=7-7]
	\arrow[color={rgb,255:red,0;green,255;blue,255}, from=6-8, to=5-9]
	\arrow[color={rgb,255:red,0;green,255;blue,255}, from=6-8, to=6-6]
	\arrow[color=red, from=6-8, to=7-9]
	\arrow[color=red, from=6-10, to=6-8]
	\arrow[from=6-10, to=7-11]
	\arrow[from=7-3, to=6-4]
	\arrow[from=7-3, to=8-4]
	\arrow[from=7-5, to=6-6]
	\arrow[from=7-5, to=7-3]
	\arrow[from=7-5, to=8-6]
	\arrow[color=red, from=7-7, to=6-8]
	\arrow[from=7-7, to=7-5]
	\arrow[color=red, from=7-7, to=8-8]
	\arrow[color=red, from=7-9, to=6-10]
	\arrow[color=red, from=7-9, to=7-7]
	\arrow[from=7-9, to=8-10]
	\arrow[from=7-11, to=7-9]
	\arrow[from=8-4, to=7-5]
	\arrow[from=8-4, to=9-5]
	\arrow[color=red, from=8-6, to=7-7]
	\arrow[from=8-6, to=8-4]
	\arrow[color=red, from=8-6, to=9-7]
	\arrow[color=red, from=8-8, to=7-9]
	\arrow[color=red, from=8-8, to=8-6]
	\arrow[from=8-8, to=9-9]
	\arrow[dashed, from=8-10, to=7-11]
	\arrow[from=8-10, to=8-8]
	\arrow[color=red, from=9-5, to=8-6]
	\arrow[from=9-5, to=10-6]
	\arrow[color=red, from=9-7, to=8-8]
	\arrow[color=red, from=9-7, to=9-5]
	\arrow[color={rgb,255:red,0;green,255;blue,255}, from=9-7, to=10-8]
	\arrow[dashed, from=9-9, to=8-10]
	\arrow[color={rgb,255:red,0;green,255;blue,255}, from=9-9, to=9-7]
	\arrow[color={rgb,255:red,0;green,255;blue,255}, from=10-6, to=9-7]
	\arrow[from=10-6, to=11-7]
	\arrow[color={rgb,255:red,0;green,255;blue,255}, dashed, from=10-8, to=9-9]
	\arrow[color={rgb,255:red,0;green,255;blue,255}, from=10-8, to=10-6]
	\arrow[dashed, from=11-7, to=10-8]
\end{tikzcd}
    \caption{The red parallelogram ($\Diamond_{4,6}$) corresponds to $\mathbb{A}_{4,6}$, while the mint tiny polygon ($\Box_{2,1}$) denotes $\mathbb{B}_{2,1}$. The purple variables are shared variables of these shapes. Observe that the inner product of $\Diamond_{4,6}$ with $\Box_{2,1}$ is $-\frac{1}{2}(0-2) = 1$, whereas its inner product with every other tiny polygon is $0$. For instance, the inner product with $\Box_{2,5}$ (generated by $X_{2,5}, X_{2,6}, X_{3,3}, X_{1,5}$) is zero; specifically, in the column $\mathbb{B}_{2,5}$, the entry corresponding to $X_{3,3}$ is $1$ and the entry for $X_{2,6}$ is $-1$, resulting in $-\frac{1}{2}(1-1) = 0$. Likewise, the inner product with $\Box_{1,1}$ (generated by $X_{1,1}, X_{1,2}, X_{2,2}$) is zero. Analogously, consider the green line ($\Diamond_{1,2}$) representing $\mathbb{A}_{1,2}$. The unique tiny polygon having a nonzero inner product with $\Diamond_{1,2}$ is $\Box_{1,3}$ (generated by $X_{1,3}, X_{1,4}, X_{2,4}$).}
    \label{Fig29}
\end{figure}

\begin{prop}
    $\mathbb{B}_{\text{deg}}^T$ is the inverse matrix of $\mathbb{A}_{\text{deg}}$.
\end{prop}

\begin{proof}
By the cyclic symmetry (Remark~\ref{rem6311}), it suffices to show that the inner product $\mathbb{B}_{k,r}^T  \operatorname{val}(\mathbb{A}_{n-d,n})$ is $1$ when $(k,r) = (d,1)$ and $0$ otherwise.

Let $(k,r) = (d,1)$. The geometric configuration of $\Box_{d,1}$ relative to $\Diamond_{n-d,n}$ is as follows ($m = \lfloor n/2 \rfloor)$:

$$\begin{tikzcd}[scale cd=0.8,column sep = tiny, row sep = tiny]
	&&& \textcolor{rgb,255:red,0;green,255;blue,255}{{X_{d,1}}} && \textcolor{rgb,255:red,0;green,255;blue,255}{\bullet} &&&&&&&&& \textcolor{rgb,255:red,0;green,255;blue,255}{{X_{1,2}}} \\
	&& \textcolor{rgb,255:red,0;green,255;blue,255}{\bullet} && \textcolor{rgb,255:red,128;green,0;blue,255}{{X_{d,2}}} && \textcolor{rgb,255:red,255;green,0;blue,0}{\cdots} && \textcolor{rgb,255:red,255;green,0;blue,0}{{X_{1,2}}} &&&&& \textcolor{rgb,255:red,0;green,255;blue,255}{\bullet} && \textcolor{rgb,255:red,128;green,0;blue,255}{{X_{1,2}}} \\
	&&& \textcolor{rgb,255:red,255;green,0;blue,0}{\cdots} &&&& \textcolor{rgb,255:red,255;green,0;blue,0}{\cdots} &&&&&&& \textcolor{rgb,255:red,255;green,0;blue,0}{\cdots} \\
	&& \textcolor{rgb,255:red,255;green,0;blue,0}{{X_{m,2-d}}} &&&& \textcolor{rgb,255:red,255;green,0;blue,0}{{X_{m,1}}} &&&&&&& \textcolor{rgb,255:red,255;green,0;blue,0}{{X_{m,1}}} \\
	& \textcolor{rgb,255:red,255;green,0;blue,0}{\cdots} &&&& \textcolor{rgb,255:red,255;green,0;blue,0}{\cdots} &&&&&&& \textcolor{rgb,255:red,255;green,0;blue,0}{\cdots} \\
	\textcolor{rgb,255:red,255;green,0;blue,0}{{X_{1,2-d}}} && \textcolor{rgb,255:red,255;green,0;blue,0}{\cdots} && \textcolor{rgb,255:red,128;green,0;blue,255}{{X_{d,1}}} && \textcolor{rgb,255:red,0;green,255;blue,255}{\bullet} &&&&& \textcolor{rgb,255:red,128;green,0;blue,255}{{X_{1,1}}} && \textcolor{rgb,255:red,0;green,255;blue,255}{\bullet} \\
	&&& \textcolor{rgb,255:red,0;green,255;blue,255}{\bullet} && \textcolor{rgb,255:red,0;green,255;blue,255}{{X_{d,2}}} &&&&&&& \textcolor{rgb,255:red,0;green,255;blue,255}{{X_{1,2}}}
	\arrow[color={rgb,255:red,0;green,255;blue,255}, from=1-4, to=2-5]
	\arrow[color={rgb,255:red,0;green,255;blue,255}, from=1-6, to=1-4]
	\arrow[color={rgb,255:red,0;green,255;blue,255}, from=1-15, to=2-16]
	\arrow[color={rgb,255:red,0;green,255;blue,255}, from=2-3, to=1-4]
	\arrow[color={rgb,255:red,0;green,255;blue,255}, from=2-5, to=1-6]
	\arrow[color={rgb,255:red,0;green,255;blue,255}, from=2-5, to=2-3]
	\arrow[color={rgb,255:red,255;green,0;blue,0}, from=2-7, to=2-5]
	\arrow[color={rgb,255:red,255;green,0;blue,0}, from=2-9, to=2-7]
	\arrow[color={rgb,255:red,0;green,255;blue,255}, from=2-14, to=1-15]
	\arrow[color={rgb,255:red,0;green,255;blue,255}, from=2-16, to=2-14]
	\arrow[color={rgb,255:red,255;green,0;blue,0}, from=3-4, to=2-5]
	\arrow[color={rgb,255:red,255;green,0;blue,0}, from=3-8, to=2-9]
	\arrow[color={rgb,255:red,255;green,0;blue,0}, from=3-15, to=2-16]
	\arrow[color={rgb,255:red,255;green,0;blue,0}, from=4-3, to=3-4]
	\arrow[color={rgb,255:red,255;green,0;blue,0}, from=4-7, to=3-8]
	\arrow[color={rgb,255:red,255;green,0;blue,0}, from=4-14, to=3-15]
	\arrow[color={rgb,255:red,255;green,0;blue,0}, from=5-2, to=4-3]
	\arrow[color={rgb,255:red,255;green,0;blue,0}, from=5-6, to=4-7]
	\arrow[color={rgb,255:red,255;green,0;blue,0}, from=5-13, to=4-14]
	\arrow[color={rgb,255:red,255;green,0;blue,0}, from=6-1, to=5-2]
	\arrow[color={rgb,255:red,255;green,0;blue,0}, from=6-3, to=6-1]
	\arrow[color={rgb,255:red,255;green,0;blue,0}, from=6-5, to=5-6]
	\arrow[color={rgb,255:red,255;green,0;blue,0}, from=6-5, to=6-3]
	\arrow[color={rgb,255:red,0;green,255;blue,255}, from=6-5, to=7-6]
	\arrow[color={rgb,255:red,0;green,255;blue,255}, from=6-7, to=6-5]
	\arrow[color={rgb,255:red,255;green,0;blue,0}, from=6-12, to=5-13]
	\arrow[color={rgb,255:red,0;green,255;blue,255}, from=6-12, to=7-13]
	\arrow[color={rgb,255:red,0;green,255;blue,255}, from=6-14, to=6-12]
	\arrow[color={rgb,255:red,0;green,255;blue,255}, from=7-4, to=6-5]
	\arrow[color={rgb,255:red,0;green,255;blue,255}, from=7-6, to=6-7]
	\arrow[color={rgb,255:red,0;green,255;blue,255}, from=7-6, to=7-4]
	\arrow[color={rgb,255:red,0;green,255;blue,255}, from=7-13, to=6-14]
\end{tikzcd}$$

The left quiver depicts the case where $d \ne 1$, while the right quiver shows the case $d=1$. In both quivers, the red plane figure represents $\Diamond_{n-d,n}$ and the mint polygon represents $\Box_{d,1}$. The variables highlighted in purple are shared by these two figures. In this case, we have $N^+ = 0$ and $N^- = 2$ (corresponding to the variables $X_{d,1}$ and $X_{d,2}$). Hence, by Proposition~\ref{prop6313}, the inner product is given by
$$
-\frac{1}{2}(0 - 2) = 1.
$$

Next, let $(k,r) \ne (d,1)$. In the glued $\mathcal{A}_n$-quiver, the polygon $\Box_{k,r}$ is represented as two disjoint pieces. For each piece, the number of shared sides with $\Diamond_{n-d,n}$ is either $0$, $1$, or $4$. In each case, the contributions from the common variables are zero. Note that each side of $\Box_{k,r}$ contains exactly one variable in $V^+_{k,r}$ and one variable in $V^-_{k,r}$.

\begin{itemize}
    \item If there are no shared sides, the piece and $\Diamond_{n-d,n}$ do not share any variables. As illustrated in the figure above, variable sharing without shared sides occurs only when $(k,r)=(d,1)$. In all other cases, sharing a variable requires sharing a side (See Figure~\ref{Fig29}). Consequently, since the variable sets are disjoint, the contribution is $0$.
    
    \item If there is \textbf{one shared side}, a side contains exactly one variable in $V^+_{k,r}$ and one variable in $V^-_{k,r}$. By Proposition~\ref{prop6313}, the contribution is 
    $$-\frac{1}{2}(1 - 1) = 0.$$
    
    \item If there are \textbf{four shared sides}, these sides contain two variables in $V^+_{k,r}$ and two variables in $V^-_{k,r}$. The contribution is 
    $$-\frac{1}{2}(2 - 2) = 0.$$
\end{itemize}

Since the total inner product is the sum of the contributions from these disjoint pieces (yielding $0 + 0 = 0$), we conclude that $\mathbb{B}_{k,r}^T \operatorname{val}(\mathbb{A}_{n-d,n}) = 0$ for all $(k,r) \ne (d,1)$. Consequently, $\mathbb{B}_{\text{deg}}^T$ is the inverse of $\mathbb{A}_{\text{deg}}$ (refer to Figure~\ref{Fig29} for a concrete illustration).
\end{proof}

\begin{rem}\label{rem6315} Consider \(l\in\mathcal A(\mathbb Z^t)\) and its tropical coordinate vector \(\mathbf a(l)\in\mathbb Z^{n(n-1)/2}\). We can write \(\mathbb B_{\deg}^{T}\mathbf a(l)\) in block form as 
\[ \mathbb B_{\deg}^{T}\mathbf a(l) = \begin{bmatrix} \mathbb B_1^T(\mathbf a(l))\\ \mathbb B_2^T(\mathbf a(l))\\ \vdots\\ \mathbb B^T_{\lfloor n/2\rfloor}(\mathbf a(l)) \end{bmatrix}, \] where \(\mathbb B_k^T(\mathbf a(l))\) is an \(N_k\)-dimensional column vector such that \(N_k\) is the length of the \(k\)th main cycle of the \(\mathcal A_n\)-quiver. By definition, its $r$th component is $\mathbb B_{k,r}^{T}\mathbf a(l)$. Consequently, Corollary~\ref{cor639} gives \( W_k^t(l) = \min_{1\leq r\leq N_k} \bigl\{\mathbb B_{k,r}^{T}\mathbf a(l)\bigr\}. \) \end{rem} 

Since \(\mathbb B_{\deg}^{T}=\mathbb A_{\deg}^{-1}\), Equation~\eqref{6.5} can be rewritten as 
\begin{equation}\label{6.6} \mathbf u = \mathbb B_{\deg}^{T} \left( \mathbf a(l) - \sum_{i=1}^{\lfloor n/2\rfloor} v_i\,\operatorname{val}(\mathcal K_i) \right). \end{equation} 

Let \(m=\lfloor n/2\rfloor\). Expanding~\eqref{6.6} according to the block decompositions of \(\mathbf u\), \(\mathbf a(l)\), and \(\mathbb B_{\deg}^{T}\), we obtain 
\begin{equation}\label{6.7} \begin{bmatrix} \mathbf u_1\\ \vdots\\ \mathbf u_i\\ \vdots\\ \mathbf u_{m-1}\\ \mathbf u_m \end{bmatrix} = \mathbb B_{\deg}^{T} \begin{bmatrix} \mathbf a(l)_1-v_1\mathbf 1_1\\ \vdots\\ \mathbf a(l)_i-\displaystyle\sum_{j=1}^{i}v_j\mathbf 1_i\\ \vdots\\ \mathbf a(l)_{m-1} -\displaystyle\sum_{j=1}^{m-1}v_j\mathbf 1_{m-1}\\ \mathbf a(l)_m-\displaystyle\sum_{j=1}^{m}v_j\mathbf 1_m \end{bmatrix} = \begin{bmatrix} \mathbb B^T_1(\mathbf a(l))+(v_1-v_2)\mathbf 1_1\\ \vdots\\ \mathbb B^T_i(\mathbf a(l))+(v_i-v_{i+1})\mathbf 1_i\\ \vdots\\ \mathbb B^T_{m-1}(\mathbf a(l))+(v_{m-1}-v_m)\mathbf 1_{m-1}\\ \mathbb B^T_m(\mathbf a(l)) +\operatorname{par}(n)v_m\mathbf 1_m \end{bmatrix}. \end{equation} 

Here, \(\mathbf u_i\), \(\mathbf a(l)_i\), and \(\mathbf 1_i\) are \(N_i\)-dimensional column vectors such that \[ \mathbf u = \begin{bmatrix} \mathbf u_1^T& \cdots& \mathbf u_m^T \end{bmatrix}^{T}, \qquad \mathbf a(l) = \begin{bmatrix} \mathbf a(l)_1^T& \cdots& \mathbf a(l)_m^T \end{bmatrix}^{T}, \qquad \mathbf 1_i = \begin{bmatrix} 1&\cdots&1 \end{bmatrix}^{T}. \] 

See Example~\ref{exmp6316} for an explicit description. The final equality in~\eqref{6.7} follows from the block structure of \(\mathbb B_{\deg}\) described in Definition~\ref{defn638}. 

More precisely, consider each row block $\mathbb B_i^T$ of \(\mathbb B_{\deg}^{T}\). Note that if \(n\) is even, the row-block decomposition consists of \(m-1\) blocks of size \(n\), followed by a final block of size \(n/2\). If \(n\) is odd, it consists of \(m\) blocks of size \(n\). Then, the contributions are computed as follows. 
\begin{itemize} 
\item For \(\mathbb B_1^T\), each row contains exactly two entries equal to \(-1\) in the first block and one entry equal to \(1\) in the second block. Its contribution is therefore \((v_1-v_2)\mathbf 1_1\), since \[ -(-v_1)-(-v_1)+(-v_1-v_2) = v_1-v_2. \] 

\item For \(\mathbb B_i^T\), where \(1<i<m\), each row contains two entries equal to \(-1\) in the \(i\)-th block and one entry equal to \(1\) in each of the \((i-1)\)st and \((i+1)\)st blocks. Its contribution is \((v_i-v_{i+1})\mathbf 1_i\), because 
\begin{align*} &-\left(-\sum_{j=1}^{i}v_j\right) -\left(-\sum_{j=1}^{i}v_j\right) +\left(-\sum_{j=1}^{i-1}v_j\right) +\left(-\sum_{j=1}^{i+1}v_j\right)\\ &\hspace{2cm} = 2\sum_{j=1}^{i}v_j - \sum_{j=1}^{i-1}v_j - \sum_{j=1}^{i+1}v_j = v_i-v_{i+1}. \end{align*}

\item For \(\mathbb B_m^T\), we distinguish the two parities of \(n\). 
\begin{itemize} \item If \(n\) is even, each row contains two entries equal to \(-1\) in the \(m\)-th block and two entries equal to \(1\) in the \((m-1)\)-st block. Hence its contribution is \(2v_m\mathbf 1_m\), since \[ 2\left(\sum_{j=1}^{m}v_j\right) - 2\left(\sum_{j=1}^{m-1}v_j\right) = 2v_m. \] 

\item If \(n\) is odd, each row contains two entries equal to \(-1\) and one entry equal to \(1\) in the \(m\)-th block, together with one entry equal to \(1\) in the \((m-1)\)-st block. Hence its contribution is \(v_m\mathbf 1_m\), since \[ 2\left(\sum_{j=1}^{m}v_j\right) - \left(\sum_{j=1}^{m}v_j\right) - \left(\sum_{j=1}^{m-1}v_j\right) = v_m. \] \end{itemize} \end{itemize} 

It follows from~\eqref{6.7} that every entry of \(\mathbf u\) is nonnegative if and only if 
\begin{equation}\label{6.8} \begin{cases} v_1-v_2 \geq -\displaystyle\min_{1\leq r\leq N_1} \{ \mathbb B_{1,r}^T \mathbf a(l) \},\\[2mm] v_2-v_3 \geq -\displaystyle\min_{1\leq r\leq N_2} \{ \mathbb B_{2,r}^T \mathbf a(l) \},\\ \hspace{1.7cm}\vdots\\ v_{m-1}-v_m \geq -\displaystyle\min_{1\leq r\leq N_{m-1}} \{ \mathbb B_{m-1,r}^T \mathbf a(l) \},\\[2mm] \operatorname{par}(n)v_m \geq -\displaystyle\min_{1\leq r\leq N_m} \{ \mathbb B_{m,r}^T \mathbf a(l) \}. \end{cases} \end{equation} 

We now choose the integers \(v_i\) recursively by taking the smallest integers satisfying~\eqref{6.8}. More precisely, we set \[ v_i-v_{i+1} = -\min_{1\leq r\leq N_i} \{ \mathbb B_{i,r}^T \mathbf a(l) \}, \qquad 1\leq i<m, \] and \[ v_m = \left\lceil \frac{ -\displaystyle\min_{1\leq r\leq N_m} \{ \mathbb B_{m,r}^T \mathbf a(l) \} }{ \operatorname{par}(n) } \right\rceil. \] 

With this choice, equation~\eqref{6.5} yields the decomposition \( \mathbf a(l) = \sum_{i=1}^{m} v_i\,\operatorname{val}(\mathcal K_i) + \mathbb A_{\deg}\mathbf u, \) where \( \mathbf u\in \mathbb Z_{\geq0}^{\,n(n-1)/2}. \) 

For each \(1\leq i<m\), the vector \(\mathbf u_i\) contains at least one zero entry, since \( v_i-v_{i+1} = -\min_{1\leq r\leq N_i} \{ \mathbb B_{i,r}^T \mathbf a(l) \}. \) If \(n\) is odd, we have \(\operatorname{par}(n)=1\). This implies \( v_m = -\min_{1\leq r\leq N_m} \{ \mathbb B_{m,r}^T \mathbf a(l) \}. \) Therefore, \(\mathbf u_m\) also contains at least one zero entry. If \(n\) is even, then \(\operatorname{par}(n)=2\). In this case, \( \min_{1\leq r\leq N_m}[\mathbf u_m]_r = \min_{1\leq r\leq N_m}\{ \mathbb B_{m,r}^T \mathbf a(l) \} + 2v_m \in\{0,1\}. \) Thus, \(\mathbf u_m\) contains at least one entry equal to \(0\) or \(1\).

\begin{exmp}\label{exmp6316} \textit{(\(n=4\) example).} 

Consider the tropical point with \( \mathbf a(l) = \begin{bmatrix} a_{1,1} & a_{1,2} & a_{1,3} & a_{1,4} & a_{2,1} & a_{2,2} \end{bmatrix}^{T}. \) We seek integers \(v_1,v_2\) and a nonnegative integral vector \( \mathbf u = \begin{bmatrix} u_{3,4} & u_{1,4} & u_{1,2} & u_{2,3} & u_{2,4} & u_{1,3} \end{bmatrix}^{T} \) such that \[ v_1\operatorname{val}(\mathcal K_1) + v_2\operatorname{val}(\mathcal K_2) + \mathbb A_{\deg}\mathbf u = \mathbf a(l). \] 

Explicitly, this equation is 
\[ \begin{bmatrix} v_1\\ v_1\\ v_1\\ v_1\\ v_1\\ v_1 \end{bmatrix} + \begin{bmatrix} 0\\ 0\\ 0\\ 0\\ v_2\\ v_2 \end{bmatrix} + \begin{bmatrix} -\frac12 & 0 & 0 & -\frac12 & 0 & -\frac12\\ -\frac12 & -\frac12 & 0 & 0 & -\frac12 & 0\\ 0 & -\frac12 & -\frac12 & 0 & 0 & -\frac12\\ 0 & 0 & -\frac12 & -\frac12 & -\frac12 & 0\\ -\frac12 & 0 & -\frac12 & 0 & -\frac12 & -\frac12\\ 0 & -\frac12 & 0 & -\frac12 & -\frac12 & -\frac12 \end{bmatrix} \begin{bmatrix} u_{3,4}\\ u_{1,4}\\ u_{1,2}\\ u_{2,3}\\ u_{2,4}\\ u_{1,3} \end{bmatrix} = \begin{bmatrix} a_{1,1}\\ a_{1,2}\\ a_{1,3}\\ a_{1,4}\\ a_{2,1}\\ a_{2,2} \end{bmatrix}.\] 

The displayed \(6\times6\) matrix is \(\mathbb A_{\deg}\). Moving the two Casimir contributions to the right-hand side, we obtain \[ \mathbb A_{\deg} \begin{bmatrix} u_{3,4}\\ u_{1,4}\\ u_{1,2}\\ u_{2,3}\\ u_{2,4}\\ u_{1,3} \end{bmatrix} = \begin{bmatrix} a_{1,1}-v_1\\ a_{1,2}-v_1\\ a_{1,3}-v_1\\ a_{1,4}-v_1\\ a_{2,1}-v_1-v_2\\ a_{2,2}-v_1-v_2 \end{bmatrix}. \] 

Multiplying both sides by \(\mathbb B_{\deg}^{T}=\mathbb A_{\deg}^{-1}\), we obtain \[ \begin{aligned} \begin{bmatrix} u_{3,4}\\ u_{1,4}\\ u_{1,2}\\ u_{2,3}\\ u_{2,4}\\ u_{1,3} \end{bmatrix} &= \begin{bmatrix} -1 & -1 & 0 & 0 & 0 & 1\\ 0 & -1 & -1 & 0 & 1 & 0\\ 0 & 0 & -1 & -1 & 0 & 1\\ -1 & 0 & 0 & -1 & 1 & 0\\ 1 & 0 & 1 & 0 & -1 & -1\\ 0 & 1 & 0 & 1 & -1 & -1 \end{bmatrix} \begin{bmatrix} a_{1,1}-v_1\\ a_{1,2}-v_1\\ a_{1,3}-v_1\\ a_{1,4}-v_1\\ a_{2,1}-v_1-v_2\\ a_{2,2}-v_1-v_2 \end{bmatrix} = \begin{bmatrix} -a_{1,1}-a_{1,2}+a_{2,2}+v_1-v_2\\ -a_{1,2}-a_{1,3}+a_{2,1}+v_1-v_2\\ -a_{1,3}-a_{1,4}+a_{2,2}+v_1-v_2\\ -a_{1,4}-a_{1,1}+a_{2,1}+v_1-v_2\\ -a_{2,1}-a_{2,2}+a_{1,1}+a_{1,3}+2v_2\\ -a_{2,1}-a_{2,2}+a_{1,2}+a_{1,4}+2v_2 \end{bmatrix}. \end{aligned} \] 

We decompose the exponent and tropical coordinate vectors into blocks: \[ \begin{aligned} \mathbf u_1 &= \begin{bmatrix} u_{3,4}&u_{1,4}&u_{1,2}&u_{2,3} \end{bmatrix}^{T}, & \mathbf u_2 &= \begin{bmatrix} u_{2,4}&u_{1,3} \end{bmatrix}^{T},\\ \mathbf a(l)_1 &= \begin{bmatrix} a_{1,1}&a_{1,2}&a_{1,3}&a_{1,4} \end{bmatrix}^{T}, & \mathbf a(l)_2 &= \begin{bmatrix} a_{2,1}&a_{2,2} \end{bmatrix}^{T}. \end{aligned} \] 

We first choose \(v_2\) to be the smallest integer satisfying \(2v_2\geq-\min_{1\leq r\leq2} \{\mathbb B_{2,r}^{T}\mathbf a(l)\}\). Since \[ \mathbb B_2^{T}\mathbf a(l) = \begin{bmatrix} -a_{2,1}-a_{2,2}+a_{1,1}+a_{1,3}\\ -a_{2,1}-a_{2,2}+a_{1,2}+a_{1,4} \end{bmatrix}, \] we obtain \[ v_2 = \left\lceil -\frac{1}{2} \min \left\{ -a_{2,1}-a_{2,2}+a_{1,1}+a_{1,3},\, -a_{2,1}-a_{2,2}+a_{1,2}+a_{1,4} \right\} \right\rceil. \] 

We then determine \(v_1\) from \(v_1-v_2=-\min_{1\leq r\leq4} \{\mathbb B_{1,r}^{T}\mathbf a(l)\}\), which gives \[ \begin{aligned} v_1 = v_2-\min\bigl\{& -a_{1,1}-a_{1,2}+a_{2,2},\, -a_{1,2}-a_{1,3}+a_{2,1},\\ & -a_{1,3}-a_{1,4}+a_{2,2},\, -a_{1,4}-a_{1,1}+a_{2,1} \bigr\}. \end{aligned} \] 

With these choices, every entry of \(\mathbf u_1\) and \(\mathbf u_2\) is a nonnegative integer. 

For example, take \(\mathbf a(l)=\begin{bmatrix}1&2&1&4&2&5\end{bmatrix}^{T}\). Then \[ \mathbb B_2^{T}\mathbf a(l) = \begin{bmatrix}-5&-1\end{bmatrix}^{T}, \qquad \mathbb B_1^{T}\mathbf a(l) = \begin{bmatrix}2&-1&0&-3\end{bmatrix}^{T}. \] 

Thus, the smallest integer satisfying \(2v_2\geq5\) is \(v_2=3\), and \(v_1-v_2=3\), so \(v_1=6\). Consequently, \[ \mathbf u_1 = \begin{bmatrix}5&2&3&0\end{bmatrix}^{T}, \qquad \mathbf u_2 = \begin{bmatrix}1&5\end{bmatrix}^{T}. \] 

Therefore, \[ \begin{aligned} \mathbf a(l) ={}& 6\,\operatorname{val}(\mathcal K_1) +3\,\operatorname{val}(\mathcal K_2)\\ &+5\,\operatorname{val}(\mathbb A_{3,4}) +2\,\operatorname{val}(\mathbb A_{1,4}) +3\,\operatorname{val}(\mathbb A_{1,2})\\ &+\operatorname{val}(\mathbb A_{2,4}) +5\,\operatorname{val}(\mathbb A_{1,3}). \end{aligned} \]\end{exmp}

We summarize the preceding construction in the following proposition. \begin{prop}\label{prop6317} Let \(m=\lfloor n/2\rfloor\), and consider the decomposition \[ \mathbf a(l) = \sum_{i=1}^{m}v_i\,\operatorname{val}(\mathcal K_i) + \mathbb A_{\deg}\mathbf u. \] 

Define \(l_{\mathrm{cas}},l_{\mathrm{geo}}\in\mathcal A(\mathbb Z^t)\) by \(\mathbf a(l_{\mathrm{cas}}):=\sum_{i=1}^{m}v_i\,\operatorname{val}( \mathcal K_i)\) and \(\mathbf a(l_{\mathrm{geo}}):=\mathbb A_{\deg}\mathbf u\). Then \(\mathbf a(l)=\mathbf a(l_{\mathrm{cas}})+\mathbf a(l_{\mathrm{geo}})\), and \(\tau_i^t(l_{\mathrm{geo}})=l_{\mathrm{geo}}\) for \(1\leq i<m\). If \(n\) is even, then \[ W_m^t(l_{\mathrm{geo}}) = \min_{1\leq r\leq N_m}[\mathbf u_m]_r \in\{0,1\}. \] 

In particular, \(\tau_m^t(l_{\mathrm{geo}})=l_{\mathrm{geo}}\) whenever \(\min_r[\mathbf u_m]_r=0\). \end{prop} 

\begin{proof} Since \(\mathbf a(l_{\mathrm{geo}})=\mathbb A_{\deg}\mathbf u\) and \(\mathbb B_{\deg}^{T}=\mathbb A_{\deg}^{-1}\), we have \(\mathbb B_{\deg}^{T}\mathbf a(l_{\mathrm{geo}})=\mathbf u\). Comparing the corresponding block decompositions gives \[ \mathbb B_i^T\bigl(\mathbf a(l_{\mathrm{geo}})\bigr) = \mathbf u_i, \qquad 1\leq i\leq m. \] 

Hence Corollary~\ref{cor639} yields \[ W_i^t(l_{\mathrm{geo}}) = \min_{1\leq r\leq N_i} \left\{ \mathbb B_{i,r}^T\mathbf a(l_{\mathrm{geo}}) \right\} = \min_{1\leq r\leq N_i}[\mathbf u_i]_r. \] 

By the construction, each block \(\mathbf u_i\) with \(1\leq i<m\) contains a zero entry. Thus \(W_i^t(l_{\mathrm{geo}})=0\) for \(1\leq i<m\), and Proposition~\ref{prop632} implies \(\tau_i^t(l_{\mathrm{geo}})=l_{\mathrm{geo}}\) for \(1\leq i<m\). If \(n\) is even, the construction gives \(\min_{1\leq r\leq N_m}[\mathbf u_m]_r\in\{0,1\}\). This implies \[ W_m^t(l_{\mathrm{geo}}) = \min_{1\leq r\leq N_m}[\mathbf u_m]_r \in\{0,1\}. \] 

If this minimum is zero, Proposition~\ref{prop632} also gives \(\tau_m^t(l_{\mathrm{geo}})=l_{\mathrm{geo}}\). \end{proof}

\subsection{Weyl Group Invariants for Even \(n\)} 
In this section, we always assume that \(n=2m\). We begin by recalling the existence and equivariance properties of the theta basis for a quiver admitting a reddening sequence. 

\begin{thm} \label{thm641} \cite[Theorem~3.5]{3}
Let \(\Sigma\) be a seed whose unfrozen quiver admits a reddening sequence. Then the algebra \(\mathcal O(\mathcal X_{|\Sigma|})\) admits a linear basis \( \Theta\bigl(\mathcal O(\mathcal X_{|\Sigma|})\bigr), \) called the \textit{theta basis} with the following properties: 

\begin{enumerate} 

\item Every element of \(\mathcal O(\mathcal X_{|\Sigma|})\) can be expressed as a finite linear combination of theta functions: \( f=\sum_{\lambda} c_\lambda\theta_\lambda, \) where only finitely many coefficients \(c_\lambda\) are nonzero.

\item The cluster modular group \(\Gamma_{|\Sigma|}\) preserves the theta basis \( \Theta\bigl(\mathcal O(\mathcal X_{|\Sigma|})\bigr) \). 
\item Every global monomial is a theta function. In particular, every Casimir monomial belongs to \(\Theta\bigl(\mathcal O(\mathcal X_{|\Sigma|})\bigr)\). 
\item For any \(l_1,l_2\in\mathcal A_{|\Sigma|}(\mathbb Z^t)\), one has \[ \theta_{l_1}\theta_{l_2} = \sum_l c_l\theta_l, \] where \(c_l\in\mathbb N\) and only finitely many \(c_l\) are nonzero. Moreover, if \(l,c\in\mathcal A_{|\Sigma|}(\mathbb Z^t)\) and \(\theta_c\) is a Casimir monomial, then, by \cite[Lemma~3.9]{3}, \[ \theta_l\theta_c=\theta_{l+c}. \] 
\item There is a natural \(\Gamma_{|\Sigma|}\)-equivariant bijection (see also \cite[Corollary~A.4]{32}) \[ \mathcal A_{|\Sigma|}(\mathbb Z^t) \xrightarrow{\sim} \Theta\bigl(\mathcal O(\mathcal X_{|\Sigma|})\bigr), \qquad l\longmapsto\theta_l. \]

In other words, for any element $g \in \Gamma_{|\Sigma|}$, we have \[ g^*(\theta_l)=\theta_{(g^{-1})^t(l)}. \]

\item Fix an arbitrary seed in the mutation class of \(\Sigma\), and let \(\mathbf a(l)\) be the tropical coordinate vector of \(l\) in the corresponding cluster chart. Then there exist finitely many nonzero coefficients \(d_{l,\mathbf v}\in\mathbb N\) such that \[ \theta_l = X^{\mathbf a(l)} + \sum_{\mathbf v>\mathbf a(l)} d_{l,\mathbf v}X^{\mathbf v}. \] \end{enumerate} \end{thm} 

Recall from the preceding section the decomposition \(l=l_{\mathrm{cas}}+l_{\mathrm{geo}}\), where \(l\in\mathcal A(\mathbb Z^t)\), \(\mathbf a(l_{\mathrm{cas}})=\sum_{i=1}^{m}v_i\,\operatorname{val}(\mathcal K_i)\) and \(\mathbf a(l_{\mathrm{geo}})=\mathbb A_{\deg}\mathbf u\). Here, the sum denotes translation by the Casimir tropical point $l_{\mathrm{cas}}$. The monomial \(\prod_{i=1}^{m}\mathcal K_i^{v_i}\) has multidegree \(\mathbf a(l_{\mathrm{cas}})\), and hence \(\theta_{l_{\mathrm{cas}}}=\prod_{i=1}^{m}\mathcal K_i^{v_i}\). This implies that \(\theta_l=\theta_{l_{\mathrm{cas}}}\theta_{l_{\mathrm{geo}}}\) by Theorem~\ref{thm641}. We then define the \textit{Weyl orbit sum} associated with \(l\) by 
\[ \theta_l^{\mathbf W_n}:=\sum_{l'\in \mathbf W_n\cdot l}\theta_{l'}. \] 

By definition, \(\theta_l^{\mathbf W_n}\) depends only on the Weyl orbit \(\mathbf W_n\cdot l\). 

\begin{thm}\label{thm642} The action of Weyl group \(\mathbf W_n\) on the theta basis is given by \( w^*(\theta_l)=\theta_{(w^{-1})^t(l)} \) for all \(w\in \mathbf W_n\) and \(l\in\mathcal A(\mathbb Z^t)\). Consequently, if \(\mathcal R\subset\mathcal A(\mathbb Z^t)\) is a set of representatives for the \(\mathbf W_n\)-orbits, then \(\{\theta_l^{\mathbf W_n}\mid l\in\mathcal R\}\) is a linear basis of the Weyl-invariant subalgebra \(\mathcal O(\mathcal X_{|\mathcal A_n|})^{\mathbf W_n}\). \end{thm} 

\begin{proof} The first assertion follows from the equivariance in Theorem~\ref{thm641}. More precisely, since \(\mathbf W_n\) acts on regular functions by its pullback, the induced action is contravariant, and hence \(w^*(\theta_l)=\theta_{(w^{-1})^t(l)}\). 

Let \(f=\sum_l d_l\theta_l\in\mathcal O(\mathcal X_{|\mathcal A_n|})^{\mathbf W_n}\). For every \(w\in \mathbf W_n\), the equality \(w^*(f)=f\) gives \[ \sum_l d_l\theta_{(w^{-1})^t(l)} = \sum_l d_l\theta_l. \] 

Reindexing the sum on the left by \(l'= (w^{-1})^t(l)\), we obtain \[ \sum_{l'}d_{w^t(l')}\theta_{l'} = \sum_{l'}d_{l'}\theta_{l'}. \] 

Since the theta functions are linearly independent, \(d_{w^t(l)}=d_l\) for every \(w\in \mathbf W_n\) and \(l\in\mathcal A(\mathbb Z^t)\). Thus, the coefficients are constant on each \(\mathbf W_n\)-orbit, and \(f\) is a finite linear combination of Weyl orbit sums. 

Conversely, the pullback action of every \(w\in \mathbf W_n\) permutes the theta functions indexed by the orbit \(\mathbf W_n\cdot l\). Therefore, every orbit sum \(\theta_l^{\mathbf W_n}\) is \(\mathbf W_n\)-invariant, and the orbit sums span \(\mathcal O(\mathcal X_{|\mathcal A_n|})^{\mathbf W_n}\). Moreover, orbit sums associated with distinct Weyl orbits have disjoint theta function supports and are therefore linearly independent. Hence \(\{\theta_l^{\mathbf W_n}\mid l\in\mathcal R\}\) is a linear basis of \(\mathcal O(\mathcal X_{|\mathcal A_n|})^{\mathbf W_n}\). \end{proof}

Next, by Proposition~\ref{prop6317}, the theta function \(\theta_{l_{\mathrm{geo}}}\) is invariant under \(\tau_i^*\) for every \(1\leq i<m\) because \(\tau_i^t(l_{\mathrm{geo}})=l_{\mathrm{geo}}\) and \(\tau_i\) is an involution. The only possible obstruction comes from the innermost cycle, for which \[ W_m^t(l_{\mathrm{geo}}) = \min_{1\leq r\leq m}[\mathbf u_m]_r \in\{0,1\}. \] 

If this minimum is zero, then \(\tau_m^t(l_{\mathrm{geo}})=l_{\mathrm{geo}}\). The equivariance of the theta basis gives \(\tau_m^*(\theta_{l_{\mathrm{geo}}})=\theta_{l_{\mathrm{geo}}}\) as \(\tau_m\) is an involution. Thus, \(\theta_{l_{\mathrm{geo}}}\) is invariant under the entire Weyl group. It remains to consider the case \(\min_r[\mathbf u_m]_r=1\). To remove this obstruction, let \(\mathcal A_n^{\mathrm{in}}\) be the frozen extension of the \(\mathcal A_n\)-quiver obtained by adjoining frozen variables \(\mathbf F_1,\ldots,\mathbf F_m\) along the innermost main cycle and adding the arrows \[ A_{m,r+1}\longrightarrow\mathbf F_r\longrightarrow A_{m,r}, \qquad 1\leq r\leq m, \] where the second indices are understood cyclically modulo \(m\).

\begin{lem}\label{lem643} In the frozen extension \(\mathcal A_n^{\mathrm{in}}\), the action of \(\tau_m\) on the mutable cluster \(\mathcal A\)-variables is given by \[ \tau_m^*(A_{i,j}) = \begin{cases} A_{i,j}W_m, & i=m,\\ A_{i,j}, & i\neq m, \end{cases} \qquad W_m = \sum_{r=1}^{m} \frac{\mathbf F_r A_{m-1,r}A_{m-1,r+m}} {A_{m,r}A_{m,r+1}}. \] 

For \(1\leq i<m\), the action of \(\tau_i\) remains unchanged. For a tropical point \(l\in\mathcal A_{|\mathcal A_n^{\mathrm{in}}|}(\mathbb Z^t)\), denote its mutable part by \(\mathbf a(l)=\mathbf a_{\mathcal A_n}(l)\) and write \(f_r=\mathbf F_r^t(l)\). Then \[ W_m^t(l) = \min_{1\leq r\leq m} \left\{ \mathbb B_{m,r}^{T}\mathbf a(l)+f_r \right\}. \] 

Consequently, \[ A_{i,j}^t\bigl(\tau_m^t(l)\bigr) = \begin{cases} \displaystyle A_{i,j}^t(l) + \min_{1\leq r\leq m} \left\{ \mathbb B_{m,r}^{T}\mathbf a(l)+f_r \right\}, & i=m,\\[4mm] A_{i,j}^t(l), & i\neq m. \end{cases} \] \end{lem} 

\begin{proof} The new frozen vertices are adjacent only to the innermost main cycle. Hence, for \(1\leq i<m\), the transformations \(\tau_i\) act on the mutable variables exactly as in the original \(\mathcal A_n\)-quiver and fix every newly added frozen variable \(\mathbf F_r\). 

For \(\tau_m\), the frozen variable \(\mathbf F_r\) contributes with exponent one precisely to the \(r\)th summand of \(W_m\) (see Proposition~\ref{prop334}). Thus, the cycle-transformation formula in Theorem~\ref{thm333} gives 
\begin{equation} \label{6.9}
W_m = \sum_{r=1}^{m} \frac{\mathbf F_r A_{m-1,r}A_{m-1,r+m}} {A_{m,r}A_{m,r+1}}. 
\end{equation}

On the other hand, by the definition of \(\mathbb B_{m,r}\), we have \[ \mathbb B_{m,r}^{T}\mathbf a(l) = -a_{m,r}-a_{m,r+1} +a_{m-1,r}+a_{m-1,r+m}. \] 

Tropicalizing \(W_m\) (\ref{6.9}) now yields \[ W_m^t(l) = \min_{1\leq r\leq m} \left\{ \mathbb B_{m,r}^{T}\mathbf a(l)+f_r \right\}. \] 

The stated formula for the tropical action of \(\tau_m\) follows from Proposition~\ref{prop632}. \end{proof}

Since the frozen extension does not change the unfrozen subquiver, the reddening sequence for the \(\mathcal A_n\)-quiver also applies to \(\mathcal A_n^{\mathrm{in}}\). Hence \(\mathcal O\bigl(\mathcal X_{|\mathcal A_n^{\mathrm{in}}|}\bigr)\) admits a theta basis. We regard \(l_{\mathrm{geo}}\) and \(l_{\mathrm{cas}}\) as tropical points of \(\mathcal A_{|\mathcal A_n^{\mathrm{in}}|}(\mathbb Z^t)\) by assigning value \(0\) to every newly added frozen variable. We then introduce the Casimir monomial \[ \mathcal K_f:=\prod_{r=1}^{m}\mathbf F_r. \]

Let \(f\in\mathcal A_{|\mathcal A_n^{\mathrm{in}}|}(\mathbb Z^t)\) denote the tropical label of \(\mathcal K_f\). Since every Casimir monomial is a theta function, \[ K_f=\theta_f, \qquad \mathbf F_r^t(f)=1 \quad (1\leq r\leq m), \] while all mutable tropical coordinates of \(f\) are zero.

\begin{thm}\label{thm644} Suppose that \( \min_{1\leq r\leq m}[\mathbf u_m]_r=1 \) for $l_{\text{geo}}$. Then \[ \tau_i^*(\theta_{l_{\mathrm{geo}}-f}) = \theta_{l_{\mathrm{geo}}-f} \qquad (1\leq i\leq m). \] \end{thm} 

\begin{proof} The mutable tropical coordinates of \(l_{\mathrm{geo}}-f\) agree with those of \(l_{\mathrm{geo}}\), while \( \mathbf F_r^t(l_{\mathrm{geo}}-f)=-1 \) for \( 1\leq r\leq m. \) Therefore, Lemma~\ref{lem643} gives \( W_m^t(l_{\mathrm{geo}}-f) = \min_{1\leq r\leq m} \left\{ [\mathbf u_m]_r-1 \right\} = 0. \) It follows that \( \tau_m^t(l_{\mathrm{geo}}-f) = l_{\mathrm{geo}}-f. \) 

For \(1\leq i<m\), the transformation \(\tau_i\) fixes the newly added frozen variables. Moreover, \(\tau_i^t(l_{\mathrm{geo}}) = l_{\mathrm{geo}} \) by Proposition~\ref{prop6317}. Hence \( \tau_i^t(l_{\mathrm{geo}}-f) = l_{\mathrm{geo}}-f \) for \( 1\leq i<m. \) As a result, the equivariance of the theta basis now yields $\tau_i^*(\theta_{l_{\mathrm{geo}}-f}) = \theta_{l_{\mathrm{geo}}-f}$ for $1\leq i\leq m$. \end{proof}

We next show that every Weyl orbit sum has a unique minimal multidegree. Let us define \( \Delta_m:=\sqrt{\prod_{i=1}^{m}\mathcal K_i}. \) 

\begin{prop}\label{prop645} For every \(l\in\mathcal A(\mathbb Z^t)\), the Weyl orbit sum \( \theta_l^{\mathbf W_n} = \sum_{l'\in \mathbf W_n\cdot l}\theta_{l'} \) has a unique minimal multidegree. \end{prop} 

\begin{proof} Recall the decomposition \( l=l_{\mathrm{cas}}+l_{\mathrm{geo}} \) constructed in the preceding section. We distinguish two cases according to the minimum entry of the innermost block \(\mathbf u_m\). 

\medskip \noindent \textit{Case 1: \(\min_r[\mathbf u_m]_r=0\).} 

By Proposition~\ref{prop6317}, \(l_{\mathrm{geo}}\) is fixed by every tropical Weyl reflection. Hence, by the equivariance of the theta basis, \(\theta_{l_{\mathrm{geo}}}\) is Weyl-invariant. Since \( \theta_l = \theta_{l_{\mathrm{cas}}}\theta_{l_{\mathrm{geo}}}, \) we obtain \[ \theta_l^{\mathbf W_n} = \theta_{l_{\mathrm{geo}}} \sum_{c\in \mathbf W_n\cdot l_{\mathrm{cas}}}\theta_c. \] 

Write \( \theta_{l_{\mathrm{cas}}} = \prod_{i=1}^{m}\mathcal K_i^{v_i} \) for some \(v_1,\ldots,v_m\in\mathbb Z\). Since the Weyl group acts on the Casimirs by signed permutations and \[ \operatorname{val}(\mathcal K_1)> \operatorname{val}(\mathcal K_2)> \cdots> \operatorname{val}(\mathcal K_m) \] with respect to the coordinatewise partial order, the Weyl orbit of \((v_1,\ldots,v_m)\) contains a unique exponent vector whose corresponding Casimir monomial has minimal multidegree \( \mathbf v^{-} = \bigl(-|v|_{(1)},\ldots,-|v|_{(m)}\bigr), \) where \( |v|_{(1)}\geq\cdots\geq|v|_{(m)} \) is the decreasing rearrangement of \((|v_1|,\ldots,|v_m|)\). 

Therefore, the Casimir orbit sum \( \sum_{c\in \mathbf W_n\cdot l_{\mathrm{cas}}}\theta_c \) has a unique minimal multidegree. Since \(\theta_{l_{\mathrm{geo}}}\) also has a unique minimal multidegree (Theorem~\ref{thm641}) and the coordinatewise partial order is compatible with addition, it follows that 
\[ \theta_l^{\mathbf W_n} = \theta_{l_{\mathrm{geo}}} \sum_{c\in \mathbf W_n\cdot l_{\mathrm{cas}}}\theta_c \] has a unique minimal multidegree.

\medskip \noindent \textit{Case 2: \(\min_r[\mathbf u_m]_r=1\).} 

Regard \(l_{\mathrm{geo}}\) as a tropical point of the frozen extension \(\mathcal A_n^{\mathrm{in}}\), as described above. We use the factorization \[ \theta_l = \theta_{l_{\mathrm{cas}}}\theta_{l_{\mathrm{geo}}} = \bigl(\theta_{l_{\mathrm{cas}}}\Delta_m^{-1}\bigr) \bigl(\theta_f\Delta_m\bigr) \theta_{l_{\mathrm{geo}}-f}. \] 

By Theorem~\ref{thm644}, \(\theta_{l_{\mathrm{geo}}-f}\) is Weyl-invariant. Moreover, a direct calculation using the Weyl action on the Casimirs and the newly added frozen variables shows that \(\theta_f\Delta_m\) is Weyl-invariant. Therefore, 
\[ \theta_l^{\mathbf W_n} = \bigl(\theta_f\Delta_m\bigr) \theta_{l_{\mathrm{geo}}-f} \sum_{c\in \mathbf W_n\cdot l_{\mathrm{cas}}} \theta_c\Delta_m^{-1}. \] 

Writing \( \theta_{l_{\mathrm{cas}}} = \prod_{i=1}^{m}\mathcal K_i^{v_i}, \) we have \( \theta_{l_{\mathrm{cas}}}\Delta_m^{-1} = \prod_{i=1}^{m}\mathcal K_i^{v_i-\frac12}. \) Thus, the relevant Casimir exponents form the half-integral vector \[ \left( v_1-\frac12,\ldots,v_m-\frac12 \right) \in \left(\frac12\mathbb Z\right)^m. \] 

The Weyl group again acts by signed permutations, and the anti-dominant chamber argument from the first case applies without change to this half-integral lattice. Hence, \( \sum_{c\in \mathbf W_n\cdot l_{\mathrm{cas}}} \theta_c\Delta_m^{-1} \) has a unique minimal multidegree. The Weyl-invariant factors \( \theta_f\Delta_m \) and \( \theta_{l_{\mathrm{geo}}-f} \) also have unique minimal multidegrees. Thus, their product with the preceding Casimir orbit sum has a unique minimal multidegree. Therefore, \(\theta_l^{\mathbf W_n}\) has a unique minimal multidegree. 

Although the expressions for \(\theta_l^{\mathbf W_n}\) in \(\mathcal O(\mathcal X_{|\mathcal A_n|})\) and \(\mathcal O(\mathcal X_{|\mathcal A_n^{\mathrm{in}}|})\) may differ, their unique mutable minimal multidegrees are determined by the corresponding tropical coordinate vectors, by Theorem~\ref{thm641}. Since \(l\) has degree zero in all newly added frozen directions, the two expressions have the same unique mutable minimal multidegree.

Moreover, although the three factors in the above factorization may separately have half-integral multidegrees, these half-integral shifts add under multiplication. Their sum is the multidegree of the original orbit sum \(\theta_l^{\mathbf W_n}\), and is therefore integral because \(l\in\mathcal A(\mathbb Z^t)\) is an integral tropical point. \end{proof}

We now identify the unique minimal multidegrees of the Weyl orbit sums in terms of the formal geodesic degrees. To this end, we first show that the Weyl orbit sums of integral Casimir monomials are generated by the formal geodesic functions. 

\begin{lem}\label{lem4319} Let \(t=\lambda+\lambda^{-1}\), and suppose that \(n\) is even. Then, for every integer \(i\) with \(0\leq i\leq n/2\), the expression \( \frac{\lambda^{n-i}+\lambda^i}{\lambda^{n/2}} \) is a polynomial in \(t\). \end{lem} 

\begin{proof} For every nonnegative integer \(k\), the expression \(\lambda^k+\lambda^{-k}\) is a polynomial in \(t=\lambda+\lambda^{-1}\). Indeed, this follows inductively from the recurrence \[ \lambda^{k+1}+\lambda^{-(k+1)} = t\bigl(\lambda^k+\lambda^{-k}\bigr) - \bigl(\lambda^{k-1}+\lambda^{-(k-1)}\bigr), \] together with the initial values \[ \lambda^0+\lambda^0=2, \qquad \lambda+\lambda^{-1}=t. \] 

Write \(n=2m\). For \(0\leq i\leq m\), we have \[ \frac{\lambda^{n-i}+\lambda^i}{\lambda^{n/2}} = \frac{\lambda^{2m-i}+\lambda^i}{\lambda^m} = \lambda^{m-i}+\lambda^{-(m-i)}. \] 

Since \(m-i\) is a nonnegative integer, the recurrence shows that this expression is a polynomial in \(t\). \end{proof} 

\begin{lem}\label{lem647} Every Weyl orbit sum of a Casimir monomial belongs to the algebra generated by the formal geodesic functions. \end{lem} \begin{proof} Let \(n=2m\). By~\cite[Theorem~4.1]{17}, we have \begin{equation}\label{6.10} \det\bigl(\mathbb A+\lambda\mathbb A^T\bigr) = \prod_{i=1}^{m} (\lambda-\mathcal K_i)(\lambda-\mathcal K_i^{-1}). \end{equation} 

By~\cite{4}, the polynomial \( \det\bigl(\mathbb A+\lambda\mathbb A^T\bigr) = \sum_{i=0}^{n}g_i\lambda^i \) is palindromic, so that \(g_{n-i}=g_i\). Hence, \[ \det\bigl(\mathbb A+\lambda\mathbb A^T\bigr) = \sum_{i=0}^{m-1} g_i\bigl(\lambda^{n-i}+\lambda^i\bigr) + g_m\lambda^m. \] 

On the other hand, Equation~\eqref{6.10} gives \[ \det\bigl(\mathbb A+\lambda\mathbb A^T\bigr) = \prod_{i=1}^{m} \left( \lambda^2+1 - (\mathcal K_i+\mathcal K_i^{-1})\lambda \right). \] 

Dividing both expressions by \(\lambda^m\) and setting \(t=\lambda+\lambda^{-1}\), we obtain 
\[ g_m + \sum_{i=0}^{m-1} g_i \frac{\lambda^{n-i}+\lambda^i}{\lambda^m} = \prod_{i=1}^{m} \left( t-(\mathcal K_i+\mathcal K_i^{-1}) \right). \] 

By Lemma~\ref{lem4319}, the left-hand side is a polynomial in \(t\). Moreover, its coefficients belong to the algebra generated by the formal geodesic functions, since each \(g_i\) is a polynomial in the entries of \(\mathbb A\). Expanding the right-hand side as \[ \sum_{r=0}^{m} (-1)^r e_r\bigl( \mathcal K_1+\mathcal K_1^{-1}, \ldots, \mathcal K_m+\mathcal K_m^{-1} \bigr)t^{m-r}, \] and comparing coefficients of \(t\), we conclude that every elementary symmetric polynomial in \( \mathcal K_1+\mathcal K_1^{-1}, \ldots, \mathcal K_m+\mathcal K_m^{-1} \) belongs to the algebra generated by the formal geodesic functions. 

By Corollary~\ref{cor354}, every Weyl orbit sum of an integral Casimir monomial is a polynomial in these elementary symmetric functions. Therefore, every such orbit sum belongs to the algebra generated by the formal geodesic functions. \end{proof}

We next consider the case when \(\min_{1\leq r\leq m}[\mathbf u_m]_r=1\). Recall that \(\Delta_m:=\sqrt{\prod_{i=1}^{m}\mathcal K_i}\).

\begin{lem}\label{lem648} For every \(\mathbf r=(r_1,\ldots,r_m)\in\mathbb Z^m\), the Weyl orbit sum of \( \Delta_m\prod_{i=1}^{m}\mathcal K_i^{r_i} = \prod_{i=1}^{m}\mathcal K_i^{r_i+\frac12} \) has the form \[ \left( \Delta_m\prod_{i=1}^{m}\mathcal K_i^{r_i} \right)^{\mathbf W_n} = \operatorname{Pf}_n G_{\mathbf r}, \] where \( \operatorname{Pf}_n := \prod_{i=1}^{m} \left(\mathcal K_i^{1/2}+\mathcal K_i^{-1/2}\right), \) and \(G_{\mathbf r}\) is a symmetric polynomial in \( \mathcal K_1+\mathcal K_1^{-1},\ldots, \mathcal K_m+\mathcal K_m^{-1}. \) Consequently, this Weyl orbit sum belongs to the Poisson algebra generated by the formal geodesic functions. \end{lem} 
\begin{proof} For every \(a\in\mathbb Z_{\geq 0}\), there exists a polynomial \(P_a(t)\in\mathbb Z[t]\) satisfying \begin{equation}\label{6.11} \mathcal K^{a+\frac12}+\mathcal K^{-a-\frac12} = \left(\mathcal K^{1/2}+\mathcal K^{-1/2}\right) P_a\left(\mathcal K+\mathcal K^{-1}\right). \end{equation} 

Indeed, these polynomials are determined recursively by \[ P_0(t)=1,\qquad P_1(t)=t-1,\qquad P_{a+1}(t)=tP_a(t)-P_{a-1}(t). \] 

For each \(i\), define \(a_i\in\mathbb Z_{\geq 0}\) by \[ \left|r_i+\frac12\right|=a_i+\frac12, \qquad\text{or equivalently}\qquad a_i=\left|r_i+\frac12\right|-\frac12. \] 

The Weyl group acts on the half-integral exponent vector \(\left(r_1+\frac12,\ldots,r_m+\frac12\right)\) by signed permutations. Thus, after summing over the independent sign changes and then over the distinct permutations of \(\mathbf a=(a_1,\ldots,a_m)\), we obtain \[ \begin{aligned} \left( \Delta_m\prod_{i=1}^{m}\mathcal K_i^{r_i} \right)^{\mathbf W_n} &= \left( \prod_{i=1}^{m}\mathcal K_i^{r_i+\frac12} \right)^{\mathbf W_n} = \sum_{\mathbf a'\in S_m\cdot\mathbf a}^{\mathrm{dist}} \prod_{i=1}^{m} \left( \mathcal K_i^{a_i'+\frac12} + \mathcal K_i^{-a_i'-\frac12} \right). \end{aligned} \] 

Applying \eqref{6.11} to each factor, we obtain \[ \begin{aligned} \left( \Delta_m\prod_{i=1}^{m}\mathcal K_i^{r_i} \right)^{\mathbf W_n} &= \prod_{i=1}^{m} \left(\mathcal K_i^{1/2}+\mathcal K_i^{-1/2}\right) \sum_{\mathbf a'\in S_m\cdot\mathbf a}^{\mathrm{dist}} \prod_{i=1}^{m} P_{a_i'}\left(\mathcal K_i+\mathcal K_i^{-1}\right) \\ &= \operatorname{Pf}_n G_{\mathbf r}, \end{aligned} \] where \[ G_{\mathbf r} := \sum_{\mathbf a'\in S_m\cdot\mathbf a}^{\mathrm{dist}} \prod_{i=1}^{m} P_{a_i'}\left(\mathcal K_i+\mathcal K_i^{-1}\right). \] 

By construction, \(G_{\mathbf r}\) is a symmetric polynomial in \( \mathcal K_1+\mathcal K_1^{-1},\ldots, \mathcal K_m+\mathcal K_m^{-1}. \) Lemma~\ref{lem647} therefore implies that \(G_{\mathbf r}\) belongs to the algebra generated by the formal geodesic functions. Moreover, \(\operatorname{Pf}_n\) belongs to the same algebra by~\cite{4}. As a result, their product \(\operatorname{Pf}_nG_{\mathbf r}\), which is the given Weyl orbit sum, also belongs to this algebra. \end{proof}

We now apply these results to the unique minimal multidegree of a Weyl orbit sum. 
\begin{thm}\label{thm649} Assume that \(n=2m\). For every \(l\in\mathcal A(\mathbb Z^t)\), the unique minimal multidegree of \(\theta_l^{\mathbf W_n}\) has the form \[ \sum_{1\leq p<q\leq n} u_{p,q}\operatorname{val}(\mathbb A_{p,q}), \qquad u_{p,q}\in\mathbb Z_{\geq0}. \] 

Equivalently, the unique minimal multidegree of every theta orbit sum belongs to the monoid generated by the unique minimal multidegrees of the formal geodesic functions (see also Example~\ref{exmp6411}). \end{thm} 

\begin{proof} Recall the decomposition \(\theta_l=\theta_{l_{\mathrm{cas}}}\theta_{l_{\mathrm{geo}}}\). Suppose first that \(\min_{1\leq r\leq m}[\mathbf u_m]_r=0\). By Proposition~\ref{prop6317}, \(\theta_{l_{\mathrm{geo}}}\) is Weyl-invariant and hence factors out of the Weyl orbit sum. Therefore, \(\theta_l^{\mathbf W_n} =\theta_{l_{\mathrm{geo}}}\theta_{l_{\mathrm{cas}}}^{\mathbf W_n}\).

The unique minimal multidegree of the first factor is \( \mathbf a(l_{\mathrm{geo}}) = \mathbb A_{\deg}\mathbf u, \) which is a nonnegative integral linear combination of the vectors \(\operatorname{val}(\mathbb A_{p,q})\). 

By Lemma~\ref{lem647}, the second factor \(\theta_{l_{\mathrm{cas}}}^{\mathbf W_n}\) is a polynomial in the formal geodesic functions \(\mathbb A_{p,q}\). Every monomial in the functions \(\mathbb A_{p,q}\) appearing in this polynomial has a unique minimal multidegree belonging to the geodesic-degree monoid. Since \(\mathbb A_{\deg}\) is invertible, distinct such monomials have distinct minimal multidegrees. It follows that the unique minimal multidegree of \(\theta_{l_{\mathrm{cas}}}^{\mathbf W_n}\) belongs to the geodesic-degree monoid. Consequently, \(\theta_l^{\mathbf W_n}\) has a unique minimal multidegree, which also belongs to the geodesic-degree monoid.

Now suppose that \( \min_{1\leq r\leq m}[\mathbf u_m]_r=1. \) Using \(K_f=\theta_f\), we have \( \theta_l = \theta_{l_{\mathrm{cas}}} \theta_{l_{\mathrm{geo}}} = \theta_{l_{\mathrm{cas}}} \theta_f \theta_{l_{\mathrm{geo}}-f}. \) Then, we rewrite this factorization as \( \theta_l = \bigl( \theta_{l_{\mathrm{cas}}}\Delta_m^{-1} \bigr) \bigl( \theta_f\Delta_m \bigr) \theta_{l_{\mathrm{geo}}-f}. \) Hence \[ \theta_l^{\mathbf W_n} = \bigl( \theta_f\Delta_m \bigr) \theta_{l_{\mathrm{geo}}-f} \left( \theta_{l_{\mathrm{cas}}}\Delta_m^{-1} \right)^{\mathbf W_n}. \] 

Moreover, Lemma~\ref{lem648} implies that \( \left( \theta_{l_{\mathrm{cas}}}\Delta_m^{-1} \right)^{\mathbf W_n} \) is a polynomial in the formal geodesic functions. By the same argument as above, using the invertibility of \(\mathbb A_{\deg}\), its unique minimal multidegree belongs to the geodesic-degree monoid. Moreover, the Weyl-invariant product \[ (\theta_f\Delta_m)\theta_{l_{\mathrm{geo}}-f} = \Delta_m\theta_{l_{\mathrm{geo}}} \] has a unique minimal multidegree belonging to the geodesic degree monoid. Indeed, \(\operatorname{val}(\Delta_m^{-1})\) is the sum of the formal geodesic degrees corresponding to the last \(m\) columns of \(\mathbb A_{\deg}\); it can be shown by using parallelogram $\blacklozenge_{p,q}$ (Example~\ref{exmp635}). Hence multiplication by \(\Delta_m\) replaces the innermost coefficient block \(\mathbf u_m\) by \(\mathbf u_m-\mathbf 1_m\). Since \(\min_r[\mathbf u_m]_r=1\), the resulting block remains nonnegative. Therefore, the same conclusion holds for the unique minimal multidegree of \(\theta_l^{\mathbf W_n}\). 

Note that although half-integral multidegrees occur in the individual factors of the above factorization, these half-integral shifts add under multiplication. Their sum is the multidegree of the original theta orbit sum and is therefore integral. The result then follows. \end{proof}

\begin{prop}\label{prop6410}
Let \(g = \sum_{[l]}c_{[l]}\theta_l^{\mathbf W_n} \) be a nonzero finite linear combination of distinct theta orbit sums. Then every minimal multidegree of \(g\) belongs to the monoid generated by \(\operatorname{val}(\mathbb A_{p,q})\).
\end{prop}

\begin{proof} 
By Proposition~\ref{prop645}, each orbit sum \(\theta_l^{\mathbf W_n}\) has a unique minimal multidegree, which we denote by \(\mathbf d_l\). Moreover, distinct theta orbit sums have distinct minimal multidegrees. If \(\mathbf d_l=\mathbf d_{l'}\), then their minimal theta labels have the same tropical coordinate vector in the fixed seed. Hence, they define the same tropical point, so \(l\) and \(l'\) belong to the same Weyl orbit. 

Let \(\mathbf d\) be a minimal multidegree of \(g\). We claim that \(\mathbf d=\mathbf d_{l_0}\) for some \(l_0\) with \(c_{[l_0]}\neq0\). Otherwise, \(\mathbf d\) occurs as a nonminimal term of some orbit sum \(\theta_l^{\mathbf W_n}\) appearing in the expansion of \(g\), and Theorem~\ref{thm641} implies that \(\mathbf d_l<\mathbf d\). If the minimal term of \(\theta_l^{\mathbf W_n}\) is canceled, then it must occur as a nonminimal term of another orbit sum, since distinct orbit sums have distinct minimal multidegrees. 

The minimal multidegree of that orbit sum is therefore strictly smaller than \(\mathbf d_l\). Repeating this argument produces a strictly decreasing sequence of minimal multidegrees. Since only finitely many orbit sums occur in the expansion of \(g\), this process must terminate at an uncanceled term of multidegree strictly smaller than \(\mathbf d\), contradicting the minimality of \(\mathbf d\). Thus, \(\mathbf d=\mathbf d_{l_0}\) for some \(l_0\) with \(c_{[l_0]}\neq0\). 

Conversely, let \(\mathbf d_l\) be minimal among the multidegrees \(\mathbf d_{l'}\) for which \(c_{[l']}\neq0\). It cannot occur as a nonminimal term of another orbit sum appearing in \(g\), since Theorem~\ref{thm641} would then yield \(\mathbf d_{l'}<\mathbf d_l\) for some \(l'\) with \(c_{[l']}\neq0\). Nor can it occur as the minimal term of a distinct orbit sum. Hence, its minimal term cannot be canceled, and \(\mathbf d_l\) is a minimal multidegree of \(g\). Therefore, the minimal multidegrees of \(g\) are exactly the minimal elements among the multidegrees \(\mathbf d_l\) occurring in its orbit-sum expansion. The final assertion follows from Theorem~\ref{thm649}. \end{proof}

\begin{exmp}\textit{(Weyl group orbit sums of theta functions and unique minimal multidegrees for $n=4$)} \label{exmp6411} 

Let $l_1,l_2\in\mathcal A(\mathbb Z^t)$ have tropical coordinate vectors \[ \mathbf a(l_1)=(0,0,1,-2,1,1)^T, \qquad \mathbf a(l_2)=(-2,-4,-1,-2,-4,-1)^T. \] 

Their decompositions from Proposition~\ref{prop6317} are \[ \begin{aligned} \mathbf a(l_1) &= 3\operatorname{val}(\mathcal K_1) +2\operatorname{val}(\mathcal K_2) +\mathbb A_{\deg}(2,1,3,4,3,0)^T,\\ \mathbf a(l_2) &= 3\operatorname{val}(\mathcal K_1) +\operatorname{val}(\mathcal K_2) +\mathbb A_{\deg}(7,3,4,2,4,1)^T. \end{aligned} \] 

For $l_1$, its innermost block is $(3,0)^T$, so Case~1 of Theorem~\ref{thm649} applies. Since the geodesic part \( \mathbb A_{\deg}(2,1,3,4,3,0)^T \) is invariant under the tropical Weyl group, the unique minimal Casimir exponent vector in the Weyl orbit is $(-3,-2)$. Therefore, the unique minimal multidegree of $\theta_{l_1}^{\mathbf W_4}$ is \[ \begin{aligned} d_1 &= \mathbb A_{\deg}(2,1,3,4,3,0)^T -3\operatorname{val}(\mathcal K_1) -2\operatorname{val}(\mathcal K_2)\\ &= (-6,-6,-5,-8,-9,-9)^T. \end{aligned} \]

For $l_2$, the innermost block is $(4,1)^T$. We first factor \[ \begin{aligned} \theta_{l_2} =\theta_{l_{2,\mathrm{cas}}}\theta_f \theta_{l_{2,\mathrm{geo}}-f}= \underbrace{(\theta_f\Delta_2) \theta_{l_{2,\mathrm{geo}}-f}}_{\mathscr G_{l_2}} \left(\theta_{l_{2,\mathrm{cas}}}\Delta_2^{-1}\right). \end{aligned} \] 

Then, we have \( \theta_{l_2}^{\mathbf W_4} = \mathscr G_{l_2} \left( \theta_{l_{2,\mathrm{cas}}}\Delta_2^{-1} \right)^{\mathbf W_4}.\) Since $\operatorname{val}(\Delta_2^{-1})$ is the sum of the last two columns of $\mathbb A_{\deg}$, multiplication by $\Delta_2$ replaces the innermost block $(4,1)^T$ by $(3,0)^T$. This means, \[ \operatorname{val}(\mathscr G_{l_2}) = \operatorname{val}((\theta_f\Delta_2) \theta_{l_{2,\mathrm{geo}}-f}) = \operatorname{val}(\Delta_2 \theta_{l_{2,\mathrm{geo}}}) = \mathbb A_{\deg}(7,3,4,2,3,0)^T. \] 

On the other hand, $\theta_{l_{2,\mathrm{cas}}}=\mathcal K_1^3\mathcal K_2$ and $\Delta_2=(\mathcal K_1\mathcal K_2)^{1/2}$, so \( \theta_{l_{2,\mathrm{cas}}}\Delta_2^{-1} = \mathcal K_1^{5/2}\mathcal K_2^{1/2}. \) This implies that the unique minimal exponent vector in its signed-permutation orbit is $(-\frac52,-\frac12)$. Consequently, 
\[ \begin{aligned} d_2 &= \operatorname{val} \left(\theta_{l_2}^{\mathbf W_4}\right) = \mathbb A_{\deg}(7,3,4,2,3,0)^T -\frac52\operatorname{val}(\mathcal K_1) -\frac12\operatorname{val}(\mathcal K_2)= (-7,-9,-6,-7,-10,-7)^T. \end{aligned} \] 

Moreover, \( \mathbb B_{\deg}^Td_1=(3,2,4,5,7,4)^T\) and \( \mathbb B_{\deg}^Td_2=(9,5,6,4,4,1)^T. \) Thus, as asserted by Theorem~\ref{thm649}, both multidegrees belong to the geodesic degree monoid. They are realized, respectively, as the unique minimal multidegrees of \[ \begin{aligned} G_1&= \mathbb A_{3,4}^{\,3}\mathbb A_{1,4}^{\,2}\mathbb A_{1,2}^{\,4} \mathbb A_{2,3}^{\,5}\mathbb A_{2,4}^{\,7}\mathbb A_{1,3}^{\,4},\\ G_2&= \mathbb A_{3,4}^{\,9}\mathbb A_{1,4}^{\,5}\mathbb A_{1,2}^{\,6} \mathbb A_{2,3}^{\,4}\mathbb A_{2,4}^{\,4}\mathbb A_{1,3}. \end{aligned} \] 

Since $d_1$ and $d_2$ are integral, $G_1$ and $G_2$ are universal Laurent polynomials. Let \[ g=c_1\theta_{l_1}^{\mathbf W_4} +c_2\theta_{l_2}^{\mathbf W_4}, \qquad c_1c_2\neq0. \] 

By Proposition~\ref{prop6410}, the incomparable multidegrees $d_1$ and $d_2$ are precisely the minimal multidegrees of $g$. Suitable scalar multiples of $G_1$ and $G_2$ eliminate the corresponding minimal terms simultaneously. The remainder is again a Weyl-invariant universal Laurent polynomial, and its minimal multidegrees are strictly greater than the eliminated ones.\end{exmp}

\begin{thm}\label{thm6412} Assume that \(n\) is even. Then \(\mathcal O(\mathcal X_{|\mathcal A_n|})^{\mathbf W_n} \subseteq\mathcal O(\mathcal A_n)\). Thus, Theorem~\ref{thm422} implies \( \mathcal O(\mathcal X_{|\mathcal A_n|})^{\mathbf W_n} = \mathcal O(\mathcal A_n) \cap \mathcal O(\mathcal X_{|\mathcal A_n|})\). \end{thm}

\begin{proof} Let \( g\in\mathcal O(\mathcal X_{|\mathcal A_n|})^{\mathbf W_n}. \) By Theorem~\ref{thm642}, the function \(g\) admits a finite expansion in theta orbit sums: \[ g=\sum_{[l]}c_{[l]}\theta_l^{\mathbf W_n}. \] 

Let \( \mathbf d_1,\ldots,\mathbf d_s \) be the minimal multidegrees of \(g\). By Proposition~\ref{prop6410}, for every \(1\leq\alpha\leq s\), there exist nonnegative integers \(u_{p,q}^{(\alpha)}\) such that \[ \mathbf d_\alpha = \sum_{p<q} u_{p,q}^{(\alpha)} \operatorname{val}(\mathbb A_{p,q}). \] 

Set \[ G_\alpha := \prod_{p<q} \mathbb A_{p,q}^{u_{p,q}^{(\alpha)}}. \] 

Since each formal geodesic function has a unique minimal multidegree, so does $G_\alpha$, with \( \operatorname{val}(G_\alpha)=\mathbf d_\alpha. \) Moreover, $G_\alpha$ is universally Laurent. Indeed, by Theorem~\ref{thm416}, each formal geodesic function has the form $M^{1/2}L$ in every cluster, where $M^{1/2}$ is determined by its unique minimal multidegree and $L$ is Laurent. 

Since the total minimal multidegree $\mathbf d_\alpha$ is integral, the square-root monomial factors in $G_\alpha$ combine into an integral Laurent monomial in every cluster. Finally, $G_\alpha$ is Weyl-invariant because each formal geodesic function is Weyl-invariant.

We now eliminate all minimal terms of \(g\) simultaneously. Since the multidegrees \(\mathbf d_1,\ldots,\mathbf d_s\) are pairwise incomparable and \(G_\alpha\) has the unique minimal multidegree \(\mathbf d_\alpha\), the support of \(G_\alpha\) contains no \(\mathbf d_\beta\) with \(\beta\neq\alpha\). Hence subtracting a scalar multiple of \(G_\alpha\) does not change the coefficient of the term of multidegree \(\mathbf d_\beta\) for \(\beta\neq\alpha\). We may therefore choose scalars \(b_1,\ldots,b_s\) such that 
\[ g_1 := g-\sum_{\alpha=1}^{s}b_\alpha G_\alpha \] has no term of multidegree \(\mathbf d_1,\ldots,\mathbf d_s\). The function \(g_1\) remains Weyl-invariant. If \(g_1\neq0\), each minimal multidegree of \(g_1\) is strictly greater, with respect to the coordinatewise partial order, than one of the eliminated multidegrees. We repeat this procedure, at each stage eliminating all minimal multidegrees of the current remainder.

At every stage, the remainder is Weyl-invariant, so Proposition~\ref{prop6410} applies again. Hence every minimal multidegree arising during the procedure belongs to the geodesic-degree monoid and is therefore coordinatewise nonpositive, since each vector \(\operatorname{val}(\mathbb A_{p,q})\) is coordinatewise nonpositive. Now set \[ D := \bigcup_{\alpha=1}^{s} \left\{ \mathbf d\in\mathbb Z^{n(n-1)/2} \mathrel{\big|} \mathbf d_\alpha\leq\mathbf d\leq\mathbf 0 \right\}. \] This is a finite partially ordered set. Moreover, if the elimination process continues for \(r\) stages, then every minimal multidegree occurring at the \(r\)th stage lies at the end of a strictly increasing chain \[ \mathbf d_{\alpha} < \mathbf e_1 < \cdots < \mathbf e_r \] contained in \(D\). Indeed, at each stage every new minimal multidegree is strictly greater than a minimal multidegree eliminated at the preceding stage. 

Since a finite partially ordered set contains no strictly increasing chains of arbitrarily large length, the elimination process must terminate after finitely many steps. It follows that \(g\) is a polynomial in the formal geodesic functions, and hence \( g\in\mathcal O(\mathcal A_n). \) Therefore, \( \mathcal O(\mathcal X_{|\mathcal A_n|})^{\mathbf W_n} \subseteq \mathcal O(\mathcal A_n). \) \end{proof}

\begin{rem}\label{rem6413} \textit{(A description of the function field of \(\mathcal X_{|\mathcal A_n|}\)).} Set \[ \mathcal R := \mathcal O(\mathcal X_{|\mathcal A_n|})^{\mathbf W_n} [\mathcal K_1^{\pm1},\ldots,\mathcal K_m^{\pm1}] \subseteq \mathcal O(\mathcal X_{|\mathcal A_n|}). \] 

The action of \(\mathbf W_n\) on the set of Casimirs \(\mathcal K_i^{\pm1}\) is faithful. Moreover, the Casimirs are roots of the characteristic polynomial in~\eqref{6.10}, whose coefficients belong to \(\mathcal O(\mathcal X_{|\mathcal A_n|})^{\mathbf W_n}\). Consequently, \[ \begin{aligned} \bigl[ \operatorname{Frac}(\mathcal O(\mathcal X_{|\mathcal A_n|})) : \operatorname{Frac}( \mathcal O(\mathcal X_{|\mathcal A_n|})^{\mathbf W_n}) \bigr] &= |\mathbf W_n|, \\ \bigl[ \operatorname{Frac}(\mathcal R) : \operatorname{Frac}( \mathcal O(\mathcal X_{|\mathcal A_n|})^{\mathbf W_n}) \bigr] &= |\mathbf W_n|. \end{aligned} \] 

Since \(\operatorname{Frac}(\mathcal R)\subseteq \operatorname{Frac}(\mathcal O(\mathcal X_{|\mathcal A_n|}))\), the equality of these degrees implies that \[ \operatorname{Frac}(\mathcal R) = \operatorname{Frac}(\mathcal O(\mathcal X_{|\mathcal A_n|})). \] 

Together with Theorem~\ref{thm6412}, this shows that the function field \(\operatorname{Frac}(\mathcal O(\mathcal X_{|\mathcal A_n|}))\) is generated by the Casimirs and the square-root-free products of formal geodesic functions. We also note that both \(\mathcal O(\mathcal X_{|\mathcal A_n|})\) and \(\mathcal O(\mathcal X_{|\mathcal A_n|})^{\mathbf W_n}\) are normal rings. Naturally, we expect that $\mathcal R = \mathcal O(\mathcal X_{|\mathcal A_n|})$ as a ring. \end{rem}

\begin{rem}\label{rem6414} 
Let \(\mathbf G_n\) denote the finite group acting by changing the choices of the square roots \(\sqrt{X_{i,j}}\). Then Theorem~\ref{thm6412} implies that \[ \mathcal O(\mathcal X_{|\mathcal A_n|})^{\mathbf W_n} = \mathcal O(\mathcal A_n)^{\mathbf G_n}. \] The right-hand side is precisely the square-root-free subalgebra of \(\mathcal O(\mathcal A_n)\). This identity may be regarded as a cluster-theoretic analogue of the Chevalley restriction theorem. \end{rem}

\section{$\imath$Quantum Groups of Type $\mathrm{AI}_n$ and Their Cluster Structure}

\subsection{$\imath$Quantum Group of Type $\mathrm{AI}_n$}
We introduce the $\imath$quantum group of type $\mathrm{AI}_n$. Our explanation is based on\cite{37}. Consider the Lie algebra \(\mathfrak g=\mathfrak{sl}_{n+1}(\mathbb C)\) with associated Cartan matrix $(a_{ij})$. Here, $a_{ij}$ equals $2$ if $i = j$, equals $-1$ if $|i - j| = 1$, and equals $0$ if $|i - j| > 1$.

\begin{defn}
The algebra $\widetilde{U}_n$ is the unital $\mathbb Q(q^{1/2})$-algebra generated by $E_i, F_i, K_i^{\pm 1}, {K'_i}^{\pm 1}$ for $1 \leq i \leq n$ with the following relations for $1 \leq i, j \leq n$:
\begin{align*}
    & K_i K_i^{-1} = K_i^{-1} K_i = K'_i {K'_i}^{-1} = {K'_i}^{-1} K'_i = 1, \\
    & [K_i, K_j] = [K'_i, K'_j] = 0, \\
    & K_i E_j = q^{a_{ij}} E_j K_i, \quad K'_i E_j = q^{-a_{ij}} E_j K'_i, \\
    & K_i F_j = q^{-a_{ij}} F_j K_i, \quad K'_i F_j = q^{a_{ij}} F_j K'_i, \\
    & [E_i, F_j] = \delta_{ij}(q - q^{-1})(K'_i - K_i), \\
    & E_i^2 E_j - (q + q^{-1})E_i E_j E_i + E_j E_i^2 = 0 \quad (\text{if } |i - j| = 1), \\
    & F_i^2 F_j - (q + q^{-1})F_i F_j F_i + F_j F_i^2 = 0 \quad (\text{if } |i - j| = 1), \\
    & [E_i, E_j] = [F_i, F_j] = 0 \quad (\text{if } |i - j| > 1).
\end{align*}
\end{defn}

$\widetilde{U}_n$ is called the Drinfeld double quantum group associated with the Lie algebra $\mathfrak{g}$. It is a Hopf algebra with the coproduct given by:
\begin{align*}
    & \Delta(E_i) = E_i \otimes 1 + K_i \otimes E_i, \quad \Delta(K_i) = K_i \otimes K_i, \\
    & \Delta(F_i) = F_i \otimes K'_i + 1 \otimes F_i, \quad \Delta(K'_i) = K'_i \otimes K'_i,
\end{align*}
for $1 \leq i \leq n$.

The quantum group $U_n$ is the unital $\mathbb Q(q^{1/2})$-algebra with generators $E_i, F_i, K_i^{\pm 1}$ for $1 \leq i \leq n$ and with relations $K_iK_i' = 1$. There is a surjective algebra homomorphism, called the central reduction map of $\widetilde{U}_n$:
\[
    \pi: \widetilde{U}_n \to U_n,
\]
given by $E_i \mapsto E_i$, $F_i \mapsto F_i$, $K_i \mapsto K_i$, and $K'_i \mapsto K_i^{-1}$ for $1 \leq i \leq n$.

We introduce the $\imath$quantum group of type AI. It is a certain coideal subalgebra of the quantum group.

\begin{defn}
The algebra $\widetilde{U}^{\imath}_n$ is the unital $\mathbb Q(q^{1/2})$-subalgebra of $\widetilde{U}_n$, generated by elements $B_i, k_i^{\pm 1}$ which are given by
\[
    B_i = F_i - q^{-1}E_i K'_i, \quad k_i = K_i K'_i
\]
for $1 \leq i \leq n$.
\end{defn}

The algebra $\widetilde{U}^{\imath}_n$ is called the universal $\imath$quantum group of type $AI_n$. Let $\widetilde{U}^{\imath 0}_n$ be the unital $\mathbb Q(q^{1/2})$-subalgebra of $\widetilde{U}^{\imath}_n$ generated by $k_i^{\pm 1}$ for $1 \leq i \leq n$. Then $\widetilde{U}^{\imath 0}_n$ is a central subalgebra of $\widetilde{U}^{\imath}_n$. 

The algebra $\widetilde{U}^{\imath}_n$ is a right coideal subalgebra of $\widetilde{U}_n$; the coproduct of $\widetilde{U}_n$ restricts to:
\[
    \Delta: \widetilde{U}^{\imath}_n \to \widetilde{U}^{\imath}_n \otimes \widetilde{U}_n.
\]
\begin{prop} \label{prop713}
The algebra $\widetilde{U}^{\imath}_n$ is the unital $\mathbb Q(q^{1/2})$-algebra generated by $B_i, k_i^{\pm 1}$ with relations:
\begin{align*}
    & k_i k_i^{-1} = k_i^{-1} k_i = 1, \quad [k_i, k_j] = [k_i, B_j] = 0, \\
    & [B_i, B_j] = 0 \quad (\text{if } |i - j| > 1), \\
    & B_j B_i^2 - (q + q^{-1})B_i B_j B_i + B_i^2 B_j = (q - q^{-1})^2 B_j k_i \quad (\text{if } |i - j| = 1)
\end{align*}
for $1 \leq i \leq n$.
\end{prop}

The $\imath$quantum group $U^{\imath}_n$ is defined to be the $\mathbb Q(q^{1/2})$-subalgebra of $U_n$, generated by the elements
\[
    B_i = F_i + q^{-1}E_i K_i^{-1} \quad (\text{for } 1 \leq i \leq n).
\]
Naturally, there is a surjective algebra homomorphism:
\[
    \pi^{\imath}: \widetilde{U}^{\imath}_n \to U^{\imath}_n,
\]
which is called the central reduction of $\widetilde{U}^{\imath}_n$, given by:
\[
    B_i \mapsto B_i, \quad k_i \mapsto -1 \quad (\text{for } 1 \leq i \leq n).
\]
The kernel of $\pi^{\imath}$ is the two-sided ideal generated by $k_i + 1$ for $1 \leq i \leq n$. Note that $\pi^{\imath}$ is not the restriction of $\pi$. They only coincide after a twist of the generators.

\subsection{Classical Image of the Embedding for the \(\imath\)Quantum Group of Type \(\mathrm{AI}_n\)} \label{ch7.2} 

Throughout this subsection, we work with the semiclassical limits of \(\widetilde U_n^{\imath}\) and \(U_n^{\imath}\), regarded as Poisson algebras. We first introduce a frozen extension of the \(\mathcal A_{n+1}\)-quiver. Recall that the first main cycle consists of the vertices \( X_{1,1},X_{1,2},\ldots,X_{1,n+1}, \) where the second indices are interpreted cyclically modulo \(n+1\). 

\begin{defn}\label{defn721} For each \(i\in\{1,\ldots,n+1\}\), introduce a frozen vertex \(Z_i\) and add the arrows \[ X_{1,i+1}\longrightarrow Z_i\longrightarrow X_{1,i}. \] 

We also add a dashed arrow \(Z_i\longrightarrow Z_{i+1}\) of weight \(1/2\) for each \(i\in\{1,\ldots,n+1\}\). Here, all indices are understood cyclically modulo \(n+1\). 

We denote the resulting frozen extension by \(\Sigma_n\). Our convention differs slightly from that of~\cite{37}: after the cyclic relabeling adopted below, the quiver in~\cite{37} does not contain the frozen vertex \(Z_{n+1}\) or any arrows incident to it. \end{defn}

\begin{exmp}\label{exmp722} For \(n=3\), the mutable part is the \(\mathcal A_4\)-quiver, whose first main cycle is \[ X_{1,1} \longrightarrow X_{1,2} \longrightarrow X_{1,3} \longrightarrow X_{1,4} \longrightarrow X_{1,1}. \] 

The frozen extension \(\Sigma_3\) contains the four frozen vertices \( Z_1,\ Z_2,\ Z_3,\ Z_4 \) with arrows \[ X_{1,i+1} \longrightarrow Z_i \longrightarrow X_{1,i}, \qquad 1\leq i\leq4, \] where \(X_{1,5}=X_{1,1}\). The frozen vertices form the dashed cycle \[ Z_1 \dashrightarrow Z_2 \dashrightarrow Z_3 \dashrightarrow Z_4 \dashrightarrow Z_1. \]

Note that deleting \(Z_4\) and its incident arrows recovers the quiver used in~\cite{37}. \end{exmp} 

We now recall the cluster realization of the classical \(\imath\)quantum group of type $\mathrm{AI}$. 
\begin{thm}[{\cite{37}}]\label{thm723} There is a Poisson embedding 
\[ \iota: \widetilde U_n^{\imath} \longrightarrow \mathcal O(\mathcal X_{|\Sigma_n|}) \] such that \[ \begin{aligned} \iota(B_i) ={}& Z_i^{-1} \left( 1+ \sum_{j=1}^{m} \prod_{k=1}^{j}X_{k,i}^{-1} \right.\\ &\left. \hspace{1.5cm} + \left( \prod_{k=1}^{m}X_{k,i}^{-1} \right) \left( \sum_{j=e}^{m-1} \prod_{k=e}^{j} X_{m-k,m-k+i}^{-1} \right) \right), \end{aligned} \] and \[ \iota(k_i) = -Z_i^{-2} \left( \prod_{k=1}^{m}X_{k,i}^{-1} \right) \left( \prod_{k=e}^{m-1} X_{m-k,m-k+i}^{-1} \right), \] where \[ 1\leq i\leq n, \qquad e=\operatorname{par}(n+1)-1, \qquad m=\left\lfloor\frac{n+1}{2}\right\rfloor. \] \end{thm} 

By abuse of notation, we identify \(B_i\) and \(k_i\) with their images under \(\iota\) for \(1\leq i\leq n\). We then extend these formulas cyclically to define \(B_{n+1}\) and \(k_{n+1}\). Our conventions differ from those of~\cite{37} in that we use the inverse variables \(X_{i,j}^{-1}\) and \(Z_i^{-1}\), together with the cyclic relabeling described above. These modifications preserve the defining Poisson relations and ensure that all \(B_i\) and \(k_i\) in this realization are universally Laurent.

\begin{cor}\label{cor724} Let \(m=\lfloor (n+1)/2\rfloor\), and set \[ M_n:=\prod_{r=1}^{n+1}Z_r \prod_{i=1}^{m}\prod_{j=1}^{N_i}X_{i,j}. \] 

Define \[ \mathcal I := \left\langle M_n-1,\ k_i^{-1}+1 \mathrel{\big|} 1\leq i\leq n \right\rangle \subset \mathcal O(\mathcal X_{|\Sigma_n|}). \] 

Then \(\iota\) induces an injective Poisson homomorphism \[ \iota\colon U_n^{\imath} \longrightarrow \mathcal O(\mathcal X_{|\Sigma_n|})/\mathcal I. \] 

Moreover, modulo \(\mathcal I\), each \(B_i\), for \(1\leq i\leq n+1\), is identified with the corresponding elementary formal geodesic function. \end{cor} 

\begin{proof} The quiver \(\Sigma_n\) is obtained from the quiver used in Song's cluster realization by adjoining the frozen vertex \(Z_{n+1}\). Imposing the central relation \(M_n=1\) recovers Song's original cluster realization and gives \[ k_{n+1} = (-1)^{n-1}\prod_{i=1}^{n}k_i^{-1}. \] 

The central reduction \[ \widetilde U_n^\imath \longrightarrow U_n^\imath \] is obtained by imposing \(k_i=-1\), or equivalently \(k_i^{-1}+1=0\), for \(1\leq i\leq n\). Together with \(M_n=1\), these relations imply \[ k_{n+1} = (-1)^{n-1}(-1)^n = -1. \] 

Thus, the relation \(k_{n+1}^{-1}+1=0\) is automatic. After imposing these central relations, the resulting homomorphism is precisely the classical limit of the cluster realization constructed in~\cite{37}. Its injectivity therefore follows from Song's injective cluster realization. 

Modulo \(\mathcal I\), the functions \(B_i\), for \(1\leq i\leq n\), are identified with the corresponding elementary formal geodesic functions (Theorem~\ref{thm723}). Since \(M_n\) and the \(k_i\) are Casimir monomials of the quiver, the same identification for \(B_{n+1}\) follows from the Bondal Poisson bracket~\eqref{1.4}. Note that a formal geodesic function generally has the form \(M^{1/2}L\), where \(M\) is a cluster monomial and \(L\) is Laurent, and therefore need not belong to \(\mathcal O(\mathcal X_{|\Sigma_n|})\) itself. Hence, this assertion should be understood as an identification modulo \(\mathcal I\). \end{proof}

\begin{exmp}\label{exmp725} For \(n=3\), we have \( m=2\) and \(e=1. \) Hence, 
\[ \begin{aligned} B_i ={}& Z_i^{-1} \left( 1+X_{1,i}^{-1} +X_{1,i}^{-1}X_{2,i}^{-1} +X_{1,i}^{-1}X_{2,i}^{-1}X_{1,i+1}^{-1} \right),\\ k_i ={}& -Z_i^{-2} X_{1,i}^{-1}X_{2,i}^{-1}X_{1,i+1}^{-1}. \end{aligned} \] 

In particular, \[ \begin{aligned} B_1 ={}& Z_1^{-1} \left( 1+X_{1,1}^{-1} +X_{1,1}^{-1}X_{2,1}^{-1} +X_{1,1}^{-1}X_{2,1}^{-1}X_{1,2}^{-1} \right),\\ k_1 ={}& -Z_1^{-2} X_{1,1}^{-1}X_{2,1}^{-1}X_{1,2}^{-1}. \end{aligned} \] 

The remaining elements \(B_i\) and \(k_i\) are obtained by cyclically shifting the second indices. \end{exmp} 

\begin{rem}\label{rem726} The elements \(B_i\) and \(k_i\) have the following properties: \begin{enumerate} 
\item Each \(k_i\) is a Casimir monomial of the \(\Sigma_n\)-quiver. 
\item The central reduction given by \(k_i=-1\) for \(1\leq i\leq n\), together with the relation \(M_n=1\), recovers the compatibility conditions used in the construction of the \(\mathcal A_{n+1}\)-quiver. Under this reduction, we have \(B_1=\mathbb A_{n,n+1}\), \(B_2=\mathbb A_{1,n+1}\), and \(B_i=\mathbb A_{i-2,i-1}\) for \(3\leq i\leq n+1\). 

Thus, unlike the \(\mathcal A_{n+1}\)-quiver, for which the individual formal geodesic functions do not belong to the ring of regular functions, each formal geodesic function in the present setting is realized by one of the \(B_i\). In this sense, the \(\Sigma_n\)-quiver provides a natural extension of the \(\mathcal A_{n+1}\)-quiver.\end{enumerate} \end{rem}

We now extend the birational Weyl group action from the \(\mathcal A_{n+1}\)-quiver to \(\Sigma_n\). Since the newly added frozen vertices are adjacent only to the first main cycle, the transformations associated with the remaining main cycles fix them. 

\begin{prop}\label{prop727} Let \( m=\left\lfloor\frac{n+1}{2}\right\rfloor. \) The birational transformations \( s_1^*,\ldots,s_m^* \) extend to Poisson birational transformations of \(\mathcal X_{|\Sigma_n|}\). Moreover, the group \[ \mathbf W_n^{\imath} := \langle s_1^*,\ldots,s_m^*\rangle \] is isomorphic to the Weyl group of type \(B_m\). \end{prop} 

\begin{proof} The frozen extension does not change the adjacency relations between main cycles. Therefore, they also generate a Weyl group of type \(B_m\). \end{proof} 

\begin{prop}\label{prop728} The transformations \(s_l^*\) act on the mutable variables \(X_{i,j}\) exactly as they do on the \(\mathcal A_{n+1}\)-quiver. Their action on the frozen variables is given by \[ s_l^*(Z_i) = \begin{cases} Z_iY_{1,i}, & l=1,\\ Z_i, & l\neq1, \end{cases} \qquad 1\leq i\leq n+1. \] \end{prop} 

\begin{proof} The frozen vertices are adjacent only to the first main cycle. Therefore, \(s_l^*\) fixes every \(Z_i\) whenever \(l\neq1\). The vertex \(Z_i\) is adjacent to the consecutive mutable vertices \(X_{1,i}\) and \(X_{1,i+1}\). Theorem \ref{thm333} therefore gives \( s_1^*(Z_i)=Z_iY_{1,i}. \) The same calculation applies to every \(i\in\{1,\ldots,n+1\}\), including the additional frozen vertex \(Z_{n+1}\). \end{proof} 

\begin{prop}\label{prop729} The elements \(B_i\) and \(k_i\) are invariant under the Weyl group \(\mathbf W_n^{\imath}\). \end{prop} 

\begin{proof} \(B_i\) and \(k_i\) are preserved by the same calculation as in Theorem~\ref{thm422}, so we omit the details. \end{proof}

\subsection{Weyl Invariants for the \(\imath\)Quantum Group of Type \(\mathrm{AI}_n\) for Odd $n$} \label{Ch7.3}

Assume for the remainder of this subsection that \(n\) is odd. Then the mutable part of \(\Sigma_n\) is the \(\mathcal A_{n+1}\)-quiver, where \(n+1\) is even. Hence its unfrozen quiver admits a reddening sequence (Theorem~\ref{thm621}), and \(\mathcal O(\mathcal X_{|\Sigma_n|})\) admits a theta basis parametrized by \(\mathcal A_{|\Sigma_n|}(\mathbb Z^t)\). Furthermore, since $n+1$ is even, the generators $s_i^*$ of $\mathbf W_n^\imath$ agree with the cycle mutations $\tau_i^*$ introduced in Section~\ref{Ch3.3}.

\begin{defn} The \emph{frozen grading} on \(\mathcal K(\mathcal X_{|\Sigma_n|})\) is the \(\mathbb Z^{n+1}\)-grading defined by \[ \deg_f(X_{i,j})=0, \qquad \deg_f(Z_r)=\mathbf e_r, \] where \(\mathbf e_r\) is the \(r\)th standard basis vector of \(\mathbb Z^{n+1}\). \end{defn} 

\begin{prop}\label{prop732} Let \(g\in\mathcal O(\mathcal X_{|\Sigma_n|})^{\mathbf W_n^\imath}\), and write its frozen-homogeneous decomposition as \(g=\sum_{\mathbf z}g_{\mathbf z}\), where \(\deg_f(g_{\mathbf z})=\mathbf z\). Then every component \(g_{\mathbf z}\) is invariant under \(\mathbf W_n^\imath\). \end{prop} 

\begin{proof} The Weyl group action preserves the frozen grading. Hence \(w^*(g_{\mathbf z})\) has frozen degree \(\mathbf z\) for every \(w\in W_n^\imath\). Since \(w^*(g)=g\), the uniqueness of the homogeneous decomposition implies \(w^*(g_{\mathbf z})=g_{\mathbf z}\). \end{proof} 

It therefore suffices to consider Weyl-invariant functions that are homogeneous with respect to the frozen grading. By \cite[Proposition~3.17]{32}, every theta function in \(O(\mathcal X_{|\Sigma_n|})\) is frozen-homogeneous. Since the Weyl group action preserves the frozen grading, each theta orbit sum is also frozen-homogeneous. Thus, each frozen-homogeneous component can be expanded in theta orbit sums of that frozen degree.

Let \(\mathcal K_{\zeta}:=\prod_{r=1}^{n+1}Z_r\). By abuse of notation, we also denote by \(\zeta\in\mathcal A_{|\Sigma_n|}(\mathbb Z^t)\) the tropical point corresponding to \(K_\zeta\). Hence, \(\mathcal K_{\zeta}=\theta_\zeta\) as $\mathcal K_{\zeta}$ is a Casimir monomial. The tropical point \(\zeta\) has zero mutable coordinates and value \(1\) on every frozen variable \(Z_r\). 

Let \(l\in\mathcal A_{|\Sigma_n|}(\mathbb Z^t)\), and decompose its tropical coordinate vector as 
\begin{equation} \label{7.1}
    \mathbf a_{|\Sigma_n|}(l) = \bigl(\mathbf a_{\mathrm{mut}}(l),\mathbf z\bigr), \qquad \mathbf z=(z_1,\ldots,z_{n+1}).
\end{equation}

Applying the geodesic decomposition of Section~\ref{ch6.3}, write \[ l=l_{\mathrm{cas}}+l_{\mathrm{geo}}, \qquad \mathbf a_{\mathrm{mut}}(l_{\mathrm{geo}}) = \mathbb A_{\deg}\mathbf u, \] where \(\mathbf u\) has nonnegative integral entries. In this decomposition, \(l_{\mathrm{geo}}\) retains the frozen coordinates \(\mathbf z\) of \(l\) so \(l_{\mathrm{cas}}\) has zero frozen coordinates. 

Recall that the entries of the first block \(\mathbf u_1\) correspond to the elementary formal geodesic functions. By construction, \([\mathbf u_1]_r\geq0\) for \(1\leq r\leq n+1\), with \([\mathbf u_1]_{r_0}=0\) for some \(r_0\). Define \[ \nu(l):= \min_{1\leq r\leq n+1} \bigl\{z_r+[\mathbf u_1]_r\bigr\}. \]

\begin{lem}\label{lem733} Let \(l\in\mathcal A_{|\Sigma_n|}(\mathbb Z^t)\) possess frozen degree \(\mathbf z=(z_1,\ldots,z_{n+1})\). Then \[ \min_r z_r \leq \nu(l) \leq \max_r z_r. \] 

In particular, \(\nu(l)\) takes only finitely many values on each fixed frozen-homogeneous component. \end{lem} 

\begin{proof} Since \([\mathbf u_1]_r\geq0\) for every \(r\), we have \[ \nu(l) = \min_r\bigl\{z_r+[\mathbf u_1]_r\bigr\} \geq \min_r z_r. \] 

On the other hand, choose \(r_0\) such that \([\mathbf u_1]_{r_0}=0\). Then \[ \nu(l) \leq z_{r_0}+[\mathbf u_1]_{r_0} = z_{r_0} \leq \max_r z_r. \]

Thus, the lemma follows. \end{proof}

Let \(l_{\mathrm{geo}}^Z := l_{\mathrm{geo}}-\nu(l)\zeta\) and \(\widetilde z_r:=z_r-\nu(l).\) Then \(\widetilde z_r+[\mathbf u_1]_r\geq0\) for every \(r\), with equality for at least one \(r\). Similar to Lemma~\ref{lem643}, we have \(\tau_1^t(l_{\mathrm{geo}}^Z) = l_{\mathrm{geo}}^Z\). We also distinguish two cases according to the minimum entry of the innermost block \(\mathbf u_m\) as we did in Proposition~\ref{prop645}.

\medskip \noindent \textit{Case 1: \(\min_r[\mathbf u_m]_r=0\).}

Since subtracting \(\nu(l)\zeta\) changes only the \(Z\)-frozen coordinates, \(l_{\mathrm{geo}}^Z\) is fixed by every $\tau_i^t$. By the equivariance of the theta basis, the function \(\mathscr G_l:=\theta_{l_{\mathrm{geo}}^Z} \) is Weyl-invariant. Casimir translation gives \(\theta_{l_{\mathrm{geo}}} =\mathcal K_{\zeta}^{\nu(l)}\mathscr G_l\) and \(\mathcal K_{\zeta}\mathcal K_1 =\left(\sqrt{\prod_{i=1}^{n+1}(k_i)}\right)^{-1}\) is Weyl-invariant. Therefore,

\begin{equation}\label{7.2}
\theta_l^{\mathbf W_n^\imath} = (\mathcal K_{\zeta}\mathcal K_1)^{\nu(l)} \mathscr G_l\left(\theta_{l_{\mathrm{cas}}}\mathcal K_1^{-\nu(l)}\right)^{\mathbf W_n^\imath}.
\end{equation}

\medskip \noindent \textit{Case 2: \(\min_r[\mathbf u_m]_r=1\).}

Consider the frozen extension along the innermost cycle introduced in the preceding section. Recall that \(\mathcal K_f=\prod_{r=1}^{m}\mathbf F_r=\theta_f\), where \(f\) is the tropical label of \(\mathcal K_f\). Since \(\min_r[\mathbf u_m]_r=1\), all tropical Weyl corrections of \(l_{\mathrm{geo}}^Z-f\) vanish (Lemma~\ref{lem643}). Hence \(\theta_{l_{\mathrm{geo}}^Z-f}\) is Weyl-invariant (Theorem~\ref{thm644}).

Let \(\Delta_m:=\sqrt{\prod_{i=1}^{m}\mathcal K_i}\), and define \( \mathscr G_l :=(\mathcal K_f\Delta_m) \theta_{l_{\mathrm{geo}}^Z-f}. \) Since \(\mathcal K_f\Delta_m\) is Weyl-invariant, so is \(\mathscr G_l\). Together with the Weyl invariance of
\(\mathcal K_{\zeta}\mathcal K_1\), this gives

\begin{equation}\label{7.3}
\theta_l^{\mathbf W_n^\imath} = (\mathcal K_{\zeta}\mathcal K_1)^{\nu(l)} \mathscr G_l\left( \theta_{l_{\mathrm{cas}}} \mathcal K_1^{-\nu(l)} \Delta_m^{-1} \right)^{\mathbf W_n^\imath}.
\end{equation}

The product \(\mathscr G_l\) has minimal multidegree in the geodesic degree monoid. Casimir translation gives \( \mathscr G_l = \Delta_m\theta_{l_{\mathrm{geo}}^Z}. \) Moreover, \(\operatorname{val}(\Delta_m^{-1})\) is the sum of the formal geodesic degrees corresponding to the last \(m\) columns of \(\mathbb A_{\deg}\); it can be shown by using parallelogram $\blacklozenge_{p,q}$ (Example~\ref{exmp635}). Thus, multiplication by \(\Delta_m\) replaces the innermost coefficient block \(\mathbf u_m\) by \(\mathbf u_m-\mathbf 1_m\). Since \(\min_r[\mathbf u_m]_r=1\), the resulting block remains nonnegative. Consequently, the mutable part of the unique minimal multidegree of \(\mathscr G_l\) belongs to the geodesic degree monoid.

\begin{lem}\label{lem734}
The unique minimal multidegree of \(\mathscr G_l\) coincides with that of a product of the elements \(B_i\), \((\sqrt{k_i})^{\pm1}\), and \(\mathbb A_{p,q}\) (see also Example~\ref{exmp735}). \end{lem}

\begin{proof}
Recall that \(\widetilde z_r=z_r-\nu(l)\) and \( \widetilde z_r+[\mathbf u_1]_r\geq0 \) for $1\leq r\leq n+1$. The relevant elements satisfy
\[
\deg_f(\sqrt{k_r})=-\mathbf e_r,
\qquad
\deg_f(B_rk_r^{-1})=\mathbf e_r.
\]
Moreover, \(B_rk_r^{-1}\) has zero mutable minimal multidegree, whereas \(\sqrt{k_r}\) has the same mutable minimal multidegree as the corresponding elementary formal geodesic function. More precisely, after disregarding the frozen coordinates, their multidegrees agree with those of \(\mathbb A_{n,n+1}\), \(\mathbb A_{1,n+1}\), and \(\mathbb A_{i-2,i-1}\) for \(r=1\), \(r=2\), and \(3\leq r\leq n+1\), respectively.

Suppose first that \(\widetilde z_r\geq0\). Then \((B_rk_r^{-1})^{\widetilde z_r}\) realizes the frozen degree \(\widetilde z_r\mathbf e_r\) without using any mutable geodesic multiplicity. If \(\widetilde z_r<0\), then \((\sqrt{k_r})^{-\widetilde z_r}\) realizes the same frozen degree and uses \(-\widetilde z_r\) copies of the corresponding elementary formal geodesic degree. The remaining formal geodesic multiplicity is \( [\mathbf u_1]_r-(-\widetilde z_r) = [\mathbf u_1]_r + \widetilde z_r\geq0.\)

Applying this construction for every \(r\), the frozen degree of \(\mathscr G_l\) is realized while leaving only nonnegative formal geodesic multiplicities. In Case~1, the remaining multiplicities are the entries of \(\mathbf u\). In Case~2, the innermost block is replaced by \(\mathbf u_m-\mathbf 1_m\), which is nonnegative because \(\min_r[\mathbf u_m]_r=1\). All remaining multiplicities can therefore be realized by a product of the formal geodesic functions \(\mathbb A_{p,q}\). \end{proof}

\begin{exmp}(Realizing the multidegree of $\mathscr G_l$ for $\Sigma_3$) \label{exmp735} 

Let $n=3$. Then the mutable part of $\Sigma_3$ is the $\mathcal A_4$-quiver. Write the geodesic multiplicity vector as \[ \mathbf u=(\mathbf u_1,\mathbf u_2)^T, \qquad \mathbf u_1=(u_{3,4},u_{1,4},u_{1,2},u_{2,3})^T, \qquad \mathbf u_2=(u_{2,4},u_{1,3})^T. \] 

Thus, the entries of $\mathbf u_1$ correspond to $\mathbb A_{3,4},\mathbb A_{1,4},\mathbb A_{1,2},\mathbb A_{2,3}$, while those of $\mathbf u_2$ correspond to $\mathbb A_{2,4},\mathbb A_{1,3}$. Consider \[ \mathbf z=(1,2,-1,3), \qquad \mathbf u_1=(0,2,3,5)^T, \qquad \mathbf u_2=(4,1)^T. \] 

We have \[ \nu(l) = \min_r\{z_r+[\mathbf u_1]_r\} = \min\{1,4,2,8\} = 1. \] 

Hence \[ \widetilde{\mathbf z} = \mathbf z-\nu(l)\mathbf1_4 = (0,1,-2,2), \qquad \widetilde{\mathbf z}+\mathbf u_1 = (0,3,1,7)^T. \] 

Thus, \(l_{\mathrm{geo}}^Z:=l_{\mathrm{geo}}-\zeta\) is invariant under \(\tau_1^*\). Since \(\min_r[\mathbf u_2]_r=1\), we are in Case~2; see \eqref{7.3}. The function \[ \mathscr G_l = (\mathcal K_f\Delta_2)\theta_{l_{\mathrm{geo}}^Z-f} = \Delta_2\theta_{l_{\mathrm{geo}}^Z} \] has the unique mutable minimal multidegree \[ \begin{aligned} \operatorname{mdeg}_{\mathrm{mut}}(\mathscr G_l) &=\mathbb A_{\deg}(0,2,3,5,3,0)^T \\ &=\left( -\frac52,-\frac52,-\frac52,-\frac{11}{2},-3,-5 \right)^T. \end{aligned} \]

We now realize this mutable multidegree together with the frozen degree $\widetilde{\mathbf z}$. The positive frozen degrees $\widetilde z_2=1$ and $\widetilde z_4=2$ are realized by \( B_2k_2^{-1} \text{ and } (B_4k_4^{-1})^2, \) respectively, without using any mutable geodesic multiplicity. The negative frozen degree $\widetilde z_3=-2$ is realized by $(\sqrt{k_3})^2$, which accounts for two copies of $\operatorname{val}(\mathbb A_{1,2})$. The remaining multiplicity vectors are therefore \[ \widehat{\mathbf u}_1=(0,2,1,5)^T, \qquad \widehat{\mathbf u}_2=(3,0)^T. \] 

It follows that \[ P_{\mathrm{geo}} := (B_2k_2^{-1})(\sqrt{k_3})^2(B_4k_4^{-1})^2 \mathbb A_{1,4}^{\,2}\mathbb A_{1,2} \mathbb A_{2,3}^{\,5}\mathbb A_{2,4}^{\,3} \] has frozen degree \( \begin{aligned} \deg_f(P_{\mathrm{geo}}) = \mathbf e_2-2\mathbf e_3+2\mathbf e_4 = \widetilde{\mathbf z} \end{aligned} \) and the same unique mutable minimal multidegree as $\mathscr G_l$.
\end{exmp}

\begin{thm}\label{thm736} Every theta orbit sum \(\theta_l^{\mathbf W_n^\imath}\) has a unique minimal multidegree. This multidegree is realized as the unique minimal multidegree of a square-root-free product of the elements \(B_i\), \(\sqrt{k_i}^{\,\pm1}\), and \(\mathbb A_{p,q}\) (See also Example~\ref{exmp738}). \end{thm} 

\begin{proof} 
In Case 1, Equation~\eqref{7.2} expresses \(\theta_l^{\mathbf W_n^\imath}\) as the product of \((\mathcal K_{\zeta}\mathcal K_1)^{\nu(l)}\), the Weyl-invariant factor $\mathscr{G}_l$, and the Casimir orbit sum \( \left(\theta_{l_{\mathrm{cas}}}\mathcal K_1^{-\nu(l)}\right)^{\mathbf W_n^\imath}. \) The latter has a unique minimal multidegree (Proposition~\ref{prop645}) and is generated by formal geodesic functions (Lemma~\ref{lem647}). The unique minimal multidegree of $\mathscr{G}_l$ is realized by a product of the elements \(B_i\), \(\sqrt{k_i}^{\,\pm1}\), and \(\mathbb A_{p,q}\) by Lemma~\ref{lem734}.

Moreover, we have \((\mathcal K_{\zeta}\mathcal K_1)^{\nu(l)} = \left(\prod_{i=1}^{n+1} \sqrt{k_i}\right)^{-\nu(l)}\). Thus, the unique minimal multidegree of \(\theta_l^{\mathbf W_n^\imath}\) is realized by a square-root-free product of the stated elements. Notice that even if \(\sqrt{k_i}^{\,\pm1}\) and \(\mathbb A_{p,q}\) may individually involve square roots of cluster monomials, their square-root monomial factors cancel in the resulting product. 

In Case 2, Equation~\eqref{7.3} expresses \(\theta_l^{\mathbf W_n^\imath}\) as the product of \((\mathcal K_{\zeta} \mathcal K_1)^{\nu(l)}\), the Weyl-invariant factor \( \mathscr{G}_l, \) and the half-integral Casimir orbit sum \( \left(\theta_{l_{\mathrm{cas}}}\mathcal K_1^{-\nu(l)} \Delta_m^{-1}\right)^{\mathbf W_n^\imath}. \) The latter is generated by formal geodesic functions by Lemma~\ref{lem648}. Moreover, Lemma~\ref{lem734} implies that the unique minimal multidegree of $\mathscr{G}_l$ is realized by a product of the elements \(B_i\), \(\sqrt{k_i}^{\,\pm1}\), and \(\mathbb A_{p,q}\). As in Case 1, the square-root monomial factors contributed by \(\sqrt{k_i}^{\,\pm1}\) cancel with those appearing in the corresponding formal geodesic functions \(\mathbb A_{p,q}\). Hence the resulting product is square-root-free. 

In either case, all factors under consideration have unique minimal multidegrees by an argument similar to that of Proposition~\ref{prop645}. Since the coordinatewise partial order is compatible with addition, their product has a unique minimal multidegree equal to that of \(\theta_l^{\mathbf W_n^\imath}\). The construction also preserves the frozen degree, so the realizing product may be chosen to have the same frozen degree as \(\theta_l^{\mathbf W_n^\imath}\). \end{proof}

\begin{prop}\label{prop737} Fix a frozen degree \(\mathbf z=(z_1,\ldots,z_{n+1})\), and set \(M(\mathbf z):=\max_{1\leq r\leq n+1}z_r\). For every theta orbit sum \(\theta_l^{\mathbf W_n^\imath}\) in this frozen-degree component, let \(d_{\mathrm{mut}}(l)\) denote the mutable part of its unique minimal multidegree. Then \[ \mathbf d_{\mathrm{mut}}(l) \leq \operatorname{val}\bigl(\mathcal K_1^{M(\mathbf z)}\bigr) \] with respect to the coordinatewise partial order. \end{prop} 

\begin{proof} In both factorizations in Theorem~\ref{thm736}, the factor \(\mathcal K_1^{\nu(l)}\) contributes \(\nu(l)\operatorname{val}(\mathcal K_1)\) to the mutable part of the unique minimal multidegree. The mutable minimal multidegrees of all remaining factors belong to the monoid generated by the vectors \(\operatorname{val}(\mathbb A_{p,q})\). In particular, Lemma~\ref{lem734} shows that the unique minimal multidegree of \(\mathscr G_l\) is realized by a product of formal geodesic functions. Since all these valuation vectors are coordinatewise nonpositive, we obtain \[ \mathbf d_{\mathrm{mut}}(l) \leq \nu(l)\operatorname{val}(\mathcal K_1) \leq M(\mathbf z)\operatorname{val}(\mathcal K_1) = \operatorname{val}\bigl(\mathcal K_1^{M(\mathbf z)}\bigr), \] where the second inequality follows from \(\nu(l)\leq M(\mathbf z)\). \end{proof}

\begin{exmp}[The minimal multidegree of a theta orbit sum for \(\Sigma_3\)]\label{exmp738} We continue with Example~\ref{exmp735}, in which \[ \mathbf z=(1,2,-1,3),\qquad \nu(l)=1,\qquad \mathbf u=(0,2,3,5,4,1)^T, \] and assume that \(\theta_{l_{\mathrm{cas}}}=\mathcal K_1^3\mathcal K_2\). Since this falls under Case~2, Equation~\eqref{7.3} yields \begin{equation}\label{7.4} \theta_l^{\mathbf W_3^\imath} = (\mathcal K_{\zeta}\mathcal K_1)\mathscr G_l \left( \theta_{l_{\mathrm{cas}}} \mathcal K_1^{-1}\Delta_2^{-1} \right)^{\mathbf W_3^\imath}. \end{equation} 

Each of the three factors in~\eqref{7.4} has a unique mutable minimal multidegree. Since \(\mathcal K_{\zeta}\) has zero mutable degree, the first factor contributes \(\operatorname{val}(\mathcal K_1)\). 

By Example~\ref{exmp735}, the mutable part of the unique minimal multidegree of \(\mathscr G_l\) is \[ \mathbf d_{\mathrm{geo}} := \mathbb A_{\deg}(0,2,3,5,3,0)^T = \left( -\frac52,-\frac52,-\frac52,-\frac{11}{2},-3,-5 \right)^T. \] 

Moreover, \(\theta_{l_{\mathrm{cas}}}\mathcal K_1^{-1}\Delta_2^{-1} =\mathcal K_1^{3/2}\mathcal K_2^{1/2}\). Thus, the unique multidegree of \(\left( \theta_{l_{\mathrm{cas}}} \mathcal K_1^{-1}\Delta_2^{-1} \right)^{\mathbf W_3^\imath}\) is \[ \mathbf d_{\mathrm{cas}} := -\frac32\operatorname{val}(\mathcal K_1) -\frac12\operatorname{val}(\mathcal K_2). \] 

Adding the three contributions, we obtain \begin{equation}\label{7.5} \begin{aligned} \mathbf d_{\text{mut}} &:= \text{the mutable part of the unique minimal multidegree of } \theta_l^{\mathbf W_3^\imath} \\ &= \operatorname{val}(\mathcal K_1) +\mathbf d_{\mathrm{geo}} +\mathbf d_{\mathrm{cas}} \\ &= \operatorname{val}(\mathcal K_1) +\mathbb A_{\deg}(0,2,3,5,3,0)^T -\frac32\operatorname{val}(\mathcal K_1) -\frac12\operatorname{val}(\mathcal K_2) \\ &= (-3,-3,-3,-6,-4,-6)^T. \end{aligned} \end{equation} 

Thus, although \(\mathbf d_{\mathrm{geo}}\) and \(\mathbf d_{\mathrm{cas}}\) are individually half-integral, their contributions combine to give the integral multidegree \( \mathbf d_{\text{mut}}\). We now construct a product that realizes both \(\mathbf d_{\text{mut}}\) and the frozen degree \(\mathbf z\). Example~\ref{exmp735} provides \[ P_{\mathrm{geo}} = (B_2k_2^{-1})(\sqrt{k_3})^2(B_4k_4^{-1})^2 \mathbb A_{1,4}^{\,2}\mathbb A_{1,2} \mathbb A_{2,3}^{\,5}\mathbb A_{2,4}^{\,3}, \] whose frozen degree is \(\widetilde{\mathbf z}=(0,1,-2,2)\) and whose mutable part of the unique minimal multidegree is \(\mathbf d_{\mathrm{geo}}\). 

The first factor in~\eqref{7.4} satisfies \[ \mathcal K_{\zeta}\mathcal K_1 = \prod_{r=1}^{4}\sqrt{k_r^{-1}}. \] 

It has frozen degree \(\mathbf 1_4=(1,1,1,1)\) and mutable degree \(\operatorname{val}(\mathcal K_1)\). For the contribution from \(\left( \theta_{l_{\mathrm{cas}}} \mathcal K_1^{-1}\Delta_2^{-1} \right)^{\mathbf W_3^\imath}\), a direct calculation gives \[ \mathbb B_{\deg}^{T}\mathbf d_{\mathrm{cas}} = (1,1,1,1,1,1)^T. \] 

Therefore, \(\mathbf d_{\mathrm{cas}}\) is realized as the unique minimal multidegree of the formal geodesic product \[ P_{\mathrm{cas}} := \mathbb A_{3,4}\mathbb A_{1,4}\mathbb A_{1,2} \mathbb A_{2,3}\mathbb A_{2,4}\mathbb A_{1,3}. \] 

Combining these three contributions, set \[ P_l := P_{\mathrm{geo}} \left(\prod_{r=1}^{4}\sqrt{k_r^{-1}}\right) P_{\mathrm{cas}}. \] 

Its frozen degree is \[ \deg_f(P_l) = \widetilde{\mathbf z}+\mathbf 1_4 = (0,1,-2,2)+(1,1,1,1) = (1,2,-1,3) = \mathbf z. \] 

Moreover, the mutable part of its unique minimal multidegree is \[ \begin{aligned} \mathbf d_{\mathrm{geo}} +\operatorname{val}(\mathcal K_1) +\mathbf d_{\mathrm{cas}} &= \mathbb A_{\deg}(0,2,3,5,3,0)^T +\operatorname{val}(\mathcal K_1) \\ &\quad -\frac32\operatorname{val}(\mathcal K_1) -\frac12\operatorname{val}(\mathcal K_2) \\ &= (-3,-3,-3,-6,-4,-6)^T = \mathbf d_{\text{mut}}. \end{aligned} \] 

Thus, \(P_l\) is a product of the elements \(B_i\), \((\sqrt{k_i})^{\pm1}\), and \(\mathbb A_{p,q}\) that simultaneously realizes the frozen degree and the mutable part of the unique minimal multidegree of \(\theta_l^{\mathbf W_3^\imath}\). \end{exmp}

\begin{thm}\label{thm739} The Weyl-invariant Poisson algebra \(\mathcal O(\mathcal X_{|\Sigma_n|})^{\mathbf W_n^\imath}\) is generated by square-root-free products of \(B_i\), \((\sqrt{k_i})^{\pm1}\), and the formal geodesic functions. \end{thm} 

\begin{proof} By Proposition~\ref{prop732}, it suffices to consider a nonzero Weyl-invariant function \(g\) of fixed frozen degree \(\mathbf z\). Expand \(g\) as a finite linear combination of theta orbit sums. Each theta orbit sum has a unique minimal multidegree, and distinct theta orbit sums have distinct minimal multidegrees. Hence, the minimal elements among the minimal multidegrees of the orbit sums appearing in \(g\) cannot cancel. The logic is similar to Proposition~\ref{prop6410}.

Let \(\mathbf d_1,\ldots,\mathbf d_s\) be the mutable minimal multidegrees of \(g\). By Theorem~\ref{thm736}, each \(\mathbf d_\alpha\) is the unique minimal multidegree of a square-root-free product \(G_\alpha\) of the elements \(B_i\), \((\sqrt{k_i})^{\pm1}\), and \(\mathbb A_{p,q}\), which may be chosen to have the same frozen degree as \(g\). Although the individual factors \((\sqrt{k_i})^{\pm1}\) and \(\mathbb A_{p,q}\) may involve square roots of cluster monomials, these square-root factors cancel in \(G_\alpha\). Since each $\mathbb A_{p,q}$ has the form $M^{1/2} L$ in every cluster by Theorem~\ref{thm416} and each $k_i$ is a Casimir monomial, this cancellation holds in every cluster. Hence, \(G_\alpha\in \mathcal O(\mathcal X_{|\Sigma_n|})^{\mathbf W_n^\imath}\). 

Subtracting suitable scalar multiples of the \(G_\alpha\) eliminates the current minimal terms while preserving Weyl invariance and the frozen degree. Repeating this procedure produces strictly increasing mutable minimal multidegrees. 

By Proposition~\ref{prop737}, all such multidegrees are bounded above by \(\operatorname{val}(\mathcal K_1^{M(\mathbf z)})\), where \(M(\mathbf z)=\max_r z_r\). Hence, they lie in the finite set \[ \bigcup_{\alpha=1}^{s} \left\{ \mathbf v: \mathbf d_\alpha\leq\mathbf v \leq \operatorname{val}\bigl(\mathcal K_1^{M(\mathbf z)}\bigr) \right\} \] inside the discrete mutable degree lattice. At each nonzero step, the minimal multidegrees strictly increase within this finite partially ordered set. Thus, by the same finite-poset argument used in the proof of Theorem~\ref{thm6412}, the elimination procedure terminates after finitely many steps. It follows that \(g\) belongs to the Poisson algebra generated by square-root-free products of \(B_i\), \((\sqrt{k_i})^{\pm1}\), and the formal geodesic functions. \end{proof}

\begin{rem}\label{rem7310}
    When the frozen degree is zero, we have $M(\mathbf z)=0$, and the upper bound $M(\mathbf z)\operatorname{val}(\mathcal K_1)$ reduces to the upper bound $0$ used in the proof of Theorem~\ref{thm6412}. Thus, Proposition~\ref{prop737} and Theorem~\ref{thm739} may be viewed as a frozen-degree refinement of the boundedness argument in the unfrozen case.
\end{rem}

Next, define the invariant reduction ideal by \[ \mathcal I^{\mathbf W_n^\imath} := \mathcal I \cap \mathcal O(\mathcal X_{|\Sigma_n|})^{\mathbf W_n^\imath}. \] 

Equivalently, \[ \mathcal I^{\mathbf W_n^\imath} = \left\langle M_n-1,\, k_i^{-1}+1 \ \middle|\ 1\leq i\leq n \right\rangle \cap \mathcal O(\mathcal X_{|\Sigma_n|})^{\mathbf W_n^\imath}. \] 

As observed above, this ideal also contains \(k_{n+1}^{-1}+1\).

\begin{thm}\label{thm7311} Let \(n\) be odd. The embedding induced by Song's cluster realization identifies the image of the classical \(\imath\)quantum group of type \(\mathrm{AI}_n\) with a quotient of the Weyl-invariant Poisson algebra: \[ \iota(U_n^\imath) \simeq \mathcal O(\mathcal X_{|\Sigma_n|})^{\mathbf W_n^\imath}/\mathcal I^{\mathbf W_n^\imath}. \] \end{thm} 

\begin{proof} By Theorem~\ref{thm739}, the Weyl-invariant Poisson algebra is generated by square-root-free products of \(B_i\), \((\sqrt{k_i})^{\pm1}\), and the formal geodesic functions. Moreover, with the cyclic indexing convention adopted above, one has \[ B_1\equiv\mathbb A_{n,n+1},\qquad B_2\equiv\mathbb A_{1,n+1},\qquad B_i\equiv\mathbb A_{i-2,i-1}\pmod{\mathcal I} \] for \(3\leq i\leq n+1\). Every non-elementary formal geodesic function can be obtained from the elementary ones by repeatedly applying the Bondal Poisson relations. 

Hence, modulo \(\mathcal I^{\mathbf W_n^\imath}\), all formal geodesic functions, as well as \(B_{n+1}\), belong to the Poisson algebra generated by \(B_1,\ldots,B_n\). Therefore, \(\mathcal O(\mathcal X_{|\Sigma_n|})^{\mathbf W_n^\imath}/\mathcal I^{\mathbf W_n^\imath}\) is generated as a Poisson algebra by the images of \(B_1,\ldots,B_n\). It follows that \[ \mathcal O(\mathcal X_{|\Sigma_n|})^{\mathbf W_n^\imath}/\mathcal I^{\mathbf W_n^\imath} \subseteq \iota(U_n^\imath). \] 

Conversely, Song's embedding induces the inclusion \[ \iota(U_n^\imath) \subseteq \mathcal O(\mathcal X_{|\Sigma_n|})^{\mathbf W_n^\imath}/\mathcal I^{\mathbf W_n^\imath}. \] 

The two inclusions establish the claimed Poisson algebra isomorphism. \end{proof} 

\begin{rem}\label{rem7312} Theorem~\ref{thm7311} may be viewed as the semiclassical \(\imath\)quantum-group counterpart of the quantum cluster isomorphism established in~\cite{3}. There, the quantum group \(\mathbf U_q(\mathfrak g)\) is realized as a quotient of the Weyl-invariant subalgebra of a Fock--Goncharov quantum cluster algebra.

Similarly, in the present type \(\mathrm{AI}_n\) setting, the classical \(\imath\)quantum group is realized as the quotient of a Weyl-invariant cluster Poisson algebra. Thus, Theorem~\ref{thm7311} provides a Poisson analogue of Shen's isomorphism in the \(\imath\)quantum-group setting. \end{rem} 

This remark suggests a quantum analogue. Namely, one may replace the classical \(\imath\)quantum group \(U_n^\imath\) and the Weyl-invariant cluster Poisson algebra by their quantum counterparts. This leads to the following conjecture.

\begin{conj}\label{conj7313} Let \(n\) be odd, and let \[ \iota_q\colon \widetilde U_n^\imath \longrightarrow {}_F\mathcal O_q(\mathcal X_{|\Sigma_n|}) \] be the quantum cluster embedding for the \(\imath\)quantum group of type \(\mathrm{AI}_n\), where \[ {}_F\mathcal O_q(\mathcal X_{|\Sigma_n|}) := \mathbb Q(q^{1/2}) \otimes_{\mathbb Z[q^{1/2},q^{-1/2}]} \mathcal O_q(\mathcal X_{|\Sigma_n|}). \] 

Then \(\iota_q\) induces an isomorphism \[ \iota_q(U_n^\imath) \simeq {}_F\mathcal O_q(\mathcal X_{|\Sigma_n|})^{\mathbf W_n^\imath} \big/\mathcal I_q^{\mathbf W_n^\imath}, \] where \( \mathcal I_q := \left\langle M_n-1,\ k_i^{-1}+1 \ \middle|\ 1\leq i\leq n \right\rangle \) is the two-sided ideal of the quantum Weyl-invariant algebra generated by the central elements and \(\mathcal I_q^{\mathbf W_n^\imath} = \mathcal I_q \cap \mathcal O_q(\mathcal X_{|\Sigma_n|})^{\mathbf W_n^\imath} \). \end{conj}

\begin{rem}\label{rem7314} The arguments developed in this and the preceding sections suggest both a conditional extension to the standard $\mathcal A_n$-quiver for odd $n$ and a quantum case. 

In the semiclassical setting, assuming the full Fock--Goncharov duality conjecture (which is equivalent to the existence of a linear basis satisfying Theorem~\ref{thm641}, but is known to fail in general \cite{35}), both the theta-orbit decomposition and the minimal multidegree elimination procedure extend to the case of odd~$n$. At present, the conjecture has been established only for the standard $\mathcal A_3$-quiver, namely the Markov quiver, by \cite{24}. This suggests that the description of the Weyl-invariant Poisson algebra obtained above should also extend to odd~$n$.

A similar picture should hold in the quantum setting whenever a quantum theta basis exists and the quantum formal geodesic functions are preserved by the quantum Weyl group action. For even $n$, the standard $\mathcal A_n$-quiver admits a reddening sequence, and hence the required quantum theta basis exists by \cite{32}. For odd $n$, the existence of such a basis is again conditional on the full Fock--Goncharov duality conjecture.\end{rem}

\section{Conclusion} In this paper, we introduced a birational Weyl group action on the \(\mathcal A_n\)-quiver and proved that this action preserves the formal geodesic functions. Using this birational Weyl group action, we showed that it suffices to consider a single irreducible component of the rank-condition locus when constructing the coisotropic reduction. Indeed, the Weyl group acts transitively on the irreducible components, and the corresponding coisotropic reductions are conjugate. This observation significantly reduces the computational complexity of the reduction problem. 

Furthermore, for even \(n\), we showed that the longest element of the Weyl group is realized by a cluster Donaldson--Thomas transformation. Consequently, \(\mathcal O(\mathcal X_{|\mathcal A_n|})\) admits a canonical theta basis. Using this basis, we identified the Weyl-invariant algebra \(\mathcal O(\mathcal X_{|\mathcal A_n|})^{\mathbf W_n}\) with \(\mathcal O(\mathcal A_n)\cap \mathcal O(\mathcal X_{|\mathcal A_n|})\). In contrast, the \(\mathcal A_n\)-quiver admits no reddening sequence when \(n\) is odd.

We also introduced a frozen extension of the \(\mathcal A_{n+1}\)-quiver, denoted by \(\Sigma_n\)-quiver. This extension encodes the cluster realization of the \(\imath\)quantum group of type \(\mathrm{AI}_n\). Using the birational Weyl-group action, we determined the exact image of this realization in the classical limit. More precisely, for odd \(n\), we proved that this image is Poisson isomorphic to a quotient of the corresponding Weyl-invariant algebra.

We conclude with several directions for future research. \begin{enumerate} 
\item \textit{Maximal green sequences.} The cluster Donaldson--Thomas transformation \(\omega_0\) is realized by a maximal green sequence for \(n=6,8,10\). Thus, we conjecture that \(\omega_0\) is realized by a maximal green sequence for every even \(n\). 

\item \textit{Generic degree for the \(\mathcal A_n\)-quiver.} We have proved that the function field generated by the initial cluster variables is generated by the formal geodesic functions and the Casimirs. A proof of this result, together with the determination of the precise number of generic solutions to the system of equations discussed in Remark~\ref{rem423}, will be presented in a forthcoming paper.

\item \textit{Quantization.} Since this paper primarily concerns the classical setting, quantizing the birational Weyl group action is a natural next step. We expect such a quantization to determine the exact image of the quantum cluster embedding of the \(\imath\)quantum group of type \(\mathrm{AI}_n\), as formulated in Conjecture~\ref{conj7313}. 

\item \textit{Coisotropic reduction.} We aim to complete the construction of the coisotropic reduction for general \(n\), as stated in Conjecture~\ref{conj524}. We have verified the existence of such a reduction for \(n=5,6,8,10\); the proof will appear in a forthcoming paper.

\item \textit{Full Fock--Goncharov duality and canonical bases.} The full Fock--Goncharov duality conjecture holds for the doubled \(\mathcal A_n\)-quiver and, when \(n\) is even, for the standard \(\mathcal A_n\)-quiver, due to the existence of reddening sequences. For odd \(n\), the conjecture remains open in general. However, it holds for \(n=3\) by~\cite{24}, since the \(\mathcal A_3\)-quiver is the Markov quiver. If the full Fock--Goncharov duality conjecture holds for the standard \(\mathcal A_n\)-quiver when \(n\) is odd, then we expect the theta-basis argument developed in this paper to extend to the odd case; see also Remark~\ref{rem7314}. \end{enumerate} 

\section{Acknowledgments} The author was supported by NSF grant DMS-2100791. I am deeply grateful to my advisor, Michael Shapiro, for his constant guidance and many fruitful discussions throughout this research. I also thank Leonid Chekhov for insightful conversations that greatly enriched my understanding of symplectic groupoids and cluster algebras. I am grateful to John Machacek and Eric Bucher for their assistance in identifying the reddening sequences for even \(n\). I would also like to thank Linhui Shen for pointing out that the projection from the doubled quiver does not preserve the Poisson bracket, contrary to an assumption made in an earlier arXiv version of this work. Finally, I thank Alexander Shapiro, David Whiting, Gus Schrader, and Scott Neville for invaluable conversations that contributed to the development of this work.

\nocite{*}
\footnotesize{\bibliographystyle{alpha}\bibliography{ref}}
\end{document}